\documentclass[twoside,12pt]{book}

\usepackage{thesis}
\usepackage[final]{graphicx}
\usepackage{amsmath,amsfonts,amssymb,amsthm}
\usepackage{verbatim}
\usepackage{color}
\usepackage{floatpag}
\usepackage{float}
\usepackage{multirow}
\usepackage[numbers,sort&compress,comma]{natbib}

\usepackage[pdftex,draft=false,breaklinks,colorlinks,linktocpage,citecolor=black,linkcolor=black,urlcolor=black]{hyperref}
\usepackage{emptypage} %no number of empty pages. %try to find later.
\usepackage[all]{hypcap} %can't remember. maybe for hyperlinking ref.
\usepackage{scalefnt}
\usepackage{lscape}  %AAA: adding this for a large matrix in the appendix.
\usepackage{subfigure}
\usepackage{amsmath} % assumes amsmath package installed
\usepackage{amssymb}  % assumes amsmath package installed
\usepackage{amsthm}
\usepackage{color}

\theoremstyle{plain}  % Bold name, italics font
\newtheorem{theorem}{Theorem}[chapter]
\newtheorem{lemma}[theorem]{Lemma}
\newtheorem{proposition}[theorem]{Proposition}
\newtheorem{corollary}[theorem]{Corollary}
\newtheorem{definition}[theorem]{Definition}

\theoremstyle{definition}

\theoremstyle{remark} % italics name, roman font
\newtheorem{example}{Example}[section]
\newtheorem{remark}{Remark}[section]

\floatpagestyle{empty}

\newcommand{\biglet}[2]{\begingroup\def\par{
    \endgraf\endgroup\lineskiplimit=0pt}
  \setbox2=\hbox{\uppercase{\sffamily #2}
    }\newdimen\tmpht \tmpht \ht2 \advance\tmpht by
  \baselineskip\font\hhuge=cmssbx10 at \tmpht \setbox1=\hbox{{\hhuge #1}}
  \count7=\tmpht \count8=\ht1\divide\count8 by 1000 \divide\count7
  by\count8 \tmpht=.001\tmpht\multiply\tmpht by \count7\font\hhuge=cmssbx10
  at \tmpht \setbox1=\hbox{{\hhuge #1}}
  \noindent \hangindent1.05\wd1
    \hangafter=-2 {\hskip-\hangindent \lower1\ht1\hbox{\raise1.0\ht2\copy1}%
    \kern-0\wd1}\copy2\lineskiplimit=-1000pt}

\floatstyle{boxed}
\newfloat{algorithm}{tbh}{loa}
\floatname{algorithm}{Algorithm}

\newcommand{\shortv}[1]{}

\renewcommand\bibname{References}

\begin{document}
\raggedbottom

%%%%%%%%%%%%%%%
%% Front Matter
%%%%%%%%%%%%%%%

%%
%% MIT Doctoral Thesis - Title Page
%%

\thispagestyle{empty}

\begin{centering}
\rule{\textwidth}{.05in}\\ \vspace{.10in}

\Large \textbf{\textsf{Algebraic Relaxations and Hardness Results \\in Polynomial Optimization and Lyapunov Analysis}} \\
\normalsize

\vspace{.1in}
by \\

\vspace{.1in}
Amir Ali Ahmadi \\

\vspace{.1in} B.S., Electrical Engineering,
University of Maryland, 2006 \\

\vspace{.0in} B.S., Mathematics,
University of Maryland, 2006 \\

\vspace{.0in} S.M., Electrical Engineering and Computer Science,
MIT, 2008 \\

\vspace{.1in}
\rule{\textwidth}{.05in}\\

\vspace{.2in}

Submitted to the Department of Electrical Engineering and Computer Science \\
in partial fulfillment of the requirements for the degree of \\ Doctor of Philosophy
in Electrical Engineering and Computer Science \\
at the Massachusetts Institute of Technology

\vspace{.1in} September 2011

\vspace{.15in} \copyright\ 2011
Massachusetts Institute of Technology.  All rights reserved. \\

\end{centering}

\vspace{.3in}

\noindent
Signature of Author: \hrulefill{} \\
\vspace{-.25in}
{\flushright Department of Electrical Engineering and Computer Science \\
September $2$, $2011$ \\
}

\vspace{.15in}

\noindent
Certified by: \hrulefill{} \\
\vspace{-.25in}
{\flushright Pablo A. Parrilo \\
Professor of Electrical Engineering and Computer Science \\
%%Finmeccanica Career Development Professor \\
Thesis Supervisor \\
}

\vspace{.15in}

\noindent
Accepted by: \hrulefill{} \\
\vspace{-.25in}
{\flushright Leslie A. Kolodziejski \\
Professor of Electrical Engineering \\
Chair, Committee for Graduate Students \\
}

\cleardoublepage

%\phantomsection
%\addcontentsline{toc}{chapter}{Abstract}
\thispagestyle{empty}

\begin{centering}
\rule{\textwidth}{.05in}\\ \vspace{.10in}

\Large \textbf{\textsf{Algebraic Relaxations and Hardness Results \\in Polynomial Optimization and Lyapunov Analysis}} \\
\normalsize

\vspace{.1in}
by \\

\vspace{.1in}
Amir Ali Ahmadi \\

%\vspace{.1in} B.S., Electrical \& Computer Engineering,
%Cornell University, $2005$ \\
%
%\vspace{.0in} S.M., Electrical Engineering \& Computer Science,
%Massachusetts Institute of Technology, 2007 \\

\vspace{.1in}
\rule{\textwidth}{.05in}\\

\vspace{.2in}

Submitted to the Department of Electrical Engineering \\ and Computer Science on September $2$, $2011$ in partial fulfillment \\ of the requirements for the degree of Doctor of Philosophy \\
\end{centering}

\section*{Abstract}
The contributions of the first half of this thesis are on the
computational and algebraic aspects of convexity in polynomial
optimization. We show that unless P=NP, there exists no polynomial
time (or even pseudo-polynomial time) algorithm that can decide
whether a multivariate polynomial of degree four (or higher even
degree) is globally convex. This solves a problem that has been
open since 1992 when N. Z. Shor asked for the complexity of
deciding convexity for quartic polynomials. We also prove that
deciding strict convexity, strong convexity, quasiconvexity, and
pseudoconvexity of polynomials of even degree four or higher is
strongly NP-hard. By contrast, we show that quasiconvexity and
pseudoconvexity of odd degree polynomials can be decided in
polynomial time.

We then turn our attention to sos-convexity---an algebraic sum of
squares (sos) based sufficient condition for polynomial convexity
that can be efficiently checked with semidefinite programming. We
show that three natural formulations for sos-convexity derived
from relaxations on the definition of convexity, its first order
characterization, and its second order characterization are
equivalent. We present the first example of a convex polynomial
that is not sos-convex. Our main result then is to prove that the
cones of convex and sos-convex polynomials (resp. forms) in $n$
variables and of degree $d$ coincide if and only if $n=1$ or $d=2$
or $(n,d)=(2,4)$ (resp. $n=2$ or $d=2$ or $(n,d)=(3,4)$). Although
for disparate reasons, the remarkable outcome is that convex
polynomials (resp. forms) are sos-convex exactly in cases where
nonnegative polynomials (resp. forms) are sums of squares, as
characterized by Hilbert in 1888.

The contributions of the second half of this thesis are on the
development and analysis of computational techniques for
certifying stability of uncertain and nonlinear dynamical systems.
We show that deciding asymptotic stability of homogeneous cubic
polynomial vector fields is strongly NP-hard. We settle some of
the converse questions on existence of polynomial and sum of
squares Lyapunov functions. We present a globally asymptotically
stable polynomial vector field with no polynomial Lyapunov
function. We show via an explicit counterexample that if the
degree of the polynomial Lyapunov function is fixed, then sos
programming can fail to find a valid Lyapunov function even though
one exists. By contrast, we show that if the degree is allowed to
increase, then existence of a polynomial Lyapunov function for a
planar or a homogeneous polynomial vector field implies existence
of a polynomial Lyapunov function that can be found with sos
programming. We extend this result to develop a converse sos
Lyapunov theorem for robust stability of switched linear systems.

In our final chapter, we introduce the framework of path-complete
graph Lyapunov functions for approximation of the joint spectral
radius. The approach is based on the analysis of the underlying
switched system via inequalities imposed between multiple Lyapunov
functions associated to a labeled directed graph. Inspired by
concepts in automata theory and symbolic dynamics, we define a
class of graphs called path-complete graphs, and show that any
such graph gives rise to a method for proving stability of
switched systems. The semidefinite programs arising from this
technique include as special case many of the existing methods
such as common quadratic, common sum of squares, and
maximum/minimum-of-quadratics Lyapunov functions. We prove
approximation guarantees for analysis via several families of
path-complete graphs and a constructive converse Lyapunov theorem
for maximum/minimum-of-quadratics Lyapunov functions.

\begin{flushleft}
Thesis Supervisor: Pablo A. Parrilo

Title: Professor of Electrical Engineering and Computer Science

%%Finmeccanica Career Development Professor
\end{flushleft}

\cleardoublepage

\thispagestyle{empty}
\vspace*{\fill}
\begin{center}
\textit{{\Large \ \     To my parents, Maryam and Hamid Reza}}
\newline
\newline
\includegraphics[height=1.028cm]{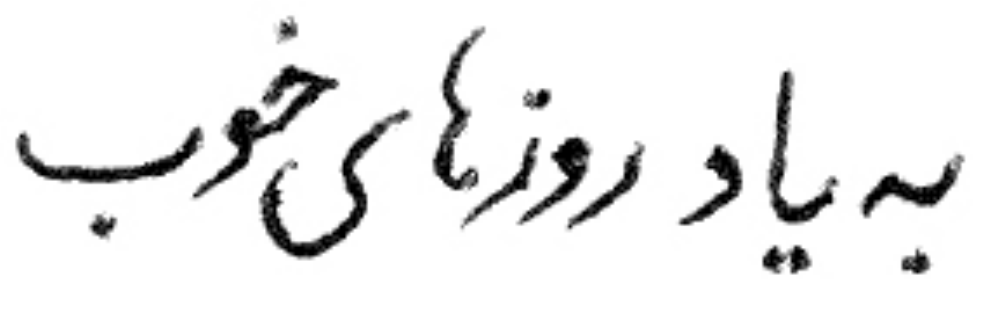}
\end{center}
\vspace*{\fill}  \cleardoublepage

\lhead[\fancyplain{}{\bf\thepage}]{\fancyplain{}{\sfeight\rightmark}}
\rhead[\fancyplain{}{\sfeight\leftmark}]{\fancyplain{}{\bf\thepage}}

\setcounter{tocdepth}{3}
\tableofcontents \cleardoublepage

%\phantomsection
%\addcontentsline{toc}{chapter}{List of Figures}
%\listoffigures
%\cleardoublepage

\phantomsection
\addcontentsline{toc}{chapter}{Acknowledgments}
\chapter*{Acknowledgements}
\markboth{Acknowledgements}{Acknowledgements}

The opportunities that have been available to me as a graduate
student at MIT have been endless, but without a doubt, the
greatest opportunity of all has been the chance to work with my
advisor Pablo Parrilo. What inspires me most about Pablo---aside
from his well-known traits like superior intelligence, humility,
and professional attitude---is his never-ending passion for
understanding things deeply. The joy that Pablo takes in
interpreting a mathematical result from all different angles,
placing it in the right level of abstraction, and simplifying it
to the point that it cannot be made simpler is a virtue that I
hope to take with me throughout my career. As Feynman once said,
``The truth always turns out to be simpler than you thought.'' On
several occasions in this thesis, Pablo's insights, or at times
simply his questions, have made me realize this fact, and for that
I am very grateful.

I would also like to thank Pablo for creating the perfect
environment for me to pursue my research ideas. Not once did he
ask me to work on a problem that I didn't choose to work on
myself, not once did he inquire about a research result before I
felt ready to present my progress, and not once did I ever have to
worry about funding for my research or for going to various
conferences. Pablo's approach was to meet with me regularly during
my Master's studies, but to be more hands-off throughout my Ph.D.
years. This worked out perfectly. I recall as a child, my father
taught me how to ride a bicycle by running alongside my bike and
holding on to its back, but then gradually letting go of his hands
(without me realizing) so I could ride on my own. I feel that
Pablo has very much done the same thing in the past few years in
placing me on the path to becoming a great researcher. I will be
grateful to him for as long as I continue on this ride.

I couldn't have asked for better thesis committee members than
Vincent Blondel and John Tsitsiklis. Among their many other
accomplishments, Vincent and John are two of the pioneering
figures in complexity theory in control and optimization, a
subject that as parts of this thesis reflect has become of much
interest to me over the last couple of years. From discussions
about complexity and the joint spectral radius to conversations
about my business idea for the MIT \$100K Entrepreneurship
Competition, both Vincent and John have always been generous with
their time and advice. In the case of Vincent, I was very
fortunate that my last year at LIDS coincided with the year that
he joined us as a visiting faculty member from Belgium. Some of
the most creative talks I have ever attended have been given by
Vincent. (A memorable example is in his LIDS seminar talk on
\emph{privacy in social networks}, where he fooled me into
believing that his wife, Gilbert Strang, and Mark Zuckerberg sent
him ``private'' messages which we saw popping up on the screen as
he was speaking!) I am thankful to Vincent for sincerely caring
about my thesis work, either in the form of a call from Belgium
providing encouragement the day prior to my defense, or by
deriving valuable results related to this thesis even after the
completion of my defense.

%I am thankful to Vincent for sincerely caring about my thesis
%work, whether in the form of calling me from Belgium the day
%before my defense to provide encouragement and support, or by
%deriving valuable results related to this thesis even after my
%defense was over.

My acquaintance with John goes back to my first year at MIT. I
remember walking out of the first lecture of his probability
course telling myself that I should attend every class that is
taught by this professor. Surely, I have done that. John has an
amazing ability to make every concept look simple and intuitive. I
am indebted to him for everything he has taught me both in and out
of the classroom. I would also like to thank him for his
invaluable contributions to a joint work that led to
Chapter~\ref{chap:nphard.convexity} of this thesis.

My gratitude extends to my other coauthors: Gerg Blekherman,
Rapha\"{e}l Jungers, Miroslav Krstic, Alex Olshevsly, and Mardavij
Roozbehani. I have learned a great deal from interactions with all
of them. I am thankful to Greg for settling two (out of two!)
mathematical bets that I made with Markus Schweighofer at a
meeting in 2009. Had he not done that, it is not clear how long it
would have taken me to get over my obsession with those problems
and move on to the problems addressed in this thesis. (I lost one
of the two bets when Greg proved that polynomials with a certain
special property exist~\cite{Blekherman_convex_not_sos};
interestingly, no explicit examples are known to this day, see
Section~\ref{sec:concluding.remarks}.) Rapha\"{e}l, Alex, and
Mardavij have all been great friends throughout my years at LIDS.
I particularly want to express my gratitude to Mardavij for always
being there when I needed someone to talk to, and to Rapha\"{e}l
for his nice gesture in giving me a copy of his
book~\cite{Raphael_Book} as a gift in our first meeting, which he
kindly signed with the note: ``Let's collaborate!'' I am grateful
to Miroslav for hosting me (together with Tara Javidi) at the
Cymer Center for Control Systems and Dynamics at UCSD, which led
to immediate and fruitful collaborations between us.

This thesis has benefited from my interactions with several
distinguished researchers outside of MIT. Among them I would like
to mention Amitabh Basu, Stephen Boyd, Etienne de Klerk, Jes\'{u}s
De Loera, Bill Helton, Monique Laurent, Jiawang Nie, Bruce
Reznick, Claus Scheiderer, Eduardo Sontag, Bernd Sturmfels, and
Andr\'{e} Tits. I had the good fortune to TA the Convex
Optimization class that Stephen co-taught with Pablo while
visiting MIT from Stanford. I am grateful to Amitabh and Jes\'{u}s
for hosting me at the Mathematics Dept. of UC Davis, and to
Eduardo for having me at the SontagFest in DIMACS. These have all
been memorable and unique experiences for me. I wish to also thank
Nima Moshtagh for the internship opportunity at Scientific
Systems.

I am deeply grateful to the faculty at LIDS, especially Sasha
Megretski, Asu Ozdaglar, Munther Dahleh, and Emilio Frazzoli for
making LIDS a great learning environment for me. I have also
learned a great deal from classes offered outside of LIDS, for
example from the beautiful lectures of Michael Sipser on Theory of
Computation, or from the exciting course of Constantinos
Daskalakis on Algorithmic Game Theory. Special thanks are also due
to members of the LIDS staff. I am grateful to Lisa Gaumond for
her friendly smile and for making sure that I always got
reimbursed on time, to Jennifer Donovan for making organizational
tasks easy, and to Brian Jones for going out of his way to fix my
computer twice when I was on the verge of losing all my files.
Janet Fischer from the EECS graduate office and my academic
advisor, Vladimir Stojanovic, are two other people that have
always helped me stay on track.

This is also a good opportunity to thank my mentors at the
University of Maryland, where I spent four amazing years as an
undergraduate student. Thomas Antonsen, Thomas Murphy, Edward Ott,
Reza Salem, Andr\'{e} Tits, and (coach!) James Yorke helped set
the foundation on which I am able to do research today.

Some of the results of this thesis would not have been possible
without the use of software packages such as YALMIP~\cite{yalmip},
SOSTOOLS~\cite{sostools}, and SeDuMi~\cite{sedumi}. I am deeply
grateful to the people who wrote these pieces of software and made
it freely available. Special thanks go to Johan L\"{o}fberg for
his patience in answering my questions about YALMIP.

My best memories at MIT are from the moments shared with great
friends, such as, Amir Khandani, Emmanuel Abbe, Marco Pavone,
Parikshit Shah, Noah Stein, Georgios Kotsalis, Borjan Gagoski,
Yola Katsargyri, Mitra Osqui, Ermin Wei, Hoda Eydgahi, Ali
ParandehGheibi, Sertac Karaman, Michael Rinehart, Mark Tobenkin,
John Enright, Mesrob Ohannessian, Rose Faghih, Ali Faghih, Sidhant
Misra, Aliaa Atwi, Venkat Chandrasekaran, Ozan Candogan, James
Saunderson, Dan Iancu, Christian Ebenbauer, Paul Njoroge, Alireza
Tahbaz-Salehi, Kostas Bimpikis, Ilan Lobel, Stavros Valavani,
Spyros Zoumpoulis, and Kimon Drakopoulos. I particularly want to
thank my longtime officemates, Pari and Noah. Admittedly, I didn't
spend enough time in the office, but whenever I did, I had a great
time in their company. Together with Amir and Emmanuel, I had a
lot of fun exploring the night life of Boston during my first two
years. Some of the memories there are unforgettable. I am thankful
to Marco for making everyday life at MIT more enjoyable with his
sense of humor. His presence was greatly missed during my last
year when NASA decided to threaten the security of all humans on
Earth and aliens in space by making the disastrous mistake of
hiring him as a Research Technologist. I also owe gratitude to
Marco's lovely wife-to-be, Manuela, for her sincere friendship.
Unfortunately, just like NASA, Manuela is a victim of Marco's
deception.

Aside from friends at MIT, my childhood friends who live in Iran
and Washington DC have made a real effort to keep our friendships
close by paying frequent visits to Boston, and for that I am
forever grateful. I would also like to thank the coaches of the
MIT tennis team, Dave Hagymas, Spritely Roche, and Charlie Maher,
for always welcoming me on practice sessions, which allowed me to
continue to pursue my childhood passion at MIT and maintain a
balanced graduate life.

The work in this thesis was partially supported by the NSF Focused
Research Group Grant on Semidefinite Optimization and Convex
Algebraic Geometry DMS-0757207, and by AFOSR MURI subaward
07688-1.

\vspace{15mm}

My heart is full of gratitude for my parents, Maryam and Hamid
Reza, my sister, Shirin, and my brother-in-law, Karim, who have
always filled my life with unconditional love and support. I am so
sorry and embarrassed for all the times I have been too ``busy''
to return the love, and so grateful to you for never expecting me
to do so. ``Even after all this time the sun never says to the
earth, `you owe me.' Look what happens with a love like that. It
lights the whole sky.''

I finally want to thank my girlfriend and best friend, Margarita,
who has the heart of an angel and who has changed my life in so
many ways since the moment we met. As it turns out, my coming to
Boston was not to get a Ph.D. degree, but to meet you. ``Who could
be so lucky? Who comes to a lake for water and sees the reflection
of moon.''

 \cleardoublepage

%\phantomsection

%\phantomsection
%\addcontentsline{toc}{chapter}{List of Algorithms}
%\listof{algorithm}{List of Algorithms}
%\cleardoublepage

%\phantomsection
%\addcontentsline{toc}{chapter}{List of Tables}
%\listoftables
%\cleardoublepage

\lhead[\fancyplain{}{\bf\thepage}]{\fancyplain{}{\sfeight\rightmark}}
\rhead[\fancyplain{}{\sfeight\leftmark}]{\fancyplain{}{\bf\thepage}}
%%%%%%%%%%%%%%%%
%% Document Body
%%%%%%%%%%%%%%%%

% \renewcommand{\thepage}{\arabic{page}}
% \setcounter{page}{1}
\chapter{Introduction}\label{chap:intro}

With the advent of modern computers in the last century and the
rapid increase in our computing power ever since, more and more
areas of science and engineering are being viewed from a
computational and algorithmic perspective---the field of
optimization and control is no exception. Indeed, what we often
regard nowadays as a satisfactory solution to a problem in this
field---may it be the optimal allocation of resources in a power
network or the planning of paths of minimum fuel consumption for a
group of satellites---is an \emph{efficient algorithm} that when
fed with an instance of the problem as input, returns in a
reasonable amount of time an output that is guaranteed to be
optimal or near optimal.
%a computational lens and explored from an algorithmic perspective

Fundamental concepts from theory of computation, such as the
notions of a Turing machine, decidability, polynomial time
solvability, and the theory of NP-completeness, have allowed us to
make precise what it means to have an (efficient) algorithm for a
problem and much more remarkably to even be able to prove that for
certain problems such algorithms do not exist. The idea of
establishing ``hardness results'' to provide rigorous explanations
for why progress on some problems tends to be relatively
unsuccessful is commonly used today across many disciplines and
rightly so. Indeed, when a problem is resisting all attempts for
an (efficient) algorithm, little is more valuable to an
unsatisfied algorithm designer than the ability to back up the
statement ``I cannot do it'' with the claim that ``it cannot be
done''.

Over the years, the line between what can or cannot be efficiently
computed has shown to be a thin one. There are many examples in
optimization and control where complexity results reveal that two
problems that on the surface appear quite similar have very
different structural properties. Consider for example the problem
of deciding given a symmetric matrix $Q$, whether $x^TQx$ is
nonnegative for all $x\in\mathbb{R}^n$, and contrast this to the
closely related problem of deciding whether $x^TQx$ is nonnegative
for all $x$'s in $\mathbb{R}^n$ that are elementwise nonnegative.
The first problem, which is at the core of semidefinite
programming, can be answered in polynomial time (in fact in
$O(n^3)$), whereas the second problem, which forms the basis of
copositive programming, is NP-hard and can easily encode many hard
combinatorial problems~\cite{nonnegativity_NP_hard}. Similar
scenarios arise in control theory. An interesting example is the
contrast between the problems of deciding stability of interval
polynomials and interval matrices.
%
%To mention a similar scenario from control theory, consider the
%problem of deciding stability of interval polynomials in contrast
%to that of deciding stability of interval matrices.
%
%Another example that demonstrates a similar interesting contrast
%is the problem of deciding stability of interval polynomials or
%interval matrices in robust control.
%
%
%A similar example from the field
%of robust control is the problem of deciding stability of interval
%polynomials or interval matrices.
If we are given a single univariate polynomial of degree $n$ or a
single $n\times n$ matrix, then standard classical results enable
us to decide in polynomial time whether the polynomial or the
matrix is (strictly) stable, i.e, has all of its roots (resp.
eigenvalues) in the open left half complex plane. Suppose now that
we are given lower and upper bounds on the coefficients of the
polynomial or on the entries of the matrix and we are asked to
decide whether all polynomials or matrices in this interval family
are stable. Can the answer still be given in polynomial time? For
the case of interval polynomials, Kharitonov famously
demonstrated~\cite{Kharitonov_interval_poly} that it can:
stability of an interval polynomial can be decided by checking
whether \emph{four} polynomials obtained from the family via some
simple rules are stable. One may naturally speculate whether such
a wonderful result can also be established for interval matrices,
but alas, NP-hardness
results~\cite{Nemirovskii_interval_matrix_NPhard} reveal that
unless P=NP, this cannot happen.

Aside from ending the quest for exact efficient algorithms, an
NP-hardness result also serves as an insightful bridge between
different areas of mathematics. Indeed, when we give a reduction
from an NP-hard problem to a new problem of possibly different
nature, it becomes apparent that the computational difficulties
associated with the first problem are intrinsic also to the new
problem. Conversely, any algorithm that has been previously
developed for the new problem can now readily be applied also to
the first problem. This concept is usually particularly
interesting when one problem is in the domain of discrete
mathematics and the other in the continuous domain, as will be the
case for problems considered in this thesis. For example, we will
give a reduction from the canonical NP-complete problem of 3SAT to
the problem of deciding stability of a certain class of
differential equations. As a byproduct of the reduction, it will
follow that a certificate of unsatisfiability of instances of 3SAT
can always be given in form of a Lyapunov function.

%, and this, at least from at a theoretical level, makes a
%connection between ideas from two different fields.

%With a viewpoint that emphasizes efficient algorithms on one hand
%and a considerable number of negative results that point to
%nonexistence of such algorithms for general problems in the other,

%
%The practical message of hardness results in optimization is
%obviously that,

In general, hardness results in optimization come with a clear
practical implication: as an algorithm designer, we either have to
give up optimality and be content with finding suboptimal
solutions, or we have to work with a subclass of problems that
have more tractable attributes. In view of this, it becomes
exceedingly relevant to identify structural properties of
optimization problems that allow for tractability of finding
optimal solutions.

One such structural property, which by and large is the most
fundamental one that we know of, is \emph{convexity}. As a
geometric property, convexity comes with many attractive
consequences. For instance, every local minimum of a convex
problem is also a global minimum. Or for example, if a point does
not belong to a convex set, this nonmembership can be certified
through a separating hyperplane. Due in part to such special
attributes, convex problems generally allow for efficient
algorithms for solving them. Among other approaches, a powerful
theory of interior-point polynomial time methods for convex
optimization was developed in~\cite{NN}. At least when the
underlying convex cone has an efficiently computable so-called
``barrier function'', these algorithms are efficient both in
theory and in practice.

%
%As for algorithms for solving convex problems, a general theory of
%interior-point polynomial time methods for convex optimization was
%developed in the early 1990s~\cite{NN}. At least when the
%underlying convex cone has an efficiently computable ``barrier
%function'', these algorithms are extremely efficient both in
%theory and in practice.

%\footnote{It is important to keep in mind that in the absence of a
%nice barrier function, there are convex problems that are
%intractable; an example that we have already mentioned is
%optimization over the copositive cone.}

Extensive and greatly successful research in the applications of
convex optimization over the last couple of decades has shown that
surprisingly many problems of practical importance can be cast as
convex optimization problems. Moreover, we have a fair number of
rules based on the calculus of convex functions that allow us to
design---whenever we have the freedom to do so---problems that are
by construction convex. Nevertheless, in order to be able to
exploit the potential of convexity in optimization in full, a very
basic question is to understand whether we are even able to
\emph{recognize} the presence of convexity in optimization
problems. In other words, can we have an efficient algorithm that
tests whether a given optimization problem is convex?

%Indeed, imagine a scenario where we have set up an optimization
%model for a real world problem and are now wondering whether the
%optimization problem is convex, perhaps to be able to argue that
%if we find a point where the gradient of the objective function
%vanishes, then that point must be a global optimum. Is there an
%efficient algorithm for us to test whether a given optimization
%problem is convex?

We will show in this thesis---answering a longstanding question of
N.Z. Shor---that unfortunately even for the simplest classes of
optimization problems where the objective function and the
defining functions of the feasible set are given by polynomials of
modest degree, the question of determining convexity is NP-hard.
We also show that the same intractability result holds for
essentially any well-known variant of convexity (generalized
convexity). These results suggest that as significant as convexity
may be in optimization, we may not be able to in general guarantee
its presence before we can enjoy its consequences.

Of course, NP-hardness of a problem does not stop us from studying
it, but on the contrary stresses the need for finding good
approximation algorithms that can deal with a large number of
instances efficiently. Towards this goal, we will devote part of
this thesis to a study of convexity from an \emph{algebraic}
viewpoint. We will argue that in many cases, a notion known as
\emph{sos-convexity}, which is an efficiently checkable algebraic
counterpart of convexity, can be a viable substitute for convexity
of polynomials. Aside from its computational implications,
sos-convexity has recently received much attention in the area of
convex algebraic
geometry~\cite{Blekherman_convex_not_sos},\cite{Monique_Etienne_Convex},\cite{Helton_Nie_SDP_repres_2},\cite{Lasserre_Jensen_inequality},\cite{Lasserre_Convex_Positive},\cite{Lasserre_set_convexity},
mainly due to its role in connecting the geometric and algebraic
aspects of convexity. In particular, the name ``sos-convexity''
comes from the work of Helton and Nie on semidefinite
representability of convex sets~\cite{Helton_Nie_SDP_repres_2}.

The basic idea behind sos-convexity is nothing more than a simple
extension of the concept of representation of nonnegative
polynomials as sums of squares. To demonstrate this idea on a
concrete example, suppose we are given the polynomial
\begin{equation}\label{eq:example.of.poly}
\begin{array}{lll}
p(x)&=&x_1^4-6x_1^3x_2+2x_1^3x_3+6x_1^2x_3^2+9x_1^2x_2^2-6x_1^2x_2x_3-14x_1x_2x_3^2+4x_1x_3^3
\\ \ &\ &+5x_3^4-7x_2^2x_3^2+16x_2^4,
\end{array}
\end{equation}
and we are asked to decide whether it is nonnegative, i.e, whether
$p(x)\geq 0$ for all $x\mathrel{\mathop:}=(x_1,x_2,x_3)$ in
$\mathbb{R}^3$. This may seem like a daunting task (and indeed it
is as deciding nonnegativity of quartic polynomials is also
NP-hard), but suppose that we could ``somehow'' come up with a
decomposition of the polynomial a sum of squares (sos):
\begin{equation}
p(x)=(x_1^2-3x_1x_2+x_1x_3+2x_3^2)^2+(x_1x_3-x_2x_3)^2+(4x_2^2-x_3^2)^2.
\end{equation}
Then, we have at our hands an \emph{explicit certificate} of
nonnegativity of $p(x)$, which can be easily checked (simply by
multiplying the terms out).

It turns out (see e.g.~\cite{PhD:Parrilo},~\cite{sdprelax}) that
because of several interesting connections between real algebra
and convex optimization discovered in recent years and quite
well-known by now, the question of existence of an sos
decomposition can be cast as a semidefinite program, which can be
solved efficiently e.g. by interior point methods. As we will see
more formally later, the notion of sos-convexity is based on an
appropriately defined sum of squares decomposition of the Hessian
matrix of a polynomial and hence it can also be checked
efficiently with semidefinite programming. Just like sum of
squares decomposition is a sufficient condition for polynomial
nonnegativity, sos-convexity is a sufficient condition for
polynomial convexity.

An important question that remains here is the obvious one: when
do nonnegative polynomials admit a decomposition as a sum of
squares? The answer to this question comes from a classical result
of Hilbert. In his seminal 1888 paper~\cite{Hilbert_1888}, Hilbert
gave a complete characterization of the degrees and dimensions in
which all nonnegative polynomials can be written as sums of
squares. In particular, he proved that there exist nonnegative
polynomials with no sum of squares decomposition, although
explicit examples of such polynomials appeared only 80 years
later. One of the main contributions of this thesis is to
establish the counterpart of Hilbert's results for the notions of
convexity and sos-convexity. In particular, we will give the first
example of a convex polynomial that is not sos-convex, and by the
end of the first half of this thesis, a complete characterization
of the degrees and dimensions in which convexity and sos-convexity
are equivalent. Some interesting and unexpected connections to
Hilbert's results will also emerge in the process.

In the second half of this thesis, we will turn to the study of
stability in dynamical systems. Here too, we will take a
computational viewpoint with our goal will being the development
and analysis of efficient algorithms for proving stability of
certain classes of nonlinear and hybrid systems.

Almost universally, the study of stability in systems theory leads
to Lyapunov's second method or one of its many variants. An
outgrowth of Lyapunov's 1892 doctoral
dissertation~\cite{PhD:Lyapunov}, Lyapunov's second method tells
us, roughly speaking, that if we succeed in finding a
\emph{Lyapunov function}---an energy-like function of the state
that decreases along trajectories---then we have proven that the
dynamical system in question is stable. In the mid 1900s, a series
of \emph{converse Lyapunov theorems} were developed which
established that any stable system indeed has a Lyapunov function
(see~\cite[Chap. 6]{Hahn_stability_book} for an overview).
Although this is encouraging, except for the simplest classes of
systems such as linear systems, converse Lyapunov theorems do not
provide much practical insight into how one may go about finding a
Lyapunov function.

%Moreover, the Lyapunov functions constructed in these classical
%theorems often belong to an infinite dimensional space of
%functions (e.g. the space of all continuously differentiable
%functions) and because of this reason the theorems do not really
%restrict the search space in

In the last few decades however, advances in the theory and
practice of convex optimization and in particular semidefinite
programming (SDP) have rejuvenated Lyapunov theory. The approach
has been to parameterize a class of Lyapunov functions with
restricted complexity (e.g., quadratics, pointwise maximum of
quadratics, polynomials, etc.) and then pose the search for a
Lyapunov function as a convex feasibility problem. A widely
popular example of this framework which we will revisit later in
this thesis is the method of sum of squares Lyapunov
functions~\cite{PhD:Parrilo},\cite{PabloLyap}. Expanding on the
concept of sum of squares decomposition of polynomials described
above, this technique allows one to formulate semidefinite
programs that search for polynomial Lyapunov functions for
polynomial dynamical systems. Sum of squares Lyapunov functions,
along with many other SDP based techniques, have also been applied
to systems that undergo switching; see
e.g.~\cite{JohRan_PWQ},\cite{PraP03},\cite{Pablo_Jadbabaie_JSR_journal}.
The analysis of these types of systems will also be a subject of
interest in this thesis.

An algorithmic approach to Lyapunov theory naturally calls for new
converse theorems. Indeed, classical converse Lyapunov theorems
only guarantee existence of Lyapunov functions within very broad
classes of functions (e.g. the class of continuously
differentiable functions) that are a priori not amenable to
computation. So there is the need to know whether Lyapunov
functions belonging to certain more restricted classes of
functions that can be computationally searched over also exist.
For example, do stable polynomial systems admit Lyapunov functions
that are polynomial? What about polynomial functions that can be
found with sum of squares techniques? Similar questions arise in
the case of switched systems. For example, do stable linear
switched systems admit sum of squares Lyapunov functions? How
about Lyapunov functions that are the pointwise maximum of
quadratics? If so, how many quadratic functions are needed? We
will answer several questions of this type in this thesis.

This thesis will also introduce a new class of techniques for
Lyapunov analysis of switched systems. The novel component here is
a general framework for formulating Lyapunov inequalities between
\emph{multiple} Lyapunov functions that together guarantee
stability of a switched system under arbitrary switching. The
relation between these inequalities has interesting links to
concepts from automata theory. Furthermore, the technique is
amenable to semidefinite programming.

Although the main ideas behind our approach directly apply to
broader classes of switched systems, our results will be presented
in the more specific context of switched linear systems. This is
mainly due to our interest in the notion of the \emph{joint
spectral radius} of a set of matrices which has intimate
connections to stability of switched linear systems. The joint
spectral radius is an extensively studied quantity that
characterizes the maximum growth rate obtained by taking arbitrary
products from a set of matrices. Computation of the joint spectral
radius, although notoriously hard~\cite{BlTi2},\cite{BlTi3}, has a
wide range of applications including continuity of wavelet
functions, computation of capacity of codes, convergence of
consensus algorithms, and combinatorics, just to name a few. Our
techniques provide several hierarchies of polynomial time
algorithms that approximate the JSR with guaranteed accuracy.

%, all based on semidefinite programming, that provide
%approximations of the JSR with guaranteed accuracy.

%
%The reasons for this are twofold. First, when the switched system
%is linear, we can also provide guarantees on the success of our
%semidefinite programs in finding Lyapunov functions, depending of
%course on the stability properties of the system. Second, there is
%an intimate connection between switched linear systems and the
%notion of the \emph{joint spectral radius} (JSR) of a set of
%matrices.

A more concrete account of the contributions of this thesis will
be given in the following section. We remark that although the
first half of the thesis is mostly concerned with convexity in
polynomial optimization and the second half with Lyapunov
analysis, a common theme throughout the thesis is the use of
algorithms that involve algebraic methods in optimization and
semidefinite programming.

\section{Outline and contributions of the thesis}

The remainder of this thesis is divided into two parts each
containing two chapters. The first part includes our complexity
results on deciding convexity in polynomial optimization
(Chapter~\ref{chap:nphard.convexity}) and our study of the
relationship between convexity and sos-convexity
(Chapter~\ref{chap:convexity.sos.convexity}). The second part
includes new results on Lyapunov analysis of polynomial
differential equations (Chapter~\ref{chap:converse.lyap}) and a
novel framework for proving stability of switched systems
(Chapter~\ref{chap:jsr}). A summary of our contributions in each
chapter is as follows.

\paragraph {Chapter~\ref{chap:nphard.convexity}.} The main result of this chapter is to prove that unless P=NP, there
cannot be a polynomial time algorithm (or even a pseudo-polynomial
time algorithm) that can decide whether a quartic polynomial is
globally convex. This answers a question of N.Z. Shor that
appeared as one of seven open problems in complexity theory for
numerical optimization in 1992~\cite{open_complexity}. We also
show that deciding strict convexity, strong convexity,
quasiconvexity, and pseudoconvexity of polynomials of even degree
four or higher is strongly NP-hard. By contrast, we show that
quasiconvexity and pseudoconvexity of odd degree polynomials can
be decided in polynomial time.

\paragraph {Chapter~\ref{chap:convexity.sos.convexity}.} Our first contribution in this chapter is to prove that three
natural sum of squares (sos) based sufficient conditions for
convexity of polynomials via the definition of convexity, its
first order characterization, and its second order
characterization are equivalent. These three equivalent algebraic
conditions, which we will refer to as sos-convexity, can be
checked by solving a single semidefinite program. We present the
first known example of a convex polynomial that is not sos-convex.
We explain how this polynomial was found with tools from sos
programming and duality theory of semidefinite optimization. As a
byproduct of this numerical procedure, we obtain a simple method
for searching over a restricted family of nonnegative polynomials
that are not sums of squares that can be of independent interest.

If we denote the set of convex and sos-convex polynomials in $n$
variables of degree $d$ with $\tilde{C}_{n,d}$ and $\tilde{\Sigma
C}_{n,d}$ respectively, then our main contribution in this chapter
is to prove that $\tilde{C}_{n,d}=\tilde{\Sigma C}_{n,d}$ if and
only if $n=1$ or $d=2$ or $(n,d)=(2,4)$. We also present a
complete characterization for forms (homogeneous polynomials)
except for the case $(n,d)=(3,4)$ which will appear
elsewhere~\cite{AAA_GB_PP_Convex_ternary_quartics}. Our result
states that the set $C_{n,d}$ of convex forms in $n$ variables of
degree $d$ equals the set $\Sigma C_{n,d}$ of sos-convex forms if
and only if $n=2$ or $d=2$ or $(n,d)=(3,4)$. To prove these
results, we present in particular explicit examples of polynomials
in $\tilde{C}_{2,6}\setminus\tilde{\Sigma C}_{2,6}$ and
$\tilde{C}_{3,4}\setminus\tilde{\Sigma C}_{3,4}$ and forms in
$C_{3,6}\setminus\Sigma C_{3,6}$ and $C_{4,4}\setminus\Sigma
C_{4,4}$, and a general procedure for constructing forms in
$C_{n,d+2}\setminus\Sigma C_{n,d+2}$ from nonnegative but not sos
forms in $n$ variables and degree $d$.

Although for disparate reasons, the remarkable outcome is that
convex polynomials (resp. forms) are sos-convex exactly in cases
where nonnegative polynomials (resp. forms) are sums of squares,
as characterized by Hilbert.

\paragraph {Chapter~\ref{chap:converse.lyap}.} This chapter is devoted to converse results on (non)-existence of polynomial and sum of squares polynomial Lyapunov functions for systems described by polynomial differential
equations. We present a simple, explicit example of a
two-dimensional polynomial vector field of degree two that is
globally asymptotically stable but does not admit a polynomial
Lyapunov function of any degree. We then study whether existence
of a polynomial Lyapunov function implies existence of one that
can be found with sum of squares techniques. We show via an
explicit counterexample that if the degree of the polynomial
Lyapunov function is fixed, then sos programming can fail to find
a valid Lyapunov function even though one exists. On the other
hand, if the degree is allowed to increase, we prove that
existence of a polynomial Lyapunov function for a planar vector
field (under an additional mild assumption) or for a homogeneous
vector field implies existence of a polynomial Lyapunov function
that is sos and that the negative of its derivative is also sos.
This result is extended to prove that asymptotic stability of
switched linear systems can always be proven with sum of squares
Lyapunov functions. Finally, we show that for the latter class of
systems (both in discrete and continuous time), if the negative of
the derivative of a Lyapunov function is a sum of squares, then
the Lyapunov function itself is automatically a sum of squares.

This chapter also includes some complexity results. We prove that
deciding asymptotic stability of homogeneous cubic polynomial
vector fields is strongly NP-hard. We discuss some byproducts of
the reduction that establishes this result, including a
Lyapunov-inspired technique for proving positivity of forms.

%%AAA: removed the following because the problem is in fact undecidable.
%and a corollary demonstrating that deciding existence of lattice
%points in basic semialgebraic sets is NP-hard even in very
%restricted cases.

\paragraph {Chapter~\ref{chap:jsr}.} In this chapter, we introduce the framework of path-complete graph
Lyapunov functions for approximation of the joint spectral radius.
The approach is based on the analysis of the underlying switched
system via inequalities imposed between multiple Lyapunov
functions associated to a labeled directed graph. The nodes of
this graph represent Lyapunov functions, and its directed edges
that are labeled with matrices represent Lyapunov inequalities.
Inspired by concepts in automata theory and symbolic dynamics, we
define a class of graphs called path-complete graphs, and show
that any such graph gives rise to a method for proving stability
of the switched system. This enables us to derive several
asymptotically tight hierarchies of semidefinite programming
relaxations that unify and generalize many existing techniques
such as common quadratic, common sum of squares, and
maximum/minimum-of-quadratics Lyapunov functions.

We compare the quality of approximation obtained by certain
families of path-complete graphs including all path-complete
graphs with two nodes on an alphabet of two matrices. We argue
that the De Bruijn graph of order one on $m$ symbols, with
quadratic Lyapunov functions assigned to its nodes, provides good
estimates of the JSR of $m$ matrices at a modest computational
cost. We prove that the bound obtained via this method is
invariant under transposition of the matrices and always within a
multiplicative factor of $1/\sqrt[4]{n}$ of the true JSR
(independent of the number of matrices).

Approximation guarantees for analysis via other families of
path-complete graphs will also be provided. In particular, we show
that the De Bruijn graph of order $k$, with quadratic Lyapunov
functions as nodes, can approximate the JSR with arbitrary
accuracy as $k$ increases. This also proves that common Lyapunov
functions that are the pointwise maximum (or minimum) of
quadratics always exist. Moreover, the result gives a bound on the
number of quadratic functions needed to achieve a desired level of
accuracy in approximation of the JSR, and also demonstrates that
these quadratic functions can be found with semidefinite
programming. \vspace{5mm}

\noindent A list of open problems for future research is presented
at the end of each chapter.
%
%We will end each chapter with a list of open problems for future
%research.

\subsection{Related publications}

The material presented in this thesis is in the most part based on
the following papers.

\vspace{3mm}

\paragraph{Chapter~\ref{chap:nphard.convexity}.} \

% \newline

\noindent A.~A. Ahmadi, A.~Olshevsky, P.~A. Parrilo, and J.~N.
Tsitsiklis.
\newblock {NP}-hardness of deciding convexity of quartic polynomials and related problems.
\newblock {\em Mathematical Programming}, 2011.
\newblock Accepted for publication. Online version available at arXiv:.1012.1908.

\paragraph{Chapter~\ref{chap:convexity.sos.convexity}.} \

%\newline

\noindent A.~A. Ahmadi and P.~A. Parrilo.
\newblock A convex polynomial that is not sos-convex.
\newblock {\em Mathematical Programming}, 2011.
\newblock DOI: 10.1007/s10107-011-0457-z.

%\newline
\vspace{3mm}

\noindent A.~A. Ahmadi and P.~A. Parrilo.
\newblock A complete characterization of the gap between convexity and sos-convexity.
\newblock In preparation, 2011.

%\newline
\vspace{3mm}

\noindent A.~A. Ahmadi, G.~Blekherman, and P.~A.Parrilo.
\newblock Convex ternary quartics are sos-convex.
\newblock In preparation, 2011.

\paragraph{Chapter~\ref{chap:converse.lyap}.} \

%\newline

\noindent A.~A. Ahmadi and P.~A. Parrilo.
\newblock Converse results on existence of sum of squares {L}yapunov functions.
\newblock In {\em Proceedings of the 50$^{th}$ IEEE Conference on Decision and Control}, 2011.

%\newline
\vspace{3mm}

\noindent A.~A. Ahmadi, M.~Krstic, and P.~A. Parrilo.
\newblock A globally asymptotically stable polynomial vector field with no polynomial {L}yapunov function.
\newblock In {\em Proceedings of the 50$^{th}$ IEEE Conference on Decision and Control}, 2011.

\paragraph{Chapter~\ref{chap:jsr}.} \

%\newline

\noindent A.~A. Ahmadi, R.~Jungers, P.~A. Parrilo, and
M.~Roozbehani.
\newblock Analysis of the joint spectral radius via {L}yapunov functions on path-complete graphs.
\newblock In {\em Hybrid Systems: Computation and Control 2011}, Lecture Notes in Computer Science. Springer, 2011.

\cleardoublepage \thispagestyle{empty} \vspace*{\fill}
\begin{center}{ \sfbHuge  Part I: \\ \ \\Computational
and Algebraic\\ \ \\Aspects of Convexity }
\end{center}
\vspace*{\fill} \addcontentsline{toc}{chapter}{I: Computational
and Algebraic Aspects of Convexity} \cleardoublepage

\chapter{Complexity of Deciding Convexity}\label{chap:nphard.convexity}

%%%Black and white version:
\def\alex#1{{#1}}
\def\alexc#1{{#1}}
\def\alexn#1{{#1}}

\long\def\jnt#1{{#1}} \long\def\old#1{} \long\def\jntn#1{{#1}}

\definecolor{DarkerGreen}{RGB}{0,170,0}
%Amirali: revision 5 in DarkerGreen.
\long\def\aaa#1{{#1}}
%Amirali: revision 7:
\definecolor{orange}{rgb}{1,0.5,0}
\long\def\aaan#1{{#1}}

%Amirali: black and white for the revision after hearing back from MPA:
\long\def\aaar#1{{#1}}

%%%%%%%%%%%%%%%%%%%%%%%%%%%%%%%%%%%%%%%%%%%%%%%%%%%%%%%%%%%%%%%%%%%%%%%%%%%%%%%%

In this chapter, we characterize the computational complexity of
deciding convexity and many of its variants in polynomial
optimization. The material presented in this chapter is based on
the work in~\cite{NPhard_Convexity_arxiv}.

\section{Introduction}
%\subsection{Motivation}
The role of \emph{convexity} in modern day mathematical
programming has proven to be remarkably fundamental, to the point
that tractability of an optimization problem is nowadays assessed,
more often than not, by whether or not the problem benefits from
some sort of underlying convexity. In the famous words of
Rockafellar~\cite{Roc_SIAM_Lagrange}: \vspace{-5pt}
\begin{itemize}
\item[] ``In fact the great watershed in optimization isn't
between linearity and nonlinearity, but convexity and
nonconvexity.''
\end{itemize}
\vspace{-5pt} But how easy is it to distinguish between convexity
and nonconvexity? Can we decide in an efficient manner if a given
optimization problem is convex?

A class of optimization problems that allow for a rigorous study
of this question from a computational complexity viewpoint is the
class of polynomial optimization problems. \alex{These are}
optimization problems where the objective is given by a polynomial
function and the feasible set is described by polynomial
inequalities. Our research in this direction was motivated by a
concrete question of N. Z. Shor that appeared as one of seven open
problems in complexity theory for numerical optimization put
together by Pardalos and Vavasis in 1992~\cite{open_complexity}:
\vspace{-5pt}
\begin{itemize}
\item[] ``Given a degree-$4$ polynomial in $n$ variables, what is
the complexity of determining whether this polynomial describes a
convex function?''
\end{itemize}
\vspace{-5pt} As we will explain in more detail shortly, the
reason why Shor's question is specifically about degree $4$
polynomials is that deciding convexity of odd degree polynomials
is trivial and deciding convexity of degree $2$ (quadratic)
polynomials can be reduced to the simple task of checking
\jnt{whether} a constant matrix is positive semidefinite. So, the
first interesting case really occurs for degree $4$ (quartic)
polynomials. Our main contribution in this chapter
(Theorem~\ref{thm:convexity.quartic.nphard} in
Section~\ref{subsec:convexity.hard.degrees}) is to show that
deciding convexity of polynomials is strongly NP-hard already for
polynomials of degree $4$.

The implication of NP-hardness of this problem is that
\aaan{unless P=NP}, there exists no algorithm that can take as
input the (\alex{rational}) coefficients of a quartic polynomial,
have running time bounded by a polynomial in the number of bits
needed to represent the coefficients, and output correctly on
every instance whether or not the polynomial is convex.
Furthermore, the fact that our NP-hardness result is in the strong
sense (as opposed to weakly NP-hard problems such as KNAPSACK)
implies, roughly speaking, that the problem remains NP-hard even
when the magnitude of the coefficients of the polynomial are
restricted to be ``small.'' For a strongly NP-hard problem, even a
pseudo-polynomial time algorithm cannot exist unless P=NP.
See~\cite{GareyJohnson_Book} for precise definitions and more
details.

%Indeed, if such an algorithm is ever found, it would immediately
%give rise to polynomial time algorithms for thousands of other
%NP-hard and NP-complete problems, for which no polynomial time
%algorithms is believed to exist. (See e.g. [*] for more details on
%the theory of NP-completeness.)

%
%Unfortunately, this intractability result \jnt{runs counter}
%%comes in contrast
%to our interest in establishing convexity of polynomials in many
%application areas.

\aaan{There are many areas of application where one would like to
establish convexity of polynomials.} Perhaps the simplest example
is in global minimization of polynomials, where it could be very
useful to decide first \jnt{whether the polynomial to be optimized
is convex. Once convexity is verified,} then every local minimum
is global and very basic techniques (e.g., gradient descent) can
find a global minimum---a task that is in general NP-hard in
\aaan{the} absence of
convexity~\cite{Minimize_poly_Pablo},~\cite{nonnegativity_NP_hard}.
As another example, if we can certify that a homogeneous
polynomial is convex, \aaa{then we define a gauge (or Minkowski)
norm based on its convex sublevel sets, which may be useful in
many applications}. In several other problems \jnt{of practical
relevance,} we might not just be interested in checking whether a
given polynomial is convex, but to \emph{parameterize} a family of
convex polynomials and perhaps search or optimize over them. For
example we might be interested in approximating the convex
envelope of a complicated nonconvex function with a convex
polynomial, or in fitting a convex polynomial to a set of data
points with minimum error~\cite{convex_fitting}. Not surprisingly,
if testing membership to the set of convex polynomials is hard,
searching and optimizing over \jnt{that set} also turns out to be
a hard problem.

We also extend our hardness result to some variants of convexity,
namely, the problems of deciding \emph{strict convexity},
\emph{strong convexity}, \emph{pseudoconvexity}, and
\emph{quasiconvexity} of polynomials. Strict convexity is a
property that is often useful to check because it guarantees
uniqueness of the optimal solution in optimization problems. The
notion of strong convexity is a common assumption in convergence
analysis of many iterative Newton-type algorithms in optimization
theory; see, e.g.,~\cite[Chaps.\ 9--11]{BoydBook}. So, in order to
ensure the theoretical convergence rates promised by many of these
algorithms, one needs to first make sure that the objective
function is strongly convex. The problem of checking
quasiconvexity (convexity of sublevel sets) of polynomials also
arises frequently in practice. For instance, if the feasible set
of an optimization problem is defined by polynomial inequalities,
by certifying quasiconvexity of the defining polynomials we can
ensure that the feasible set is convex. In several statistics and
clustering problems, we are interested in finding minimum volume
convex sets that contain a set of data points in space. This
problem can be tackled by searching over the set of quasiconvex
polynomials~\cite{convex_fitting}. In economics, quasiconcave
functions are prevalent as desirable utility
functions~\cite{Testing_convexity_economics},~\cite{Quasiconcave_programming}.
In control and systems theory, it is useful at times to search for
quasiconvex Lyapunov functions whose convex sublevel sets contain
relevant information about the trajectories of a dynamical
system~\cite{Chesi_Hung_journal},~\cite{AAA_PP_CDC10_algeb_convex}.
Finally, the notion of pseudoconvexity is a natural generalization
of convexity that inherits many of the attractive properties of
convex functions. For example, every stationary point or every
local minimum of a pseudoconvex function must be a global minimum.
Because of these nice features, pseudoconvex programs have been
studied extensively in nonlinear
programming~\cite{Mangasarian_Pseudoconvex_fns},%\cite{Generalized_convexity_Book},
~\cite{pseudoconvex_nonnegative_vars}.

As an outcome of close to a century of research in convex
analysis, numerous necessary, sufficient, and exact conditions for
convexity and all of its variants are available; see,
e.g.,~\cite[Chap.~3]{BoydBook},~\cite{Second_order_pseudoconvexity},~\cite{Matrix_theoretic_quasiconvexity},~\cite{Criteria_comparison_quasiconvexity_pseudoconvexity},~\cite{Testing_convexity_economics},
%~\cite{Generalized_convexity_Book},
~\cite{NLP_Book_Mangasarian} and references therein for a by no
means exhaustive list. Our results %in this paper
suggest that none of the exact characterizations of these notions
can be efficiently checked for polynomials. In fact, when turned
upside down, many of these equivalent formulations reveal new
NP-hard problems; see, e.g.,
Corollary~\ref{cor:nonnegativity.quartic.nphard} and
\ref{cor:semialgeb.set.nphard}.
%and~\ref{cor:max.of.quartics.nphard}.

\subsection{Related Literature}
There are several results in the literature \jnt{on} the
complexity of various special cases of polynomial optimization
problems. The interested reader can find many of these results in
the edited volume of
Pardalos~\cite{Pardalos_Book_Complexity_num_opt} or in the survey
papers of de Klerk~\cite{deKlerk_complexity_survey}, and Blondel
and Tsitsiklis~\cite{BlTi1}. A very general and fundamental
concept in certifying feasibility of polynomial equations and
inequalities is the Tarski--Seidenberg quantifier elimination
theory~\cite{Tarski_quantifier_elim},~\cite{Seidenberg_quantifier_elim},
from which it follows that all of the problems that we consider in
this chapter are algorithmically \emph{decidable}. This means that
there are algorithms that on all instances of our problems of
interest halt in finite time and always output the correct yes--no
answer. Unfortunately, algorithms based on quantifier elimination
or similar decision algebra techniques have running times that are
at least exponential in the number of variables
\cite{Algo_real_algeb_geom_Book}, and in practice can only solve
problems with very few parameters.

When we turn to the issue of polynomial time solvability, perhaps
the most relevant result for our purposes is the NP-hardness of
deciding nonnegativity of quartic polynomials and biquadratic
forms (see Definition~\ref{defn:biquad.forms}); the main reduction
that we give in this chapter will in fact be from the latter
problem. \jnt{As we will see in
Section~\ref{subsec:convexity.hard.degrees}, it turns out that
deciding convexity of quartic forms is \aaan{equivalent to}
checking nonnegativity of a special \aaan{class} of biquadratic
forms, which are themselves a special \aaan{class} of quartic
forms.} The NP-hardness of checking nonnegativity of quartic forms
follows, e.g., as a direct consequence of NP-hardness of testing
matrix copositivity, a result proven by Murty and
\aaa{Kabadi}~\cite{nonnegativity_NP_hard}. As for the hardness of
checking nonnegativity of biquadratic forms, we know of two
different proofs. The first one is due to
Gurvits~\cite{Gurvits_quantum_entag_hard}, \jnt{who} proves that
the entanglement problem in quantum mechanics (i.e., the problem
of distinguishing separable quantum states from entangled ones) is
NP-hard. A dual \jnt{reformulation of this} %viewpoint to this
result shows directly that checking nonnegativity of biquadratic
forms is NP-hard; see~\cite{Pablo_Sep_Entang_States}. The second
proof is due to Ling et al.~\cite{Ling_et_al_Biquadratic},
\jnt{who} use a theorem of Motzkin and Straus to give a very short
and elegant reduction from the maximum \aaar{clique} problem in
graphs.

The only work in the literature on the hardness of deciding
polynomial convexity that we are aware of is the work of Guo on
the complexity of deciding convexity of quartic polynomials over
simplices~\cite{complexity_simplex_convexity}. Guo discusses some
of the difficulties that arise from this problem, but he does not
prove that deciding convexity of polynomials over simplices is
NP-hard. Canny shows in~\cite{Canny_PSPACE} that the existential
theory of the real numbers can be decided in PSPACE. From this, it
follows that testing several properties of polynomials, including
nonnegativity and convexity, can be done in polynomial space.
In~\cite{Nie_PMI_SDP_repres.}, Nie proves that the related notion
of \emph{matrix convexity} is NP-hard for polynomial matrices
whose
entries are quadratic forms. %Finally, the authors have some earlier
%hardness results on testing a slightly modified version of convexity
%of sextic polynomials~\cite{Convexity_Sextic_NPhard} and quartic
%polynomials~\cite{Convexity_quartic_NPhard_old_proof}. The results
%in~\cite{Convexity_Sextic_NPhard}
%and~\cite{Convexity_quartic_NPhard_old_proof} are implied by the
%stronger results of this paper, but the reductions in there are
%totally different and much more combinatorial in nature. The gadgets
%used in those reductions provide novel ways of linking discrete
%combinatorial problems to continuous ones, and could be of
%independent interest to some readers.

On the \alex{algorithmic} side, several techniques have been
proposed both for testing convexity of sets and convexity of
functions. \aaan{Rademacher and Vempala present and analyze
randomized algorithms for testing the relaxed notion of
\emph{approximate convexity}~\cite{Test_Geom_Convexity}.
%Delanoue and Henrion give sufficient conditions for set convexity
%based on ideas from interval
%analysis~\cite{set_convexity_interval_analysis}.
In~\cite{Lasserre_set_convexity}, Lasserre proposes a semidefinite
programming hierarchy for testing convexity of basic closed
semialgebraic sets; a problem that we also prove to be NP-hard
(see Corollary~\ref{cor:semialgeb.set.nphard}).}
%In~\cite{Lasserre_set_convexity}, Lasserre proposes a semidefinite
%programming hierarchy for testing convexity of basic closed
%semialgebraic sets; a problem that we also prove to be NP-hard
%(see Corollary~\ref{cor:semialgeb.set.nphard}). Delanoue and
%Henrion give sufficient conditions for set convexity based on
%ideas from interval
%analysis~\cite{set_convexity_interval_analysis}. Rademacherand and
%Vempala present and analyze randomized algorithms for testing the
%relaxed notion of \emph{approximate
%convexity}~\cite{Test_Geom_Convexity}.
As for testing convexity of functions, an approach that some
convex optimization parsers (e.g.,  \texttt{CVX}~\cite{cvx}) take
is to start with some ground set of convex functions and then
check \aaan{whether} the desired function can be obtained
\aaan{by} applying a set of convexity preserving operations to the
functions in the ground
set~\cite{Crusius_thesis},~\cite[p.~79]{BoydBook}.
\aaan{Techniques of this type that are based on the calculus of
convex functions are successful for a large range of applications.
However, when applied to general polynomial functions, they can
only detect a subclass of convex polynomials.}
%
%
%Although this approach is successful for a large range of
%applications, when applied to polynomial functions, it can only
%detect a small subset of convex polynomials.
%

\jnt{Related} to convexity of polynomials, a concept that has
attracted recent attention is the algebraic notion of
\emph{sos-convexity} (see
Definition~\ref{defn:sos-convex})~\cite{Helton_Nie_SDP_repres_2},
\cite{Lasserre_Jensen_inequality},
\cite{Lasserre_Convex_Positive}, \cite{AAA_PP_CDC10_algeb_convex},
\cite{convex_fitting}, \cite{Chesi_Hung_journal},
\cite{AAA_PP_not_sos_convex_journal}. This is a powerful
sufficient condition for convexity \aaa{that} relies on an
appropriately defined sum of squares decomposition of the Hessian
matrix, and can be efficiently checked by solving a single
semidefinite program. The study of sos-convexity will be the main
focus of our next chapter. In particular, we will present explicit
counterexamples to show that not every convex polynomial is
sos-convex. \aaa{The NP-hardness result in this chapter certainly
justifies the existence of such counterexamples and more generally
suggests that \emph{any} polynomial time algorithm attempted for
checking polynomial convexity is doomed to fail on some hard
instances.}

\subsection{Contributions and organization of this chapter}
\jnt{The} \alex{main} contribution of this chapter is to
\aaa{establish} the computational complexity of deciding
convexity, strict convexity, strong convexity, pseudoconvexity,
and quasiconvexity of polynomials for any given degree. (See
Table~\ref{table:summary} in Section~\ref{sec:summary.conclusions}
for a quick summary.) The results are mainly divided in three
sections, with Section~\ref{sec:convexity} covering convexity,
Section~\ref{sec:strict.strong} covering strict and strong
convexity, and Section~\ref{sec:quasi.pseudo} covering
quasiconvexity and pseudoconvexity. \aaa{These three sections
follow a similar \aaan{pattern} and are each divided into three
parts: first, the definitions and basics, second, the degrees for
which the questions can be answered in polynomial time, and third,
the degrees for which the questions are NP-hard.}

Our main reduction, which establishes NP-hardness of checking
convexity of quartic forms, is given in
Section~\ref{subsec:convexity.hard.degrees}. This hardness result
is extended to strict and strong convexity in
Section~\ref{subsec:strict.strong.hard.degrees}, and to
quasiconvexity and pseudoconvexity in
Section~\ref{subsec:quasi.pseudo.hard.degrees}. \aaa{By contrast,}
we show \aaa{in Section~\ref{subsec:quasi.pseudo.easy.degrees}}
that quasiconvexity and pseudoconvexity of odd degree polynomials
can be decided in polynomial time. A summary of the chapter and
some concluding remarks are presented in
Section~\ref{sec:summary.conclusions}.

\section{Complexity of deciding convexity}\label{sec:convexity}
\subsection{Definitions and basics}\label{subsec:convexity.basics}
A \aaa{(multivariate)} \emph{polynomial} $p(x)$ in variables
$x\mathrel{\mathop:}=(x_1,\ldots,x_n)^T$ is a function from
$\mathbb{R}^n$ to $\mathbb{R}$ that is a finite linear combination
of monomials:
\begin{equation}
p(x)=\sum_{\alpha}c_\alpha x^\alpha=\sum_{\alex{\alpha_1, \ldots,
\alpha_n}} c_{\alex{\alpha_1,\ldots,\alpha_n}} x_1^{\alpha_1}
\cdots x_n^{\alpha_n} ,
\end{equation}
\alex{where the sum is over \aaa{$n$-tuples of} nonnegative
integers $\alpha_i$}. An algorithm for testing some property of
polynomials will have as its input an ordered list of the
coefficients $c_\alpha$. Since our complexity results are based on
models of digital computation, where the input must be represented
by a finite number of bits, the coefficients $c_\alpha$ for us
will always be rational numbers, which upon clearing the
denominators can be taken to be integers. So, for the remainder of
the chapter, even when not explicitly stated, we will always have
$c_\alpha \in \mathbb{Z}$.

%When we discuss the complexity of algorithms for testing
%properties of polynomials, we c
%
%
%an ordered list of the coefficients $c_\alpha$ is considered as
%the input to these algorithms. Because we adhere to models of
%digital computation where the input should be represented by a
%finite number of bits, the coefficients $c_\alpha$ for us will
%always be rational numbers, which upon clearing of the
%denominators can be taken to be integers. So, for the remainder of
%the paper, even when not explicitly mentioned, we will always have
%$c_\alpha \in \mathbb{Z}$.

The \emph{degree} of a monomial $x^\alpha$ is equal to $\alpha_1 +
\cdots + \alpha_n$. The degree of a polynomial $p(x)$ is defined
to be the highest degree of its component monomials. A simple
counting argument shows that a polynomial of degree $d$ in $n$
variables has $\binom{n+d}{d}$ coefficients. A \emph{homogeneous
polynomial} (or a \emph{form}) is a polynomial where all the
monomials have the same degree. A form $p(x)$ of degree $d$ is a
homogeneous function of degree $d$ (since it satisfies $p(\lambda
x)=\lambda^d p(x)$), and has $\binom{n+d-1}{d}$ coefficients.

A polynomial $p(x)$ is said to be \emph{nonnegative} or
\emph{positive semidefinite (psd)} if $p(x)\geq0$ for all
$x\in\mathbb{R}^n$. Clearly, a necessary condition for a
polynomial to be psd is for its degree to be even. We say that
$p(x)$ is a \emph{sum of squares (sos)}, if there exist
polynomials $q_{1}(x),\ldots,q_{m}(x)$ such that
$p(x)=\sum_{i=1}^{m}q_{i}^{2}(x)$. Every sos polynomial is
obviously psd. A \emph{polynomial matrix} $P(x)$ is a matrix with
polynomial entries. We say that a polynomial matrix $P(x)$ is
\emph{PSD} (denoted $P(x)\succeq0$) if it is positive semidefinite
in the matrix sense for every value of the indeterminates $x$.
\jnt{(Note the upper case convention for matrices.)} It is easy to
see that $P(x)$ is PSD if and only if the scalar polynomial
$y^TP(x)y$ in variables $(x;y)$ is psd.
%We say that $p(x)$ is a sum of squares (sos), if there exist
%polynomials $q_{1}(x),...,q_{m}(x)$ such that
%\begin{equation}\label{eq:sos.decomp.q_i}
%p(x)=\sum_{i=1}^{m}q_{i}^{2}(x).
%\end{equation}

\aaa{We recall that a polynomial $p(x)$ is convex if and only if
its Hessian matrix, which will be generally denoted by $H(x)$, is
PSD.}

\subsection{Degrees that are easy}\label{subsec:convexity.easy.degrees}
The question of deciding convexity is trivial \aaa{for odd degree
polynomials}. Indeed, it is easy to check that linear polynomials
($d=1$) are always convex and that polynomials of odd degree
$d\geq3$ can never be convex. \aaa{The case of quadratic
polynomials ($d=2$) is also straightforward.}
%Perhaps the easiest way to see the
%latter claim is to note that convex functions restricted to
%arbitrary lines remain convex. But a univariate odd degree
%polynomial is clearly never convex.
\aaa{A quadratic polynomial $p(x)=\frac{1}{2}x^TQx+q^Tx+c$ is
convex if and only if the constant matrix $Q$ is positive
\jnt{semidefinite}. This can be decided in polynomial time for
example by performing Gaussian pivot steps along the main diagonal
of $Q$~\cite{nonnegativity_NP_hard} or by computing the
characteristic polynomial of $Q$ exactly and then checking that
the signs of its coefficients alternate~\cite[p.
403]{HJ_Matrix_Analysis_Book}.}

%Let us now address the case of quadratic polynomials ($d=2$). We
%recall first a basic result from linear algebra.
%
%\begin{theorem}\label{thm:Q.psd.iff.char.poly.alternate} (e.g.~\cite[p. 403]{HJ_Matrix_Analysis_Book})
%An $n \times n$ symmetric matrix $Q$ is positive semidefinite
%(PSD) if and only if all the coefficients of its characteristic
%polynomial $r(\lambda)=\det(\lambda
%I-Q)=\lambda^n+r_{n-1}\lambda^{n-1}+\cdots+r_o$ alternate in sign,
%i.e., they satisfy $r_i(-1)^{n-i}\geq0$.
%\end{theorem}
%
%Since determinants can be computed in polynomial
%time~\cite{BlTi1}, the theorem above clearly gives a polynomial
%time algorithm for checking if a matrix is positive
%semidefinite.\footnote{Alternatively, positive semidefiniteness of
%a matrix can be checked by performing at most $n$ Gaussian
%eliminations with a computational cost of
%$O(n^3)$~\cite{nonnegativity_NP_hard}.}
%
%\begin{theorem}\label{thm:convex.quadratics.poly.time}
%Convexity of quadratic polynomials can be decided in polynomial
%time.
%\end{theorem}
%\begin{proof}
%Let $p(x)=\frac{1}{2}x^TQx+q^Tx+c$ be a quadratic polynomial. By
%Theorem~\ref{thm:classical.first.2nd.order.charac.}, $p(x)$ is
%convex if and only if its Hessian matrix, which is the constant
%matrix $Q$ is positive semidefinite.
%Theorem~\ref{thm:Q.psd.iff.char.poly.alternate} gives a way of
%checking this condition in polynomial time.
%\end{proof}
Unfortunately, the results that come next suggest that the case of
quadratic polynomials is essentially the only nontrivial case
where convexity can be efficiently \aaa{decided}.

%%The following remark is WRONG! got rid of it.
%\begin{remark} We \jnt{note} %shall also point out
%that although deciding convexity of
%quadratic polynomials on the whole space is easy, the problem
%\jnt{can} become hard if we restrict our attention to certain regions
%of space. For example, the problem of deciding whether a quadratic
%polynomial $p(x)=\frac{1}{2}x^TQx+q^Tx+c$ is convex on the
%nonnegative orthant $\{x\geq0\}$ is equivalent to testing whether
%$x^TQx$ is nonnegative for $x\geq0$. This is exactly the matrix
%copositivity problem, which is known to be
%NP-hard~\cite{nonnegativity_NP_hard}.
%\end{remark}

\subsection{Degrees that are hard}\label{subsec:convexity.hard.degrees}
The main hardness result of this chapter is the following theorem.
\begin{theorem}\label{thm:convexity.quartic.nphard}
Deciding convexity of degree four polynomials is \jnt{strongly}
NP-hard. This is true even when the polynomials are restricted to
be homogeneous.
\end{theorem}
We will give a reduction from the problem of deciding
nonnegativity of biquadratic forms. We start by recalling some
basic facts about biquadratic forms and sketching the idea of the
proof.

\begin{definition}\label{defn:biquad.forms}
A \emph{biquadratic form} \alex{$b(x;y)$ is a \aaa{form in the
variables\\ $x=(x_1, \ldots, x_n)^T$ and $y=(y_1, \ldots, y_m)^T$
that} can be written as}
\begin{equation}\label{eq:defn.biquad.form}
b(x;y)\old{\mathrel{\mathop:}}=\sum_{i\leq j, \, k\leq
l}\alpha_{ijkl}x_ix_jy_ky_l.
\end{equation}
%is a form in two sets of variables $x=(x_1, \ldots, x_n)^T$ and
%$y=(y_1, \ldots, y_m)^T$, such that for fixed $x$, $b(x;y)$
%becomes a quadratic form in $y$ and for fixed $y$, it becomes a
%quadratic form in $x$.
\end{definition}
\aaa{Note that for fixed $x$, $b(x;y)$ becomes a quadratic form in
$y$, and for fixed $y$, it becomes a quadratic form in $x$.} Every
biquadratic form is a quartic form, but the converse is of course
not true. It follows from a result of Ling et
al.~\cite{Ling_et_al_Biquadratic} that deciding nonnegativity of
biquadratic forms is strongly NP-hard. \aaar{This claim is not
precisely stated in this form in~\cite{Ling_et_al_Biquadratic}.
For the convenience of the reader, let us make the connection more
explicit before we proceed, as this result underlies everything
that follows.

The argument in~\cite{Ling_et_al_Biquadratic} is based on a
reduction from CLIQUE (given a graph $G(V,E)$ and a positive
integer $k\leq |V|$, decide whether $G$ contains a clique of size
$k$ or more) whose (strong) NP-hardness is
well-known~\cite{GareyJohnson_Book}. For a given graph $G(V,E)$ on
$n$ nodes, if we define the biquadratic form $b_G(x;y)$ in the
variables $x=(x_1, \ldots, x_n)^T$ and $y=(y_1, \ldots, y_n)^T$ by
$$b_G(x;y)=-2 \sum_{(i,j)\in E} x_ix_jy_iy_j,$$ then Ling et al.~\cite{Ling_et_al_Biquadratic} use a theorem of Motzkin and
Straus~\cite{Motzkin_Straus} to show
\begin{equation}\label{eq:biquadratic.CLIQUE}
\min_{||x||=||y||=1} b_G(x;y)=-1+\frac{1}{\omega(G)}.
\end{equation}
Here, $\omega(G)$ denotes the clique number of the graph $G$,
i.e., the size of a maximal clique.\footnote{Equation
(\ref{eq:biquadratic.CLIQUE}) above is stated
in~\cite{Ling_et_al_Biquadratic} with the stability number
$\alpha(G)$ in place of the clique number $\omega(G)$. This seems
to be a minor typo.} From this, we see that for any value of $k$,
$\omega(G)\leq k$ if and only if $$\min_{||x||=||y||=1}
b_G(x;y)\geq\frac{1-k}{k},$$ which by homogenization holds if and
only if the biquadratic form
$$\hat{b}_G(x;y)=- 2k \sum_{(i,j)\in E} x_i x_j y_i y_j - (1-k)
\left(\sum_{i=1}^n x_i^2\right) \left( \sum_{i=1}^n y_i^2
\right)$$ is nonnegative. Hence, by checking nonnegativity of
$\hat{b}_G(x;y)$ for all values of $k\in\{1,\ldots,n-1\}$, we can
find the exact value of $\omega(G)$. It follows that deciding
nonnegativity of biquadratic forms is NP-hard, and in view of the
fact that the coefficients of $\hat{b}_G(x;y)$ are all integers
with absolute value at most $2n-2$, the NP-hardness claim is in
the strong sense. Note also that the result holds even when $n=m$
in Definition~\ref{defn:biquad.forms}. In the sequel, we will
always have $n=m$.

}

%As Pablo noted, it is well-known that deciding nonnegativity of
%biquadratic forms is NP-hard (even when $n=m$). In particular,
%there is a simple reduction in the literature from the maximum
%independent set problem that uses the Motzkin-Straus theorem. In
%the sequel, we will always have $n=m$.
%
%We say that a polynomial $p(x)$ is nonnegative or \emph{positive
%semidefinite (psd)} if $p(x)\geq0, \ \forall x\in\mathbb{R}^n$. A
%\emph{polynomial matrix} $P(x)$ is a matrix with polynomial
%entries. We say that a polynomial matrix is \emph{PSD} if it is
%positive semidefinite in the matrix sense for every value of the
%indeterminates $x$. Clearly, $P(x)$ is PSD if and only if the
%scalar polynomial $y^TP(x)y$ in variables $(x;y)$ is psd.

It is not difficult to see that any biquadratic form $b(x;y)$ can
be written in the form \begin{equation} \label{eq:b_represent}
b(x;y)=y^TA(x)y\end{equation} (or of course as $x^TB(y)x$) for
some symmetric polynomial matrix $A(x)$ whose entries are
quadratic forms. Therefore, it is strongly NP-hard to decide
whether a symmetric polynomial matrix with quadratic form entries
is PSD. \alex{One might hope that this would lead to a quick proof
\aaa{of NP-hardness of testing convexity of quartic forms},
because} the Hessian of a quartic form is exactly a symmetric
polynomial matrix with quadratic form entries. \aaa{However, the
major problem that stands in the way is that not every polynomial
matrix is a \emph{valid Hessian}. Indeed, if any of the partial
derivatives between the entries of $A(x)$ do not commute (e.g., if
$\frac{\partial{A_{11}(x)}}{\partial{x_2}}\neq\frac{\partial{A_{12}(x)}}{\partial{x_1}}$),
then $A(x)$ cannot be the matrix of second derivatives of some
polynomial. This is because all mixed third partial derivatives of
polynomials must commute.}

Our task is therefore to prove that even with these additional
constraints \jnt{on} the entries of $A(x)$, the problem of
deciding positive semidefiniteness of such matrices remains
NP-hard. \aaa{We will show} that any given symmetric $n \times n$
matrix $A(x)$, whose entries are quadratic forms, \jnt{can be
embedded} in a $2n \times 2n$ polynomial matrix $H(x,y)$, again
with quadratic form entries, \jnt{so} that $H(x,y)$ is a valid
Hessian and $A(x)$ is PSD if and only if $H(x,y)$ is. In fact,
\alex{we will directly construct the polynomial $f(x,y)$ whose
Hessian is the matrix $H(x,y)$.} \aaa{This is done in the next
theorem, which establishes the correctness of our main reduction.
Once this theorem is proven, the proof of
Theorem~\ref{thm:convexity.quartic.nphard} will become immediate.}
%instead of constructing $H(x,y)$ from $A(x)$ and
%integrating it twice to get a quartic form $f(x,y)$ (that will be
%convex if and only if $A(x)$ is PSD), we will show how to directly
%construct $f(x,y)$ from $A(x)$.

%At this point, we only need one last easy definition. A polynomial
%$p(x)$ is said to be a \emph{sum of squares (sos)}, if it can be
%written as
%$$p(x)=\sum_{i=1}^{k}q_i^2(x)$$ for some other polynomials $q_i(x)$.
%Existence of a sum of squares decomposition for a polynomial is an
%\emph{explicit certificate} of its nonnegativity. As we will see
%shortly in our proof, the constructive nature of sos decomposition
%proves to be advantageous for use in a reduction.

%(of Theorem~\ref{thm:convexity.quartic.nphard})

\aaa{
\begin{theorem}\label{thm:main.reduction}
 \alex{Given a biquadratic form $b(x;y)$,
%, which,
%as described above, can be represented as $\frac{1}{2}y^TA(x)y$, where $A(x)$ is a symmetric polynomial matrix.
define the} the $n \times n$ polynomial matrix $C(x,y)$ by setting
\begin{equation}\label{eq:C(x,y).defn.}
\alex{[}C(x,y)\alex{]_{ij}}\mathrel{\mathop:}=\frac{\partial{b(x;y)}}{\partial{x_{\alex{i}}}\partial{y_{\alex{j}}}},
\end{equation} and let $\gamma$ be the largest coefficient, in
absolute value, \alex{of any monomial present in some entry of the
matrix $C(x,y)$.}
%the scalar polynomial
%$v^TC(x,y)w$ in variables $x,y,$ and auxiliary variables $v$ and $w$
%in $\mathbb{R}^n$; i.e.,
%\begin{equation}\label{eq:scalar.gamma.defn.}
%\gamma\mathrel{\mathop:}=\max |\mbox{coefficients}(v^TC(x,y)w)|.
%\end{equation}
Let $f$ be the form given by
%\footnote{The
%notation $\sum_{\ }^{k}$ means that there are ${k}$ terms in the
%summation.}
\begin{equation}\label{eq:f(x,y).defn.}
f(x,y)\mathrel{\mathop:}=b(x;y) + \frac{n^2\gamma}{2}\Big(
\sum_{i=1}^n x_i^4+\sum_{i=1}^n
y_i^4+\sum_{\substack{i,j=1,\ldots,n\\i\aaan{<} j}}
x_i^2x_j^2+\sum_{\substack{i,j=1,\ldots,n\\i\aaan{<} j}}
y_i^2y_j^2\Big).\end{equation}
%\smallskip
%
%\noindent \textbf{Claim:}
%
Then, $b(x;y)$ is psd if and only if $f(x,y)$ is convex.
\end{theorem}
}

%\bigskip
\begin{proof}
%\alex{Observe that once this claim is established, our proof will
%be complete. This is because, as we remarked earlier, testing
%whether $b(x;y)$ is psd is known to be NP-hard
%\cite{Ling_et_al_Biquadratic}, so that testing convexity of the
%polynomial $f(x,y)$ is also NP-hard.}
%Before we prove this claim, let us note that $f(x,y)$ is indeed a
%quartic form and that the reduction is clearly polynomial in length.

%\bigskip

\aaa{Before we prove the claim, let us make a few observations and
try to shed light on the intuition behind this construction}. We
will use $H(x,y)$ to denote the Hessian of $f$. This is a $2n
\times 2n$ polynomial matrix \alex{whose entries are quadratic
forms.} The polynomial $f$ is convex if and only if $z^TH(x,y)z$
is psd. For bookkeeping purposes, let us split the variables $z$
as $z\mathrel{\mathop:}=(z_x, z_y)^T$, where $z_x$ and $z_y$ each
belong to $\mathbb{R}^n$. It will also be helpful to give a name
to the second group of terms in the definition of $f(x,y)$ in
(\ref{eq:f(x,y).defn.}). So, let
\begin{equation}\label{eq:g(x,y).defn.}
g(x,y)\mathrel{\mathop:}=\frac{n^2\gamma}{2}\aaa{\Big(
\sum_{i=1}^n x_i^4+\sum_{i=1}^n
y_i^4+\sum_{\substack{i,j=1,\ldots,n\\i\aaan{<} j}}
x_i^2x_j^2+\sum_{\substack{i,j=1,\ldots,n\\i\aaan{<} j}}
y_i^2y_j^2\Big)}.
\end{equation}
We denote the Hessian matrices of $b(x,y)$ and $g(x,y)$ with
$H_b(x,y)$ and $H_g(x,y)$ respectively. Thus,
$H(x,y)=H_b(x,y)+H_g(x,y)$. \aaa{Let us first focus on the
structure of $H_b(x,y)$}. \alex{Observe that if we define}
$$\alex{[}A(x)\alex{]_{ij}}=\frac{\partial{b(x;y)}}{\partial{y_{\alex{i}}}\partial{y_{\alex{j}}}},$$
\alex{then $A(x)$ depends only on $x$, and  \begin{equation}
 \frac{1}{2} y^T A(x) y = b(x;y). \label{eq:yAy=b} \end{equation} }
\alex{Similarly, if we let}
$$\alex{[}B(y)\alex{]_{ij}}=\frac{\partial{b(x;y)}}{\partial{x_{\alex{i}}}\partial{x_{\alex{j}}}},$$
\alex{then $B(y)$ depends only on $y$, and
\begin{equation}\label{eq:x.B.x=b} \frac{1}{2}x^TB(y)x=b(x;y).
\end{equation}}\alex{From Eq.\ (\ref{eq:yAy=b}), we have that $b(x;y)$ is psd if and only if $A(x)$ is \aaa{PSD}; from  Eq.\ (\ref{eq:x.B.x=b}),
we see that $b(x;y)$ is psd if and only if $B(y)$ is \aaa{PSD}. }

Putting the blocks together, we have
\begin{equation}\label{eq:Hb}
H_b(x,y)=\begin{bmatrix} B(y) & C(x,y) \\ C^T(x,y) & A(x)
\end{bmatrix}.
\end{equation} The matrix $C(x,y)$ is not in general symmetric.
\aaan{The entries of $C(x,y)$ consist of square-free monomials
that are each a multiple of} $x_iy_j$ for some \jnt{$i$, $j$,
with} $1\leq i,j \leq n$; (see (\ref{eq:defn.biquad.form}) and
(\ref{eq:C(x,y).defn.})).
%If $b(x;y)$ is psd, even though $A(x)$ and
%$B(y)$ must be PSD, $H_b(x,y)$  in (\ref{eq:Hb}) may obviously not
%be PSD (or else every psd biquadratic form would be convex).

\aaa{The Hessian $H_g(x,y)$ of the polynomial $g(x,y)$ in
(\ref{eq:g(x,y).defn.}) is given by}

%Now, let us briefly comment on the polynomial $g(x,y)$ in
%(\ref{eq:g(x,y).defn.}) whose role will become clear shortly. One
%can check that $g(x,y)$ is convex.
%
%Its Hessian $H_g(x,y)$ is given
%by
\begin{equation}\label{eq:Hg}
H_g(x,y)=\frac{n^2\gamma}{2}\begin{bmatrix}H_g^{11}(x) & 0 \\ 0 &
H_g^{22}(y)
\end{bmatrix},
\end{equation}
where
\begin{equation}
H_g^{11}(x)=\begin{bmatrix}
12x_1^2+2\aaa{\displaystyle\sum_{\substack{i=1,\ldots,n\\i\neq 1}}} x_i^2 & 4x_1x_2 & \cdots & 4x_1x_n  \\
4x_1x_2 & 12x_2^2+2\aaa{\displaystyle\sum_{\substack{i=1,\ldots,n\\i\neq 2}}} x_i^2  & \cdots & 4x_2x_n \\
\vdots & \aaan{\vdots} &\ddots &\vdots \\
4x_1x_n &  \cdots &4x_{n-1}x_n & 12x_n^2+2\aaa{\displaystyle\sum_{\substack{i=1,\ldots,n\\i\neq n}}} x_i^2  \\
\end{bmatrix},
\end{equation} and \begin{equation}
H_g^{22}(y)=\begin{bmatrix}
12y_1^2+2\aaa{\displaystyle\sum_{\substack{i=1,\ldots,n\\i\neq 1}}} y_i^2 & 4y_1y_2 & \cdots & 4y_1y_n  \\
4y_1y_2 & 12y_2^2+2\aaa{\displaystyle\sum_{\substack{i=1,\ldots,n\\i\neq 2}}} y_i^2  & \cdots & 4y_2y_n \\
\vdots & \aaan{\vdots} &\ddots &\vdots \\
4y_1y_n &  \cdots &4y_{n-1}y_n & 12y_n^2+2\aaa{\displaystyle\sum_{\substack{i=1,\ldots,n\\i\neq n}}} y_i^2  \\
\end{bmatrix}.
\end{equation}
\aaa{Note that all diagonal elements of $H_g^{11}(x)$ \aaan{and}
$H_g^{22}(y)$ contain the square of every variable $x_1, \ldots,
x_n$ \aaan{and} $y_1, \ldots, y_n$ \aaan{respectively}.}

\aaa{\aaan{We fist give an intuitive summary of the rest of the
proof.} If $b(x;y)$ is not psd, then $B(y)$ and $A(x)$ are not PSD
and hence $H_b(x,y)$ is not PSD. Moreover, adding $H_g(x,y)$ to
$H_b(x,y)$ cannot help make $H(x,y)$ PSD because the dependence of
the diagonal blocks of $H_b(x,y)$ and $H_g(x,y)$ on $x$ and $y$
\aaan{runs} backwards. On the other hand, if $b(x;y)$ is psd, then
$H_b(x,y)$ will have PSD diagonal blocks. \aaan{In principle,
$H_b(x,y)$ might} still not be PSD because of the off-diagonal
block $C(x,y)$. However, the squares in the diagonal elements of
$H_g(x,y)$ \aaan{will be shown to} dominate the monomials of
$C(x,y)$ \aaan{and} make $H(x,y)$ PSD.}

Let us now prove the theorem formally. One direction is easy: if
$b(x;y)$ is not psd, then $f(x,y)$ is not convex. Indeed, if there
exist $\bar{x}$ and $\bar{y}$ in $\mathbb{R}^n$ such that
$b(\bar{x};\bar{y})<0$, then
$$z^TH(x,y)z \Big \vert_{z_x=0, x=\bar{x}, y=0,
z_y=\bar{y}}=\bar{y}^TA(\bar{x})\bar{y}=2b(\bar{x};\bar{y})<0.$$
%Alternatively, we could have also argued that $$z^TH(x,y)z \Big
%\vert_{z_x=\bar{x}, x=0, y=\bar{y},
%z_y=0}=\bar{x}^TB(\bar{y})\bar{x}=2b(\bar{x};\bar{y})<0.$$
%
%The argument we just gave explains to some extent our choice of
%the polynomial $g(x,y)$ in (\ref{eq:g(x,y).defn.}). The monomials
%in $g$ are chosen in a way that a copy of $A(x)$ (and $B(y)$) is
%always embedded in a principal block of the Hessian $H(x,y)$. So,
%even though $g$ is convex, if the biquadratic $b(x;y)$ is not psd,
%$f$ will never become convex no matter how much we scale up the
%coefficients of $g$. This gives us the luxury of being very lavish
%with the choice of $\gamma$ in (\ref{eq:scalar.gamma.defn.}) in
%order to make the Hessian of $g$ dominate the terms coming from
%$C(x,y)$, i.e., the off-diagonal block of the Hessian of $b(x;y)$.
%As we will see next, this is the key to making the other direction
%of the proof to work.

%
%It is not difficult to see that as long as $g(x,y)$ does not
%include a monomial that has total degree $2$ in $x$ and $2$ in
%$y$, the proceeding argument would still go through. So, in
%principle, we could have also included monomials like $x_i^3x_j,
%x_i^3y_j, x_iy_j^3, y_i^3y_j, x_ix_jx_kx_l, y_iy_jy_ky_l$ in $g$,
%but as we will see next, it turns out that they are not needed.

\aaan{For the converse,} suppose that $b(x;y)$ is psd; we will
prove that $z^TH(x,y)z$ is psd and hence $f(x,y)$ is convex. We
have
\begin{equation}\nonumber
\begin{array}{lll}
z^TH(x,y)z&=&z_y^TA(x)z_y+z_x^TB(y)z_x+2z_x^TC(x,y)z_y \\ \ &\ &\
\\ \ &\
&+\frac{n^2\gamma}{2}z_x^TH_g^{11}(x)z_x+\frac{n^2\gamma}{2}z_y^TH_g^{22}(y)z_y.
\end{array}
\end{equation}
Because $z_y^TA(x)z_y$ and $z_x^TB(y)z_x$ are psd by assumption
(see \jnt{(\ref{eq:yAy=b}) and} (\ref{eq:x.B.x=b})), it suffices
to show that $z^TH(x,y)z-z_y^TA(x)z_y-z_x^TB(y)z_x$ is psd. In
fact, we will show that $z^TH(x,y)z-z_y^TA(x)z_y-z_x^TB(y)z_x$ is
a sum of squares.

After \jnt{some} regrouping of terms we can write
\begin{equation}\label{eq:p1+p2+p3}
z^TH(x,y)z-z_y^TA(x)z_y-z_x^TB(y)z_x=p_1(x,y,z)+p_2(x,z_x)+p_3(y,z_y),
\end{equation}
where
\begin{equation}\label{eq:p1}
p_1(x,y,z)=2z_x^TC(x,y)z_y+n^2\gamma
\Big(\sum_{\aaa{i=1}}^nz_{x,i}^2\Big)\Big(\sum_{\aaa{i=1}}^nx_i^2\Big)+n^2\gamma\Big(\sum_{\aaa{i=1}}^nz_{y,i}^2\Big)\Big(\sum_{\aaa{i=1}}^ny_i^2\Big),
\end{equation}
\begin{equation}\label{eq:p2}
p_2(x,z_x)=n^2\gamma z_x^T\begin{bmatrix}
5x_1^2& 2x_1x_2 & \cdots & 2x_1x_n  \\
2x_1x_2 & 5x_2^2  & \cdots & 2x_2x_n \\
\vdots & \aaan{\vdots} &\ddots &\vdots \\
2x_1x_n &  \cdots &2x_{n-1}x_n & 5x_n^2 \\
\end{bmatrix}z_x,
\end{equation}
and
\begin{equation}\label{eq:p3}
p_3(y,z_y)=n^2\gamma z_y^T\begin{bmatrix}
5y_1^2 & 2y_1y_2 & \cdots & 2y_1y_n  \\
2y_1y_2 & 5y_2^2  & \cdots & 2y_2y_n \\
\vdots & \aaan{\vdots} &\ddots &\vdots \\
2y_1y_n &  \cdots &2y_{n-1}y_n & 5y_n^2 \\
\end{bmatrix}z_y.
\end{equation}

We show that (\ref{eq:p1+p2+p3}) is sos by showing that $p_1$,
$p_2$, and $p_3$ are each individually sos. % The argument for $p_2$
%and $p_3$ is of course identical.
\aaa{To see that $p_2$ is sos, simply note that we can rewrite it
as \alex{\[ p_2(x,z_x) = n^2 \gamma \left[ 3 \sum_{k=1}^n
z_{x,k}^2 x_k^2 + 2 \Big(\sum_{k=1}^n z_{x,k} x_k\Big)^2 \right].
\]} The argument for $p_3$ is of course identical.} \alex{To show that $p_1$ is sos,} we argue as follows. If we
multiply out \aaan{the first term} $2z_x^TC(x,y)z_y$, we
\jnt{obtain} a polynomial \alex{with monomials of the form}
%with at
%most $n^4$ monomials, all of the type
\begin{equation}\label{eq:typical.monomial}
\pm2\beta_{i,j,k,l}z_{x,k}x_iy_jz_{y,l},
\end{equation}
where $0\leq \beta_{i,j,k,l} \leq \gamma$, \alex{by the definition
of} $\gamma$. Since
\begin{equation}
\pm2\beta_{i,j,k,l}z_{x,k}x_iy_jz_{y,l}+\beta_{i,j,k,l}z_{x,k}^2x_i^2+\beta_{i,j,k,l}y_j^2z_{y,l}^2=\beta_{i,j,k,l}(z_{x,k}x_i\pm
y_jz_{y,l})^2,
\end{equation} \alex{by pairing up the terms of $2z_x^T C(x,y) z_y$
with \aaa{fractions of} the squared terms $z_{x,k}^2 x_i^2$ and
$z_{y,l}^2 y_j^2$, we get a sum of squares. Observe that there are
more than enough squares for each monomial of $2z_x^T C(x,y) z_y$
because each such monomial $\pm2\beta_{i,j,k,l}z_{x,k}x_i y_j
z_{y,l}$ occurs at most once, so that each of the terms $z_{x,k}^2
x_i^2$ and $z_{y,l}^2 y_j^2$ will be needed at most $n^2$ times,
each time with a coefficient of at most $\gamma$.} Therefore,
$p_1$ is sos, \alex{and this completes the proof.}
%To prove that $p_2$ is sos, note that it can be written as
%\begin{equation}
%p_2(x,z_x)=n^2\gamma u^TQu,
%\end{equation}
%where $u$ is the vector of monomials
%\begin{equation}
%u=(x_1z_{x,1}, x_2z_{x,2}, \ldots, x_nz_{x,n})^T,
%\end{equation}
%and
%\begin{equation}
%Q=\begin{bmatrix}
%5 & 2 & \cdots & 2 \\
%2 & 5  & \cdots & 2 \\
%\vdots & \cdots &\ddots &\vdots \\
%2&  \cdots &2& 5\\
%\end{bmatrix}.
%\end{equation}
%We observe that $Q=2\mathbf{J}_n+3\mathbf{I}_n$, where
%$\mathbf{J}_n$ is the all-ones matrix. So, the smallest eigenvalue
%of $Q$ is always $3$. Therefore, it is positive definite and has a
%Cholesky factorization $Q=L^TL$. Hence, $p_2(x,z_x)=n^2\gamma
%(Lu)^T(Lu)$ is sos.
\end{proof}
\aaa{We can now \aaan{complete the proof of} strong NP-hardness of
deciding convexity of quartic forms.
\begin{proof}[Proof of Theorem~\ref{thm:convexity.quartic.nphard}]
As we remarked earlier, deciding nonnegativity of biquadratic
forms is known to be strongly
NP-hard~\cite{Ling_et_al_Biquadratic}. Given such a biquadratic
form $b(x;y)$, we can construct the polynomial $f(x,y)$ as in
(\ref{eq:f(x,y).defn.}). Note that $f(x,y)$ has degree four and is
homogeneous. Moreover, the reduction from $b(x;y)$ to $f(x,y)$
\aaan{runs in polynomial time} as we are only adding to $b(x;y)$
$2n+2 \binom{n}{2}$ new monomials with coefficient
$\frac{n^2\gamma}{2}$, and the size of $\gamma$ is by definition
only polynomially larger than the size of any coefficient of
$b(x;y)$. Since by Theorem~\ref{thm:main.reduction} convexity of
$f(x,y)$ is equivalent to nonnegativity of $b(x;y)$, we conclude
that deciding convexity of quartic forms is strongly NP-hard.
\end{proof}}
\paragraph {An algebraic version of the reduction.} Before we
proceed further with our results, we \jnt{make} a slight detour
and present an algebraic \aaan{analogue} of this reduction,
\jnt{which relates} sum of squares biquadratic forms to sos-convex
polynomials. Both of these concepts are well-studied in the
literature, in particular in regards to their connection to
semidefinite programming; see,
e.g.,~\cite{Ling_et_al_Biquadratic},~\cite{AAA_PP_not_sos_convex_journal},
and references therein.

\begin{definition}\label{defn:sos-convex}
A polynomial $p(x)$, with its Hessian denoted by $H(x)$, is
\emph{sos-convex} if the polynomial $y^TH(x)y$ is a sum of squares
in variables (x;y).\footnote{Three other equivalent definitions of
sos-convexity are presented in the next chapter.}
\end{definition}

\begin{theorem}\label{thm:algeb.ver.of.reduction}
Given a biquadratic form \aaa{$b(x;y)$}, let $f(x,y)$ be the
quartic form defined as in (\ref{eq:f(x,y).defn.}). Then $b(x;y)$
is a sum of squares if and only if $f(x,y)$ is sos-convex.
\end{theorem}
\begin{proof}
The proof is very similar to the proof of
\aaa{Theorem~\ref{thm:main.reduction}} and is left to the reader.
\end{proof}

We will revisit Theorem~\ref{thm:algeb.ver.of.reduction} in the
next chapter when we study the connection between convexity and
sos-convexity.

%Perhaps of independent interest,
%\aaa{Theorems~\ref{thm:main.reduction}} and
%\ref{thm:algeb.ver.of.reduction} imply that our reduction gives an
%explicit way of constructing convex but not sos-convex quartic
%forms (see~\cite{AAA_PP_not_sos_convex_journal}), starting from
%any example of a psd but not sos biquadratic form
%(see~\cite{Choi_Biquadratic}).

\paragraph {Some NP-hardness \jnt{results, obtained} as corollaries.} NP-hardness of checking convexity of quartic
forms directly establishes NP-hardness\aaa{\footnote{\aaa{All of
our NP-hardness results in this chapter are in the strong sense.
For the sake of brevity, from now on we refer to strongly NP-hard
problems simply as NP-hard problems.}}}
 of \jnt{several problems of interest.}
Here, we mention a \jnt{few examples.}

\begin{corollary}\label{cor:nonnegativity.quartic.nphard}
It is NP-hard to decide \alex{nonnegativity of a homogeneous}
polynomial $q$ of degree four, \alex{of the form}
\begin{equation}\nonumber
q(x,y)=\frac{1}{2}p(x)+\frac{1}{2}p(y)-p\left(\textstyle{\frac{x+y}{2}}\right),
\end{equation}
for some homogeneous quartic polynomial $p$.
\end{corollary}
\begin{proof}
\jnt{Nonnegativity of $q$ is equivalent to convexity of $p$, and
the result} follows directly from
Theorem~\ref{thm:convexity.quartic.nphard}.% and
%Theorem~\ref{thm:lambda.0.5.enough}.
\end{proof}

\begin{definition}\label{def:basic.semialgeb.set}
A set $\mathcal{S}\subset\mathbb{R}^n$ is \emph{basic closed
semialgebraic} if it can be written as
\begin{equation}\label{eq:semialgebraic.set}
\mathcal{S}=\{x\in\mathbb{R}^n|\ f_i(x)\geq0,\ i=1,\ldots,m\},
\end{equation}
for some positive integer $m$ and some polynomials $f_i(x)$.
\end{definition}

\begin{corollary}\label{cor:semialgeb.set.nphard}
Given a basic closed semialgebraic set $\mathcal{S}$ as in
(\ref{eq:semialgebraic.set}), where at least one of the defining
polynomials $f_i(x)$ has degree four, it is NP-hard to decide
whether $\mathcal{S}$ is a convex set.
\end{corollary}
\begin{proof}
Given a quartic polynomial $p(x)$, consider the basic closed
semialgebraic set $$\mathcal{E}_p=\{(x,t)\in\mathbb{R}^{n+1}|\
t-p(x)\geq0\},$$ describing the epigraph of $p(x)$. Since $p(x)$
is convex if and only if its epigraph is a convex set, the result
follows.\footnote{Another proof of this corollary is given by the
NP-hardness of checking convexity of sublevel sets of quartic
polynomials (Theorem~\ref{thm:quasi.pseudo.quartic.nphard} in
Section~\ref{subsec:quasi.pseudo.hard.degrees}).}
\end{proof}

\paragraph {Convexity of polynomials of even degree larger than
four.} We end this section by extending our hardness result to
polynomials of higher degree.
\begin{corollary}\label{cor:convexity.higher.degree.hard}
It is NP-hard to check convexity of polynomials of any fixed even
degree $d\geq4$.
\end{corollary}
\begin{proof}
We have already established the result for polynomials of degree
four. Given such a degree four polynomial
$p(x)\mathrel{\mathop:}=p(x_1,\ldots,x_n)$ and an even degree
$d\geq6$, consider the polynomial $$q(x,x_{n+1})=p(x)+x_{n+1}^d$$
in $n+1$ variables. It is clear (e.g., from the block diagonal
structure of the Hessian of $q$) that $p(x)$ is convex if and only
if $q(x)$ is convex. The \jnt{result} follows.
\end{proof}

%\aaar{
%\begin{remark}\label{rmk:higher.degree.dont.do.homog}
%Corollary~\ref{cor:convexity.higher.degree.hard} does not
%establish NP-hardness of checking convexity for \emph{forms} of
%fixed even degree $d\geq 6$. If needed, such a refinement can be
%done. One possibility, which we just sketch, is to give a
%reduction from the problem of deciding nonnegativity of forms of
%fixed even degree $d\geq 4$. Given such a form $p(x)$, one can
%construct a form $q(x)$ of degree $d+2$ in such a way that $p(x)$
%is a diagonal element of the Hessian of $q(x)$, and $p(x)$ is
%nonnegative if and only if $q(x)$ is convex. A construction of
%this type, although for a different purpose, is given
%in~\cite[Thm. 5.11]{AAA_PP_table_sos-convexity}.
%\end{remark}
%}

\section{Complexity of deciding strict convexity and strong
convexity}\label{sec:strict.strong}

\subsection{Definitions and basics}\label{subsec:strict.strong.basics}

\begin{definition}
A function $f:\mathbb{R}^n\rightarrow\mathbb{R}$ is \emph{strictly
convex} if %its domain is a convex set and
for all $x\neq y$ %in the domain
 and all $\lambda \in (0,1)$, we have
\begin{equation}\label{eq:strict.convexity.defn.}
f(\lambda x+(1-\lambda)y)< \lambda f(x)+(1-\lambda)f(y).
\end{equation}
\end{definition}

\begin{definition}
A twice differentiable function
$f:\mathbb{R}^n\rightarrow\mathbb{R}$ is \emph{strongly convex} if
its Hessian $H(x)$ satisfies
\begin{equation}\label{eq:strong.convexity.defn}
H(x)\succeq mI,
\end{equation}
for a scalar $m>0$ and for all $x$.% in the domain of $f$.
\end{definition}

We have the standard implications
\begin{equation}\label{eq:implications.strong.strict.convex}
\mbox{strong convexity} \ \Longrightarrow \  \mbox{strict
convexity}\ \Longrightarrow \ \mbox{convexity},
\end{equation}
but \jnt{none of the converse implications is true.}
%the converse of neither implication is true.% even for
%polynomials. For example linear polynomials are convex but not
%strictly convex, and the univariate polynomial $x^4$ is strictly
%convex but not strongly convex.

\subsection{Degrees that are easy}\label{subsec:strict.strong.easy.degrees}
From the implications in
(\ref{eq:implications.strong.strict.convex}) and our previous
discussion, it is clear that odd degree polynomials can never be
strictly convex \aaan{or} strongly convex. \aaa{We cover the case
of quadratic polynomials in the following straightforward
proposition.}

\aaa{
\begin{proposition}%\label{thm:strict=strong.for.quadratic}
For a quadratic polynomial $p(x)=\frac{1}{2}x^TQx+q^Tx+c$, the
notions of strict convexity and strong convexity are equivalent,
and can be decided in polynomial time.
\end{proposition}
}
\begin{proof}
Strong convexity always implies strict convexity. For the
\jnt{reverse} direction, assume that $p(x)$ is not strongly
convex. In view of (\ref{eq:strong.convexity.defn}), this means
that the matrix $Q$ is not positive definite. If $Q$ has a
negative eigenvalue, $p(x)$ is not convex, let alone strictly
convex. If $Q$ has a zero eigenvalue, let $\bar{x}\neq 0$ be the
corresponding eigenvector. Then $p(x)$ restricted to the
\aaan{line} from the origin to $\bar{x}$ is linear and hence not
strictly convex.

To see that these properties can be checked in polynomial time,
note that $p(x)$ is strongly convex if and only if the symmetric
matrix $Q$ is positive definite. \aaan{By Sylvester's criterion,}
positive definiteness of an $n \times n$ symmetric matrix is
equivalent to positivity of its $n$ leading principal minors,
\alex{each of which can be computed in polynomial time.}
\end{proof}

%\begin{remark}
%A quadratic polynomial can certainly be convex but not strictly
%convex. An example is $(x_1+x_2)^2$.
%\end{remark}
%
%\begin{proposition}
%Strict convexity and strong convexity of quadratic polynomials can
%be decided in polynomial time.
%\end{proposition}

%\begin{proof}
%By the previous theorem, both notions are equivalent for quadratic
%polynomials. A quadratic polynomial $p(x)=\frac{1}{2}x^TQx+q^Tx+c$
%is strongly convex if and only if its Hessian, which is the constant
%symmetric matrix $Q$, is positive definite (see
%(\ref{eq:strong.convexity.defn})). Positive definiteness of an $n
%\times n$ symmetric matrix is equivalent to positivity of its $n$
%leading principal minors, \alex{each of which can be computed in polynomial
%time.}
%\end{proof}

\subsection{Degrees that are hard}\label{subsec:strict.strong.hard.degrees}
With little effort, we can extend our NP-hardness result in the
previous section to \jnt{address} strict convexity and strong
convexity.
%The key that makes things simple for us is the fact
%that our NP-hardness result in
%Theorem~\ref{thm:convexity.quartic.nphard} stands even when the
%polynomials are homogenous.

\begin{proposition}\label{thm:strong.convexity.NPhard.even.deg} %these were originally thms. that's the reason for the label.
It is NP-hard to decide strong convexity of polynomials of any
fixed even degree $d\geq4$.
\end{proposition}

\begin{proof}
\aaa{We give a reduction from the problem of deciding convexity of
quartic forms. Given a homogenous quartic polynomial
$p(x)\mathrel{\mathop:}=p(x_1,\ldots,x_n)$ and an even degree
$d\geq4$, consider the polynomial
\begin{equation}\label{eq:q.in.reduction.strong.convexity}
q(x,x_{n+1}):=p(x)+x_{n+1}^d+\textstyle{\frac{1}{2}}(x_1^2+\cdots+x_n^2+x_{n+1}^2)
\end{equation}
in $n+1$ variables. We claim that $p$ is convex if and only if $q$
is strongly convex. Indeed, if $p(x)$ is convex, then so is
$p(x)+x_{n+1}^d$. Therefore, the Hessian of $p(x)+x_{n+1}^d$ is
PSD. On the other hand, the Hessian of the term
$\frac{1}{2}(x_1^2+\cdots+x_n^2+x_{n+1}^2)$ is the identity
matrix. So, the minimum eigenvalue of the Hessian of
$q(x,x_{n+1})$ is positive and bounded \aaan{below by one}. Hence,
$q$ is strongly convex.

Now suppose that $p(x)$ is not convex. Let us denote the Hessians
of $p$ and $q$ respectively by $H_p$ and $H_q$. If $p$ is not
convex, then there exists a point $\bar{x}\in\mathbb{R}^n$ such
that
$$\lambda_{\min}(H_p(\bar{x}))<0,$$ where $\lambda_{\min}$ here
denotes the minimum eigenvalue. Because $p(x)$ is homogenous of
degree four, we have
$$\lambda_{\min}(H_p(c\bar{x}))=c^2\lambda_{\min}(H_p(\bar{x})),$$
for any scalar $c\in\mathbb{R}$. Pick $c$ large enough such that
$\lambda_{\min}(H_p(c\bar{x}))<1$. Then it is easy to see that
$H_q(c\bar{x},0)$ has a negative eigenvalue and hence $q$ is not
convex, let alone strongly convex.}
\end{proof}

\begin{remark}
It is worth noting that homogeneous polynomials of degree $d>2$
can never be strongly convex (because their Hessians vanish at the
origin). Not surprisingly, the polynomial $q$ in the proof of
\aaa{Proposition}~\ref{thm:strong.convexity.NPhard.even.deg} is
not homogeneous.
\end{remark}

\begin{proposition}
It is NP-hard to decide strict convexity of polynomials of any
fixed even degree $d\geq4$.
\end{proposition}

\begin{proof}
The proof is almost identical to the proof of
\aaa{Proposition}~\ref{thm:strong.convexity.NPhard.even.deg}. Let
$q$ be defined as in (\ref{eq:q.in.reduction.strong.convexity}).
If $p$ is convex, then we established that $q$ is strongly convex
and hence also strictly convex. If $p$ is not convex, we showed
that $q$ is not convex and hence also not strictly convex.
\end{proof}

\section{Complexity of deciding quasiconvexity and
pseudoconvexity}\label{sec:quasi.pseudo}
\subsection{Definitions and basics}\label{subsec:quasi.pseudo.basics}

\begin{definition}\label{def:quasiconvex.fn}
A function $f:\mathbb{R}^n\rightarrow\mathbb{R}$ is
\emph{quasiconvex} if \aaan{its sublevel sets}
\begin{equation}\label{eq:sublevel.sets}
\mathcal{S}(\alpha)\mathrel{\mathop:}=\{x\in\mathbb{R}^n \mid
f(x)\leq\alpha\},
\end{equation}for all $\alpha\in\mathbb{R}$, \aaan{are} convex.
\end{definition}

\begin{definition}\label{def:pseudoconvex.fn}
A differentiable function $f:\mathbb{R}^n\rightarrow\mathbb{R}$ is
\emph{pseudoconvex} if the implication
\begin{equation}\label{eq:pseudoconvexity.defn}
\nabla f(x)^T(y-x)\geq0  \ \Longrightarrow \ f(y)\geq f(x)
\end{equation}
holds for all $x$ and $y$ \aaa{in $\mathbb{R}^n$}.
\end{definition}

The following implications are well-known \aaa{(see e.g.~\cite[p.
143]{NLP_Book_Bazaraa})}:
\begin{equation}\label{eq:implications.convex.quasi.pseudo}
\mbox{convexity} \ \Longrightarrow \  \mbox{pseudoconvexity}\
\Longrightarrow \ \mbox{quasiconvexity},
\end{equation}
but the converse of neither implication is true in general.

\subsection{Degrees that are easy}\label{subsec:quasi.pseudo.easy.degrees}
\aaa{As we remarked earlier, linear polynomials are always convex
and hence also pseudoconvex and quasiconvex. Unlike convexity,
however, it is possible for polynomials of odd degree $d\geq3$ to
be pseudoconvex or quasiconvex.}
% In fact, we already presented simple
%cubic examples such as $x^3$ that was quasiconvex, and $x^3+x$
%that was pseudoconvex.
We will show in this section that \alex{somewhat} surprisingly,
quasiconvexity and pseudoconvexity of polynomials of any fixed odd
degree can be decided in polynomial time. Before we present these
results, we will \alex{cover} the easy case of quadratic
polynomials.

\aaa{
\begin{proposition}\label{thm:conv=quasi=pseudo.for.quadratics}
For a quadratic polynomial $p(x)=\frac{1}{2}x^TQx+q^Tx+c$, the
notions of convexity, pseudoconvexity, and quasiconvexity are
equivalent, and can be decided in polynomial time.
\end{proposition}
}
\begin{proof}
We \alex{argue} that the quadratic polynomial \alex{$p(x)$} is
convex if and only if
it is quasiconvex. %In view of implication
%(\ref{eq:implications.convex.quasi.pseudo}), this also shows the
%equivalence to pseudoconvexity.
\alex{Indeed, if $p(x)$ is not convex, then $Q$ has a negative
eigenvalue; letting $\bar{x}$ be \aaan{a} corresponding
eigenvector, we have that $p(t \bar{x})$ is a \jnt{quadratic}
polynomial in $t$, with negative leading coefficient, so $p(t
\bar{x})$ is not quasiconvex, \jnt{as a function of $t$.} This,
however, implies that $p(x)$ is not quasiconvex.}

\aaa{We have already argued in
Section~\ref{subsec:convexity.easy.degrees} that convexity of
quadratic polynomials can be decided in polynomial time.}
%Convex functions are always quasiconvex. For the converse
%direction, assume that $p(x)=\frac{1}{2}x^TQx+q^Tx+c$ is not
%convex. This means that the symmetric matrix $Q$ has at least one
%negative eigenvalue. Let $\bar{x}\neq0$ be an eigenvector
%corresponding to a negative eigenvalue of $Q$. So, we have
%$\bar{x}^TQ\bar{x}<0$. Fix any $\epsilon>0$. Because the quadratic
%term in $p(x)$ dominates the linear and constant term for $||x||$
%large, we can find a scalar $b$ large enough such that
%$p(b\bar{x})<c-\epsilon$ and $p(-b\bar{x})<c-\epsilon$. On the
%other hand, the origin is on the line between $b\bar{x}$ and
%$-b\bar{x}$ and satisfies $p(0)=c$. This shows that the sublevel
%set $\mathcal{S}_p(c-\epsilon)$ is not convex, and hence $p(x)$ is
%not quasiconvex.
\end{proof}

\subsubsection{Quasiconvexity of polynomials of odd degree}
\label{subsubsec:quasi.odd.degree} In this subsection, we provide
a polynomial time algorithm for checking whether an odd-degree
polynomial is quasiconvex. \aaan{Towards this goal, we will first
show that quasiconvex polynomials of odd degree have a very
particular structure
(Proposition~\ref{prop:pseudo.odd.degree.special.struc.}).}

%We will begin by arguing that the set of quasiconvex polynomials
%of odd degree is quite small; in particular, we will show that any
%such polynomial has a very particular representation.

Our first lemma concerns quasiconvex polynomials of odd degree in
\aaan{one} variable. \aaar{The proof is easy and left to the
reader. A version of this lemma is provided in~\cite[p.\
99]{BoydBook}, though there also without proof.}

% \jntn{A version of this lemma is provided
%in~\cite[p.\ 99]{BoydBook}, without proof. We provide here a
%proof, for completeness.}

\begin{lemma} \label{lemma:quasiconvex-odd-degree} Suppose that
$p(t)$ is a quasiconvex univariate polynomial of odd degree. Then,
$p(t)$ is monotonic.
\end{lemma}

%%AAA:commenting out the proof because of the comment by the referee.
%\begin{proof} Suppose that monotonicity does not hold for the points $a < b < c$. This can happen
%in two ways: \[ p(a)<p(b)\ \ {\rm and}\ \ p(c)<p(b),\] or \[ p(a)
%> p(b)\ \ {\rm and}\ \   p(c)>p(b).
%\] In the first case, pick $\alpha$ such that $p(b) > \alpha >
%\max(p(a),p(c))$. The \aaan{sublevel} set $\{x\mid p(x) \leq
%\alpha \}$ includes $a$ and $c$, but not $b$, and is therefore
%non-convex, which contradicts the quasiconvexity of $p$.
%
%In the second case, pick $\alpha$ such that $\min(p(a),p(c)) >
%\alpha > p(b)$, and consider the \aaan{sublevel} set
%$\aaan{\mathcal{S}(\alpha)} = \{ x \mid p(x) \leq \alpha \}$. If
%the leading coefficient of $p$ is positive, then
%$\aaan{\mathcal{S}(\alpha)}$ contains contains $b$ and some point
%to the left of $a$ (since $\lim_{t \rightarrow -\infty} p(t) =
%-\infty$) but not $a$, which is a contradiction. Similarly, if the
%leading coefficient of $p$ is negative, then
%$\aaan{\mathcal{S}(\alpha)}$ contains $b$ and some point to the
%right of $c$ but not $c$, which is a contradiction.
%\end{proof}

Next, we use the \aaan{preceding lemma} to characterize the
complements of \aaan{sublevel} sets of quasiconvex polynomials of
odd degree.

\begin{lemma} \label{lemma:halfspaces} Suppose that $p(x)$
is a quasiconvex polynomial of odd degree $d$. Then the set $\{ x
\mid  p(x) \geq \alpha \}$ is convex.
\end{lemma}

\begin{proof} Suppose not. In that case, there exist $x,y,z$ such that
$z$ is on the line segment connecting $x$ and $y$, \jntn{and such
that} $p(x), p(y) \geq \alpha$ but $p(z) < \alpha$. Consider the
polynomial
\[ q(t)=p(x + t(y-x)).\] This is, of course, a quasiconvex
polynomial with $q(0)=p(x)$, $q(1)=p(y)$, and $q(t')=p(z)$, for
some $t' \in (0,1)$. If $q(t)$ has degree $d$, then, by Lemma
\ref{lemma:quasiconvex-odd-degree}, it must be monotonic, which
immediately provides a contradiction.

Suppose now that $q(t)$ has degree less than $d$. Let us attempt
to perturb $x$ to $x+x'$, and $y$ to $y+y'$, so that the new
polynomial \[ \hat{q}(t)=p \left( x+x' + t(y+y'-x-x') \right)\]
has the following two properties: (i) $\hat{q}(t)$ is a polynomial
of degree $d$, and (ii) $\hat{q}(0)
> \hat{q}(t')$, $\hat{q}(1) > \hat{q}(t')$. If such perturbation
vectors $x', y'$ can be found, then we obtain a contradiction as
in the previous paragraph.

To satisfy condition (ii), it suffices (by continuity) to take
$x',y'$ with $\|x'\|, \|y'\|$ small enough. Thus, we \jntn{only}
need to argue that we can find arbitrarily small $x',y'$ that
satisfy condition (i). Observe that  the \jntn{coefficient of
$t^d$ in the polynomial} $\hat{q}(t)$ is a nonzero polynomial in
$x+x', y+y'$; let us denote \jntn{that coefficient} as
$\jntn{r}(x+x',y+y')$. \jntn{Since $r$ is a nonzero polynomial, it
cannot vanish at all points of any given ball. Therefore, even
when considering a small ball around $(x,y)$ (to satisfy condition
(ii)), we can find $(x+x',y+y')$ in that ball, with
$r(x+x',y+y')\neq 0$, thus establishing that the degree of $\hat
q$ is indeed $d$.} This completes the proof.

\old{Since the only polynomial which is zero on a ball of positive
radius is the zero polynomial, and $l$ is nonzero, we have that no
matter how small a ball ${\cal B}$ \aaa{we} take around $(x,y)$
there are some $x',y'$ with $(x+x',y+y') \in {\cal B}$ and
$l(x+x', y+y') \neq 0$. \aaa{[* Can you rephrase this argument?
Right now it is a bit confusing because it says that we know the
leading coefficient is non-zero but then for any ball we can take
the leading coefficient to be non-zero .... ?? *]} This completes
the proof.}
\end{proof}

%\jntn{We now proceed to the desired representation of quasiconvex
%polynomials of odd degree.}

\jntn{We now proceed to \aaan{a characterization} of quasiconvex
polynomials of odd degree.}

%\jntn{CHANGED THE NEXT THM TO A PROPOSITION}

%%%%%%%% OLD
\old{ \alex{Of course, by definition, the closed set
$\{\aaa{x\mid} p(x) \leq \alpha \}$ is convex for a quasiconvex
$p(x)$; the previous lemma says that the \aaa{closure of the}
complement of this set is also convex. It is a simple exercise to
see that \aaa{if a closed set is convex, and the closure of its
complement is also convex, then the set must be a halfspace.}
\aaa{[* Is the last sentence OK? *]} Thus, we conclude that every
\aaa{sub}level set $\{\aaa{x\mid} p(x) \leq \alpha\}$ is a
halfspace, i.e., it can be written as $\{\jnt{x\mid} q^T x \leq
c\}$ for some $q \in \mathbb{R}^n, c \in \mathbb{R}$. This of
course implies that the set $\{x\mid p(x) = \alpha \}$ has the
representation $\{x\mid q^T x = \aaa{c}\}$.}

\alex{Moreover, we can say something \aaa{more} about the normal
vectors of these halfspaces. Suppose that the set $\{x\mid p(x) =
\aaa{\alpha_1} \}$ can be written as $\{x\mid q_1^T x = c_1 \}$
and that the set $\{x\mid p(x) = \aaa{\alpha_2} \}$ can be written
as $\{x \mid q_2^T x = c_2 \}$. Then $q_1$ and $q_2$ are multiples
of each other. Indeed, this follows from the definition of a
\aaa{sub}level set: if $q_1, q_2$ were not multiples of each
other, one could find points which are in exactly one of the two
sets $\{x\mid p(x) \leq \aaa{\alpha_1}\}, \{x\mid p(x) \leq
\aaa{\alpha_2}\}$. But this cannot be because one of these sets
must be contained in the other. \aaa{[* Do the last two sentences
need rephrasing? Even if one set is contained in another, one
could find points that are exactly in one of the two sets. *]} }

\alex{Let us then pick one of these normal vectors $q$
arbitrarily, i.e., let $\{x\mid q^T x = c \}$ be a representation
of the level set $\{x\mid p(x) = 1 \}$. The discussion in the
previous paragraph implies that $p(x)$ is completely determined by
$q^T x$, i.e., if $q^T x_1 = q^T x_2$ then $p(x_1)=p(x_2)$. }

\alex{These observations are now sufficient to establish the
following theorem.} }
%%%%%%%%%  END OLD

\begin{proposition}\label{prop:pseudo.odd.degree.special.struc.} %label should be prop:quasi.... but i didn't change it.
Let $p(x)$ be a polynomial of odd degree $d$. Then, $p(x)$ is
quasiconvex if and only if it can be written as
\begin{equation} \label{quasiconvexdecomposition}
p(x)=h(\xi^Tx),
\end{equation}
for some \jntn{nonzero} $\xi \in \mathbb{R}^n$, and for some
\jntn{monotonic univariate} polynomial \jntn{$h(t)$ of degree $d$.
If, in addition, we require the nonzero component of $\xi$ with
the smallest index to be equal to unity, then $\xi$ and $h(t)$ are
uniquely determined by $p(x)$.}
\end{proposition}

\begin{proof} \old{\jntn{PROOF REWRITTEN}} It is easy to see that any polynomial
that can be written in the above form is quasiconvex. In order to
prove the converse, let us assume that $p(x)$ is quasiconvex. By
the definition of quasiconvexity, the closed set
$\aaan{\mathcal{S}}(\alpha)=\{x\mid p(x) \leq \alpha \}$ is
convex. On the other hand, Lemma \ref{lemma:halfspaces} states
that the closure of the complement of $\aaan{\mathcal{S}}(\alpha)$
is also convex. It is not hard to verify that, as a consequence of
these two properties, the set $\aaan{\mathcal{S}}(\alpha)$ must be
a halfspace. Thus, for any given $\alpha$, the sublevel set
$\aaan{\mathcal{S}}(\alpha)$ can be written as $\{x\mid
\xi(\alpha)^T x \leq c\aaan{(\alpha)}\}$ for some $\xi(\alpha) \in
\mathbb{R}^n$ and $ c(\alpha) \in \mathbb{R}$. This of course
implies that the level sets $\{x\mid p(x) = \alpha \}$ are
hyperplanes of the form $\{x\mid \xi(\alpha)^T x = c(\alpha)\}$.

We note that the sublevel sets are necessarily nested: if
$\alpha<\beta$, then $\aaan{\mathcal{S}}(\alpha)\subseteq
\aaan{\mathcal{S}}(\beta)$. An elementary consequence of this
property is that the hyperplanes must be collinear, i.e., that the
vectors $\xi(\alpha)$ must be positive multiples of each other.
Thus, by suitably scaling the coefficients $c(\alpha)$, we can
assume, without loss of generality, that $\xi(\alpha)=\xi$, for
some $\xi \in \mathbb{R}^n$, and for all $\alpha$. We then have
that $\{x\mid p(x)=\alpha\}=\{x\mid \xi^T x= c(\alpha)\}$.
Clearly, there is a one-to-one correspondence between $\alpha$ and
$c(\alpha)$, and therefore the value of $p(x)$ is completely
determined by $\xi^T x$. In particular, there exists a function
$h(t)$ such that $p(x)=h(q^Tx)$. Since $p(x)$ is a polynomial of
degree $d$, it follows that $h(t)$ is a univariate polynomial of
degree $d$. Finally, we observe that if $h(t)$ is not monotonic,
then $p(x)$ is not quasiconvex. This proves that a representation
of the desired form exists. Note that by suitably scaling $\xi$,
we can also impose the condition that the nonzero component of
$\xi$ with the smallest index is equal to one.

Suppose that now that $p(x)$ can also be represented in the form
$p(x)=\bar h(\bar \xi^T x)$ for some other polynomial $\bar h(t)$
and vector $\aaan{\bar{\xi}}$. Then, the gradient vector of $p(x)$
must be proportional to both $\xi$ and $\bar\xi$. The vectors
$\xi$ and $\bar \xi$ are therefore collinear. Once we impose the
requirement that the nonzero component of $\xi$ with the smallest
index is equal to one, we obtain that $\xi=\bar\xi$ and,
consequently, $h=\bar h$. This establishes the claimed uniqueness
of the representation.
\end{proof}
%%%%%%%% OLD
\old{ \alex{\noindent {\it \aaa{Remark.}} Observe that if Eq.\
(\ref{quasiconvexdecomposition}) holds for one \jnt{choice of}
$q,h$, then it holds for infinitely many \jnt{choices of $q,h$.}
For example, one can multiply every entry of the vector $q$ by
some number $c$ while multiplying the $i$th coefficient of $h$ by
$1/c^i$. If the original $h$ is monotonic, the new $h$ will be as
well, since it can be written as $h(x/c)$. In particular, we may
assume without loss of generality that $h$ is monic.
Alternatively, we may assume without loss of generality that
$q_1=1$; this may be assured by the above transformation unless it
happens that $q_1=0$, in which case $p(x)$ does not depend on
$x_1$.} }
%%%%%%%%% END OLD

\bigskip
\noindent {\it Remark.} It is not hard to see that if $p(x)$ is
homogeneous and quasiconvex, then one can additionally conclude
that $h(t)$ can be taken to be $h(t)=t^d$, where $d$ is the degree
of $p(x)$.

\begin{theorem}\label{thm:quasi.odd.degree.poly.time}
\jntn{For any fixed odd degree $d$, the} quasiconvexity of
polynomials of degree $d$ can be checked in polynomial time.
\end{theorem}

\begin{proof} \old{\jntn{NEW PROOF}}
The algorithm consists of attempting to build a representation of
$p(x)$ of the form given in Proposition
\ref{prop:pseudo.odd.degree.special.struc.}. The polynomial $p(x)$
is quasiconvex if and only if the attempt is successful.

Let us proceed under the assumption that $p(x)$ is quasiconvex. We
differentiate $p(x)$ symbolically to obtain its gradient vector.
Since a representation of the form given in Proposition
\ref{prop:pseudo.odd.degree.special.struc.} exists, the gradient
is of the form $\nabla p(x) =\xi h'(\xi^T x)$, where $h'(t)$ is
the derivative of $h(t)$. In particular, the different components
of the gradient are polynomials that are proportional to each
other. (If they are not proportional, we conclude that $p(x)$ is
not quasiconvex, and the algorithm terminates.) By considering the
ratios between different components, we can identify the vector
$\xi$, up to a scaling factor. By imposing the additional
requirement that the nonzero component of $\xi$ with the smallest
index is equal to one, we can identify $\xi$ uniquely.

We now proceed to identify the polynomial $h(t)$. For
$k=1,\ldots,d+1$, we evaluate $p(k\aaan{\xi})$, which must be
equal to $h(\xi^T\xi  k)$. We thus obtain the values of $h(t)$ at
$d+1$ distinct points, from which $h(t)$ is completely determined.
We then verify that $h(\xi^T x)$ is indeed equal to $p(x)$. This
is easily done, in polynomial time, by writing out the $O(n^d)$
coefficients of these two polynomials in $x$ and verifying that
they are equal. (If they are not all equal, we conclude that
$p(x)$ is not quasiconvex, and the algorithm terminates.)

Finally, we test whether the above constructed univariate
polynomial $h$ is monotonic, i.e., whether its derivative $h'(t)$
is either nonnegative or nonpositive. This can be accomplished,
e.g., by quantifier elimination or by other well-known algebraic
techniques for counting the number and the multiplicity of real
roots of univariate polynomials;
see~\cite{Algo_real_algeb_geom_Book}. Note that this requires only
a constant number of arithmetic operations since the degree $d$ is
fixed. If $h$ fails this test, then $p(x)$ is not quasiconvex.
Otherwise, our attempt has been successful and we decide that
$p(x)$ is indeed quasiconvex.
\end{proof}
%%%%%%%% OLD
\old{
\begin{proof} \alex{Our input is a \jnt{a degree $d$ polynomial $p(x)$.} We will assume, without loss of generality, that $p(x)$ depends on each of the variables $x_1, \ldots, x_n$; we can ensure
that this is the case by going through all the monomials of $p$
and throwing out variables which do not appear in at least one
monomial.}

\alex{We attempt to find some $h,q$ that satisfy Eq.
(\ref{quasiconvexdecomposition}), with the additional proviso that
$q_1=1$. If we fail, then \jntn{Proposition}
\ref{prop:pseudo.odd.degree.special.struc.} and the ensuing Remark
allow us to conclude that the polynomial $p(x)$ is not
quasiconvex. If we succeed, and we are able to additionally check
that $h$ is monotonic, we can conclude that $p(x)$ is indeed
quasiconvex. }

\alex{We will adopt the notation \aaa{$c(i_1,i_2,\ldots,i_n)$} for
the coefficient of \aaa{$x_1^{i_1} x_2^{i_2}\cdots x_n^{i_n}$} in
$p(x)$. \aaa{[* This notation is inconsistent with the notation of
Section~\ref{subsec:convexity.basics}. *]} Our first observation
is that Eq. (\ref{quasiconvexdecomposition}) cannot hold if
$c(d,0, \ldots,0)=0$, in which case we can simply output that
$p(x)$ is not quasiconvex. Indeed, if $c(d,0, \ldots,0)=0$, then
if Eq. (\ref{quasiconvexdecomposition}) holds, $q_1=0$, and $p(x)$
is independent of $x_1$, which we assumed is not the case.}

\alex{Assume, then, that $c(d,0, \ldots, 0)$ is not zero. We
observe that, by the binomial formula, the numbers $q_2, \ldots,
q_n$ satisfy:}

\alex{\[ d q_2 = \frac{c(d-1,1,0,\ldots,0)}{c(d,0,\ldots,0)},\]
and \[ d q_3 = \frac{c(d-1,0,1,0\ldots,0)}{c(d,0,\ldots,0)},\] and
similarly \[ d q_k =
\frac{c(d-1,0,\ldots,0,1,0,\ldots,0)}{c(d,0,\ldots,0)},\] where
the $1$ in $c(d-1,0,\ldots,0,1,0,\ldots,0)$ is located in the
$k$th position. Moreover, writing $h(t)=\sum_{i=0}^d \jnt{h_i}
t^{i}$, the coefficients of $h$ can be read off \jnt{from
$h_0=c(0,0,\ldots,0)$ and}
\[ \jnt{h_i} = c(i,0,\ldots,0),\jnt{\qquad i=1,\ldots,\aaa{d}.}\] Thus if Eq. (\ref{quasiconvexdecomposition})
holds with $q_1=1$, it must hold with the $q_2, \ldots, q_n$ and
$h(t)$ satisfying the above equations.}

\alex{To complete our goal of checking whether $p(x)$ can be
written as $h(1 \cdot x_1 + q_2 x_2 + \cdots + q_n x_n)$, we must
check that each coefficient of $p(x)$ equals to the corresponding
coefficient of $h(q^T x)$.   To do this, we examine each
coefficient \aaa{[* Why should we check all coefficients? Wouldn't
it be less ambiguous to the reader if we say that we solve for
$q_i$ and $h_i$ based on some coefficients and then check that the
rest of the coefficients match? *]}
\[ c(i_1,i_2,\ldots,i_n)
\] and check that it equals
\[ \jnt{h_{i_1+\cdots+i_n}}
%h(\sum_{k=1}^n i_k)
~q_1^{i_1} \cdots q_n^{i_n}  \frac{ (i_1 + \cdots \aaa{+}
i_n)!}{i_1! i_2! \cdots i_n!}. \]} If any of these conditions is
violated, we output that $p(x)$ is not quasiconvex. If all of them
are satisfied, we proceed to the following step.

\alex{Now that we have established the decomposition $p(x)=h(q^T
x)$, we need to check whether the univariate polynomial $h$ is
monotonic, i.e., that its derivative $h'$ is either nonnegative or
nonpositive. This may be accomplished \aaa{e.g.} by quantifier
elimination \aaa{or by other well-known algebraic techniques for
counting the number and the multiplicity of real roots of
univariate polynomials; see~\cite{Algo_real_algeb_geom_Book}. Note
that this step of the algorithm takes constant time since the
degree $d$ is fixed.} We output that $p$ is quasiconvex if $h$ is
monotonic, and that it isn't otherwise.}
\end{proof}
}
%%%%%%%%% END OLD
\old{ \jntn{NEED TO CHANGE $q$ to $\xi$ IN THE NEXT SUBSECTION}}

\subsubsection{Pseudoconvexity of polynomials of odd degree}
\aaan{In analogy to
Proposition~\ref{prop:pseudo.odd.degree.special.struc.}, we
present next a characterization of odd degree pseudoconvex
polynomials, which gives rise to a polynomial time algorithm for
checking this property.}

\alex{\begin{corollary}\label{thm:quasi.odd.degree.special.struc.} %why is this the wrong label?
Let $p(x)$ be a polynomial of odd degree $d$. Then, $p(x)$ is
pseudoconvex if and only if $p(x)$ can be written in the form
\begin{equation} \label{pseudoconvexrepresentation}
p(x)=h(\aaan{\xi}^T x),
\end{equation}
for some $\aaan{\xi} \in \mathbb{R}^n$ and some univariate
polynomial $h$ of degree $d$ such that \aaa{its derivative}
$h'(t)$ has no real roots.
\end{corollary}}

\alex{ \noindent {\it \aaa{Remark.}} Observe that polynomials $h$
with $h'$ having no real roots \jnt{comprise a subset of the set}
of monotonic polynomials.}

\begin{proof} \alex{\jnt{Suppose that $p(x)$ is pseudoconvex.}
Since a pseudoconvex polynomial is quasiconvex, it admits a
representation $h(\aaan{\xi}^T x)$ where $h$ is monotonic. If
$h'(t)=0$ for some $t$, then picking $a = t \cdot
\aaan{\xi}/\|\aaan{\xi}\|_2^2$, we have that $\nabla p(a)=0$, so
that by pseudoconvexity, $p(x)$ is minimized at $a$. This,
however, is impossible \aaa{since an odd degree polynomial is
never bounded below}.}
% since as in the proof of Lemma
%\ref{lemma:halfspaces}, we can find a line $u+tv$ such that the
%univariate polynomial $p(u+tv)$ has odd degree, and hence $\inf_{x
%\in \mathbb{R}^n} p(x)=-\infty$.}
\alex{Conversely, suppose $p(x)$ can be represented as in Eq.
(\ref{pseudoconvexrepresentation}). \jnt{Fix some} $x,y$, and
define the polynomial $\aaa{u}(t)=p(x+t(y-x))$. Since
$\aaa{u}(t)=h(\aaan{\xi}^T x + t \aaan{\xi}^T (y-x))$, we have
that either (i) $\aaa{u}(t)$ is constant, or (ii) $\aaa{u}'(t)$
has no real roots. Now if $\nabla p(x)(y-x) \geq 0$, then
$\aaa{u}'(0) \geq 0$. Regardless of whether (i) or (ii) holds,
this implies that $\aaa{u}'(t) \geq 0$ everywhere, so that
$\aaa{u}(1) \geq \aaa{u}(0)$ or $p(y) \geq p(x)$.}
 \end{proof}
%\smallskip
%
%\label{subsubsec:pseudo.odd.degree} We now describe a polynomial
%time algorithm for deciding pseudoconvexity of odd degree
%polynomials.
\begin{corollary}\label{thm:pseudo.odd.degree.poly.time}
\aaan{For any fixed odd degree $d$, the} pseudoconvexity of
polynomials of degree $d$ can be checked in polynomial time.
\end{corollary}
\begin{proof} \alex{This is a simple modification of our algorithm
for testing quasiconvexity
(Theorem~\ref{thm:quasi.odd.degree.poly.time}). \aaan{The first
step of the algorithm is in fact identical: once we impose the
additional requirement that the nonzero component of $\xi$ with
the smallest index should be equal to one, we can uniquely
determine the vector $\xi$ and the coefficients of the univariate
polynomial $h(t)$ that satisfy Eq.
(\ref{pseudoconvexrepresentation}) . (If we fail, $p(x)$ is not
quasiconvex and hence also not pseudoconvex.) Once we have $h(t)$,
we can check whether $h'(t)$ has no real roots e.g. by computing
the signature of the Hermite \aaa{form of $h'(t)$;
see~\cite{Algo_real_algeb_geom_Book}}.}}

\end{proof}

\begin{remark}
Homogeneous polynomials of odd degree $d\geq3$ are never
pseudoconvex. The reason is that the gradient of these polynomials
vanishes at the origin, but yet the origin is not a global minimum
since odd degree polynomials are unbounded below.
\end{remark}

\subsection{Degrees that are hard}\label{subsec:quasi.pseudo.hard.degrees}
The main result of this section is the following theorem.

\begin{theorem}\label{thm:quasi.pseudo.quartic.nphard}
It is NP-hard to check quasiconvexity/pseudoconvexity of degree
four polynomials. This is true even when the polynomials are
restricted to be homogeneous.
\end{theorem}

In view of Theorem~\ref{thm:convexity.quartic.nphard}, which
established NP-hardness of deciding convexity of
\emph{homogeneous} quartic polynomials,
Theorem~\ref{thm:quasi.pseudo.quartic.nphard} follows immediately
from the following result.

%\footnote{A slight variant of
%Theorem~\ref{thm:quasiconvexity.homog.same.convexity} has appeared
%in~\cite{AAA_PP_CDC10_algeb_convex}.}

\begin{theorem}\label{thm:quasiconvexity.homog.same.convexity}
For a homogeneous polynomial $p(x)$ of even degree $d$, the
notions of convexity, pseudoconvexity, and quasiconvexity are all
equivalent.\footnote{The result is more generally true for
differentiable functions that are homogeneous of even degree.
Also, the requirements of homogeneity and having an even degree
both need to be present. Indeed, $x^3$ and $x^4-8x^3+18x^2$ are
both quasiconvex but not convex, the first being homogeneous of
odd degree and the second being nonhomogeneous of even degree.}
\end{theorem}

We start the proof of this theorem by first proving an easy lemma.

\begin{lemma} \label{lem:quasiconvex.homog.then.psd}
Let $p(x)$ be a quasiconvex homogeneous polynomial of even degree
\aaan{$d\geq2$}. Then $p(x)$ is nonnegative.
\end{lemma}

\begin{proof}
Suppose, to derive a contradiction, that there exist some
$\epsilon>0$ and $\bar{x}\in\mathbb{R}^n$ such that
$p(\bar{x})=-\epsilon$. Then by homogeneity of even degree we must
have $p(-\bar{x})=p(\bar{x})=-\epsilon$. On the other hand,
homogeneity of $p$ implies that $p(0)=0$. Since the origin is on
the line between $\bar{x}$ and $-\bar{x}$, this shows that the
\aaan{sublevel set} $\mathcal{S}(-\epsilon)$ is not convex,
contradicting the quasiconvexity of $p$.
\end{proof}

\begin{proof}[Proof of
Theorem~\ref{thm:quasiconvexity.homog.same.convexity}] \alex{We
show that a quasiconvex homogeneous polynomial of even degree is
convex. In view of implication
(\ref{eq:implications.convex.quasi.pseudo}), this \alex{proves the
theorem.}}

\alex{Suppose that $p(x)$ is a quasiconvex polynomial. Define
$\mathcal{S}=\{x\in\mathbb{R}^n \mid\  p(x)\leq1\}$. By
homogeneity, for any $a \in \mathbb{R}^n$ with $p(a)>0$, we have
that \[ \frac{a}{p(a)^{1/d}} \in \mathcal{S}.\] By quasiconvexity,
this \jnt{implies} that for any $a,b$ with $p(a),p(b)>0$, any
point on the line connecting $a/p(a)^{1/d}$ and $b/p(b)^{1/d}$ is
in $\mathcal{S}$. In particular, consider
\[ c=\frac{a+b}{p(a)^{1/d}+p(b)^{1/d}}.\]  Because $c$ can be written as
\[ %\frac{a+b}{p(a)^{1/d} + p(b)^{1/d}}
c = \left( \frac{p(a)^{1/d}}{p(a)^{1/d} + p(b)^{1/d}}\right)
\left( \frac{a}{p(a)^{1/d}} \right) + \left(
\frac{p(b)^{1/d}}{p(a)^{1/d} + p(b)^{1/d}} \right) \left(
\frac{b}{p(b)^{1/d}} \right),\] we have that $c \in \mathcal{S}$,
\jnt{i.e., $p(c)\leq 1$. By homogeneity, this  inequality} can be
restated as
\[ p(a+b) \leq (p(a)^{1/d} + p(b)^{1/d})^d,\] and therefore
\begin{equation} p\Big(\frac{a+b}{2}\Big) \leq \left( \frac{p(a)^{1/d} + p(b)^{1/d}}{2}\right) ^d \leq \frac{p(a)+p(b)}{2}, \label{convexity-where-positive} \end{equation}
where the last inequality is due to the convexity of $x^d$.}

\alex{Finally, note that for any polynomial $p$, the set
\jnt{$\{x\mid p(x) \neq 0$\}} is dense in $\mathbb{R}^n$ (here we
again appeal to the fact that the only polynomial that is zero on
a ball of positive radius is the zero polynomial); and since $p$
is nonnegative due to Lemma \ref{lem:quasiconvex.homog.then.psd},
the set \jnt{$\{x\mid p(x)
> 0$\}} is dense in $\mathbb{R}^n$.
\jnt{Using the continuity of $p$, it follows that Eq.
(\ref{convexity-where-positive}) holds not only when $a,b$ satisfy
$p(a),p(b)>0$, but for all $a$, $b$. Appealing to the continuity
of $p$ again, we see that for all $a,b$, $p(\lambda a +
(1-\lambda) b) \leq \lambda p(a) + (1-\lambda) p(b)$, for all
$\lambda \in [0,1]$. This establishes} that $p$ is convex.}

\end{proof}

\paragraph{Quasiconvexity/pseudoconvexity of polynomials of even degree larger than four.}

\begin{corollary}\label{cor:quasi.nphard.d>=4}
It is NP-hard to decide quasiconvexity of polynomials of any fixed
even degree $d\geq4$.
\end{corollary}

\begin{proof}\label{thm:quasi.nphard.even.degree}
We have already proved the result for $d=4$. To establish the
result for even degree $d\geq6$, recall that we have established
NP-hardness of deciding convexity of homogeneous quartic
polynomials. Given such a quartic form
$p(x)\mathrel{\mathop:}=p(x_1,\ldots,x_n)$, consider the
polynomial
\begin{equation}\label{eq:q.in.reduction.quasi.even.degree}
q(x_1,\ldots, x_{n+1}) = p(x_1,\ldots,x_n) + x_{n+1}^d.
\end{equation}
We claim that $q$ is quasiconvex if and only if $p$ is convex.
Indeed, if $p$ is convex, then obviously so is $q$, and therefore
$q$ is quasiconvex. Conversely, if $p$ is not convex, then by
Theorem~\ref{thm:quasiconvexity.homog.same.convexity}, it is not
quasiconvex. So, there exist points $a,b,c\in\mathbb{R}^n$, with
$c$ on the line connecting $a$ and $b$, such that $p(a)\leq1$,
$p(b)\leq1$, but $p(c)>1$. Considering points $(a,0)$, $(b,0)$,
$(c,0)$, we see that $q$ is not quasiconvex. It follows that it is
NP-hard to decide quasiconvexity of polynomials of even degree
four or larger.
\end{proof}

\begin{corollary}
It is NP-hard to decide pseudoconvexity of polynomials of any
fixed even degree $d\geq4$.
\end{corollary}

\begin{proof}
The proof is almost identical to the proof of
\aaa{Corollary~\ref{cor:quasi.nphard.d>=4}}. \aaa{Let} $q$ be
defined as in (\ref{eq:q.in.reduction.quasi.even.degree}). If $p$
is convex, then $q$ is convex and hence also pseudoconvex. If $p$
is not convex, we showed that $q$ is not quasiconvex and hence
also not pseudoconvex.
\end{proof}

\section{Summary and conclusions}\label{sec:summary.conclusions}
In this chapter, we studied the computational complexity of
testing convexity and some of its variants, for polynomial
functions. The notions that we considered and the implications
among them are summarized below:

 \vspace{3mm} \scalefont{.87} \noindent strong convexity
$\Longrightarrow$ strict convexity $\Longrightarrow$ convexity
$\Longrightarrow$ pseudoconvexity $\Longrightarrow$
quasiconvexity. \vspace{3mm} \normalsize

%\begin{equation}\nonumber
%$$
%\mbox{strong convexity} \ \Longrightarrow \  \mbox{strict
%convexity}\ \Longrightarrow \ \mbox{convexity}\ \Longrightarrow \
%\mbox{pseudoconvexity}\ \Longrightarrow \ \mbox{quasiconvexity}.
%$$
%\end{equation}

Our complexity results as a function of the degree of the
polynomial are listed in Table~\ref{table:summary}. %%%\ref{table:summary}.
%\begin{table} [h] \label{table:summary}
%  \caption{Summary of our complexity results. A yes (no) entry means that the question is trivial for that particular entry because the answer is always yes (no) independent of the input. By $\in$P, we mean that the problem can be solved in polynomial time.}
%  \begin{center}
%\begin{tabular}{l||c|c|c|c}
%
%  % after \\: \hline or \cline{col1-col2} \cline{col3-col4} ...
%\textbf{property vs. degree} & 1 & 2 & odd $\geq3$ & even $\geq4$ \\
%  \hline \hline
%  \noindent\large{convexity} & yes & $\in$P & no & strongly NP-hard \\
%  %\hline
%  \large{strict convexity} & no & $\in$P & no & strongly NP-hard  \\
%  %\hline
%  \large{strong convexity} & no & $\in$P & no & strongly NP-hard \\
%  %\hline
%  \large{quasiconvexity} & yes & $\in$P & $\in$P & strongly NP-hard \\
%  %\hline
%\large{pseudoconvexity} & yes & $\in$P & $\in$P & strongly NP-hard \\
%
%\end{tabular}
%\end{center}
%\end{table}
\begin{table}
\begin{center}
\begin{tabular}{l||c|c|c|c}
\textbf{property vs. degree} & 1 & 2 & odd $\geq3$ & even $\geq4$ \\
  \hline \hline
   \noindent\large{strong convexity} & no & P & no & strongly NP-hard \\
  %\hline
  \large{strict convexity} & no & P & no & strongly NP-hard  \\
  %\hline
  \large{convexity} & yes & P & no & strongly NP-hard \\
  %\hline
  \large{pseudoconvexity} & yes & P & P & strongly NP-hard \\
  %\hline
  \large{quasiconvexity} & yes & P & P & strongly NP-hard \\
\end{tabular}
\end{center}
\caption{Summary of our complexity results. A yes (no) entry means
that the question is trivial for that particular entry because the
answer is always yes (no) independent of the input. By \aaan{P},
we mean that the problem can be solved in polynomial time.}
\label{table:summary}
\end{table}
We gave polynomial time algorithms for checking pseudoconvexity
and quasiconvexity of odd degree polynomials that can be useful in
many applications. Our negative results, on the other hand, imply
(\jnt{under P$\neq$NP}) the impossibility of a polynomial time (or
even pseudo-polynomial time) algorithm for testing any of the
properties listed in Table~\ref{table:summary} for polynomials of
even degree four or larger. \aaa{Although the implications of
convexity are very significant in optimization theory, our results
suggest that \aaan{unless additional structure is present,}
ensuring the mere presence of convexity is likely an intractable
task.} It is therefore natural to wonder whether there are other
properties of optimization problems that share some of the
attractive consequences of convexity, but are easier to check.

The hardness results of this chapter also lay emphasis on the need
for finding good approximation algorithms for recognizing
convexity that can deal with a large number of instances. This is
our motivation for the next chapter as we turn our attention to
the study of algebraic counterparts of convexity that can be
efficiently checked with semidefinite programming.

%, and our goal will be to understand when
%
%
%Towards this end, we will devote the next chapter to the study of
%sos-convexity, an algebraic counterpart of convexity that can be
%efficiently checked with semidefinite programming.
%
%semidefinite programming based relaxations relying on algebraic
%concepts such as sum of squares decomposition of polynomials
%currently seem to be very promising techniques for recognizing
%convexity of polynomials and basic semialgebraic sets. The study
%of these algebraic relaxations and characterizing the special
%cases where they are exact is the subject of the next chapter.

\chapter{Convexity and SOS-Convexity}\label{chap:convexity.sos.convexity}

%AAA: for right aligning entries of matrices
\makeatletter
\renewcommand*\env@matrix[1][c]{\hskip -\arraycolsep
  \let\@ifnextchar\new@ifnextchar
  \array{*\c@MaxMatrixCols #1}}
\makeatother

%%%%%%%%%%%%%%%%%%%%%%%%%%%%%%%%%%%%%%%%%%%%%%%%%%%%%%%%%%%%%%%%%%%%%%%%%%%%%%%%
The overall contribution of this chapter is a complete
characterization of the containment of the sets of convex and
sos-convex polynomials in every degree and dimension. The content
of this chapter is mostly based on the work
in~\cite{AAA_PP_table_sos-convexity}, but also includes parts
of~\cite{AAA_PP_not_sos_convex_journal}
and~\cite{AAA_GB_PP_Convex_ternary_quartics}.

\section{Introduction}
\subsection{Nonnegativity and sum of squares}
One of the cornerstones of real algebraic geometry is Hilbert's
seminal paper in 1888~\cite{Hilbert_1888}, where he gives a
complete characterization of the degrees and dimensions in which
nonnegative polynomials can be written as sums of squares of
polynomials. In particular, Hilbert proves in~\cite{Hilbert_1888}
that there exist nonnegative polynomials that are not sums of
squares, although explicit examples of such polynomials appeared
only about 80 years later and the study of the gap between
nonnegative and sums of squares polynomials continues to be an
active area of research to this day.

%
%The main results of Hilbert's paper include the rather surprising
%fact that all nonnegative bivariate quartic polynomials are sums
%of squares and the first proof of existence of nonnegative
%polynomials that are not sums of squares, which he specifically
%showed must exist within bivariate sextic and ternary quartic
%polynomials. Interestingly, explicit examples of such polynomials
%appeared only about 80 years later with the first examples due
%respectively to Motzkin~\cite{MotzkinSOS} and
%Robinson~\cite{RobinsonSOS}; see~\cite{Reznick}.

Motivated by a wealth of new applications and a modern viewpoint
that emphasizes efficient computation, there has also been a great
deal of recent interest from the optimization community in the
representation of nonnegative polynomials as sums of squares
(sos). Indeed, many fundamental problems in applied and
computational mathematics can be reformulated as either deciding
whether certain polynomials are nonnegative or searching over a
family of nonnegative polynomials. It is well-known however that
if the degree of the polynomial is four or larger, deciding
nonnegativity is an NP-hard problem. (As we mentioned in the last
chapter, this follows e.g. as an immediate corollary of
NP-hardness of deciding matrix
copositivity~\cite{nonnegativity_NP_hard}.) On the other hand, it
is also well-known that deciding whether a polynomial can be
written as a sum of squares can be reduced to solving a
semidefinite program, for which efficient algorithms e.g. based on
interior point methods is available. The general machinery of the
so-called ``sos relaxation'' has therefore been to replace the
intractable nonnegativity requirements with the more tractable sum
of squares requirements that obviously provide a sufficient
condition for polynomial nonnegativity.

Some relatively recent applications that sum of squares
relaxations have found span areas as diverse as control
theory~\cite{PhD:Parrilo},~\cite{PositivePolyInControlBook},
quantum computation~\cite{Pablo_Sep_Entang_States}, polynomial
games~\cite{Pablo_poly_games}, combinatorial
optimization~\cite{Stability_number_SOS}, geometric theorem
proving~\cite{Pablo_Geometry_Packing_SOS}, and many others.

\subsection{Convexity and sos-convexity}

Aside from nonnegativity, \emph{convexity} is another fundamental
property of polynomials that is of both theoretical and practical
significance. In the previous chapter, we already listed a number
of applications of establishing convexity of polynomials including
global optimization, convex envelope approximation, Lyapunov
analysis, data fitting, defining norms, etc. Unfortunately,
however, we also showed that just like nonnegativity, convexity of
polynomials is NP-hard to decide for polynomials of degree as low
as four. Encouraged by the success of sum of squares methods as a
viable substitute for nonnegativity, our focus in this chapter
will be on the analogue of sum of squares for polynomial
convexity: a notion known as \emph{sos-convexity}.

As we mentioned in our previous chapters in passing, sos-convexity
(which gets its name from the work of Helton and Nie
in~\cite{Helton_Nie_SDP_repres_2}) is a sufficient condition for
convexity of polynomials based on an appropriately defined sum of
squares decomposition of the Hessian matrix; see the equivalent
Definitions~\ref{defn:sos-convex} and~\ref{def:sos.convex}. The
main computational advantage of sos-convexity stems from the fact
that the problem of deciding whether a given polynomial is
sos-convex amounts to solving a single semidefinite program. We
will explain how this is exactly done in Section~\ref{sec:prelims}
of this chapter where we briefly review the well-known connection
between sum of squares decomposition and semidefinite programming.

Besides its computational implications, sos-convexity is an
appealing concept since it bridges the geometric and algebraic
aspects of convexity. Indeed, while the usual definition of
convexity is concerned only with the geometry of the epigraph, in
sos-convexity this geometric property (or the nonnegativity of the
Hessian) must be certified through a ``simple'' algebraic
identity, namely the sum of squares factorization of the Hessian.
The original motivation of Helton and Nie for defining
sos-convexity was in relation to the question of semidefinite
representability of convex sets~\cite{Helton_Nie_SDP_repres_2}.
But this notion has already appeared in the literature in a number
of other
settings~\cite{Lasserre_Jensen_inequality},~\cite{Lasserre_Convex_Positive},~\cite{convex_fitting},~\cite{Chesi_Hung_journal}.
In particular, there has been much recent interest in the role of
convexity in semialgebraic geometry
~\cite{Lasserre_Jensen_inequality},~\cite{Blekherman_convex_not_sos},~\cite{Monique_Etienne_Convex},~\cite{Lasserre_set_convexity}
and sos-convexity is a recurrent figure in this line of research.

\subsection{Contributions and organization of this chapter}
The main contribution of this chapter is to establish the
counterpart of Hilbert's characterization of the gap between
nonnegativity and sum of squares for the notions of convexity and
sos-convexity. We start by presenting some background material in
Section~\ref{sec:prelims}. In
Section~\ref{sec:equiv.defs.of.sos.convexity}, we prove an
algebraic analogue of a classical result in convex analysis, which
provides three equivalent characterizations for sos-convexity
(Theorem~\ref{thm:sos.convexity.3.equivalent.defs}). This result
substantiates the fact that sos-convexity is \emph{the} right sos
relaxation for convexity. In Section~\ref{sec:first.examples}, we
present two explicit examples of convex polynomials that are not
sos-convex, one of them being the first known such example. In
Section~\ref{sec:full.characterization}, we provide the
characterization of the gap between convexity and sos-convexity
(Theorem~\ref{thm:full.charac.polys} and
Theorem~\ref{thm:full.charac.forms}).
Subsection~\ref{subsec:proof.equal.cases} includes the proofs of
the cases where convexity and sos-convexity are equivalent and
Subsection~\ref{subsec:proof.non.equal.cases} includes the proofs
of the cases where they are not. In particular,
Theorem~\ref{thm:minimal.2.6.and.3.6} and
Theorem~\ref{thm:minimal.3.4.and.4.4} present explicit examples of
convex but not sos-convex polynomials that have dimension and
degree as low as possible, and
Theorem~\ref{thm:conv_not_sos_conv_forms_n3d} provides a general
construction for producing such polynomials in higher degrees.
Some concluding remarks and an open problem are presented in
Section~\ref{sec:concluding.remarks}.

This chapter also includes two appendices. In Appendix A, we
explain how the first example of a convex but not sos-convex
polynomial was found with software using sum of squares
programming techniques and the duality theory of semidefinite
optimization. As a byproduct of this numerical procedure, we
obtain a simple method for searching over a restricted family of
nonnegative polynomials that are not sums of squares. In Appendix
B, we give a formal (computer assisted) proof of validity of one
of our minimal convex but not sos-convex polynomials.

\section{Preliminaries}\label{sec:prelims}
\subsection{Background on nonnegativity and sum of squares}\label{subsec:nonnegativity.sos.basics}

For the convenience of the reader, we recall some basic concepts
from the previous chapter and then introduce some new ones.
%
%A (multivariate) \emph{polynomial} $p\mathrel{\mathop:}=p(x)$ in
%variables $x\mathrel{\mathop:}=(x_1,\ldots,x_n)^T$ is a function
%from $\mathbb{R}^n$ to $\mathbb{R}$ that is a finite linear
%combination of monomials:
%\begin{equation}\nonumber
%p(x)=\sum_{\alpha}c_\alpha x^\alpha=\sum_{\alpha_1, \ldots,
%\alpha_n} c_{\alpha_1,\ldots,\alpha_n} x_1^{\alpha_1} \cdots
%x_n^{\alpha_n} ,
%\end{equation}
%where the sum is over $n$-tuples of nonnegative integers
%$\alpha_i$.
We will be concerned throughout this chapter with polynomials with
real coefficients. The ring of polynomials in $n$ variables with
real coefficients is denoted by $\mathbb{R}[x]$. A polynomial $p$
is said to be \emph{nonnegative} or \emph{positive semidefinite
(psd)} if $p(x)\geq0$ for all $x\in\mathbb{R}^n$. We say that $p$
is a \emph{sum of squares (sos)}, if there exist polynomials
$q_{1},\ldots,q_{m}$ such that $p=\sum_{i=1}^{m}q_{i}^{2}$. We
denote the set of psd (resp. sos) polynomials in $n$ variables and
degree $d$ by $\tilde{P}_{n,d}$ (resp. $\tilde{\Sigma}_{n,d}$).
Any sos polynomial is clearly psd, so we have
$\tilde{\Sigma}_{n,d}\subseteq \tilde{P}_{n,d}$.
%\begin{equation}\label{eq:sos.decomp.q_i}
%p(x)=\sum_{i=1}^{m}q_{i}^{2}(x).
%\end{equation}
Recall that a \emph{homogeneous polynomial} (or a \emph{form}) is
a polynomial where all the monomials have the same degree. A form
$p$ of degree $d$ is a homogeneous function of degree $d$ since it
satisfies $p(\lambda x)=\lambda^d p(x)$ for any scalar
$\lambda\in\mathbb{R}$. We say that a form $p$ is \emph{positive
definite} if $p(x)>0$ for all $x\neq0$ in $\mathbb{R}^n$.
Following standard notation, we denote the set of psd (resp. sos)
homogeneous polynomials in $n$ variables and degree $d$ by
$P_{n,d}$ (resp. $\Sigma_{n,d}$). Once again, we have the obvious
inclusion $\Sigma_{n,d}\subseteq P_{n,d}$. All of the four sets
$\Sigma_{n,d}, P_{n,d}, \tilde{\Sigma}_{n,d}, \tilde{P}_{n,d}$ are
closed convex cones. The closedness of the sum of squares cone may
not be so obvious. This fact was first proved by
Robinson~\cite{RobinsonSOS}. We will make crucial use of it in the
proof of Theorem~\ref{thm:sos.convexity.3.equivalent.defs} in the
next section.

Any form of degree $d$ in $n$ variables can be ``dehomogenized''
into a polynomial of degree $\leq d$ in $n-1$ variables by setting
$x_n=1$. Conversely, any polynomial $p$ of degree $d$ in $n$
variables can be ``homogenized'' into a form $p_h$ of degree $d$
in $n+1$ variables, by adding a new variable $y$, and letting $$
p_h(x_1,\ldots,x_n,y)\mathrel{\mathop:}=y^{d} \, p\left({x_1}/{y},
\ldots, {x_n}/{y}\right).$$ The properties of being psd and sos
are preserved under homogenization and
dehomogenization~\cite{Reznick}.

A very natural and fundamental question that as we mentioned
earlier was answered by Hilbert is to understand in what
dimensions and degrees nonnegative polynomials (or forms) can be
represented as sums of squares, i.e, for what values of $n$ and
$d$ we have $\tilde{\Sigma}_{n,d}=\tilde{P}_{n,d}$ or
$\Sigma_{n,d}=P_{n,d}$. Note that because of the argument in the
last paragraph, we have $\tilde{\Sigma}_{n,d}=\tilde{P}_{n,d}$ if
and only if $\Sigma_{n+1,d}=P_{n+1,d}$. Hence, it is enough to
answer the question just for polynomials or just for forms and the
answer to the other one comes for free.

\begin{theorem}[Hilbert,~\cite{Hilbert_1888}]\label{thm:Hilbert}
$\tilde{\Sigma}_{n,d}=\tilde{P}_{n,d}$ if and only if $n=1$ or
$d=2$ or $(n,d)=(2,4)$. Equivalently, $\Sigma_{n,d}=P_{n,d}$ if
and only if $n=2$ or $d=2$ or $(n,d)=(3,4)$.
\end{theorem}

The proofs of $\tilde{\Sigma}_{1,d}=\tilde{P}_{1,d}$ and
$\tilde{\Sigma}_{n,2}=\tilde{P}_{n,2}$ are relatively simple and
were known before Hilbert. On the other hand, the proof of the
fairly surprising fact that $\tilde{\Sigma}_{2,4}=\tilde{P}_{2,4}$
(or equivalently $\Sigma_{3,4}=P_{3,4}$) is rather involved. We
refer the interested reader to
\cite{NewApproach_Hilbert_Ternary_Quatrics},
\cite{Scheiderer_ternary_quartic},
\cite{Choi_Lam_extremalPSDforms}, and references in~\cite{Reznick}
for some modern expositions and alternative proofs of this result.
Hilbert's other main contribution was to show that these are the
only cases where nonnegativity and sum of squares are equivalent
by giving a nonconstructive proof of existence of polynomials in
$\tilde{P}_{2,6}\setminus\tilde{\Sigma}_{2,6}$ and
$\tilde{P}_{3,4}\setminus\tilde{\Sigma}_{3,4}$ (or equivalently
forms in $P_{3,6}\setminus\Sigma_{3,6}$ and
$P_{4,4}\setminus\Sigma_{4,4}$). From this, it follows with simple
arguments that in all higher dimensions and degrees there must
also be psd but not sos polynomials; see~\cite{Reznick}. Explicit
examples of such polynomials appeared in the 1960s starting from
the celebrated Motzkin form~\cite{MotzkinSOS}:
\begin{equation}\label{eq:Motzkin.form}
M(x_1,x_2,x_3)=x_1^4x_2^2+x_1^2x_2^4-3x_1^2x_2^2x_3^2+x_3^6,
\end{equation}
which belongs to $P_{3,6}\setminus\Sigma_{3,6}$, and continuing a
few years later with the Robinson form~\cite{RobinsonSOS}:
\scalefont{.92}
\begin{equation}\label{eq:Robinston.form}
R(x_1,x_2,x_3,x_4)=x_1^2(x_1-x_4)^2+x_2^2(x_2-x_4)^2+x_3^2(x_3-x_4)^2+2x_1x_2x_3(x_1+x_2+x_3-2x_4),
\end{equation} \normalsize
which belongs to $P_{4,4}\setminus\Sigma_{4,4}$.

Several other constructions of psd polynomials that are not sos
have appeared in the literature since. An excellent survey
is~\cite{Reznick}. See also~\cite{Reznick_Hilbert_construciton}
and~\cite{Blekherman_nonnegative_and_sos}.

%A polynomial $p(x)$ of degree $d$ in $n$ variables has
%$n+d\choose{d}$ coefficients, whereas the number of coefficients
%of a form of degree $d$ in $n$ variables is $n+d-1\choose{d}$.
%It is easy to show that if a form of degree $d$ is sos, then $d$
%is even and the polynomials $q_i$ in the sos decomposition are
%forms of degree $d/2$.

\subsection{Connection to semidefinite programming and matrix
generalizations}\label{subsec:sos.sdp.and.matrix.generalize} As we
remarked before, what makes sum of squares an appealing concept
from a computational viewpoint is its relation to semidefinite
programming. It is well-known (see e.g. \cite{PhD:Parrilo},
\cite{sdprelax}) that a polynomial $p$ in $n$ variables and of
even degree $d$ is a sum of squares if and only if there exists a
positive semidefinite matrix $Q$ (often called the Gram matrix)
such that
$$p(x)=z^{T}Qz,$$
where $z$ is the vector of monomials of degree up to $d/2$
\begin{equation}\label{eq:monomials}
z=[1,x_{1},x_{2},\ldots,x_{n},x_{1}x_{2},\ldots,x_{n}^{d/2}].
\end{equation}
The set of all such matrices $Q$ is the feasible set of a
semidefinite program (SDP). For fixed $d$, the size of this
semidefinite program is polynomial in $n$. Semidefinite programs
can be solved with arbitrary accuracy in polynomial time. There
are several implementations of semidefinite programming solvers,
based on interior point algorithms among others, that are very
efficient in practice and widely used; see~\cite{VaB:96} and
references therein.

%The conversion step of going from an sos decomposition problem to
%an SDP problem is fully algorithmic and has been implemented in
%software packages such as SOSTOOLS~\cite{sostools} and
%YALMIP~\cite{yalmip}.

The notions of positive semidefiniteness and sum of squares of
scalar polynomials can be naturally extended to polynomial
matrices, i.e., matrices with entries in $\mathbb{R}[x]$. We say
that a symmetric polynomial matrix $U(x)\in \mathbb{R}[x]^{m
\times m}$ is positive semidefinite if $U(x)$ is positive
semidefinite in the matrix sense for all $x\in \mathbb{R}^n$, i.e,
if $U(x)$ has nonnegative eigenvalues for all $x\in \mathbb{R}^n$.
It is straightforward to see that this condition holds if and only
if the polynomial $y^{T}U(x)y$ in $m+n$ variables $[x; y]$ is psd.
A homogeneous polynomial matrix $U(x)$ is said to be positive
definite, if it is positive definite in the matrix sense, i.e.,
has positive eigenvalues, for all $x\neq0$ in $\mathbb{R}^n$. The
definition of an sos-matrix is as follows
\cite{Kojima_SOS_matrix}, \cite{Symmetry_groups_Gatermann_Pablo},
\cite{matrix_sos_Hol}.

\begin{definition}\label{def:sos-matrix}
A symmetric polynomial matrix $U(x)\in~\mathbb{R}[x]^{m \times
m}$,$ \ x\in \mathbb{R}^n,$ is an \emph{sos-matrix} if there
exists a polynomial matrix $V(x)\in \mathbb{R}[x]^{s \times m}$
for some $s\in\mathbb{N}$, such that $U(x)~=~V^{T}(x)V(x)$.
\end{definition}

It turns out that a polynomial matrix $U(x)\in \mathbb{R}[x]^{m
\times m}$, $\ x\in~\mathbb{R}^n,$ is an sos-matrix if and only if
the scalar polynomial $y^{T}U(x)y$ is a sum of squares in
$\mathbb{R}[x; y]$; see~\cite{Kojima_SOS_matrix}. This is a useful
fact because in particular it gives us an easy way of checking
whether a polynomial matrix is an sos-matrix by solving a
semidefinite program. Once again, it is obvious that being an
sos-matrix is a sufficient condition for a polynomial matrix to be
positive semidefinite.

%The converse is not true except for special cases, e.g. for
%univariate matrices~\cite{CLRrealzeros},~\cite{SOS_KYP}.

\subsection{Background on convexity and sos-convexity}\label{subsec:convexity.sos.convexity.basics}

A polynomial $p$ is (globally) convex if for all $x$ and $y$ in
$\mathbb{R}^n$ and all $\lambda \in [0,1]$, we have
\begin{equation}\label{eq:convexity.defn.}
p(\lambda x+(1-\lambda)y)\leq \lambda p(x)+(1-\lambda)p(y).
\end{equation}
Since polynomials are continuous functions, the inequality in
(\ref{eq:convexity.defn.}) holds if and only if it holds for a
fixed value of $\lambda\in(0,1)$, say, $\lambda=\frac{1}{2}$. In
other words, $p$ is convex if and only if
\begin{equation}\label{eq:convexity.with.lambda.0.5}
p\left(\textstyle{\frac{1}{2}}
x+\textstyle{\frac{1}{2}}y\right)\leq \textstyle{\frac{1}{2}
p(x)}+\textstyle{\frac{1}{2}}p(y)
\end{equation} for all $x$ and
$y$; see e.g.~\cite[p. 71]{Rudin_RealComplexAnalysis}. Recall from
the previous chapter that except for the trivial case of linear
polynomials, an odd degree polynomial is clearly never convex.

For the sake of direct comparison with a result that we derive in
the next section
(Theorem~\ref{thm:sos.convexity.3.equivalent.defs}), we recall
next a classical result from convex analysis on the first and
second order characterization of convexity. The proof can be found
in many convex optimization textbooks, e.g.~\cite[p.
70]{BoydBook}. The theorem is of course true for any twice
differentiable function, but for our purposes we state it for
polynomials.

\begin{theorem}\label{thm:classical.first.2nd.order.charac.}
Let $p\mathrel{\mathop:}=p(x)$ be a polynomial. Let $\nabla
p\mathrel{\mathop:}=\nabla p(x)$ denote its gradient and let
$H\mathrel{\mathop:}=H(x)$ be its Hessian, i.e., the $n \times n$
symmetric matrix of second derivatives. Then the following are
equivalent.

\textbf{(a)} $p\left(\textstyle{\frac{1}{2}}
x+\textstyle{\frac{1}{2}}y\right)\leq \textstyle{\frac{1}{2}
p(x)}+\textstyle{\frac{1}{2}}p(y),\quad \forall
x,y\in\mathbb{R}^n$; (i.e., $p$ is convex).

\textbf{(b)} $p(y) \geq p(x)+\nabla p(x)^T(y-x),\quad \forall
x,y\in\mathbb{R}^n.$

\textbf{(c)} $y^TH(x)y\geq0,\quad \forall x,y\in\mathbb{R}^n$;
(i.e., $H(x)$ is a positive semidefinite polynomial matrix).
\end{theorem}

Helton and Nie proposed in~\cite{Helton_Nie_SDP_repres_2} the
notion of \emph{sos-convexity} as an sos relaxation for the second
order characterization of convexity (condition \textbf{(c)}
above).

\begin{definition}\label{def:sos.convex}
A polynomial $p$ is \emph{sos-convex} if its Hessian
$H\mathrel{\mathop:}=H(x)$ is an sos-matrix.
\end{definition}

With what we have discussed so far, it should be clear that
sos-convexity is a sufficient condition for convexity of
polynomials that can be checked with semidefinite programming. In
the next section, we will show some other natural sos relaxations
for polynomial convexity, which will turn out to be equivalent to
sos-convexity.

We end this section by introducing some final notation:
$\tilde{C}_{n,d}$ and $\tilde{\Sigma C}_{n,d}$ will respectively
denote the set of convex and sos-convex polynomials in $n$
variables and degree $d$; $C_{n,d}$ and $\Sigma C_{n,d}$ will
respectively denote set of convex and sos-convex homogeneous
polynomials in $n$ variables and degree $d$. Again, these four
sets are closed convex cones and we have the obvious inclusions
$\tilde{\Sigma C}_{n,d}\subseteq\tilde{C}_{n,d}$ and $\Sigma
C_{n,d}\subseteq C_{n,d}$.

\section{Equivalent algebraic relaxations for convexity of
polynomials}\label{sec:equiv.defs.of.sos.convexity} An obvious way
to formulate alternative sos relaxations for convexity of
polynomials is to replace every inequality in
Theorem~\ref{thm:classical.first.2nd.order.charac.} with its sos
version. In this section we examine how these relaxations relate
to each other. We also comment on the size of the resulting
semidefinite programs.

Our result below can be thought of as an algebraic analogue of
Theorem~\ref{thm:classical.first.2nd.order.charac.}.
\begin{theorem} \label{thm:sos.convexity.3.equivalent.defs}
Let $p\mathrel{\mathop:}=p(x)$ be a polynomial of degree $d$ in
$n$ variables with its gradient and Hessian denoted respectively
by $\nabla p\mathrel{\mathop:}=\nabla p(x) $ and
$H\mathrel{\mathop:}=H(x)$. Let $g_{\lambda}$, $g_\nabla$, and
$g_{\nabla^2}$ be defined as
\begin{equation} \label{eq:defn.g_lambda.g_grad.g_grad2}
\begin{array}{lll}
g_{\lambda}(x,y)&=&(1-\lambda)p(x)+\lambda p(y)-p((1-\lambda)
x+\lambda y),\\
g_\nabla(x,y)&=&p(y)-p(x)-\nabla p(x)^T(y-x), \\
g_{\nabla^2}(x,y)&=&y^{T}H(x)y.
\end{array}
\end{equation}
Then the following are equivalent:

\textbf{(a)}  \  $g_{\frac{1}{2}}(x,y)$ is sos\footnote{The
constant $\frac{1}{2}$ in $g_{\frac{1}{2}}(x,y)$ of condition
\textbf{(a)} is arbitrary and is chosen for convenience. One can
show that $g_{\frac{1}{2}}$ being sos implies that $g_{\lambda}$
is sos for any fixed $\lambda\in[0,1]$. Conversely, if
$g_{\lambda}$ is sos for some $\lambda\in(0,1)$, then
$g_{\frac{1}{2}}$ is sos. The proofs are similar to the proof of
\textbf{(a)$\Rightarrow$(b)}. }.

%\textbf{(a')} $g_\lambda(x,y)$ is sos for any fixed
%$\lambda\in[0,1].$

\textbf{(b)}  \  $g_\nabla(x,y)$ is sos.

\textbf{(c)}   \  $g_{\nabla^2}(x,y)$ is sos; (i.e., $H(x)$ is an
sos-matrix).
\end{theorem}

\begin{proof}
\textbf{(a)$\Rightarrow$(b)}: Assume $g_{\frac{1}{2}}$ is sos. We
start by proving that $g_{\frac{1}{2^k}}$ will also be sos for any
integer $k\geq2$. A little bit of straightforward algebra yields
the relation
\begin{equation}\label{eq:g_diatic_relation}
g_{\frac{1}{2^{k+1}}}(x,y)=\textstyle{\frac{1}{2}}g_{\frac{1}{2^{k}}}(x,y)+g_{\frac{1}{2}}\left(x,\textstyle{\frac{2^{k}-1}{2^{k}}}x+\textstyle{\frac{1}{2^{k}}}y\right).
\end{equation}
The second term on the right hand side of
(\ref{eq:g_diatic_relation}) is always sos because
$g_{\frac{1}{2}}$ is sos. Hence, this relation shows that for any
$k$, if $g_{\frac{1}{2^k}}$ is sos, then so is
$g_{\frac{1}{2^{k+1}}}$. Since for $k=1$, both terms on the right
hand side of (\ref{eq:g_diatic_relation}) are sos by assumption,
induction immediately gives that $g_{\frac{1}{2^k}}$ is sos for
all $k$.

%When $k=2$, both terms on the right hand side of this equation are
%sos by assumption and hence (\ref{eq:g_diatic_relation}) implies
%that $g_{\frac{1}{4}}$ must be sos. Once this is established, we
%can reapply equation (\ref{eq:g_diatic_relation}) with $k=3$ to
%conclude that $g_{\frac{1}{8}}$ is sos, and so on.

Now, let us rewrite $g_{\lambda}$ as
\begin{equation}\nonumber
g_{\lambda}(x,y)=p(x)+\lambda(p(y)-p(x))-p(x+\lambda(y-x)).
\end{equation}
We have
\begin{equation}\label{eq:g.lambda.rewritten}
\frac{g_{\lambda}(x,y)}{\lambda}=p(y)-p(x)-\frac{p(x+\lambda(y-x))-p(x)}{\lambda}.
\end{equation}
Next, we take the limit of both sides of
(\ref{eq:g.lambda.rewritten}) by letting
$\lambda=\frac{1}{2^k}\rightarrow0$ as $k\rightarrow\infty$.
Because $p$ is differentiable, the right hand side of
(\ref{eq:g.lambda.rewritten}) will converge to $g_\nabla$. On the
other hand, our preceding argument implies that
$\frac{g_{\lambda}}{\lambda}$ is an sos polynomial (of degree $d$
in $2n$ variables) for any $\lambda=\frac{1}{2^k}$. Moreover, as
$\lambda$ goes to zero, the coefficients of
$\frac{g_{\lambda}}{\lambda}$ remain bounded since the limit of
this sequence is $g_\nabla$, which must have bounded coefficients
(see (\ref{eq:defn.g_lambda.g_grad.g_grad2})). By closedness of
the sos cone, we conclude that the limit $g_\nabla$ must be sos.

\textbf{(b)$\Rightarrow$(a)}: Assume $g_\nabla$ is sos. It is easy
to check that
\begin{equation} \nonumber
g_{\frac{1}{2}}(x,y)=\textstyle{\frac{1}{2}}g_\nabla\left(\textstyle{\frac{1}{2}}x+\textstyle{\frac{1}{2}}y,x\right)+\textstyle{\frac{1}{2}}g_\nabla\left(\textstyle{\frac{1}{2}}x+\textstyle{\frac{1}{2}}y,y\right),
\end{equation}
and hence $g_{\frac{1}{2}}$ is sos.

\textbf{(b)$\Rightarrow$(c)}: Let us write the second order Taylor
approximation of $p$ around $x$:
\begin{equation}\nonumber%\label{eq:Taylor.second.order}
\begin{array}{ll}
p(y)=&p(x)+\nabla^{T}p(x)(y-x) \\
\   &+\frac{1}{2}(y-x)^{T}H(x)(y-x)+o(||y-x||^2).
\end{array}
\end{equation}
After rearranging terms, letting $y=x+\epsilon z$ (for
$\epsilon>0$), and dividing both sides by $\epsilon^2$ we get:
\begin{equation}\label{eq:Taylor.second.order.rearranged}
%\begin{array}{ll}
(p(x+\epsilon
z)-p(x))/\epsilon^2-\nabla^{T}p(x)z/\epsilon=\frac{1}{2}z^{T}H(x)z+1/\epsilon^2o(\epsilon^2||z||^2).
%\end{array}
\end{equation}
The left hand side of (\ref{eq:Taylor.second.order.rearranged}) is
$g_\nabla(x,x+\epsilon z)/\epsilon^2$ and therefore for any fixed
$\epsilon>0$, it is an sos polynomial by assumption. As we take
$\epsilon\rightarrow0$, by closedness of the sos cone, the left
hand side of (\ref{eq:Taylor.second.order.rearranged}) converges
to an sos polynomial. On the other hand, as the limit is taken,
the term $\frac{1}{\epsilon^2}o(\epsilon^2||z||^2)$ vanishes and
hence we have that $z^TH(x)z$ must be sos.

\textbf{(c)$\Rightarrow$(b)}: Following the strategy of the proof
of the classical case in~\cite[p. 165]{Tits_lec.notes}, we start
by writing the Taylor expansion of $p$ around $x$ with the
integral form of the remainder:
\begin{equation}\label{eq:Taylor.expan.Cauchy.rem}
%\begin{array}{ll}
p(y)=p(x)+\nabla^{T}p(x)(y-x)+\int_0^1(1-t)(y-x)^{T}H(x+t(y-x))(y-x)dt.
%\end{array}
\end{equation}
Since $y^{T}H(x)y$ is sos by assumption, for any $t\in[0,1]$ the
integrand
$$(1-t)(y-x)^{T}H(x+t(y-x))(y-x)$$ is an sos polynomial of degree $d$ in $x$ and
$y$. From (\ref{eq:Taylor.expan.Cauchy.rem}) we have
$$g_\nabla=\int_0^1(1-t)(y-x)^{T}H(x+t(y-x))(y-x)dt.$$
It then follows that $g_\nabla$ is sos because integrals of sos
polynomials, if they exist, are sos.

%This is a standard fact, but to see it intuitively one can think
%of discretizing $t$ and writing the integral as a limit of a
%sequence of Riemann sums. Every finite Riemann sum will be a
%polynomial of degree $d$ and because the sos cone is closed, the
%limit of this sequence must be sos.
%
%The conclusion that $g_\nabla$ is sos then follows by noting that
%from (\ref{eq:Taylor.expan.Cauchy.rem}) we have
%$$g_\nabla=\int_0^1(1-t)(y-x)^{T}H(x+t(y-x))(y-x)d_t,$$
%and by appealing to the fact that integrals of sos polynomials, if
%they exist, are sos. To see this fact, one can think of
%discretizing $t$ and writing the integral as a limit of a sequence
%of Riemann sums. Every finite Riemann sum will be a polynomial of
%degree $d$ and because the sos cone is closed, the limit of this
%sequence must be sos.
\end{proof}
We conclude that conditions \textbf{(a)}, \textbf{(b)}, and
\textbf{(c)} are equivalent sufficient conditions for convexity of
polynomials, and can each be checked with a semidefinite program
as explained in
Subsection~\ref{subsec:sos.sdp.and.matrix.generalize}. It is easy
to see that all three polynomials $g_{\frac{1}{2}}(x,y)$,
$g_\nabla(x,y)$, and $g_{\nabla^2}(x,y)$ are polynomials in $2n$
variables and of degree $d$. (Note that each differentiation
reduces the degree by one.) Each of these polynomials have a
specific structure that can be exploited for formulating smaller
SDPs. For example, the symmetries
$g_{\frac{1}{2}}(x,y)=g_{\frac{1}{2}}(y,x)$ and
$g_{\nabla^2}(x,-y)=g_{\nabla^2}(x,y)$ can be taken advantage of
via symmetry reduction techniques developed
in~\cite{Symmetry_groups_Gatermann_Pablo}.

The issue of symmetry reduction aside, we would like to point out
that formulation \textbf{(c)} (which was the original definition
of sos-convexity) can be significantly more efficient than the
other two conditions. The reason is that the polynomial
$g_{\nabla^2}(x,y)$ is always quadratic and homogeneous in $y$ and
of degree $d-2$ in $x$. This makes $g_{\nabla^2}(x,y)$ much more
sparse than $g_\nabla(x,y)$ and $g_{\nabla^2}(x,y)$, which have
degree $d$ both in $x$ and in $y$. Furthermore, because of the
special bipartite structure of $y^TH(x)y$, only monomials of the
form $x_i^ky_j$ will appear in the vector of monomials
(\ref{eq:monomials}). This in turn reduces the size of the Gram
matrix, and hence the size of the SDP. It is perhaps not too
surprising that the characterization of convexity based on the
Hessian matrix is a more efficient condition to check. After all,
this is a local condition (curvature at every point in every
direction must be nonnegative), whereas conditions \textbf{(a)}
and \textbf{(b)} are both global.

\begin{remark}
There has been yet another proposal for an sos relaxation for
convexity of polynomials in~\cite{Chesi_Hung_journal}. However, we
have shown in~\cite{AAA_PP_CDC10_algeb_convex} that the condition
in~\cite{Chesi_Hung_journal} is at least as conservative as the
three conditions in
Theorem~\ref{thm:sos.convexity.3.equivalent.defs} and also
significantly more expensive to check.
\end{remark}

\begin{remark}\label{rmk:sos-convexity.restriction}
Just like convexity, the property of sos-convexity is preserved
under restrictions to affine subspaces. This is perhaps most
directly seen through characterization \textbf{(a)} of
sos-convexity in
Theorem~\ref{thm:sos.convexity.3.equivalent.defs}, by also noting
that sum of squares is preserved under restrictions. Unlike
convexity however, if a polynomial is sos-convex on every line (or
even on every proper affine subspace), this does not imply that
the polynomial is sos-convex.
\end{remark}

As an application of
Theorem~\ref{thm:sos.convexity.3.equivalent.defs}, we use our new
characterization of sos-convexity to give a short proof of an
interesting lemma of Helton and Nie.

%\begin{lemma}[Helton and Nie~\cite{Helton_Nie_SDP_repres_2}]\label{lem:helton.nie.sos-convex.then.sos}
%Every sos-convex form is sos.
%\end{lemma}

%AAA: When I try to put the Lemma number in Helton and Nie, [ [] ] doesn't work, so I have to cheat it.
\begin{lemma}\label{lem:helton.nie.sos-convex.then.sos}\emph{(Helton and Nie~\cite[Lemma 8]{Helton_Nie_SDP_repres_2})}.
Every sos-convex form is sos.
\end{lemma}

\begin{proof}
Let $p$ be an sos-convex form of degree $d$. We know from
Theorem~\ref{thm:sos.convexity.3.equivalent.defs} that
sos-convexity of $p$ is equivalent to the polynomial
%\begin{equation}\nonumber
$g_{\frac{1}{2}}(x,y)=\textstyle{\frac{1}{2}
p(x)}+\textstyle{\frac{1}{2}}p(y)-p\left(\textstyle{\frac{1}{2}}
x+\textstyle{\frac{1}{2}}y\right)$
%\end{equation}
being sos. But since sos is preserved under restrictions and
$p(0)=0$, this implies that
$$g_{\frac{1}{2}}(x,0)=\textstyle{\frac{1}{2}
p(x)}-p(\frac{1}{2}x)=\left(\frac{1}{2}-(\frac{1}{2})^d
\right)p(x)$$ is sos.
\end{proof}

Note that the same argument also shows that convex forms are psd.

%We shall end this section by commenting on the necessity of these
%conditions. In~\cite{AAA_PP_not_sos_convex_journal},
%\cite{AAA_PP_CDC09_HessianNotFactor}, the authors showed with an
%explicit example that there exist convex polynomials that are not
%sos-convex. In view of
%Theorems~\ref{thm:classical.first.2nd.order.charac.}~and~\ref{thm:main.result.equivalence},
%it follows immediately that the same example rejects necessity of
%conditions \textbf{(a)} and \textbf{(b)}. However, standard
%Positivstellensatz results related to Hilbert's 17th problem can
%be used to make these conditions necessary under mild assumptions
%but at the cost of solving larger and larger SDPs; see
%e.g.~\cite{Reznick_Unif_denominator}.

\section{Some constructions of convex but not sos-convex
polynomials}\label{sec:first.examples} It is natural to ask
whether sos-convexity is not only a sufficient condition for
convexity of polynomials but also a necessary one. In other words,
could it be the case that if the Hessian of a polynomial is
positive semidefinite, then it must factor? To give a negative
answer to this question, one has to prove existence of a convex
polynomial that is not sos-convex, i.e, a polynomial $p$ for which
one (and hence all) of the three polynomials $g_{\frac{1}{2}},
g_\nabla,$ and $g_{\nabla^2}$ in
(\ref{eq:defn.g_lambda.g_grad.g_grad2}) are psd but not sos. Note
that existence of psd but not sos polynomials does not imply
existence of convex but not sos-convex polynomials on its own. The
reason is that the polynomials $g_{\frac{1}{2}}, g_\nabla,$ and
$g_{\nabla^2}$ all possess a very special
structure.\footnote{There are many situations where requiring a
specific structure on polynomials makes psd equivalent to sos. As
an example, we know that there are forms in
$P_{4,4}\setminus\Sigma_{4,4}$. However, if we require the forms
to have only even monomials, then all such nonnegative forms in 4
variables and degree 4 are sums of
squares~\cite{Even_quartics_4vars_sos}.} For example, $y^TH(x)y$
has the structure of being quadratic in $y$ and a Hessian in $x$.
(Not every polynomial matrix is a valid Hessian.) The Motzkin or
the Robinson polynomials in (\ref{eq:Motzkin.form}) and
(\ref{eq:Robinston.form}) for example are clearly not of this
structure.

\subsection{The first example}
In~\cite{AAA_PP_not_sos_convex_journal},\cite{AAA_PP_CDC09_HessianNotFactor},
we presented the first example of a convex polynomial that is not
sos-convex\footnote{Assuming P$\neq$NP, and given the NP-hardness
of deciding polynomial convexity proven in the previous chapter,
one would expect to see convex polynomials that are not
sos-convex. However, we found the first such polynomial before we
had proven the NP-hardness result. Moreover, from complexity
considerations, even assuming P$\neq$NP, one cannot conclude
existence of convex but not sos-convex polynomials for any fixed
finite value of the number of variables $n$.}:
\begin{equation}\label{eq:first.convex.not.sos.convex}
\begin{array}{rlll}
p(x_1,x_2,x_3)&=&32x_1^8+118x_1^6x_2^2+40x_1^6x_3^2+25x_1^4x_2^4-43x_1^4x_2^2x_3^2-35x_1^4x_3^4
\\
\\ \quad&\
&+3x_1^2x_2^4x_3^2-16x_1^2x_2^2x_3^4+24x_1^2x_3^6+16x_2^8+44x_2^6x_3^2+70x_2^4x_3^4\\
\\ \quad&\
&+60x_2^2x_3^6+30x_3^8.
\end{array}
\end{equation}
As we will see later in this chapter, this form which lives in
$C_{3,8}\setminus\Sigma C_{3,8}$ turns out to be an example in the
smallest possible number of variables but not in the smallest
degree.

In Appendix A, we will explain how the polynomial in
(\ref{eq:first.convex.not.sos.convex}) was found. The proof that
this polynomial is convex but not sos-convex is omitted and can be
found in~\cite{AAA_PP_not_sos_convex_journal}. However, we would
like to highlight an idea behind this proof that will be used
again in this chapter. As the following lemma demonstrates, one
way to ensure a polynomial is \emph{not sos-convex} is by
enforcing one of the principal minors of its Hessian matrix to be
\emph{not sos}.

%First, a natural question is that if a polynomial is not
%sos-convex, then how do we go about proving that it is convex? The
%same question of course applies to polynomial nonnegativity: how
%can prove nonnegativity of a polynomial which is not sos?
%Depending on the polynomial at hand, several different
%tricks/methods can do the job. For example, the Motzkin form in
%(\ref{eq:Motzkin.form}) was shown to be nonnegative by using the
%arithmetic-geometric inequality~\cite{MotzkinSOS}. The approach
%that we will take however throughout this chapter is based on
%using higher order hierarchies of semidefinite programs that
%approximate the psd cone from the inside with cones that strictly
%contain the sos cone. One such hierarchy which is particularly
%simple can be obtained by using a result of
%Reznick~\cite{Reznick_Unif_denominator} related to Hilbert's 17th
%problem. The 17th problem asks whether every psd form must be a
%sum of squares of rational functions. In 1927,
%Artin~\cite{Artin_Hilbert17} answered Hilbert's question in the
%affirmative. This result implies that if a form $g$ is psd, then
%there must exist an sos form $s$, such that $g\cdot s$ is sos.
%Reznick showed in~\cite{Reznick_Unif_denominator} that if $g$ is
%positive definite, one can always take $s(x)=(\sum_i x_i^2)^{r}$,
%for sufficiently large $r$. For all the polynomials that we are
%concerned with in this chapter, it suffices to take $r=1$.

\begin{lemma}\label{lem:sos.matrix.then.minor.sos}
If $P(x)\in \mathbb{R}[x]^{m \times m}$ is an sos-matrix, then all
its $2^{m}-1$ principal minors\footnote{The principal minors of an
$m \times m$ matrix $A$ are the determinants of all $k \times k$
($1\leq k \leq m$) sub-blocks whose rows and columns come from the
same index set $S \subset \{1, \ldots, m\}$.} are sos polynomials.
In particular, $\det(P)$ and the diagonal elements of $P$ must be
sos polynomials.
\end{lemma}

\begin{proof}
We first prove that $\det(P)$ is sos. By
Definition~\ref{def:sos-matrix}, we have $P(x)=M^{T}(x)M(x)$ for
some $s \times m$ polynomial matrix $M(x)$. If $s=m$, we have
$$\det(P)=\det(M^{T})\det(M)=(\det(M))^2$$ and the result is
immediate. If $s>m$, the result follows from the Cauchy-Binet
formula\footnote{Given matrices $A$ and $B$ of size $m \times s$
and $s \times m$ respectively, the Cauchy-Binet formula states
that
$$\det(AB)=\sum_S \det(A_S)\det(B_S),$$ where $S$ is a subset of $\{1, \ldots, s\}$ with $m$ elements, $A_S$ denotes the $m \times m$ matrix whose columns are the columns of $A$ with index from $S$, and similarly $B_S$ denotes the $m \times m$ matrix whose rows are the rows of $B$ with index from
$S$.}. We have $$ \begin{array}{rll}\det(P)&=&\sum_S \det(M^{T})_S
\det(M_S) \\ \\ \quad&=&\sum_S \det(M_S)^{T} \det(M_S)\\
\\ \quad &=&\sum_S(\det(M_S))^2. \end{array}$$ Finally, when $s<m$, $\det(P)$ is zero
which is trivially sos. In fact, the Cauchy-Binet formula also
holds for $s=m$ and $s<m$, but we have separated these cases for
clarity of presentation.

Next, we need to prove that the \aaa{minors corresponding to
smaller principal blocks} of $P$ are also sos. Define
$\mathcal{M}=\{1, \ldots, m\},$ and let $I$ and $J$ be nonempty
subsets of $\mathcal{M}$. Denote by $P_{IJ}$ a sub-block of $P$
with row indices from $I$ and column indices from $J$. It is easy
to see that
$$P_{JJ}=(M^{T})_{J\mathcal{M}}M_{\mathcal{M}J}=(M_{\mathcal{M}J})^TM_{\mathcal{M}J}.$$
Therefore, $P_{JJ}$ is an sos-matrix itself. By the proceeding
argument $\det(P_{JJ})$ must be sos, and hence all the principal
minors are sos.
\end{proof}

\begin{remark}
Interestingly, the converse of
Lemma~\ref{lem:sos.matrix.then.minor.sos} does not hold. A
counterexample is the Hessian of the form $f$ in
(\ref{eq:minim.form.3.6}) that we will present in the next
section. All 7 principal minors of the $3\times 3$ Hessian this
form are sos polynomials, even though the Hessian is not an
sos-matrix. This is in contrast with the fact that a polynomial
matrix is positive semidefinite if and only if all its principal
minors are psd polynomials. The latter statement follows
immediately from the well-known fact that a constant matrix is
positive semidefinite if and only if all its principal minors are
nonnegative.
\end{remark}

\subsection{A ``clean'' example}
We next present another example of a convex but not sos-convex
form whose construction is in fact related to our proof of
NP-hardness of deciding convexity of quartic forms from
Chapter~\ref{chap:nphard.convexity}. The example is in
$C_{6,4}\setminus\Sigma C_{6,4}$ and by contrast to the example of
the previous subsection, it will turn out to be minimal in the
degree but not in the number of variables. What is nice about this
example is that unlike the other examples in this chapter it has
not been derived with the assistance of a computer and
semidefinite programming:

\begin{equation}\label{eq:clean.quartic.convex.not.sos.convex}
\begin{array}{rlll}
q(x_1,\ldots,x_6)&=&x_1^4+x_2^4+x_3^4+x_4^4+x_5^4+x_6^4 \\ \\
     \quad&\
&+2(x_1^2x_2^2+x_1^2x_3^2+x_2^2x_3^2
+x_4^2x_5^2+x_4^2x_6^2+x_5^2x_6^2)\\ \\
\quad&\ &     +\frac{1}{2}(x_1^2x_4^2+x_2^2x_5^2+x_3^2x_6^2)+
x_1^2x_6^2+x_2^2x_4^2+x_3^2x_5^2 \\ \\\quad&\ &
-(x_1x_2x_4x_5+x_1x_3x_4x_6+x_2x_3x_5x_6).
\end{array}
\end{equation}
The proof that this polynomial is convex but not sos-convex can be
extracted from Theorems~\ref{thm:main.reduction}
and~\ref{thm:algeb.ver.of.reduction} of
Chapter~\ref{chap:nphard.convexity}. The reader can observe that
these two theorems put together give us a general procedure for
producing convex but not sos-convex quartic forms from \emph{any}
example of a psd but not sos biquadratic form\footnote{The reader
can refer to Definition~\ref{defn:biquad.forms} of the previous
chapter to recall the definition of a biquadratic form.}. The
biquadratic form that has led to the form above is that of Choi
in~\cite{Choi_Biquadratic}.

The example in (\ref{eq:clean.quartic.convex.not.sos.convex}) also
shows that convex forms that possess strong symmetry properties
can still fail to be sos-convex. The symmetries in this form are
inherited from the rich symmetry structure of the biquadratic form
of Choi (see~\cite{Symmetry_groups_Gatermann_Pablo}). In general,
symmetries are of interest in the study of positive semidefinite
and sums of squares polynomials because the gap between psd and
sos can often behave very differently depending on the symmetry
properties; see e.g.~\cite{Symmetric_quartics_sos}.

%%%AAA: Pablo told me to remove the following counterexample.
%We claim that the form \scalefont{.9}
%\begin{equation}\label{eq:symnmetric.quartic.convex.not.sos.convex}
%\begin{array}{rlll}
%s(x_1,\ldots,x_6)&=&4x_4x_5x_1^2+5x_1x_2x_6^2-10x_1x_3x_4x_6+5x_3^2x_4x_5+9x_1x_3x_4^2+13x_1x_3x_5^2+3x_1x_3x_6^2
%\\
%\  &\ &\ \\
%\  &\
%&+9x_4x_6x_1^2-10x_5x_6x_1^2-11x_1x_3x_5x_6+23x_2^2x_4x_5+5x_5x_1x_2x_6-10x_2x_3x_4^2+13x_2^2x_4x_6
%\\
%\  &\ &\ \\
%\  &\
%&+5x_2x_3x_4x_5-11x_2x_3x_4x_6+13x_2^2x_5x_6+6x_3^2x_4^2+3x_1x_3x_4x_5-5x_2x_3x_5x_6+13x_2x_3x_5^2
%\\
%\  &\ &\ \\
%\  &\
%&+12x_2^2x_5^2+12x_2^2x_4^2+12x_3^2x_5^2+12x_3^2x_6^2+7x_2x_3x_6^2+3x_3^2x_4x_6+7x_3^2x_5x_6+12x_2^2x_6^2
%\\
%\  &\ &\ \\
%\  &\
%&+12x_4^2x_1^2+12x_5^2x_1^2+6x_6^2x_1^2+31x_4x_5x_1x_2+3x_4x_6x_1x_2+4x_4^2x_1x_2+23x_5^2x_1x_2
%\\
%\  &\ &\ \\
%\  &\
%&+12(x_1^4+x_2^4+x_3^4+x_4^4+x_5^4+x_6^4+x_1^2x_2^2+x_1^2x_3^2+x_2^2x_3^2+x_4^2x_5^2+x_4^2x_6^2+x_5^2x_6^2)
%\end{array}
%\end{equation} \normalsize
%is convex, not sos-convex, and satisfies the symmetry
%$$s(x_1,x_2,x_3,x_4,x_5,x_6)=s(x_4,x_5,x_6,x_1,x_2,x_3).$$ Just
%like the form $q$ in
%(\ref{eq:clean.quartic.convex.not.sos.convex}), the form $s$ in
%(\ref{eq:symnmetric.quartic.convex.not.sos.convex}) is also
%derived from a psd but not sos biquaratic form. This biquadratic
%form is a symmetric one and is presented
%in~\cite{AAA_GB_PP_Convex_ternary_quartics}.

\section{Characterization of the gap between convexity and
sos-convexity}\label{sec:full.characterization}

Now that we know there exist convex polynomials that are not
sos-convex, our final and main goal is to give a complete
characterization of the degrees and dimensions in which such
polynomials can exist. This is achieved in the next theorem.

\begin{theorem}\label{thm:full.charac.polys}
$\tilde{\Sigma C}_{n,d}=\tilde{C}_{n,d}$ if and only if $n=1$ or
$d=2$ or $(n,d)=(2,4)$.
\end{theorem}

We would also like to have such a characterization for homogeneous
polynomials. Although convexity is a property that is in some
sense more meaningful for nonhomogeneous polynomials than for
forms, one motivation for studying convexity of forms is in their
relation to norms~\cite{Blenders_Reznick}. Also, in view of the
fact that we have a characterization of the gap between
nonnegativity and sums of squares both for polynomials and for
forms, it is very natural to inquire the same result for convexity
and sos-convexity. The next theorem presents this characterization
for forms.

\begin{theorem}\label{thm:full.charac.forms}
$\Sigma C_{n,d}=C_{n,d}$ if and only if $n=2$ or $d=2$ or
$(n,d)=(3,4)$.
\end{theorem}

The result $\Sigma C_{3,4}=C_{3,4}$ of this theorem is to be
presented in full detail
in~\cite{AAA_GB_PP_Convex_ternary_quartics}. The remainder of this
chapter is solely devoted to the proof of
Theorem~\ref{thm:full.charac.polys} and the proof of
Theorem~\ref{thm:full.charac.forms} except for the case
$(n,d)=(3,4)$. Before we present these proofs, we shall make two
important remarks.

\begin{remark}\label{rmk:difficulty.homogz.dehomogz} {\bf Difficulty with homogenization and dehomogenization.}
Recall from Subsection~\ref{subsec:nonnegativity.sos.basics} and
Theorem~\ref{thm:Hilbert} that characterizing the gap between
nonnegativity and sum of squares for polynomials is equivalent to
accomplishing this task for forms. Unfortunately, the situation is
more complicated for convexity and sos-convexity and that is the
reason why we are presenting Theorems~\ref{thm:full.charac.polys}
and~\ref{thm:full.charac.forms} as separate theorems. The
difficulty arises from the fact that unlike nonnegativity and sum
of squares, convexity and sos-convexity are not always preserved
under homogenization. (Or equivalently, the properties of being
not convex and not sos-convex are not preserved under
dehomogenization.) In fact, any convex polynomial that is not psd
will no longer be convex after homogenization. This is because
convex forms are psd but the homogenization of a non-psd
polynomial is a non-psd form. Even if a convex polynomial is psd,
its homogenization may not be convex. For example the univariate
polynomial $10x_1^4-5x_1+2$ is convex and psd, but its
homogenization $10x_1^4-5x_1x_2^3+2x_2^4$ is not
convex.\footnote{What is true however is that a nonnegative form
of degree $d$ is convex if and only if the $d$-th root of its
dehomogenization is a convex function~\cite[Prop.
4.4]{Blenders_Reznick}.} To observe the same phenomenon for
sos-convexity, consider the trivariate form $p$ in
(\ref{eq:first.convex.not.sos.convex}) which is convex but not
sos-convex and define $\tilde{p}(x_2,x_3)=p(1,x_2,x_3)$. Then, one
can check that $\tilde{p}$ is sos-convex (i.e., its $2\times 2$
Hessian factors) even though its homogenization which is $p$ is
not sos-convex~\cite{AAA_PP_not_sos_convex_journal}.
\end{remark}

\begin{remark}\label{rmk:resemb.to.Hilbert} {\bf Resemblance to the result of Hilbert.} The reader
may have noticed from the statements of Theorem~\ref{thm:Hilbert}
and Theorems~\ref{thm:full.charac.polys}
and~\ref{thm:full.charac.forms} that the cases where convex
polynomials (forms) are sos-convex are exactly the same cases
where nonnegative polynomials are sums of squares! We shall
emphasize that as far as we can tell, our results do not follow
(except in the simplest cases) from Hilbert's result stated in
Theorem~\ref{thm:Hilbert}. Note that the question of convexity or
sos-convexity of a polynomial $p(x)$ in $n$ variables and degree
$d$ is about the polynomials $g_{\frac{1}{2}}(x,y),
g_\nabla(x,y),$ or $g_{\nabla^2}(x,y)$ defined in
(\ref{eq:defn.g_lambda.g_grad.g_grad2}) being psd or sos. Even
though these polynomials still have degree $d$, it is important to
keep in mind that they are polynomials \emph{in $2n$ variables}.
Therefore, there is no direct correspondence with the
characterization of Hilbert. To make this more explicit, let us
consider for example one particular claim of
Theorem~\ref{thm:full.charac.forms}: $\Sigma C_{2,4}=C_{2,4}$. For
a form $p$ in 2 variables and degree 4, the polynomials
$g_{\frac{1}{2}}, g_\nabla,$ and $g_{\nabla^2}$ will be forms in 4
variables and degree 4. We know from Hilbert's result that in this
situation psd but not sos forms do in fact exist. However, for the
forms in 4 variables and degree 4 that have the special structure
of $g_{\frac{1}{2}}, g_\nabla,$ or $g_{\nabla^2}$, psd turns out
to be equivalent to sos.
%
%Therefore, even if we take e.g. the relatively simple case of a
%homogeneous polynomial $p$ in 2 variables and degree 4, then the
%polynomials $g_{\frac{1}{2}}, g_\nabla,$ and $g_{\nabla^2}$ will
%be homogeneous polynomials in 4 variables and degree 4. Hilbert's
%result states that in this situation there are psd forms that are
%not sos. However, we prove that $\Sigma C_{2,4}=C_{2,4}$. This, of
%course, is due to the fact that there is additional structure in
%the polynomials $g_{\frac{1}{2}}, g_\nabla,$ and $g_{\nabla^2}$.
\end{remark}

The proofs of Theorems~\ref{thm:full.charac.polys}
and~\ref{thm:full.charac.forms} are broken into the next two
subsections. In Subsection~\ref{subsec:proof.equal.cases}, we
provide the proofs for the cases where convexity and sos-convexity
are equivalent. Then in
Subsection~\ref{subsec:proof.non.equal.cases}, we prove that in
all other cases there exist convex polynomials that are not
sos-convex.

\subsection{Proofs of Theorems~\ref{thm:full.charac.polys} and~\ref{thm:full.charac.forms}: cases where $\tilde{\Sigma C}_{n,d}=\tilde{C}_{n,d}, \\ \Sigma
C_{n,d}=C_{n,d}$}\label{subsec:proof.equal.cases}

When proving equivalence of convexity and sos-convexity, it turns
out to be more convenient to work with the second order
characterization of sos-convexity, i.e., with the form
$g_{\nabla^2}(x,y)=y^TH(x)y$ in
(\ref{eq:defn.g_lambda.g_grad.g_grad2}). The reason for this is
that this form is always quadratic in $y$, and this allows us to
make use of the following key theorem, henceforth referred to as
the ``biform theorem''.

\begin{theorem}[e.g.~\cite{CLRrealzeros}]\label{thm:biform.thm}
Let $f\mathrel{\mathop:}=f(u_1,u_2,v_1,\ldots,v_m)$ be a form in
the variables $u\mathrel{\mathop:}=(u_1,u_2)^T$ and
$v\mathrel{\mathop:}=(v_1,,\ldots,v_m)^T$ that is a quadratic form
in $v$ for fixed $u$ and a form (of however large degree) in $u$
for fixed $v$. Then $f$ is psd if and only if it is
sos.\footnote{Note that the results $\Sigma_{2,d}=P_{2,d}$ and
$\Sigma_{n,2}=P_{n,2}$ are both special cases of this theorem.}
\end{theorem}

The biform theorem has been proven independently by several
authors. See~\cite{CLRrealzeros} and~\cite{SOS_KYP} for more
background on this theorem and in particular~\cite[Sec.
7]{CLRrealzeros} for a an elegant proof and some refinements. We
now proceed with our proofs which will follow in a rather
straightforward manner from the biform theorem.

\begin{theorem}
$\tilde{\Sigma C}_{1,d}=\tilde{C}_{1,d}$ for all $d$. $\Sigma
C_{2,d}=C_{2,d}$ for all $d$.
\end{theorem}

\begin{proof}
For a univariate polynomial, convexity means that the second
derivative, which is another univariate polynomial, is psd. Since
$\tilde{\Sigma}_{1,d}=\tilde{P}_{1,d}$, the second derivative must
be sos. Therefore, $\tilde{\Sigma C}_{1,d}=\tilde{C}_{1,d}$. To
prove $\Sigma C_{2,d}=C_{2,d}$, suppose we have a convex bivariate
form $p$ of degree $d$ in variables
$x\mathrel{\mathop:}=(x_1,x_2)^T$. The Hessian
$H\mathrel{\mathop:}=H(x)$ of $p$ is a $2\times 2$ matrix whose
entries are forms of degree $d-2$. If we let
$y\mathrel{\mathop:}=(y_1,y_2)^T$, convexity of $p$ implies that
the form $y^TH(x)y$ is psd. Since $y^TH(x)y$ meets the
requirements of the biform theorem above with
$(u_1,u_2)=(x_1,x_2)$ and $(v_1,v_2)=(y_1,y_2)$, it follows that
$y^TH(x)y$ is sos. Hence, $p$ is sos-convex.
\end{proof}

\begin{theorem}
$\tilde{\Sigma C}_{n,2}=\tilde{C}_{n,2}$ for all $n$. $\Sigma
C_{n,2}=C_{n,2}$ for all $n$.
\end{theorem}

\begin{proof}
Let $x\mathrel{\mathop:}=(x_1,\ldots,x_n)^T$ and
$y\mathrel{\mathop:}=(y_1,\ldots,y_n)^T$. Let
$p(x)=\frac{1}{2}x^TQx+b^Tx+c$ be a quadratic polynomial. The
Hessian of $p$ in this case is the constant symmetric matrix $Q$.
Convexity of $p$ implies that $y^TQy$ is psd. But since
$\Sigma_{n,2}=P_{n,2}$, $y^TQy$ must be sos. Hence, $p$ is
sos-convex. The proof of $\Sigma C_{n,2}=C_{n,2}$ is identical.
\end{proof}

\begin{theorem}
$\tilde{\Sigma C}_{2,4}=\tilde{C}_{2,4}$.
\end{theorem}

\begin{proof}
Let $p(x)\mathrel{\mathop:}=p(x_1,x_2)$ be a convex bivariate
quartic polynomial. Let $H\mathrel{\mathop:}=H(x)$ denote the
Hessian of $p$ and let $y\mathrel{\mathop:}=(y_1,y_2)^T$. Note
that $H(x)$ is a $2\times 2$ matrix whose entries are (not
necessarily homogeneous) quadratic polynomials. Since $p$ is
convex, $y^TH(x)y$ is psd. Let $\bar{H}(x_1,x_2,x_3)$ be a
$2\times 2$ matrix whose entries are obtained by homogenizing the
entries of $H$. It is easy to see that $y^T\bar{H}(x_1,x_2,x_3)y$
is then the form obtained by homogenizing $y^TH(x)y$ and is
therefore psd. Now we can employ the biform theorem
(Theorem~\ref{thm:biform.thm}) with $(u_1,u_2)=(y_1,y_2)$ and
$(v_1,v_2,v_3)=(x_1,x_2,x_3)$ to conclude that
$y^T\bar{H}(x_1,x_2,x_3)y$ is sos. But upon dehomogenizing by
setting $x_3=1$, we conclude that $y^TH(x)y$ is sos. Hence, $p$ is
sos-convex.
\end{proof}

\begin{theorem}[Ahmadi, Blekherman, Parrilo~\cite{AAA_GB_PP_Convex_ternary_quartics}]
$\Sigma C_{3,4}=C_{3,4}$.
\end{theorem}

Unlike Hilbert's results $\tilde{\Sigma}_{2,4}=\tilde{P}_{2,4}$
and $\Sigma_{3,4}=P_{3,4}$ which are equivalent statements and
essentially have identical proofs, the proof of $\Sigma
C_{3,4}=C_{3,4}$ is considerably more involved than the proof of
$\tilde{\Sigma C}_{2,4}=\tilde{C}_{2,4}$. Here, we briefly point
out why this is the case and refer the reader
to~\cite{AAA_GB_PP_Convex_ternary_quartics} for more details.

If $p(x)\mathrel{\mathop:}=p(x_1,x_2,x_3)$ is a ternary quartic
form, its Hessian $H(x)$ is a $3\times 3$ matrix whose entries are
quadratic forms. In this case, we can no longer apply the biform
theorem to the form $y^TH(x)y$. In fact, the matrix
\begin{equation}\label{eq:choi.matrix}
C(x)=\begin{bmatrix} x_1^2+2x_2^2&-x_1x_2&-x_1x_3 \\ \\
-x_1x_2&x_2^2+2x_3^2&-x_2x_3 \\ \\
-x_1x_3&-x_2x_3&x_3^2+2x_1^2
\end{bmatrix},
\end{equation}
due to Choi~\cite{Choi_Biquadratic} serves as an explicit example
of a $3\times 3$ matrix with quadratic form entries that is
positive semidefinite but not an sos-matrix; i.e., the biquadratic
form $y^TC(x)y$ is psd but not sos. However, the matrix $C(x)$
above is \emph{not} a valid Hessian, i.e., it cannot be the matrix
of the second derivatives of any polynomial. If this was the case,
the third partial derivatives would commute. On the other hand, we
have in particular
$$\frac{\partial C_{1,1}(x)}{\partial x_3}=0\neq-x_3=\frac{\partial C_{1,3}(x)}{\partial
x_1}.$$

A \emph{biquadratic Hessian form} is a biquadratic form $y^TH(x)y$
where $H(x)$ is the Hessian of some quartic form. Biquadratic
Hessian forms satisfy a special symmetry property. Let us call a
biquadratic form $b(x;y)$ \emph{symmetric} if it satisfies the
symmetry relation $b(y;x)=b(x;y)$. It is an easy exercise to show
that biquadratic Hessian forms satisfy $y^TH(x)y=x^TH(y)x$ and are
therefore symmetric biquadratic forms. This symmetry property is a
rather strong condition that is not satisfied e.g. by the Choi
biquadratic form $y^TC(x)y$ in (\ref{eq:choi.matrix}).

A simple dimension counting argument shows that the vector space
of biquadratic forms, symmetric biquadratic forms, and biquadratic
Hessian forms in variables $(x_1,x_2,x_3;y_1,y_2,y_3)$
respectively have dimensions $36$, $21$, and $15$. Since the
symmetry requirement drops the dimension of the space of
biquadratic forms significantly, and since sos polynomials are
known to generally cover much larger volume in the set of psd
polynomials in presence of symmetries (see
e.g.~\cite{Symmetric_quartics_sos}), one may initially suspect (as
we did) that the equivalence between psd and sos ternary Hessian
biquadratic forms is a consequence of the symmetry property. Our
next theorem shows that interestingly enough this is not the case.

\begin{theorem}\label{thm:sym.biquad.psd.not.sos}
There exist symmetric biquadratic forms in two sets of three
variables that are positive semidefinite but not a sum of squares.
\end{theorem}

\begin{proof}
We claim that the following biquadratic form has the required
properties:
\begin{equation}\label{eq:sym.biquad.psd.not.sos}
\begin{array}{rlll}
b(x_1,x_2,x_3;y_1,y_2,y_3)&=&4y_1y_2x_1^2+4x_1x_2y_1^2+9y_1y_3x_1^2+9x_1x_3y_1^2-10y_2y_3x_1^2 \\ \  &\ &\ \\
\  &\
&-10x_2x_3y_1^2+12y_1^2x_1^2+12y_2^2x_1^2+12x_2^2y_1^2+6y_3^2x_1^2
 \\ \  &\ &\ \\
\  &\
&+6x_3^2y_1^2+23x_2^2y_1y_2+23y_2^2x_1x_2+13x_2^2y_1y_3+13x_1x_3y_2^2
 \\ \  &\ &\ \\
\  &\
&+13y_2y_3x_2^2+13x_2x_3y_2^2+12x_2^2y_2^2+12x_2^2y_3^2+12y_2^2x_3^2
 \\ \  &\ &\ \\
\  &\ & +5x_3^2y_1y_2+5y_3^2x_1x_2+12x_3^2y_3^2+3x_3^2y_1y_3+3y_3^2x_1x_3 \\  \  &\ &\ \\
\  &\ &+7x_3^2y_2y_3+7y_3^2x_2x_3+31y_1y_2x_1x_2-10x_1x_3y_1y_3
 \\ \  &\ &\ \\
\  &\
&-11x_1x_3y_2y_3-11y_1y_3x_2x_3+5x_1x_2y_2y_3+5y_1y_2x_2x_3 \\ \  &\ &\ \\
\  &\ &+3x_1x_3y_1y_2+3y_1y_3x_1x_2-5x_2x_3y_2y_3.
\end{array}
\end{equation}
The fact that $b(x;y)=b(y;x)$ can readily be seen from the order
in which we have written the monomials. The proof that $b(x;y)$ is
psd but not sos is given
in~\cite{AAA_GB_PP_Convex_ternary_quartics} and omitted from here.
\end{proof}

In view of the above theorem, it is rather remarkable that all
positive semidefinite biquadratic Hessian forms in
$(x_1,x_2,x_3;y_1,y_2,y_3)$ turn out to be sums of squares, i.e.,
that $\Sigma C_{3,4}=C_{3,4}$.

\subsection{Proofs of Theorems~\ref{thm:full.charac.polys} and~\ref{thm:full.charac.forms}: cases where $\tilde{\Sigma C}_{n,d}\subset\tilde{C}_{n,d},\\ \Sigma
C_{n,d}\subset C_{n,d}$}\label{subsec:proof.non.equal.cases}

The goal of this subsection is to establish that the cases
presented in the previous subsection are the \emph{only} cases
where convexity and sos-convexity are equivalent. We will first
give explicit examples of convex but not sos-convex
polynomials/forms that are ``minimal'' jointly in the degree and
dimension and then present an argument for all dimensions and
degrees higher than those of the minimal cases.

\subsubsection{Minimal convex but not sos-convex polynomials/forms}
 The minimal examples of convex but not sos-convex
 polynomials (resp. forms) turn out to belong to $\tilde{C}_{2,6}\setminus\tilde{\Sigma C}_{2,6}$ and $\tilde{C}_{3,4}\setminus\tilde{\Sigma C}_{3,4}$ (resp. $C_{3,6}\setminus\Sigma C_{3,6}$ and $C_{4,4}\setminus\Sigma C_{4,4}$).
 Recall from Remark~\ref{rmk:difficulty.homogz.dehomogz} that we
 lack a general argument for going from convex but not sos-convex
 forms to polynomials or vice versa. Because of this, one would need to
 present four different polynomials in the sets mentioned above
 and prove that each polynomial is (i) convex and (ii) not
 sos-convex. This is a total of eight arguments to make which is
 quite cumbersome. However, as we will see in the proof of Theorem~\ref{thm:minimal.2.6.and.3.6} and~\ref{thm:minimal.3.4.and.4.4} below, we have been able to find examples that
 act ``nicely'' with respect to particular ways of dehomogenization. This will allow us to reduce the
 total number of claims we have to prove from eight to four.

%%%AAA: Pablo asked me to remove the following lemma.
%\begin{lemma}\label{lem:conv.sos-conv.preserv.dh}
%Let $p\mathrel{\mathop:}=p(x_1,\ldots,x_n)$ be a form. Define the
%polynomial $\tilde{p}$ as
%$$\tilde{p}(x_1,\ldots,x_{n-1})\mathrel{\mathop:}=p(x_1,\ldots,x_{n-1},a_0+a_1x_1+\cdots+a_{n-1}x_{n-1}),$$
%where $a_0,\ldots,a_{n-1}$ are some constants. If $p$ is convex,
%then so is $\tilde{p}$. If $p$ is sos-convex, then so is
%$\tilde{p}$.
%\end{lemma}
%
%\begin{proof}
%We only prove the claim about sos-convexity. The proof for
%convexity is similar and in fact standard. Let
%$x\mathrel{\mathop:}=(x_1,\ldots,x_n)^T$,
%$y\mathrel{\mathop:}=(y_1,\ldots,y_n)^T$,
%$\tilde{x}\mathrel{\mathop:}=(x_1,\ldots,x_{n-1})^T$, and
%$\tilde{y}\mathrel{\mathop:}=(y_1,\ldots,y_{n-1})^T$. If $p$ is
%sos-convex, then by
%Theorem~\ref{thm:sos.convexity.3.equivalent.defs} the form
%$$g_{\frac{1}{2}}(x,y)=\frac{1}{2}p(x)+\frac{1}{2}p(y)-p\textstyle{\left(\frac{1}{2}x+\frac{1}{2}x\right)}$$
%is sos. Let
%$$\tilde{g}_{\frac{1}{2}}(\tilde{x},\tilde{y})=\frac{1}{2}\tilde{p}(\tilde{x})+\frac{1}{2}\tilde{p}(\tilde{y})-\tilde{p}\textstyle{\left(\frac{1}{2}\tilde{x}+\frac{1}{2}\tilde{x}\right)}.$$
%It is easy to check that
%$$\tilde{g}_{\frac{1}{2}}(\tilde{x},\tilde{y})=g_{\frac{1}{2}}(\tilde{x},a_0+a_1x_1+\cdots+a_{n-1}x_{n-1},\tilde{y},a_0+a_1y_1+\cdots+a_{n-1}y_{n-1}).$$
%Therefore, $\tilde{g}_{\frac{1}{2}}$ is sos, which means that
%$\tilde{p}$ is sos-convex.
%\end{proof}

 The polynomials that we are about to present next have been
 found with the assistance of a computer and by employing some ``tricks'' with semidefinite
 programming similar to those presented in Appendix A.\footnote{The approach of Appendix A, however, does not lead to examples that are minimal. But the idea is similar.} In this process, we have made use of software
 packages YALMIP~\cite{yalmip}, SOSTOOLS~\cite{sostools}, and the
 SDP solver SeDuMi~\cite{sedumi}, which we acknowledge here. To
 make the chapter relatively self-contained and to emphasize the
 fact that using \emph{rational sum of squares certificates} one
 can make such computer assisted proofs fully formal, we present the proof of Theorem~\ref{thm:minimal.2.6.and.3.6} below in the Appendix B. On the other hand, the proof of Theorem~\ref{thm:minimal.3.4.and.4.4}, which is very similar in style to the proof of Theorem~\ref{thm:minimal.2.6.and.3.6}, is largely omitted to save space. All of the proofs are available in electronic
 form and in their entirety at \texttt{http://aaa.lids.mit.edu/software}.

\begin{theorem}\label{thm:minimal.2.6.and.3.6} $\tilde{\Sigma C}_{2,6}$ is a proper subset of
$\tilde{C}_{2,6}$. $\Sigma C_{3,6}$ is a proper subset of
$C_{3,6}$.
\end{theorem}

\begin{proof}
We claim that the form
\begin{equation}\label{eq:minim.form.3.6}
\begin{array}{lll}
f(x_1,x_2,x_3)&=&77x_1^6-155x_1^5x_2+445x_1^4x_2^2+76x_1^3x_2^3+556x_1^2x_2^4+68x_1x_2^5 \\
\  &\ &\ \\
   \  &\ &+240x_2^6-9x_1^5x_3-1129x_1^3x_2^2x_3+62x_1^2x_2^3x_3+1206x_1x_2^4x_3 \\
   \  &\ &\ \\
    \ &\ &-343x_2^5x_3+363x_1^4x_3^2+773x_1^3x_2x_3^2+891x_1^2x_2^2x_3^2-869x_1x_2^3x_3^2\\
    \  &\ &\ \\
    \ &\ &+1043x_2^4x_3^2-14x_1^3x_3^3-1108x_1^2x_2x_3^3-216x_1x_2^2x_3^3-839x_2^3x_3^3\\
    \  &\ &\ \\
    \ &\ &+721x_1^2x_3^4+436x_1x_2x_3^4+378x_2^2x_3^4+48x_1x_3^5-97x_2x_3^5+89x_3^6
\end{array}
\end{equation}
belongs to $C_{3,6}\setminus\Sigma C_{3,6}$, and the
polynomial\footnote{The polynomial $f(x_1,x_2,1)$ turns out to be
sos-convex, and therefore does not do the job. One can of course
change coordinates, and then in the new coordinates perform the
dehomogenization by setting $x_3=1$.}
\begin{equation}\label{eq:minim.poly.2.6}
\tilde{f}(x_1,x_2)=f(x_1,x_2,1-\frac{1}{2}x_2)
\end{equation}
belongs to $\tilde{C}_{2,6}\setminus\tilde{\Sigma C}_{2,6}$. Note
that since convexity and sos-convexity are both preserved under
restrictions to affine subspaces (recall
Remark~\ref{rmk:sos-convexity.restriction}), it suffices to show
that the form $f$ in (\ref{eq:minim.form.3.6}) is convex and the
polynomial $\tilde{f}$ in (\ref{eq:minim.poly.2.6}) is not
sos-convex. Let $x\mathrel{\mathop:}=(x_1,x_2,x_2)^T$,
$y\mathrel{\mathop:}=(y_1,y_2,y_3)^T$,
$\tilde{x}\mathrel{\mathop:}=(x_1,x_2)^T$,
$\tilde{y}\mathrel{\mathop:}=(y_1,y_2)^T$, and denote the Hessian
of $f$ and $\tilde{f}$ respectively by $H_f$ and $H_{\tilde{f}}$.
In Appendix B, we provide rational Gram matrices which prove that
the form
\begin{equation}\label{eq:y.H_f.y.xi^2.3.6.example}
(x_1^2+x_2^2)\cdot y^TH_f(x)y
\end{equation}
is sos. This, together with nonnegativity of $x_1^2+x_2^2$ and
continuity of $y^TH_f(x)y$, implies that $y^TH_f(x)y$ is psd.
Therefore, $f$ is convex. The proof that $\tilde{f}$ is not
sos-convex proceeds by showing that $H_{\tilde{f}}$ is not an
sos-matrix via a separation argument. In Appendix B, we present a
separating hyperplane that leaves the appropriate sos cone on one
side and the polynomial
\begin{equation}\label{eq:y.H_f_tilda.y.xi^2.2.6.example}
\tilde{y}^TH_{\tilde{f}}(\tilde{x})\tilde{y}
\end{equation}
on the other.
\end{proof}

\begin{theorem}\label{thm:minimal.3.4.and.4.4} $\tilde{\Sigma C}_{3,4}$ is a proper subset of
$\tilde{C}_{3,4}$. $\Sigma C_{4,4}$ is a proper subset of
$C_{4,4}$.
\end{theorem}

\begin{proof}
We claim that the form
\begin{equation}\label{eq:minim.form.4.4}
\begin{array}{lll}
h(x_1,\ldots,x_4)&=&1671x_1^4-4134x_1^3x_2-3332x_1^3x_3+5104x_1^2x_2^2+4989x_1^2x_2x_3 \\
\  &\ &\ \\
 \  &\ &+3490x_1^2x_3^2-2203x_1x_2^3-3030x_1x_2^2x_3-3776x_1x_2x_3^2 \\
 \  &\ &\ \\
 \  &\ &-1522x_1x_3^3+1227x_2^4-595x_2^3x_3+1859x_2^2x_3^2+1146x_2x_3^3 \\
 \  &\ &\ \\
 \  &\ &+1195728x_4^4-1932x_1x_4^3-2296x_2x_4^3-3144x_3x_4^3+1465x_1^2x_4^2 \\
 \  &\ &\ \\
 \  &\ &-1376x_1^3x_4-263x_1x_2x_4^2+2790x_1^2x_2x_4+2121x_2^2x_4^2+979x_3^4 \\
 \  &\ &\ \\
 \  &\ &-292x_1x_2^2x_4-1224x_2^3x_4+2404x_1x_3x_4^2+2727x_2x_3x_4^2 \\
  \  &\ &\ \\
 \  &\ & -2852x_1x_3^2x_4-388x_2x_3^2x_4-1520x_3^3x_4+2943x_1^2x_3x_4\\
  \  &\ &\ \\
 \  &\ &-5053x_1x_2x_3x_4 +2552x_2^2x_3x_4+3512x_3^2x_4^2
\end{array}
\end{equation}
belongs to $C_{4,4}\setminus\Sigma C_{4,4}$, and the polynomial
\begin{equation}\label{eq:minim.poly.3.4}
\tilde{h}(x_1,x_2,x_3)=h(x_1,x_2,x_3,1)
\end{equation}
belongs to $\tilde{C}_{3,4}\setminus\tilde{\Sigma C}_{3,4}$. Once
again, it suffices to prove that $h$ is convex and $\tilde{h}$ is
not sos-convex. Let $x\mathrel{\mathop:}=(x_1,x_2,x_3,x_4)^T$,
$y\mathrel{\mathop:}=(y_1,y_2,y_3,y_4)^T$, and denote the Hessian
of $h$ and $\tilde{h}$ respectively by $H_h$ and $H_{\tilde{h}}$.
The proof that $h$ is convex is done by showing that the form
\begin{equation}\label{eq:y.H_h.y.xi^2.4.4.example}
(x_2^2+x_3^2+x_4^2)\cdot y^TH_h(x)y
\end{equation}
is sos.\footnote{The choice of multipliers in
(\ref{eq:y.H_f.y.xi^2.3.6.example}) and
(\ref{eq:y.H_h.y.xi^2.4.4.example}) is motivated by a result of
Reznick in~\cite{Reznick_Unif_denominator} explained in Appendix
A.} The proof that $\tilde{h}$ is not sos-convex is done again by
means of a separating hyperplane.
\end{proof}

\subsubsection{Convex but not sos-convex polynomials/forms in all higher degrees and
dimensions}\label{subsubsec:increasing.degree.and.vars} Given a
convex but not sos-convex polynomial (form) in $n$ variables , it
is very easy to argue that such a polynomial (form) must also
exist in a larger number of variables. If $p(x_1,\ldots,x_n)$ is a
form in $C_{n,d}\setminus\Sigma C_{n,d}$, then
$$\bar{p}(x_1,\ldots,x_{n+1})=p(x_1,\ldots,x_n)+x_{n+1}^d$$
belongs to $C_{n+1,d}\setminus\Sigma C_{n+1,d}$. Convexity of
$\bar{p}$ is obvious since it is a sum of convex functions. The
fact that $\bar{p}$ is not sos-convex can also easily be seen from
the block diagonal structure of the Hessian of $\bar{p}$: if the
Hessian of $\bar{p}$ were to factor, it would imply that the
Hessian of $p$ should also factor. The argument for going from
$\tilde{C}_{n,d}\setminus\tilde{\Sigma C}_{n,d}$ to
$\tilde{C}_{n+1,d}\setminus\tilde{\Sigma C}_{n+1,d}$ is identical.

Unfortunately, an argument for increasing the degree of convex but
not sos-convex forms seems to be significantly more difficult to
obtain. In fact, we have been unable to come up with a natural
operation that would produce a from in $C_{n,d+2}\setminus\Sigma
C_{n,d+2}$ from a form in $C_{n,d}\setminus\Sigma C_{n,d}$. We
will instead take a different route: we are going to present a
general procedure for going from a form in
$P_{n,d}\setminus\Sigma_{n,d}$ to a form in
$C_{n,d+2}\setminus\Sigma C_{n,d+2}$.  This will serve our purpose
of constructing convex but not sos-convex forms in higher degrees
and is perhaps also of independent interest in itself. For
instance, it can be used to construct convex but not sos-convex
forms that inherit structural properties (e.g. symmetry) of the
known examples of psd but not sos forms. The procedure is
constructive modulo the value of two positive constants ($\gamma$
and $\alpha$ below) whose existence will be shown
nonconstructively.

%\footnote{The procedure can be thought of as a generalization of
%the approach in our earlier work
%in~\cite{AAA_PP_not_sos_convex_journal}.}

Although the proof of the general case is no different, we present
this construction for the case $n=3$. The reason is that it
suffices for us to construct forms in $C_{3,d}\setminus\Sigma
C_{3,d}$ for $d$ even and $\geq 8$. These forms together with the
two forms in $C_{3,6}\setminus\Sigma C_{3,6}$ and
$C_{4,4}\setminus\Sigma C_{4,4}$ presented in
(\ref{eq:minim.form.3.6}) and (\ref{eq:minim.form.4.4}), and with
the simple procedure for increasing the number of variables cover
all the values of $n$ and $d$ for which convex but not sos-convex
forms exist.

For the remainder of this section, let
$x\mathrel{\mathop:}=(x_1,x_2,x_3)^T$ and
$y\mathrel{\mathop:}=(y_1,y_2,y_3)^T$.

\begin{theorem}\label{thm:conv_not_sos_conv_forms_n3d}
Let $m\mathrel{\mathop:}=m(x)$ be a ternary form of degree $d$
(with $d$ necessarily even and $\geq 6$) satisfying the following
three requirements:
%\begin{itemize}
\begin{description}
\item[R1:] $m$ is positive definite. \item[R2:] $m$ is not a sum
of squares. \item[R3:] The Hessian $H_m$ of $m$ is positive
definite at the point $(1,0,0)^T$.
%\begin{equation}\nonumber
%\Big[\frac{\partial{\ }}{\partial{x}\partial{x}} \int\int m(x)
%dx_1 dx_1 \Big] \Big\vert_{x_2=0, x_3=0} \succ 0 \quad \forall
%x_1\neq0,
%\end{equation}
%i.e., the $3\times 3$ Hessian of the form $\int\int m(x) dx_1
%dx_1$, when restricted to $x_2=0, x_3=0$, is positive definite.
\end{description}
%\end{itemize}
Let $g\mathrel{\mathop:}=g(x_2,x_3)$ be any bivariate form of
degree $d+2$ whose Hessian is positive definite. \\ Then, there
exists a constant $\gamma>0$, such that the form $f$ of degree
$d+2$ given by
\begin{equation}\label{eq:construction.of.f(x)}
f(x)=\int_0^{x_1}\int_0^s m(t,x_2,x_3) dt ds + \gamma g(x_2,x_3)
\end{equation}
is convex but not sos-convex.
\end{theorem}

Before we prove this theorem, let us comment on how one can get
examples of forms $m$ and $g$ that satisfy the requirements of the
theorem. The choice of $g$ is in fact very easy. We can e.g. take
\begin{equation}\nonumber
g(x_2,x_3)=(x_2^2+x_3^2)^{\frac{d+2}{2}},
\end{equation}
which has a positive definite Hessian. As for the choice of $m$,
essentially any psd but not sos ternary form can be turned into a
form that satisfies requirements {\bf R1}, {\bf R2}, and {\bf R3}.
Indeed if the Hessian of such a form is positive definite at just
one point, then that point can be taken to $(1,0,0)^T$ by a change
of coordinates without changing the properties of being psd and
not sos. If the form is not positive definite, then it can made so
by adding a small enough multiple of a positive definite form to
it. For concreteness, we construct in the next lemma a family of
forms that together with the above theorem will give us convex but
not sos-convex ternary forms of any degree $\geq 8$.

\begin{lemma}\label{lem:choice.of.m(x)}
For any even degree $d\geq 6$, there exists a constant $\alpha>0$,
such that the form
\begin{equation}\label{eq:choice.of.m(x)}
m(x)=x_1^{d-6}(x_1^2x_2^4+x_1^4x_2^2-3x_1^2x_2^2x_3^2+x_3^6)+\alpha(x_1^2+x_2^2+x_3^2)^{\frac{d}{2}}
\end{equation}
satisfies the requirements {\bf R1}, {\bf R2}, and {\bf R3} of
Theorem~\ref{thm:conv_not_sos_conv_forms_n3d}.
\end{lemma}

\begin{proof}
The form $$x_1^2x_2^4+x_1^4x_2^2-3x_1^2x_2^2x_3^2+x_3^6$$ is the
familiar Motzkin form in (\ref{eq:Motzkin.form}) that is psd but
not sos~\cite{MotzkinSOS}. For any even degree $d\geq 6$, the form
$$x_1^{d-6}(x_1^2x_2^4+x_1^4x_2^2-3x_1^2x_2^2x_3^2+x_3^6)$$ is a
form of degree $d$ that is clearly still psd and less obviously
still not sos; see~\cite{Reznick}. This together with the fact
that $\Sigma_{n,d}$ is a closed cone implies existence of a small
positive value of $\alpha$ for which the form $m$ in
(\ref{eq:choice.of.m(x)}) is positive definite but not a sum of
squares, hence satisfying requirements {\bf R1} and {\bf R2}.

Our next claim is that for any positive value of $\alpha$, the
Hessian $H_m$ of the form $m$ in (\ref{eq:choice.of.m(x)})
satisfies
%\begin{equation}\label{eq:Hessian.restricted.to.x2=x3=0}
%\Big[\frac{\partial{\ }}{\partial{x}\partial{x}} \int\int m(x)
%dx_1 dx_1 \Big] \Big\vert_{x_2=0, x_3=0}=\begin{bmatrix} c_1x_1^d
%& 0 & 0 \\ 0 & c_2x_1^d & 0 \\ 0& 0& c_3x_1^d
%\end{bmatrix}
%\end{equation}
\begin{equation}\label{eq:Hessian.at.1.0.0.}
H_m(1,0,0)=\begin{bmatrix} c_1 & 0 & 0 \\ 0 & c_2 & 0 \\
0& 0& c_3
\end{bmatrix}
\end{equation}
for some positive constants $c_1,c_2,c_3$, therefore also passing
requirement {\bf R3}. To see the above equality, first note that
since $m$ is a form of degree $d$, its Hessian $H_m$ will have
entries that are forms of degree $d-2$. Therefore, the only
monomials that can survive in this Hessian after setting $x_2$ and
$x_3$ to zero are multiples of $x_1^{d-2}$. It is easy to see that
an $x_1^{d-2}$ monomial in an off-diagonal entry of $H_m$ would
lead to a monomial in $m$ that is not even. On the other hand, the
form $m$ in (\ref{eq:choice.of.m(x)}) only has even monomials.
This explains why the off-diagonal entries of the right hand side
of (\ref{eq:Hessian.at.1.0.0.}) are zero. Finally, we note that
for any positive value of $\alpha$, the form $m$ in
(\ref{eq:choice.of.m(x)}) includes positive multiples of $x_1^d$,
$x_1^{d-2}x_2^2$, and $x_1^{d-2}x_3^2$, which lead to positive
multiples of $x_1^{d-2}$ on the diagonal of $H_m$. Hence, $c_1,
c_2$, and $c_3$ are positive.
\end{proof}

Next, we state a lemma that will be employed in the proof of
Theorem~\ref{thm:conv_not_sos_conv_forms_n3d}.

\begin{lemma}\label{lem:y.Hm.y>0.on.some.part}
Let $m$ be a trivariate form satisfying the requirements {\bf R1}
and {\bf R3} of Theorem~\ref{thm:conv_not_sos_conv_forms_n3d}. Let
$H_{\hat{m}}$ denote the Hessian of the form $\int_0^{x_1}\int_0^s
m(t,x_2,x_3) dt ds$. Then, there exists a positive constant
$\delta,$ such that
$$y^TH_{\hat{m}}(x)y>0$$ on the set
\begin{equation}\label{eq:the.set.S}
\mathcal{S}\mathrel{\mathop:}=\{(x,y) \ |\  ||x||=1, ||y||=1, \
(x_2^2+x_3^2<\delta \ \mbox{or}\ y_2^2+y_3^2<\delta)\}.
\end{equation}
\end{lemma}

\begin{proof}
We observe that when $y_2^2+y_3^2=0$, we have
$$y^TH_{\hat{m}}(x)y=y_1^2m(x),$$ which by requirement {\bf R1} is
positive when $||x||=||y||=1$. By continuity of the form
$y^TH_{\hat{m}}(x)y$, we conclude that there exists a small
positive constant $\delta_y$ such that $y^TH_{\hat{m}}(x)y>0$ on
the set
$$\mathcal{S}_y\mathrel{\mathop:}=\{(x,y) \ |\  ||x||=1, ||y||=1, \ y_2^2+y_3^2<\delta_y\}.$$
Next, we leave it to the reader to check that
$$H_{\hat{m}}(1,0,0)=\frac{1}{d(d-1)}H_m(1,0,0).$$ Therefore, when $x_2^2+x_3^2=0$, requirement {\bf R3} implies that
$y^TH_{\hat{m}}(x)y$ is positive when $||x||=||y||=~1$. Appealing
to continuity again, we conclude that there exists a small
positive constant $\delta_x$ such that $y^TH_{\hat{m}}(x)y>0$ on
the set
$$\mathcal{S}_x\mathrel{\mathop:}=\{(x,y) \ |\  ||x||=1, ||y||=1, \
x_2^2+x_3^2<\delta_x\}.$$ If we now take
$\delta=\min\{\delta_y,\delta_x\}$, the lemma is established.
\end{proof}

%The last lemma that we need has already been proven in our earlier
%work in~\cite{AAA_PP_not_sos_convex_journal}.
%
%
%\begin{lemma}[\cite{AAA_PP_not_sos_convex_journal}]\label{lem:sos.mat.then.minor.sos}
%All principal minors of an sos-matrix are sos
%polynomials.\footnote{As a side note, we remark that the converse
%of Lemma~\ref{lem:sos.mat.then.minor.sos} is not true even for
%polynomial matrices that are valid Hessians. For example, all 7
%principal minors of the $3\times 3$ Hessian of the form $f$ in
%(\ref{eq:minim.form.3.6}) are sos polynomials, even though this
%Hessian is not an sos-matrix.}
%\end{lemma}

We are now ready to prove
Theorem~\ref{thm:conv_not_sos_conv_forms_n3d}.

\begin{proof}[Proof of Theorem~\ref{thm:conv_not_sos_conv_forms_n3d}]
We first prove that the form $f$ in
(\ref{eq:construction.of.f(x)}) is not sos-convex. By
Lemma~\ref{lem:sos.matrix.then.minor.sos}, if $f$ was sos-convex,
then all diagonal elements of its Hessian would have to be sos
polynomials. On the other hand, we have from
(\ref{eq:construction.of.f(x)}) that
$$\frac{\partial{f(x)}}{\partial{x_1}\partial{x_1}}=m(x),$$ which by requirement {\bf R2} is not sos. Therefore $f$ is not sos-convex.

It remains to show that there exists a positive value of $\gamma$
for which $f$ becomes convex. Let us denote the Hessians of $f$,
$\int_0^{x_1}\int_0^s m(t,x_2,x_3) dt ds$, and $g$, by $H_f$,
$H_{\hat{m}}$, and $H_g$ respectively. So, we have
$$H_f(x)=H_{\hat{m}}(x)+\gamma H_g(x_2,x_3).$$ (Here, $H_g$ is a
$3\times 3$ matrix whose first row and column are zeros.)
Convexity of $f$ is of course equivalent to nonnegativity of the
form $y^TH_f(x)y$. Since this form is bi-homogeneous in $x$ and
$y$, it is nonnegative if and only if $y^TH_f(x)y\geq0$ on the
bi-sphere
$$\mathcal{B}\mathrel{\mathop:}=\{(x,y) \ | \ ||x||=1,
||y||=1\}.$$ Let us decompose the bi-sphere as
$$\mathcal{B}=\mathcal{S}\cup\bar{\mathcal{S}},$$
where $\mathcal{S}$ is defined in (\ref{eq:the.set.S}) and
\begin{equation}\nonumber
\bar{\mathcal{S}}\mathrel{\mathop:}=\{(x,y) \ |\  ||x||=1,
||y||=1, x_2^2+x_3^2\geq\delta,  y_2^2+y_3^2\geq\delta\}.
\end{equation}
Lemma~\ref{lem:y.Hm.y>0.on.some.part} together with positive
definiteness of $H_g$ imply that $y^TH_f(x)y$ is positive on
$\mathcal{S}$. As for the set $\bar{\mathcal{S}}$, let
\begin{equation}\nonumber
\beta_1=\min_{x,y,\in\bar{\mathcal{S}}} y^TH_{\hat{m}}(x)y,
\end{equation}
and
\begin{equation}\nonumber
\beta_2=\min_{x,y,\in\bar{\mathcal{S}}} y^TH_g(x_2,x_3)y.
\end{equation}
By the assumption of positive definiteness of $H_g$, we have
$\beta_2>0$. If we now let $$\gamma>\frac{|\beta_1|}{\beta_2},$$
then $$\min_{x,y,\in\bar{\mathcal{S}}}
y^TH_f(x)y>\beta_1+\frac{|\beta_1|}{\beta_2}\beta_2\geq0.$$ Hence
$y^TH_f(x)y$ is nonnegative (in fact positive) everywhere on
$\mathcal{B}$ and the proof is completed.
\end{proof}

Finally, we provide an argument for existence of bivariate
polynomials of degree $8,10,12,\ldots$ that are convex but not
sos-convex.

\begin{corollary}\label{cor:bivariate.polys.8.10.12...}
Consider the form $f$ in (\ref{eq:construction.of.f(x)})
constructed as described in
Theorem~\ref{thm:conv_not_sos_conv_forms_n3d}. Let
$$\tilde{f}(x_1,x_2)=f(x_1,x_2,1).$$
Then, $\tilde{f}$ is convex but not sos-convex.
\end{corollary}

\begin{proof}
The polynomial $\tilde{f}$ is convex because it is the restriction
of a convex function. It is not difficult to see that
$$\frac{\partial{\tilde{f}}(x_1,x_2)}{\partial{x_1}\partial{x_1}}=m(x_1,x_2,1),$$
which is not sos. Therefore from
Lemma~\ref{lem:sos.matrix.then.minor.sos} $\tilde{f}$ is not
sos-convex.
\end{proof}

Corollary~\ref{cor:bivariate.polys.8.10.12...} together with the
two polynomials in $\tilde{C}_{2,6}\setminus\tilde{\Sigma
C}_{2,6}$ and $\tilde{C}_{3,4}\setminus\tilde{\Sigma C}_{3,4}$
presented in (\ref{eq:minim.poly.2.6}) and
(\ref{eq:minim.poly.3.4}), and with the simple procedure for
increasing the number of variables described at the beginning of
Subsection~\ref{subsubsec:increasing.degree.and.vars} cover all
the values of $n$ and $d$ for which convex but not sos-convex
polynomials exist.

\section{Concluding remarks and an open problem}\label{sec:concluding.remarks}

\begin{figure}[h]
\centering \scalebox{0.33}
{\includegraphics{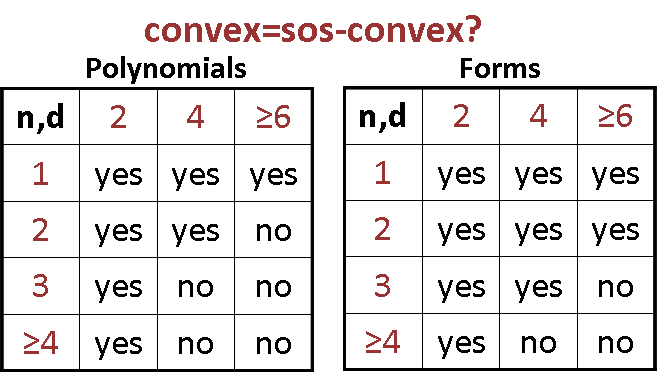}} \caption{The
tables answer whether every convex polynomial (form) in $n$
variables and of degree $d$ is sos-convex.}
\label{fig:sos.convexity.tables}
\end{figure}

A summary of the results of this chapter is given in
Figure~\ref{fig:sos.convexity.tables}. To conclude, we would like
to point out some similarities between nonnegativity and convexity
that deserve attention: (i) both nonnegativity and convexity are
properties that only hold for even degree polynomials, (ii) for
quadratic forms, nonnegativity is in fact equivalent to convexity,
(iii) both notions are NP-hard to check exactly for degree 4 and
larger, and most strikingly (iv) nonnegativity is equivalent to
sum of squares \emph{exactly} in dimensions and degrees where
convexity is equivalent to sos-convexity. It is unclear to us
whether there can be a deeper and more unifying reason explaining
these observations, in particular, the last one which was the main
result of this chapter.

Another intriguing question is to investigate whether one can give
a direct argument proving the fact that $\tilde{\Sigma
C}_{n,d}=\tilde{C}_{n,d}$ if and only if $\Sigma
C_{n+1,d}=C_{n+1,d}$. This would eliminate the need for studying
polynomials and forms separately, and in particular would provide
a short proof of the result $\Sigma_{3,4}=C_{3,4}$ given
in~\cite{AAA_GB_PP_Convex_ternary_quartics}.

Finally, an open problem related to the work in this chapter is to
\emph{find an explicit example of a convex form that is not a sum
of squares}. Blekherman~\cite{Blekherman_convex_not_sos} has shown
via volume arguments that for degree $d\geq 4$ and asymptotically
for large $n$ such forms must exist, although no examples are
known. In particular, it would interesting to determine the
smallest value of $n$ for which such a form exists. We know from
Lemma~\ref{lem:helton.nie.sos-convex.then.sos} that a convex form
that is not sos must necessarily be not sos-convex. Although our
several constructions of convex but not sos-convex polynomials
pass this necessary condition, the polynomials themselves are all
sos. The question is particularly interesting from an optimization
viewpoint because it implies that the well-known sum of squares
relaxation for minimizing
polynomials~\cite{Shor},~\cite{Minimize_poly_Pablo} may not be
exact even for the easy case of minimizing convex polynomials.

\newpage
\vspace{5mm}
\section{Appendix A: How the first convex but not sos-convex
polynomial was found}

%\addcontentsline{toc}{section}{Appendix A: How the first convex
%but not sos-convex polynomial was found}

In this appendix, we explain how the polynomial in
(\ref{eq:first.convex.not.sos.convex}) was found by solving a
carefully designed sos-program\footnote{The term ``sos-program''
is usually used to refer to semidefinite programs that have sum of
squares constraints.}. The simple methodology described here
allows one to search over a restricted family of nonnegative
polynomials that are not sums of squares. The procedure can
potentially be useful in many different settings and this is our
main motivation for presenting this appendix.

%
%The task of finding a polynomial $p(x)$ that is convex but not
%sos-convex is equivalent to finding a polynomial matrix $H(x)$
%that is a \emph{valid Hessian} (i.e., it is a matrix of second
%derivatives), and satisfies the following requirement on the
%scalar polynomial $y^{T}H(x)y$ in $[x;y]$:

Our goal is to find a polynomial $p\mathrel{\mathop:}=p(x)$ whose
Hessian $H\mathrel{\mathop:}=H(x)$ satisfies:
\begin{equation}\label{eq:y'H(x)y}
y^{T}H(x)y \quad \mbox{psd but not sos.}
\end{equation}
%Indeed, if such a matrix $H(x)$ is found, the desired polynomial
%$p(x)$ can be recovered from it by integration.
Unfortunately, a constraint of type (\ref{eq:y'H(x)y}) that
requires a polynomial to be psd but \emph{not} sos is a non-convex
constraint and cannot be easily handled with sos-programming. This
is easy to see from a geometric viewpoint. The feasible set of an
sos-program, being a semidefinite program, is always a convex set.
On the other hand, for a fixed degree and dimension, the set of
psd polynomials that are not sos is non-convex. Nevertheless, we
describe a technique that allows one to search over a \emph{convex
subset} of the set of psd but not sos polynomials using
sos-programming. Our strategy can simply be described as follows:
(i) Impose the constraint that the polynomial should \emph{not} be
sos by using a separating hyperplane (dual functional) for the sos
cone. (ii) Impose the constraint that the polynomial should be psd
by requiring that the polynomial times a nonnegative multiplier is
sos.

By definition, the dual cone $\Sigma_{n,d}^*$ of the sum of
squares cone $\Sigma_{n,d}$ is the set of all linear functionals
$\mu$ that take nonnegative values on it, i.e,
\begin{equation*}\label{eq:dual_cone_definition}
\Sigma_{n,d}^{*}\mathrel{\mathop:}=\{\mu\in \mathcal{H}_{n,d}^{*},
\ \ \langle \mu,p \rangle \geq 0 \; \; \forall p\in \Sigma_{n,d}
\}.
\end{equation*}
Here, the dual space $\mathcal{H}_{n,d}^{*}$ denotes the space of
all linear functionals on the space $\mathcal{H}_{n,d}$ of forms
in $n$ variables and degree $d$, and $\langle .,. \rangle$
represents the pairing between elements of the primal and the dual
space. If a form is not sos, we can find a dual functional $\mu
\in \Sigma_{n,d}^{*}$ that separates it from the closed convex
cone $\Sigma_{n,d}$. The basic idea behind this is the well known
separating hyperplane theorem in convex analysis; see
e.g.~\cite{BoydBook,Rockafellar}.

As for step (ii) of our strategy above, our approach for
guaranteeing that of a form $g$ is nonnegative will be to require
$g(x)\cdot\left( \sum_i x_i^2 \right)^r $ be sos for some integer
$r\geq 1$. Our choice of the multiplier $\left( \sum_i x_i^2
\right)^r$ as opposed to any other psd multiplier is motivated by
a result of Reznick~\cite{Reznick_Unif_denominator} on Hilbert's
17th problem. The 17th problem, which was answered in the
affirmative by Artin~\cite{Artin_Hilbert17}, asks whether every
psd form must be a sum of squares of rational functions. The
affirmative answer to this question implies that if a form $g$ is
psd, then there must exist an sos form $s$, such that $g\cdot s$
is sos. Reznick showed in~\cite{Reznick_Unif_denominator} that if
$g$ is positive definite, one can always take $s(x)=(\sum_i
x_i^2)^{r}$, for sufficiently large $r$. For all polynomials that
we needed prove psd in this chapter, taking $r=1$ has been good
enough.

For our particular purpose of finding a convex but not sos-convex
polynomial, we apply the strategy outlined above to make the first
diagonal element of the Hessian psd but not sos (recall
Lemma~\ref{lem:sos.matrix.then.minor.sos}). More concretely, the
polynomial in (\ref{eq:first.convex.not.sos.convex}) was derived
from a feasible solution to the following sos-program:

%\begin{equation} \label{sos.program.how.we.found.it}
%\begin{array}{rl}
%\ &p(x)\in\mathcal{H}_{3,8} \\
%%\ &\ \\
%\ &H(x)=\dfrac{\partial^2p}{\partial x^2} \\
%\ &\left(x_1^2+x_2^2+x_3^2\right)y^{T}H(x)y \quad \quad \mbox{sos} \\
%%(\mbox{for some integer}\  r \geq 1)&\ \\
%\ &\langle \mu, H_{1,1} \rangle =-1   \quad (\mbox{for some dual
%functional}\  \mu \in \Sigma_{3,6}^*).
%\end{array}
%\end{equation}

%\floatname{algorithm}{SOS-Program}
%%\newfloat{algorithm}{h}{SOS-program}
%\begin{algorithm*}
%\caption{\ } \label{alg:how.we.found.it}
%\begin{algorithmic}%[1]
%%\STATE {\bf SOS-PROGRAM *}
%\State Parameterize $p(x)\in\mathcal{H}_{3,8}$.  \State Compute
%$H(x)=\frac{\partial^2p}{\partial x^2}$. \State Impose the
%constraints
% \begin{equation}\label{eq:algo_xi^2.y.H.y} (x_1^2+x_2^2+x_3^2)y^{T}H(x)y \quad \mbox{sos}, \end{equation}
%\begin{equation}\label{eq:algo.mu.H11=-1}\quad \quad \quad \quad \quad \ \ \ \ \    \langle \mu, H_{1,1} \rangle=-1\end{equation}
%\quad \quad \quad \quad \quad \quad \quad \quad \quad \quad \quad
% \; (for some dual functional $\mu \in \Sigma_{3,6}^*$).
%\end{algorithmic}
%\end{algorithm*}
\begin{samepage}
\renewcommand{\labelitemi}{$-$}
\begin{itemize}
\item Parameterize $p\in\mathcal{H}_{3,8}$ and compute its Hessian
$H=\frac{\partial^2p}{\partial x^2}$.
%
% \item Compute$H(x)=\frac{\partial^2p}{\partial x^2}$.
%
 \item Impose the constraints
 \begin{equation}\label{eq:algo_xi^2.y.H.y} (x_1^2+x_2^2+x_3^2)\cdot y^{T}H(x)y \quad \mbox{sos}, \end{equation}
\begin{equation}\label{eq:algo.mu.H11=-1}\quad \quad \quad \quad \quad \ \ \ \ \    \langle \mu, H_{1,1} \rangle=-1\end{equation}
\quad \quad \quad \quad \quad \quad \quad \quad \quad \quad \quad
 \; (for some dual functional $\mu \in \Sigma_{3,6}^*$).
\end{itemize}
\end{samepage}

%\floatname{algorithm}{SOS-Program}
%\begin{algorithm*}
%\caption{\ } \label{alg:how.we.found.it}
%\begin{algorithmic}[1]
%%\STATE {\bf SOS-PROGRAM *}
%\State Parameterize $p(x)$ as a form of degree $8$ in $3$ variables.
%\State Compute the Hessian $H(x)=\frac{\partial^2p}{\partial x^2}$.
%\State Impose the constraint
% \begin{equation}\label{eq:algo_xi^2.y.H.y} (x_1^2+x_2^2+x_3^2)^ry^{T}H(x)y \quad \mbox{sos}. \end{equation}
%for some integer $r \geq 1$.
%\State Impose the constraint
%\begin{equation}\label{eq:b.vec(w)<0}\quad \quad \quad \quad \quad \ \ \ \ \ \   \langle \mu, H_{1,1} \rangle <0.\end{equation}
%for some (carefully chosen) dual functional $\mu \in \Sigma_{3,6}^*$.
%\end{algorithmic}
%\end{algorithm*}

The decision variables of this sos-program are the coefficients of
the polynomial $p$ that also appear in the entries of the Hessian
matrix $H$. (The polynomial $H_{1,1}$ in (\ref{eq:algo.mu.H11=-1})
denotes the first diagonal element of $H$.) The dual functional
$\mu$ must be fixed a priori as explained in the sequel. Note that
all the constraints are linear in the decision variables and
indeed the feasible set described by these constraints is a convex
set. Moreover, the reader should be convinced by now that if the
above sos-program is feasible, then the solution $p$ is a convex
polynomial that is not sos-convex.

%Note that the constraints (\ref{eq:algo_xi^2.y.H.y}) and
%(\ref{eq:b.vec(w)<0}) are linear in the decision variables and
%indeed the feasible set described by these constraints is a convex
%set.

The reason why we chose to parameterize $p$ as a form in
$\mathcal{H}_{3,8}$ is that a minimal case where a diagonal
element of the Hessian (which has $2$ fewer degree) can be psd but
not sos is among the forms in $\mathcal{H}_{3,6}$.
%
%our strategy was to force a diagonal element of the Hessian (which
%has $2$ fewer degree) to be psd but not sos. We know that a
%minimal case where this can happen is among the forms in
%$\mathcal{H}_{3,6}$.
%
The role of the dual functional $\mu\in\Sigma_{3,6}^{*}$ in
(\ref{eq:algo.mu.H11=-1}) is to separate the polynomial $H_{1,1}$
from $\Sigma_{3,6}$. Once an ordering on the monomials of
$H_{1,1}$ is fixed, this constraint can be imposed numerically as
\begin{equation}
\langle \mu, H_{1,1} \rangle=b^{T}\vec{H}_{1,1}=-1,
\end{equation}
where $\vec{H}_{1,1}$ denotes the vector of coefficients of the
polynomial $H_{1,1}$ and $b\in \mathbb{R}^{28}$ represents our
separating hyperplane, which must be computed prior to solving the
above sos-program.

There are several ways to obtain a separating hyperplane for
$\Sigma_{3,6}$. Our approach was to find a hyperplane that
separates the Motzkin form \aaa{$M$} in (\ref{eq:Motzkin.form})
from $\Sigma_{3,6}$. This can be done in at least a couple of
different ways. For example, we can formulate a semidefinite
program that requires the Motzkin form to be sos. This program is
clearly infeasible. Any feasible solution to its dual semidefinite
program will give us the desired separating hyperplane.
Alternatively, we can set up an sos-program that finds the
Euclidean projection \aaa{$M^p\mathrel{\mathop:}=M^p(x)$} of the
Motzkin form \aaa{$M$} onto the cone $\Sigma_{3,6}$. Since the
projection is done onto a convex set, the hyperplane tangent to
$\Sigma_{3,6}$ at \aaa{$M^p$} will be supporting $\Sigma_{3,6}$,
and can serve as our separating hyperplane.

To conclude, we remark that in contrast to previous techniques of
constructing examples of psd but not sos polynomials that are
usually based on some obstructions associated with the number of
zeros of polynomials (see e.g.~\cite{Reznick}), our approach has
the advantage that the resulting polynomials are positive
definite. Furthermore, additional linear or semidefinite
constraints can easily be incorporated in the search process to
impose e.g. various symmetry or sparsity patterns on the
polynomial of interest.

\newpage
\vspace{5mm}
\section{Appendix B: Certificates complementing the proof of
Theorem~\ref{thm:minimal.2.6.and.3.6}}

%\addcontentsline{toc}{section}{Appendix B: Certificates
%complementing the proof of Theorem~\ref{thm:minimal.2.6.and.3.6}}

Let $x\mathrel{\mathop:}=(x_1,x_2,x_2)^T$,
$y\mathrel{\mathop:}=(y_1,y_2,y_3)^T$,
$\tilde{x}\mathrel{\mathop:}=(x_1,x_2)^T$,
$\tilde{y}\mathrel{\mathop:}=(y_1,y_2)^T$, and let $f, \tilde{f},
H_f,$ and $H_{\tilde{f}}$ be as in the proof of
Theorem~\ref{thm:minimal.2.6.and.3.6}. This appendix proves that
the form $(x_1^2+x_2^2)\cdot y^TH_f(x)y$ in
(\ref{eq:y.H_f.y.xi^2.3.6.example}) is sos and that the polynomial
$\tilde{y}^TH_{\tilde{f}}(\tilde{x})\tilde{y}$ in
(\ref{eq:y.H_f_tilda.y.xi^2.2.6.example}) is not sos, hence
proving respectively that $f$ is convex and $\tilde{f}$ is not
sos-convex.

A rational sos decomposition of $(x_1^2+x_2^2)\cdot y^TH_f(x)y$,
which is a form in $6$ variables of degree $8$, is as follows:
\begin{equation}\nonumber
(x_1^2+x_2^2)\cdot y^TH_f(x)y=\frac{1}{84}z^TQz,
\end{equation}
where $z$ is the vector of monomials
\begin{equation}\nonumber
\begin{array}{ll}
z=&[  x_2x_3^2y_3,
  x_2x_3^2y_2,
  x_2x_3^2y_1,
  x_2^2x_3y_3,
  x_2^2x_3y_2,
  x_2^2x_3y_1,
     x_2^3y_3,
     x_2^3y_2,
     x_2^3y_1,\\
\ & \ \\
\ &    x_1x_3^2y_3,
  x_1x_3^2y_2,
  x_1x_3^2y_1,
 x_1x_2x_3y_3,
 x_1x_2x_3y_2,
 x_1x_2x_3y_1,
  x_1x_2^2y_3,
  x_1x_2^2y_2,
  x_1x_2^2y_1,\\
\ &  \ \\
\ &   x_1^2x_3y_3,
  x_1^2x_3y_2,
  x_1^2x_3y_1,
  x_1^2x_2y_3,
  x_1^2x_2y_2,
  x_1^2x_2y_1,
     x_1^3y_3,
     x_1^3y_2,
     x_1^3y_1]^T,
\end{array}
\end{equation}
and $Q$ is the $27\times 27$ positive definite
matrix\footnote{Whenever we state a matrix is positive definite,
this claim is backed up by a rational $LDL^T$ factorization of the
matrix that the reader can find online at
\texttt{http://aaa.lids.mit.edu/software}.} presented on the next
page
\begin{equation}\nonumber
Q=\begin{bmatrix} Q_1 & Q_2
\end{bmatrix},
\end{equation}

\newpage
\begin{landscape}
\setcounter{MaxMatrixCols}{30} \scalefont{.5}
\begin{equation}\nonumber
Q_1=\begin{bmatrix}[r]224280  & -40740  & 20160  & -81480  & 139692  & 93576  & -50540  & -27804  & -48384  & 0  & -29400  & -32172  & 21252  & 103404   \\
-40740  & 63504  & 36624  & 114324  & -211428  & -8316  & -67704  & 15372  & -47376  & 29400  & 0  & 16632  & 75936  & 15540   \\
20160  & 36624  & 121128  & 52920  & -27972  & -93072  & -42252  & -77196  & -58380  & 32172  & -16632  & 0  & 214284  & -57960   \\
-81480  & 114324  & 52920  & 482104  & -538776  & 36204  & -211428  & 362880  & -70644  & 19068  & 13524  & -588  & 179564  & 20258   \\
139692  & -211428  & -27972  & -538776  & 1020600  & -94416  & 338016  & -288120  & 188748  & -46368  & -33684  & -46620  & -113638  & -119112   \\
93576  & -8316  & -93072  & 36204  & -94416  & 266448  & -75348  & 216468  & 5208  & 3360  & -33432  & -31080  & -221606  & 254534   \\
-50540  & -67704  & -42252  & -211428  & 338016  & -75348  & 175224  & -144060  & 101304  & -20692  & 7826  & -4298  & -77280  & -108192   \\
-27804  & 15372  & -77196  & 362880  & -288120  & 216468  & -144060  & 604800  & 28560  & -35350  & -840  & -35434  & -132804  & 134736   \\
-48384  & -47376  & -58380  & -70644  & 188748  & 5208  & 101304  & 28560  & 93408  & -21098  & 2786  & -11088  & -104496  & -22680   \\
0  & 29400  & 32172  & 19068  & -46368  & 3360  & -20692  & -35350  & -21098  & 224280  & -40740  & 20160  & 35028  & 89964   \\
-29400  & 0  & -16632  & 13524  & -33684  & -33432  & 7826  & -840  & 2786  & -40740  & 63504  & 36624  & 51828  & -196476   \\
-32172  & 16632  & 0  & -588  & -46620  & -31080  & -4298  & -35434  & -11088  & 20160  & 36624  & 121128  & 29148  & -9408   \\
21252  & 75936  & 214284  & 179564  & -113638  & -221606  & -77280  & -132804  & -104496  & 35028  & 51828  & 29148  & 782976  & -463344   \\
103404  & 15540  & -57960  & 20258  & -119112  & 254534  & -108192  & 134736  & -22680  & 89964  & -196476  & -9408  & -463344  & 1167624   \\
267456  & -48132  & 27552  & -49742  & 78470  & 124236  & -100464  & 61404  & -90384  & 55524  & -50064  & -145908  & 41016  & -15456   \\
60872  & -25690  & -142478  & 22848  & -113820  & 259980  & -72996  & 237972  & 20412  & -95580  & -47964  & -27780  & -438732  & 514500   \\
37730  & -99036  & -10150  & -83160  & 473088  & 34188  & 167244  & 57120  & 159264  & 10752  & -93048  & -183540  & 230832  & -49980   \\
-119210  & 9170  & 81648  & 244356  & -41664  & -194124  & -9996  & 214368  & 19152  & -89184  & 2940  & -48480  & 204708  & -85344   \\
-116508  & 81564  & 26124  & 155832  & -308280  & -78180  & -74088  & 14616  & -49644  & 40320  & 87108  & 225456  & 135744  & 8568   \\
30660  & -14952  & 11844  & -21420  & 62604  & 14364  & 13608  & 1176  & 5124  & 59388  & -18144  & -99624  & 31332  & -178248   \\
35700  & 11340  & 52836  & -70788  & 86184  & 9396  & 12264  & -108024  & -11256  & 259056  & -86520  & -3528  & -19334  & 142128   \\
102156  & 2856  & 64536  & 22176  & -4200  & 77532  & -70896  & 54348  & -49616  & 72744  & -78876  & -144998  & 29316  & 23856   \\
-86100  & 47148  & 71820  & 230916  & -223692  & -131628  & -72156  & 59640  & -31416  & -75096  & -39396  & -44520  & 158508  & 308196   \\
95364  & -504  & -8412  & -23100  & 28140  & 81648  & -26768  & -25200  & -13944  & -51002  & -39228  & 71232  & 130298  & 298956   \\
11256  & 5208  & 32158  & -33264  & 45444  & 3122  & 6888  & -34440  & -5628  & 61320  & -19152  & 8988  & 18060  & -19467   \\
0  & -1344  & -3696  & -34692  & 33768  & 5964  & 9492  & -20244  & 5208  & -30072  & -9912  & 58884  & -50883  & 151956   \\
-51422  & 49056  & 32592  & 160370  & -229068  & -36792  & -68796
& 57708  & -39564  & 55944  & 31164  & -8008  & 141876  & -126483
\end{bmatrix},
\end{equation}
\normalsize

\setcounter{MaxMatrixCols}{30} \scalefont{.5}
\begin{equation}\nonumber
Q_2=\begin{bmatrix}[r]267456  & 60872  & 37730  & -119210  & -116508  & 30660  & 35700  & 102156  & -86100  & 95364  & 11256  & 0  & -51422   \\
-48132  & -25690  & -99036  & 9170  & 81564  & -14952  & 11340  & 2856  & 47148  & -504  & 5208  & -1344  & 49056   \\
27552  & -142478  & -10150  & 81648  & 26124  & 11844  & 52836  & 64536  & 71820  & -8412  & 32158  & -3696  & 32592   \\
-49742  & 22848  & -83160  & 244356  & 155832  & -21420  & -70788  & 22176  & 230916  & -23100  & -33264  & -34692  & 160370   \\
78470  & -113820  & 473088  & -41664  & -308280  & 62604  & 86184  & -4200  & -223692  & 28140  & 45444  & 33768  & -229068   \\
124236  & 259980  & 34188  & -194124  & -78180  & 14364  & 9396  & 77532  & -131628  & 81648  & 3122  & 5964  & -36792   \\
-100464  & -72996  & 167244  & -9996  & -74088  & 13608  & 12264  & -70896  & -72156  & -26768  & 6888  & 9492  & -68796   \\
61404  & 237972  & 57120  & 214368  & 14616  & 1176  & -108024  & 54348  & 59640  & -25200  & -34440  & -20244  & 57708   \\
-90384  & 20412  & 159264  & 19152  & -49644  & 5124  & -11256  & -49616  & -31416  & -13944  & -5628  & 5208  & -39564   \\
55524  & -95580  & 10752  & -89184  & 40320  & 59388  & 259056  & 72744  & -75096  & -51002  & 61320  & -30072  & 55944   \\
-50064  & -47964  & -93048  & 2940  & 87108  & -18144  & -86520  & -78876  & -39396  & -39228  & -19152  & -9912  & 31164   \\
-145908  & -27780  & -183540  & -48480  & 225456  & -99624  & -3528  & -144998  & -44520  & 71232  & 8988  & 58884  & -8008   \\
41016  & -438732  & 230832  & 204708  & 135744  & 31332  & -19334  & 29316  & 158508  & 130298  & 18060  & -50883  & 141876   \\
-15456  & 514500  & -49980  & -85344  & 8568  & -178248  & 142128  & 23856  & 308196  & 298956  & -19467  & 151956  & -126483   \\
610584  & 21840  & 127932  & -65184  & -323834  & 195636  & 90972  & 339794  & -100716  & -96012  & 24864  & -114219  & 36876   \\
21840  & 466704  & -110628  & -106820  & -54012  & -90636  & -111790  & -14952  & 63672  & 107856  & -67788  & 61404  & -88284   \\
127932  & -110628  & 1045968  & 142632  & -410592  & 171024  & 86268  & 176820  & 96516  & 199752  & 13524  & -70784  & -42756   \\
-65184  & -106820  & 142632  & 569856  & 21518  & -30156  & -159264  & -23016  & 410004  & -71484  & -62076  & -13860  & 74032   \\
-323834  & -54012  & -410592  & 21518  & 604128  & -229992  & -75516  & -297276  & 182385  & 75684  & -3528  & 94500  & 138432   \\
195636  & -90636  & 171024  & -30156  & -229992  & 169512  & 104748  & 187341  & -136332  & -145719  & 35364  & -94836  & 24612   \\
90972  & -111790  & 86268  & -159264  & -75516  & 104748  & 381920  & 147168  & -182595  & -36876  & 105504  & -24612  & -7560   \\
339794  & -14952  & 176820  & -23016  & -297276  & 187341  & 147168  & 346248  & -59304  & -137928  & 64932  & -90888  & 28392   \\
-100716  & 63672  & 96516  & 410004  & 182385  & -136332  & -182595  & -59304  & 776776  & 48972  & -98784  & 19152  & 180852   \\
-96012  & 107856  & 199752  & -71484  & 75684  & -145719  & -36876  & -137928  & 48972  & 494536  & -28392  & 118188  & -130200   \\
24864  & -67788  & 13524  & -62076  & -3528  & 35364  & 105504  & 64932  & -98784  & -28392  & 60984  & 0  & -3780   \\
-114219  & 61404  & -70784  & -13860  & 94500  & -94836  & -24612  & -90888  & 19152  & 118188  & 0  & 74760  & -65100   \\
36876  & -88284  & -42756  & 74032  & 138432  & 24612  & -7560  &
28392  & 180852  & -130200  & -3780  & -65100  & 194040
\end{bmatrix}.
\end{equation}
\normalsize
\end{landscape}

\newpage

Next, we prove that the polynomial
$\tilde{y}^TH_{\tilde{f}}(\tilde{x})\tilde{y}$ in
(\ref{eq:y.H_f_tilda.y.xi^2.2.6.example}) is not sos. Let us first
present this polynomial and give it a name:

\begin{equation}\nonumber
\begin{array}{lll}
t(\tilde{x},\tilde{y})\mathrel{\mathop:}=\tilde{y}^TH_{\tilde{f}}(\tilde{x})\tilde{y}&=&294x_1x_2y_2^2-6995x_2^4y_1y_2-10200x_1y_1y_2-4356x_1^2x_2y_1^2
\\ \\ \ &\ &-2904x_1^3y_1y_2-11475x_1x_2^2y_1^2+13680x_2^3y_1y_2+4764x_1x_2y_1^2\\\\ \ &\ &+4764x_1^2y_1y_2+6429x_1^2x_2^2y_1^2+294x_2^2y_1y_2
-13990x_1x_2^3y_2^2\\
\\ \ &\ &-12123x_1^2x_2y_2^2-3872x_2y_1y_2+\frac{2143}{2}x_1^4y_2^2+20520x_1x_2^2y_2^2
\\ \\ \ &\
&+29076x_1x_2y_1y_2-24246x_1x_2^2y_1y_2+14901x_1x_2^3y_1y_2
\\ \\ \ &\
&+15039x_1^2x_2^2y_1y_2+8572x_1^3x_2y_1y_2+\frac{44703}{4}x_1^2x_2^2y_2^2+1442y_1^2\\
\\ \ &\ &-12360x_2y_2^2
-5100x_2y_1^2+\frac{147513}{4}x_2^2y_2^2+7269x_2^2y_1^2\\
\\ \ &\ &+\frac{772965}{32}x_2^4y_2^2
+\frac{14901}{8}x_2^4y_1^2-1936x_1y_2^2-84x_1y_1^2+\frac{3817}{2}y_2^2
\\ \\ \ &\
&+7269x_1^2y_2^2+4356x_1^2y_1^2-3825x_1^3y_2^2-180x_1^3y_1^2+632y_1y_2\\
\\ \ &\ &+2310x_1^4y_1^2+5013x_1x_2^3y_1^2-22950x_1^2x_2y_1y_2-45025x_2^3y_2^2\\
\\ \ &\ &-1505x_1^4y_1y_2-4041x_2^3y_1^2-3010x_1^3x_2y_1^2+5013x_1^3x_2y_2^2.
\end{array}
\end{equation}
Note that $t$ is a polynomial in $4$ variables of degree $6$ that
is quadratic in $\tilde{y}$. Let us denote the cone of sos
polynomials in $4$ variables $(\tilde{x},\tilde{y})$ that have
degree $6$ and are quadratic in $\tilde{y}$ by
$\hat{\Sigma}_{4,6}$, and its dual cone by $\hat{\Sigma}_{4,6}^*$.
Our proof will simply proceed by presenting a dual functional
$\xi\in\hat{\Sigma}_{4,6}^*$ that takes a negative value on the
polynomial $t$. We fix the following ordering of monomials in what
follows:
\begin{equation}\label{eq:monomial.ordering.t}
\begin{array}{ll}
v=&[          y_2^2,
             y_1y_2,
              y_1^2,
           x_2y_2^2,
          x_2y_1y_2,
           x_2y_1^2,
         x_2^2y_2^2,
        x_2^2y_1y_2,
         x_2^2y_1^2,
         x_2^3y_2^2,
        x_2^3y_1y_2,
         x_2^3y_1^2,
         x_2^4y_2^2, \\
         \ &    x_2^4y_1y_2,
         x_2^4y_1^2,
                  x_1y_2^2,
          x_1y_1y_2,
           x_1y_1^2,
        x_1x_2y_2^2,
       x_1x_2y_1y_2,
        x_1x_2y_1^2,
      x_1x_2^2y_2^2,
     x_1x_2^2y_1y_2,\\
         \ &      x_1x_2^2y_1^2,
      x_1x_2^3y_2^2,
     x_1x_2^3y_1y_2,
           x_1x_2^3y_1^2,
         x_1^2y_2^2,
        x_1^2y_1y_2,
         x_1^2y_1^2,
           x_1^2x_2y_2^2,
     x_1^2x_2y_1y_2,
      x_1^2x_2y_1^2,\\
         \ &    x_1^2x_2^2y_2^2,
   x_1^2x_2^2y_1y_2,
    x_1^2x_2^2y_1^2,
         x_1^3y_2^2,
                      x_1^3y_1y_2,
           x_1^3y_1^2,
      x_1^3x_2y_2^2,
     x_1^3x_2y_1y_2,
      x_1^3x_2y_1^2,
               x_1^4y_2^2,\\
         \ &        x_1^4y_1y_2,
         x_1^4y_1^2]^T.
\end{array}
\end{equation}
Let $\vec{t}$ represent the vector of coefficients of $t$ ordered
according to the list of monomials above; i.e., $t=\vec{t}^Tv$.
Using the same ordering, we can represent our dual functional
$\xi$ with the vector
\begin{equation}\nonumber
\begin{array}{ll}
c=&[   19338,
       -2485,
       17155,
        6219,
       -4461,
       11202,
        4290,
       -5745,
       13748,
        3304,
       -5404,\\
         \ &       13227,
        3594,
               -4776,
       19284,
               2060,
        3506,
        5116,
         366,
       -2698,
        6231,
        -487,
       -2324,\\
         \ &
        4607,
         369,
       -3657,
        3534,
        6122,
         659,
        7057,
                 1646,
        1238,
        1752,
        2797,
        -940,
        4608,\\
         \ &
        -200,
        1577,
       -2030,
        -513,
       -3747,
             2541,
               15261,
                              220,
        7834]^T.
\end{array}
\end{equation}
We have
\begin{equation}\nonumber
\langle\xi,t\rangle=c^{T}\vec{t}=-\frac{364547}{16}<0.
\end{equation}
On the other hand, we claim that $\xi\in\hat{\Sigma}_{4,6}^*$;
i.e., for any form $w\in\hat{\Sigma}_{4,6}$, we should have
\begin{equation}\label{eq:c.w>=0.t}
\langle\xi,w\rangle=c^{T}\vec{w}\geq0,
\end{equation}
where $\vec{w}$ here denotes the coefficients of $w$ listed
according to the ordering in (\ref{eq:monomial.ordering.t}).
Indeed, if $w$ is sos, then it can be written in the form
\[
w(x)=\tilde{z}^{T}\tilde{Q}\tilde{z}= \mathrm{Tr} \ \tilde{Q}
\cdot \tilde{z}\tilde{z}^{T},
\]
for some symmetric positive semidefinite matrix $\tilde{Q}$, and a
vector of monomials
\[
\tilde{z}=[  y_2,
     y_1,
     x_2y_2,
     x_2y_1,
     x_1y_2,
     x_1y_1,
     x_2^2y_2,
     x_2^2y_1,
     x_1x_2y_2,
     x_1x_2y_1,
     x_1^2y_2,
     x_1^2y_1   ]^T.
\]
It is not difficult to see that
\begin{equation}\label{eq:c.vec(w)=traceQzzz'.t}
c^{T}\vec{w}= \mathrm{Tr}  \, \tilde{Q} \cdot
(\tilde{z}\tilde{z}^{T}) \vert_c,
\end{equation}
where by $(\tilde{z}\tilde{z}^{T})\vert_c$ we mean a matrix where
each monomial in $\tilde{z}\tilde{z}^{T}$ is replaced with the
corresponding element of the vector $c$. This yields the matrix
\setcounter{MaxMatrixCols}{30} \scalefont{.55}
\[(\tilde{z}\tilde{z}^{T})\vert_c=
\begin{bmatrix}[r]
19338 & -2485 & 6219 & -4461 & 2060 & 3506 & 4290 & -5745 & 366 & -2698 & 6122 & 659 \\
-2485 & 17155 & -4461 & 11202 & 3506 & 5116 & -5745 & 13748 & -2698 & 6231 & 659 & 7057 \\
6219 & -4461 & 4290 & -5745 & 366 & -2698 & 3304 & -5404 & -487 & -2324 & 1646 & 1238 \\
-4461 & 11202 & -5745 & 13748 & -2698 & 6231 & -5404 & 13227 & -2324 & 4607 & 1238 & 1752 \\
2060 & 3506 & 366 & -2698 & 6122 & 659 & -487 & -2324 & 1646 & 1238 & -200 & 1577 \\
3506 & 5116 & -2698 & 6231 & 659 & 7057 & -2324 & 4607 & 1238 & 1752 & 1577 & -2030 \\
4290 & -5745 & 3304 & -5404 & -487 & -2324 & 3594 & -4776 & 369 & -3657 & 2797 & -940 \\
-5745 & 13748 & -5404 & 13227 & -2324 & 4607 & -4776 & 19284 & -3657 & 3534 & -940 & 4608 \\
366 & -2698 & -487 & -2324 & 1646 & 1238 & 369 & -3657 & 2797 & -940 & -513 & -3747 \\
-2698 & 6231 & -2324 & 4607 & 1238 & 1752 & -3657 & 3534 & -940 & 4608 & -3747 & 2541 \\
6122 & 659 & 1646 & 1238 & -200 & 1577 & 2797 & -940 & -513 & -3747 & 15261 & 220 \\
659 & 7057 & 1238 & 1752 & 1577 & -2030 & -940 & 4608 & -3747 &
2541 & 220 & 7834
\end{bmatrix},
\] \normalsize
which is positive definite. Therefore, equation
(\ref{eq:c.vec(w)=traceQzzz'.t}) along with the fact that
$\tilde{Q}$ is positive semidefinite implies that
(\ref{eq:c.w>=0.t}) holds. This completes the proof.

\cleardoublepage \thispagestyle{empty} \vspace*{\fill}
\begin{center}
{ \sfbHuge Part II:\\ \  \\ Lyapunov Analysis and Computation }
\end{center}
\vspace*{\fill} \addcontentsline{toc}{chapter}{II: Lyapunov
Analysis and Computation} \cleardoublepage

\chapter{Lyapunov Analysis of Polynomial Differential
Equations}\label{chap:converse.lyap}

In the last two chapters of this thesis, our focus will turn to
Lyapunov analysis of dynamical systems. The current chapter
presents new results on Lyapunov analysis of polynomial vector
fields. The content here is based on the works
in~\cite{AAA_PP_CDC11_converseSOS_Lyap}
and~\cite{AAA_MK_PP_CDC11_no_Poly_Lyap}, as well as some more
recent results.

\section{Introduction}
We will be concerned for the most part of this chapter with a
continuous time dynamical system
\begin{equation}\label{eq:CT.dynamics}
\dot{x}=f(x),
\end{equation}
where $f:\mathbb{R}^n\rightarrow\mathbb{R}^n$ is a polynomial and
has an equilibrium at the origin, i.e., $f(0)=0$. Arguably, the
class of polynomial differential equations are among the most
widely encountered in engineering and sciences. For stability
analysis of these systems, it is most common (and quite natural)
to search for Lyapunov functions that are polynomials themselves.
When such a candidate Lyapunov function is used, then conditions
of Lyapunov's theorem reduce to a set of polynomial inequalities.
For instance, if establishing global asymptotic stability of the
origin is desired, one would require a radially unbounded
polynomial Lyapunov candidate
$V(x):\mathbb{R}^n\rightarrow\mathbb{R}$ to vanish at the origin
and satisfy
\begin{eqnarray}
V(x)&>&0\quad \forall x\neq0 \label{eq:V.positive} \\
\dot{V}(x)=\langle\nabla V(x),f(x)\rangle&<&0\quad \forall x\neq0.
\label{eq:Vdot.negative}
\end{eqnarray}
Here, $\dot{V}$ denotes the time derivative of $V$ along the
trajectories of (\ref{eq:CT.dynamics}), $\nabla V(x)$ is the
gradient vector of $V$, and $\langle .,. \rangle$ is the standard
inner product in $\mathbb{R}^n$. In some other variants of the
analysis problem, e.g. if LaSalle's invariance principle is to be
used, or if the goal is to prove boundedness of trajectories of
(\ref{eq:CT.dynamics}), then the inequality in
(\ref{eq:Vdot.negative}) is replaced with
\begin{equation}\label{eq:Vdot.nonpositive}
\dot{V}(x)\leq 0 \quad \forall x.
\end{equation}
In any case, the problem arising from this analysis approach is
that even though polynomials of a given degree are finitely
parameterized, the computational problem of searching for a
polynomial $V$ satisfying inequalities of the type
(\ref{eq:V.positive}), (\ref{eq:Vdot.negative}),
(\ref{eq:Vdot.nonpositive}) is intractable.
%In fact, even deciding if a given polynomial $V$ of degree four or
%larger satisfies (\ref{eq:V.positive}) is
%NP-hard~\cite{nonnegativity_NP_hard}.
An approach pioneered in~\cite{PhD:Parrilo} and widely popular by
now is to replace the positivity (or nonnegativity) conditions by
the requirement of the existence of a sum of squares (sos)
decomposition:
\begin{eqnarray}
V& \mbox{sos}\label{eq:V.SOS} \\
-\dot{V}=-\langle\nabla V,f\rangle& \mbox{sos}.
\label{eq:-Vdot.SOS}
\end{eqnarray}

As we saw in the previous chapter, sum of squares decomposition is
a sufficient condition for polynomial nonnegativity that can be
efficiently checked with semidefinite programming. For a fixed
degree of a polynomial Lyapunov candidate $V$, the search for the
coefficients of $V$ subject to the constraints (\ref{eq:V.SOS})
and (\ref{eq:-Vdot.SOS}) is a semidefinite program (SDP). We call
a Lyapunov function satisfying both sos conditions in
(\ref{eq:V.SOS}) and (\ref{eq:-Vdot.SOS}) a \emph{sum of squares
Lyapunov function}. We emphasize that this is the sensible
definition of a sum of squares Lyapunov function and not what the
name may suggest, which is a Lyapunov function that is a sum of
squares. Indeed, the underlying semidefinite program will find a
Lyapunov function $V$ if and only if $V$ satisfies \emph{both}
conditions (\ref{eq:V.SOS}) and (\ref{eq:-Vdot.SOS}).

%if a Lyapunov function $V$ satisfies (\ref{eq:V.SOS}) but not the
%derivative condition (\ref{eq:-Vdot.SOS}), then the semidefinite
%program will \emph{not} find this Lyapunov function.

Over the last decade, the applicability of sum of squares Lyapunov
functions has been explored and extended in many directions and a
multitude of sos techniques have been developed to tackle a range
of problems in systems and control. We refer the reader to the by
no means exhaustive list of works
\cite{PositivePolyInControlBook},
\cite{AutControlSpecial_PositivePolys}, \cite{Chesi_book},
\cite{ControlAppsSOS}, \cite{PraP03}, \cite{PapP02},
\cite{Pablo_Rantzer_synthesis},
\cite{Chest.et.al.sos.robust.stability},
\cite{AAA_PP_CDC08_non_monotonic}, \cite{Erin_Pablo_Contraction},
\cite{Tedrake_LQR_Trees} and references therein.
%Despite the wealth of research in this area, the converse question
%of whether the existence of a polynomial Lyapunov function implies
%the existence of a Lyapunov function satisfying the sos conditions
%in (\ref{eq:V.SOS}), (\ref{eq:-Vdot.SOS}) has remained unresolved.
Despite the wealth of research in this area, the converse question
of whether the existence of a polynomial Lyapunov function implies
the existence of a sum of squares Lyapunov function has remained
elusive.
%
%, i.e, a polynomial Lyapunov function satisfying \emph{both} sos
%conditions in (\ref{eq:V.SOS}), (\ref{eq:-Vdot.SOS})
%
This question naturally comes in two variants:

{\bf Problem 1:} Does existence of a polynomial Lyapunov function
of a given degree imply existence of a polynomial Lyapunov
function of the \emph{same degree} that satisfies the sos
conditions in (\ref{eq:V.SOS}) and (\ref{eq:-Vdot.SOS})?

{\bf Problem 2:} Does existence of a polynomial Lyapunov function
of a given degree imply existence of a polynomial Lyapunov
function of \emph{possibly higher degree} that satisfies the sos
conditions in (\ref{eq:V.SOS}) and (\ref{eq:-Vdot.SOS})?

The notion of stability of interest in this chapter, for which we
will study the questions above, is global asymptotic stability
(GAS); see e.g.~\cite[Chap. 4]{Khalil:3rd.Ed} for a precise
definition. Of course, a fundamental question that comes before
the problems mentioned above is the following:

{\bf Problem 0:} If a polynomial dynamical system is globally
asymptotically stable, does it admit a polynomial Lyapunov
function?

\subsection{Contributions and organization of this chapter}
In this chapter, we give explicit counterexamples that answer {\bf
Problem 0} and {\bf Problem 1} in the negative. This is done in
Section~\ref{sec:no.poly.Lyap} and
Subsection~\ref{subsec:the.counterexample} respectively. On the
other hand, in Subsection~\ref{subsec:converse.sos.results}, we
give a positive answer to {\bf Problem~2} for the case where the
vector field is homogeneous
(Theorem~\ref{thm:poly.lyap.then.sos.lyap}) or when it is planar
and an additional mild assumption is met
(Theorem~\ref{thm:poly.lyap.then.sos.lyap.PLANAR}). The proofs of
these two theorems are quite simple and rely on powerful
Positivstellensatz results due to Scheiderer
(Theorems~\ref{thm:claus} and~\ref{thm:claus.3vars}). In
Section~\ref{sec:extension.to.switched.sys}, we extend these
results to derive a converse sos Lyapunov theorem for robust
stability of switched linear systems. It will be proven that if
such a system is stable under arbitrary switching, then it admits
a common polynomial Lyapunov function that is sos and that the
negative of its derivative is also sos
(Theorem~\ref{thm:converse.sos.switched.sys}). We also show that
for switched linear systems (both in discrete and continuous
time), if the inequality on the decrease condition of a Lyapunov
function is satisfied as a sum of squares, then the Lyapunov
function itself is automatically a sum of squares
(Propositions~\ref{prop:switch.DT.V.automa.sos}
and~\ref{prop:switch.CT.V.automa.sos}). We list a number of
related open problems in Section~\ref{sec:summary.future.work}.

Before these contributions are presented, we establish a hardness
result for the problem of deciding asymptotic stability of cubic
homogeneous vector fields in the next section. We also present
some byproducts of this result, including a Lyapunov-inspired
technique for proving positivity of forms.

\section{Complexity considerations for deciding stability of polynomial vector fields}\label{sec:complexity.cubic.vec.field}

It is natural to ask whether stability of equilibrium points of
polynomial vector fields can be decided in finite time. In fact,
this is a well-known question of Arnold that appears
in~\cite{Arnold_Problems_for_Math}:

\vspace{-5pt}
\begin{itemize}
\item[] ``Is the stability problem for stationary points
algorithmically decidable? The well-known Lyapounov
theorem\footnote{The theorem that Arnold is referring to here is
the indirect method of Lyapunov related to linearization. This is
not to be confused with Lyapunov's direct method (or the second
method), which is what we are concerned with in sections that
follow.} solves the problem in the absence of eigenvalues with
zero real parts. In more complicated cases, where the stability
depends on higher order terms in the Taylor series, there exists
no algebraic criterion.

Let a vector field be given by polynomials of a fixed degree, with
rational coefficients. Does an algorithm exist, allowing to
decide, whether the stationary point is stable?''
\end{itemize}
\vspace{-5pt}

Later in~\cite{Costa_Doria_undecidabiliy}, the question of Arnold
is quoted with more detail:

\vspace{-5pt}
\begin{itemize}
\item[] ``In my problem the coefficients of the polynomials of
known degree and of a known number of variables are written on the
tape of the standard Turing machine in the standard order and in
the standard representation. The problem is whether there exists
an algorithm (an additional text for the machine independent of
the values of the coefficients) such that it solves the stability
problem for the stationary point at the origin (i.e., always stops
giving the answer ``stable'' or ``unstable'').

I hope, this algorithm exists if the degree is one.  It also
exists when the dimension is one. My conjecture has always been
that there is no algorithm for some sufficiently high degree and
dimension, perhaps for dimension $3$ and degree $3$ or even $2$. I
am less certain about what happens in dimension $2$. Of course the
nonexistence of a general algorithm for a fixed dimension working
for arbitrary degree or for a fixed degree working for an
arbitrary dimension, or working for all polynomials with arbitrary
degree and dimension would also be interesting.''
\end{itemize}
\vspace{-5pt}

To our knowledge, there has been no formal resolution to these
questions, neither for the case of stability in the sense of
Lyapunov, nor for the case of asymptotic stability (in its local
or global version). In~\cite{Costa_Doria_undecidabiliy}, da Costa
and Doria show that if the right hand side of the differential
equation contains elementary functions (sines, cosines,
exponentials, absolute value function, etc.), then there is no
algorithm for deciding whether the origin is stable or unstable.
They also present a dynamical system in~\cite{Costa_Doria_Hopf}
where one cannot decide whether a Hopf bifurcation will occur or
whether there will be parameter values such that a stable fixed
point becomes unstable. In earlier work, Arnold himself
demonstrates some of the difficulties that arise in stability
analysis of polynomial systems by presenting a parametric
polynomial system in $3$ variables and degree $5$, where the
boundary between stability and instability in parameter space is
not a semialgebraic set~\cite{Arnold_algebraic_unsolve}. A
relatively larger number of undecidability results are available
for questions related to other properties of polynomial vector
fields, such as
reachability~\cite{Undecidability_vec_fields_survey} or
boundedness of domain of
definition~\cite{Bounded_Defined_ODE_Undecidable}, or for
questions about stability of hybrid systems~\cite{TsiLinSat},
\cite{BlTi2}, \cite{BlTi_stab_contr_hybrid},
\cite{Deciding_stab_mortal_PWA}. We refer the interested reader to
the survey papers in~\cite{Survey_CT_Computation},
\cite{Undecidability_vec_fields_survey},
\cite{Sontag_complexity_comparison},
\cite{BlTi_complexity_3classes}, \cite{BlTi1}.

%There are, however, other properties of polynomial vector fields
%that are known to be undecidable. For example, this is the case
%for a class of reachability problems *. We refer the interested
%reader to survey papers in *,*,*. Interestingly, there seems to be
%many more undecidability results available in the literature for
%questions related to stability of hybrid systems *,*,*.

We are also interested to know whether the answer to the
undecidability question for asymptotic stability changes if the
dynamics is restricted to be homogeneous. A polynomial vector
field $\dot{x}=f(x)$ is \emph{homogeneous} if all entries of $f$
are homogeneous polynomials of the same degree. Homogeneous
systems are extensively studied in the literature on nonlinear
control~\cite{Stability_homog_poly_ODE}, \cite{Stabilize_Homog},
\cite{homog.feedback}, \cite{Baillieul_Homog_geometry},
\cite{Cubic_Homog_Planar}, \cite{HomogHomog},
\cite{homog.systems}, and some of the results of this chapter
(both negative and positive) are derived specifically for this
class of systems. A basic fact about homogeneous vector fields is
that for these systems the notions of local and global stability
are equivalent. Indeed, a homogeneous vector field of degree $d$
satisfies $f(\lambda x)=\lambda^d f(x)$ for any scalar $\lambda$,
and therefore the value of $f$ on the unit sphere determines its
value everywhere. It is also well-known that an asymptotically
stable homogeneous system admits a homogeneous Lyapunov
funciton~\cite{Hahn_stability_book},\cite{HomogHomog}.

Naturally, questions regarding complexity of deciding asymptotic
stability and questions about existence of Lyapunov functions are
related. For instance, if one proves that for a class of
polynomial vector fields, asymptotic stability implies existence
of a polynomial Lyapunov function together with a computable upper
bound on its degree, then the question of asymptotic stability for
that class becomes decidable. This is due to the fact that given
any polynomial system and any integer $d$, the question of
deciding whether the system admits a polynomial Lyapunov function
of degree $d$ can be answered in finite time using quantifier
elimination.

For the case of linear systems (i.e., homogeneous systems of
degree $1$), the situation is particularly nice. If such a system
is asymptotically stable, then there always exists a
\emph{quadratic} Lyapunov function. Asymptotic stability of a
linear system $\dot{x}=Ax$ is equivalent to the easily checkable
algebraic criterion that the eigenvalues of $A$ be in the open
left half complex plane. Deciding this property of the matrix $A$
can formally be done in polynomial time, e.g. by solving a
Lyapunov equation~\cite{BlTi1}.

Moving up in the degree, it is not difficult to show that if a
homogeneous polynomial vector field has even degree, then it can
never be asymptotically stable; see e.g.~\cite[p.
283]{Hahn_stability_book}. So the next interesting case occurs for
homogeneous vector fields of degree $3$. We will prove below that
determining asymptotic stability for such systems is strongly
NP-hard. This gives a lower bound on the complexity of this
problem. It is an interesting open question to investigate whether
in this specific setting, the problem is also undecidable.

One implication of our NP-hardness result is that unless P=NP, we
should not expect sum of squares Lyapunov functions of ``low
enough'' degree to always exist, even when the analysis is
restricted to cubic homogeneous vector fields. The semidefinite
program arising from a search for an sos Lyapunov function of
degree $2d$ for such a vector field in $n$ variables has size in
the order of ${n+d \choose d+1}$. This number is polynomial in $n$
for fixed $d$ (but exponential in $n$ when $d$ grows linearly in
$n$). Therefore, unlike the case of linear systems, we should not
hope to have a bound on the degree of sos Lyapunov functions that
is independent of the dimension.

%%%AAA: next three lines are not quite correct, so commented out.
%at least a linear growth in the minimum degree of an sos Lyapunov
%function for some hard class of instances of cubic vector fields
%with increasing dimension.

We postpone our study of existence of sos Lyapunov functions to
Section~\ref{sec:(non)-existence.sos.lyap} and proceed for now
with the following complexity result.

\begin{theorem}\label{thm:asym.stability.nphard}
Deciding asymptotic stability of homogeneous cubic polynomial
vector fields is strongly NP-hard.
\end{theorem}
%Instead of relating the solution of a combinatorial problem to the
%behavior of the trajectories of a cubic vector field, we will
%think in terms of a Lyapunov function that proves asymptotic
%stability.

The main intuition behind the proof of this theorem is the
following idea: We will relate the solution of a combinatorial
problem not to the behavior of the trajectories of a cubic vector
field that are hard to get a handle on, but instead to properties
of a Lyapunov function that proves asymptotic stability of this
vector field. As we will see shortly, insights from Lyapunov
theory make the proof of this theorem quite simple. The reduction
is broken into two steps:
\begin{center}
ONE-IN-THREE 3SAT \\ $\downarrow$ \\ positivity of quartic forms
\\ $\downarrow$\\  asymptotic stability of cubic vector fields
\end{center}
In the course of presenting these reductions, we will also discuss
some corollaries that are not directly related to our study of
asymptotic stability, but are of independent interest.

%First, we give a reduction from ONE-IN-THREE 3SAT to the problem
%of deciding positive definiteness of quartic forms. Then, we give
%a reduction from the latter problem to the problem of deciding
%asymptotic stability of cubic vector fields.

\subsection{Reduction from ONE-IN-THREE 3SAT to positivity of quartic forms}

As we remarked in Chapter~\ref{chap:nphard.convexity}, NP-hardness
of deciding nonnegativity (i.e., positive semidefiniteness) of
quartic forms is well-known. The proof commonly cited in the
literature is based on a reduction from the matrix copositivity
problem~\cite{nonnegativity_NP_hard}: given a symmetric $n \times
n$ matrix $Q$, decide whether $x^TQx\geq0$ for all $x$'s that are
elementwise nonnegative. Clearly, a matrix $Q$ is copositive if
and only if the quartic form $z^TQz$, with
$z_i\mathrel{\mathop:}=x_i^2$, is nonnegative. The original
reduction~\cite{nonnegativity_NP_hard} proving NP-hardness of
testing matrix copositivity is from the subset sum problem and
only establishes weak NP-hardness. However, reductions from the
stable set problem to matrix copositivity are also
known~\cite{deKlerk_StableSet_copositive},
\cite{copositivity_NPhard} and they result in NP-hardness in the
strong sense. Alternatively, strong NP-hardness of deciding
nonnegativity of quartic forms follows immediately from
NP-hardness of deciding convexity of quartic forms (proven in
Chapter~\ref{chap:nphard.convexity}) or from NP-hardness of
deciding nonnegativity of biquadratic forms (proven
in~\cite{Ling_et_al_Biquadratic}).

For reasons that will become clear shortly, we are interested in
showing hardness of deciding \emph{positive definiteness} of
quartic forms as opposed to positive semidefiniteness. This is in
some sense even easier to accomplish. A very straightforward
reduction from 3SAT proves NP-hardness of deciding positive
definiteness of polynomials of degree $6$. By using ONE-IN-THREE
3SAT instead, we will reduce the degree of the polynomial from $6$
to $4$.

\begin{proposition}\label{prop:positivity.quartic.NPhard}
It is strongly\footnote{Just like our results in
Chapter~\ref{chap:nphard.convexity}, the NP-hardness results of
this section will all be in the strong sense. From here on, we
will drop the prefix ``strong'' for brevity.} NP-hard to decide
whether a homogeneous polynomial of degree $4$ is positive
definite.
\end{proposition}

\begin{proof}
We give a reduction from ONE-IN-THREE 3SAT which is known to be
NP-complete~\cite[p. 259]{GareyJohnson_Book}. Recall that in
ONE-IN-THREE 3SAT, we are given a 3SAT instance (i.e., a
collection of clauses, where each clause consists of exactly three
literals, and each literal is either a variable or its negation)
and we are asked to decide whether there exists a $\{0,1\}$
assignment to the variables that makes the expression true with
the additional property that each clause has \emph{exactly one}
true literal.

To avoid introducing unnecessary notation, we present the
reduction on a specific instance. The pattern will make it obvious
that the general construction is no different. Given an instance
of ONE-IN-THREE 3SAT, such as the following
\begin{equation}\label{eq:reduciton.1-in-3.3sat.instance}
(x_1\vee\bar{x}_2\vee x_4)\wedge (\bar{x}_2\vee\bar{x}_3\vee
x_5)\wedge (\bar{x}_1\vee x_3\vee \bar{x}_5)\wedge (x_1\vee
x_3\vee x_4),
\end{equation}
we define the quartic polynomial $p$ as follows:
\begin{equation}\label{eq:reduction.p}
\begin{array}{lll}
p(x)&=&\sum_{i=1}^5 x_i^2(1-x_i)^2\\ \ &\
&+(x_1+(1-x_2)+x_4-1)^2+((1-x_2)+(1-x_3)+x_5-1)^2
\\ \ &\ &+((1-x_1)+x_3+(1-x_5)-1)^2+(x_1+x_3+x_4-1)^2.
\end{array}
\end{equation}
Having done so, our claim is that $p(x)>0$ for all $x\in
\mathbb{R}^5$ (or generally for all $x\in \mathbb{R}^n$) if and
only if the ONE-IN-THREE 3SAT instance is not satisfiable. Note
that $p$ is a sum of squares and therefore nonnegative. The only
possible locations for zeros of $p$ are by construction among the
points in $\{0,1\}^5$. If there is a satisfying Boolean assignment
$x$ to (\ref{eq:reduciton.1-in-3.3sat.instance}) with exactly one
true literal per clause, then $p$ will vanish at point $x$.
Conversely, if there are no such satisfying assignments, then for
any point in $\{0,1\}^5$, at least one of the terms in
(\ref{eq:reduction.p}) will be positive and hence $p$ will have no
zeros.

It remains to make $p$ homogeneous. This can be done via
introducing a new scalar variable $y$. If we let
\begin{equation}\label{eq:reduction.ph}
p_h(x,y)=y^4 p(\textstyle{\frac{x}{y}}),
\end{equation}
then we claim that $p_h$ (which is a quartic form) is positive
definite if and only if $p$ constructed as in
(\ref{eq:reduction.p}) has no zeros.\footnote{In general,
homogenization does not preserve positivity. For example, as shown
in~\cite{Reznick}, the polynomial $x_1^2+(1-x_1x_2)^2$ has no
zeros, but its homogenization $x_1^2y^2+(y^2-x_1x_2)^2$ has zeros
at the points $(1,0,0)^T$ and $(0,1,0)^T$. Nevertheless,
positivity is preserved under homogenization for the special class
of polynomials constructed in this reduction, essentially because
polynomials of type (\ref{eq:reduction.p}) have no zeros at
infinity.} Indeed, if $p$ has a zero at a point $x$, then that
zero is inherited by $p_h$ at the point $(x,1)$. If $p$ has no
zeros, then (\ref{eq:reduction.ph}) shows that $p_h$ can only
possibly have zeros at points with $y=0$. However, from the
structure of $p$ in (\ref{eq:reduction.p}) we see that
$$p_h(x,0)=x_1^4+\cdots+x_5^4,$$ which cannot be zero (except at
the origin). This concludes the proof.
\end{proof}

We present a simple corollary of the reduction we just gave on a
problem that is of relevance in polynomial integer
programming.\footnote{We are thankful to Amitabh Basu and
Jes\'{u}s De Loera for raising this question during a visit at UC
Davis, and for later insightful discussions.} Recall from
Chapter~\ref{chap:nphard.convexity}
(Definition~\ref{def:basic.semialgeb.set}) that a basic
semialgebraic set is a set defined by a finite number of
polynomial inequalities:
\begin{equation}\label{eq:semialgebraic.set.chap.Lyap}
\mathcal{S}=\{x\in\mathbb{R}^n|\ f_i(x)\geq0,\ i=1,\ldots,m\}.
\end{equation}
\begin{corollary}\label{cor:lattice.points.nphard}
Given a basic semialgebraic set, it is NP-hard to decide if the
set contains a lattice point, i.e., a point with integer
coordinates. This is true even when the set is defined by one
constraint ($m=1$) and the defining polynomial has degree $4$.
\end{corollary}
\begin{proof}
Given an instance of ONE-IN-THREE 3SAT, we define a polynomial $p$
of degree $4$ as in (\ref{eq:reduction.p}), and let the basic
semialgebraic set be given by
$$\mathcal{S}=\{x\in\mathbb{R}^n|\ -p(x)\geq0\}.$$ Then, by
Proposition~\ref{prop:positivity.quartic.NPhard}, if the
ONE-IN-THREE 3SAT instance is not satisfiable, the set
$\mathcal{S}$ is empty and hence has no lattice points.
Conversely, if the instance is satisfiable, then $\mathcal{S}$
contains at least one point belonging to $\{0,1\}^n$ and therefore
has a lattice point.
\end{proof}

By using the celebrated result on undecidability of checking
existence of integer solutions to polynomial equations (Hilbert's
10th problem), one can show that the problem considered in the
corollary above is in fact
undecidable~\cite{Undecidable_Hilbert10Survey}. The same is true
for quadratic integer programming when both the dimension $n$ and
the number of constraints $m$ are allowed to grow as part of the
input~\cite{QuadraticIP_undecidable}. The question of deciding
existence of lattice points in polyhedra (i.e., the case where
degree of $f_i$ in (\ref{eq:semialgebraic.set.chap.Lyap}) is $1$
for all $i$) is also interesting and in fact very well-studied.
For polyhedra, if both $n$ and $m$ are allowed to grow, then the
problem is NP-hard. This can be seen e.g. as a corollary of the
NP-hardness of the INTEGER KNAPSACK problem (though this is
NP-hardness in the weak sense); see~\cite[p.
247]{GareyJohnson_Book}. However, if $n$ is fixed and $m$ grows,
it follows from a result of Lenstra~\cite{Lenstra_IP} that the
problem can be decided in polynomial time. The same is true if $m$
is fixed and $n$ grows~\cite[Cor. 18.7c]{Schrijver_LP_IP_Book}.
See also~\cite{Papadimitriou_IP}.

%%%AAA: removed and rewrote because realized our corollary is already undecidable.
%%
%Let us put the above corollary in some context. The question of
%deciding existence of lattice points in polyhedra (i.e., the case
%where degree of $f_i$ in (\ref{eq:semialgebraic.set.chap.Lyap}) is
%$1$ for all $i$) is well-studied. For polyhedra, if both the
%dimension $n$ and the number of constraints $m$ are allowed to
%grow, then the problem of deciding existence of lattice points is
%NP-hard. This follows e.g. as a corollary of the NP-hardness of
%the INTEGER KNAPSACK problem (though this is NP-hardness in the
%weak sense); see~\cite[p. 247]{GareyJohnson_Book}. However, if $n$
%is fixed and $m$ grows, it follows from a result of
%Lenstra~\cite{Lenstra_IP} that the problem can be decided in
%polynomial time. The same is true if $m$ is fixed and $n$
%grows~\cite[Cor. 18.7c]{Schrijver_LP_IP_Book}. See
%also~\cite{Papadimitriou_IP}. By contrast,
%Corollary~\ref{cor:lattice.points.nphard} shows that when the
%defining polynomials are quartic, even with $m$ fixed to $1$, the
%problem is (strongly) NP-hard. Interestingly, if we allow both $m$
%and $n$ grow, then already for the case where the defining
%polynomials have degree $2$, the problem is known to be
%undecidable~\cite{QuadraticIP_undecidable}.

\subsection{Reduction from positivity of quartic forms to asymptotic stability of cubic vector fields}
We now present the second step of the reduction and finish the
proof of Theorem~\ref{thm:asym.stability.nphard}.

\begin{proof}[Proof of Theorem~\ref{thm:asym.stability.nphard}]
We give a reduction from the problem of deciding positive
definiteness of quartic forms, whose NP-hardness was established
in Proposition~\ref{prop:positivity.quartic.NPhard}. Given a
quartic form $V\mathrel{\mathop:}=V(x)$, we define the polynomial
vector field
\begin{equation}\label{eq:xdot.cubic.reduction}
\dot{x}=-\nabla V(x).
\end{equation}
Note that the vector field is homogeneous of degree $3$. We claim
that the above vector field is (locally or equivalently globally)
asymptotically stable if and only if $V$ is positive definite.
First, we observe that by construction
\begin{equation}\label{eq:Vdot<=0.always}
\dot{V}(x)=\langle \nabla V(x), \dot{x} \rangle=-||\nabla
V(x)||^2\leq 0.
\end{equation}
Suppose $V$ is positive definite. By Euler's identity for
homogeneous functions,\footnote{Euler's identity is easily derived
by differentiating both sides of the equation $V(\lambda
x)~=~\lambda^d V(x)$ with respect to $\lambda$ and setting
$\lambda=1$.} we have $V(x)=\frac{1}{4}x^T\nabla V(x).$ Therefore,
positive definiteness of $V$ implies that $\nabla V(x)$ cannot
vanish anywhere except at the origin. Hence, $\dot{V}(x)<0$ for
all $x\neq 0$. In view of Lyapunov's theorem (see e.g.~\cite[p.
124]{Khalil:3rd.Ed}), and the fact that a positive definite
homogeneous function is radially unbounded, it follows that the
system in (\ref{eq:xdot.cubic.reduction}) is globally
asymptotically stable.

For the converse direction, suppose
(\ref{eq:xdot.cubic.reduction}) is GAS. Our first claim is that
global asymptotic stability together with $\dot{V}(x)\leq 0$
implies that $V$ must be positive semidefinite. This follows from
the following simple argument, which we have also previously
presented in~\cite{AAA_PP_ACC11_Lyap_High_Deriv} for a different
purpose. Suppose for the sake of contradiction that for some
$\hat{x}\in\mathbb{R}^n$ and some $\epsilon>0,$ we had
$V(\hat{x})=-\epsilon<0$. Consider a trajectory $x(t;\hat{x})$ of
system (\ref{eq:xdot.cubic.reduction}) that starts at initial
condition $\hat{x}$, and let us evaluate the function $V$ on this
trajectory. Since $V(\hat{x})=-\epsilon$ and $\dot{V}(x)\leq 0$,
we have $V(x(t;\hat{x}))\leq-\epsilon$ for all $t>0$. However,
this contradicts the fact that by global asymptotic stability, the
trajectory must go to the origin, where $V$, being a form,
vanishes.

To prove that $V$ is positive definite, suppose by contradiction
that for some nonzero point $x^*\in\mathbb{R}^n$ we had
$V(x^*)=0$. Since we just proved that $V$ has to be positive
semidefinite, the point $x^*$ must be a global minimum of $V$.
Therefore, as a necessary condition of optimality, we should have
$\nabla V(x^*)=0$. But this contradicts the system in
(\ref{eq:xdot.cubic.reduction}) being GAS, since the trajectory
starting at $x^*$ stays there forever and can never go to the
origin.
\end{proof}

Perhaps of independent interest, the reduction we just gave
suggests a method for proving positive definiteness of forms.
Given a form $V$, we can construct a dynamical system as in
(\ref{eq:xdot.cubic.reduction}), and then any method that we may
have for proving stability of vector fields (e.g. the use of
various kinds of Lyapunov functions) can serve as an algorithm for
proving positivity of $V$. In particular, if we use a polynomial
Lyapunov function $W$ to prove stability of the system in
(\ref{eq:xdot.cubic.reduction}), we get the following corollary.

\begin{corollary}\label{cor:positivity.forms.Lyap}
Let $V$ and $W$ be two forms of possibly different degree. If $W$
is positive definite, and $\langle \nabla W, \nabla V \rangle$ is
positive definite, then $V$ is positive definite.
\end{corollary}

One interesting fact about this corollary is that its algebraic
version with sum of squares replaced for positivity is not true.
In other words, we can have $W$ sos (and positive definite),
$\langle \nabla W, \nabla V \rangle$ sos (and positive definite),
but $V$ not sos. This gives us a way of proving positivity of some
polynomials that are not sos, using \emph{only} sos certificates.
Given a form $V$, since the expression $\langle \nabla W, \nabla V
\rangle$ is linear in the coefficients of $W$, we can use sos
programming to search for a form $W$ that satisfies $W$ sos and
$\langle \nabla W, \nabla V \rangle$ sos, and this would prove
positivity of $V$. The following example demonstrates the
potential usefulness of this approach.

\begin{example}
Consider the following form of degree $6$:
\begin{equation}\label{eq:V.positive.Motzkin}
V(x)=
x_1^4x_2^2+x_1^2x_2^4-3x_1^2x_2^2x_3^2+x_3^6+\frac{1}{250}(x_1^2+x_2^2+x_3^2)^3.
\end{equation}
One can check that this polynomial is not a sum of squares. (In
fact, this is the Motzkin form presented in equation
(\ref{eq:Motzkin.form}) of
Chapter~\ref{chap:convexity.sos.convexity} slightly perturbed.) On
the other hand, we can use YALMIP~\cite{yalmip} together with the
SDP solver SeDuMi~\cite{sedumi} to search for a form $W$
satisfying
\begin{equation}\label{eq:W.sos.gradW.gradV.sos}
\begin{array} {rl}
W & \mbox{sos} \\
\langle \nabla W, \nabla V \rangle & \mbox{sos.}
\end{array}
\end{equation}
If we parameterize $W$ as a quadratic form, no feasible solution
will be returned form the solver. However, when we increase the
degree of $W$ from $2$ to $4$, the solver returns the following
polynomial
\begin{equation}\nonumber
\begin{array}{lll}
W(x)&=&9x_2^4+9x_1^4-6x_1^2x_2^2+6x_1^2x_3^2+6x_2^2x_3^2+3x_3^4-x_1^3x_2-x_1x_2^3
\\ \ &\ &-x_1^3x_3-3x_1^2x_2x_3-3x_1x_2^2x_3-x_2^3x_3-4x_1x_2x_3^2-x_1x_3^3-x_2x_3^3
\end{array}
\end{equation}
that satisfies both sos constrains in
(\ref{eq:W.sos.gradW.gradV.sos}). The Gram matrices in these sos
decompositions are positive definite. Therefore, $W$ and $\langle
\nabla W, \nabla V \rangle$ are positive definite forms. Hence, by
Corollary~\ref{cor:positivity.forms.Lyap}, we have a proof that
$V$ in (\ref{eq:V.positive.Motzkin}) is positive definite.
$\triangle$
\end{example}

Interestingly, approaches of this type that use gradient
information for proving positivity of polynomials with sum of
squares techniques have been studied by Nie, Demmel, and Sturmfels
in~\cite{Gradient_Ideal_SOS}, though the derivation there is not
Lyapunov-inspired.

\section{Non-existence of polynomial Lyapunov functions}\label{sec:no.poly.Lyap}

As we mentioned at the beginning of this chapter, the question of
global asymptotic stability of polynomial vector fields is
commonly addressed by seeking a Lyapunov function that is
polynomial itself. This approach has become further prevalent over
the past decade due to the fact that we can use sum of squares
techniques to algorithmically search for such Lyapunov functions.
The question therefore naturally arises as to whether existence of
polynomial Lyapunov functions is necessary for global stability of
polynomial systems. In this section, we give a negative answer to
this question by presenting a remarkably simple counterexample. In
view of the fact that globally asymptotically stable linear
systems always admit quadratic Lyapunov functions, it is quite
interesting to observe that the following vector field that is
arguably ``the next simplest system'' to consider does not admit a
polynomial Lyapunov function of any degree.

\begin{theorem}\label{thm:GAS.poly.no.poly.Lyap}
Consider the polynomial vector field
\begin{equation}\label{eq:counterexample.sys}
\begin{array}{rll}
\dot{x}&=&-x+xy \\
\dot{y}&=&-y.
\end{array}
\end{equation}
The origin is a globally asymptotically stable equilibrium point,
but the system does not admit a polynomial Lyapunov function.
\end{theorem}
\begin{proof}
Let us first show that the system is GAS. Consider the Lyapunov
function
\begin{equation}\nonumber
V(x,y)=\ln(1+x^2)+y^2,
\end{equation}
which clearly vanishes at the origin, is strictly positive for all
$(x,y)\neq(0,0)$, and is radially unbounded. The derivative of
$V(x,y)$ along the trajectories of (\ref{eq:counterexample.sys})
is given by
\begin{equation}\nonumber
\begin{array}{rll}
\dot{V}(x,y)&=&\frac{\partial{V}}{\partial{x}}\dot{x}+\frac{\partial{V}}{\partial{y}}\dot{y} \\
%\ &\ &\ \\
\ &=&\frac{2x^2(y-1)}{1+x^2}-2y^2 \\
\ &=&-\frac{x^2 + 2y^2 + x^2 y^2 + (x-xy)^2}{1+x^2},
\end{array}
\end{equation}
which is obviously strictly negative for all $(x,y)\neq(0,0)$. In
view of Lyapunov's stability theorem (see e.g.~\cite[p.
124]{Khalil:3rd.Ed}), this shows that the origin is globally
asymptotically stable.

Let us now prove that no positive definite polynomial Lyapunov
function (of any degree) can decrease along the trajectories of
system (\ref{eq:counterexample.sys}).
%\footnote{In fact our argument shows the
%  following stronger statement: No (nonconstant) nonnegative
%  polynomial function can decrease along all trajectories of system
%  (\ref{eq:counterexample.sys}), even without a radially unboundedness
%  requirement.}
The proof will be based on simply considering the value of a
candidate Lyapunov function at two specific points. We will look
at trajectories on the nonnegative orthant, with initial
conditions on the line $(k,\alpha k)$ for some constant
$\alpha>0$, and then observe the location of the crossing of the
trajectory with the horizontal line $y=\alpha$. We will argue that
by taking $k$ large enough, the trajectory will have to travel
``too far east'' (see Figure~\ref{fig:trajectories}) and this will
make it impossible for any polynomial Lyapunov function to
decrease.

\begin{figure}%[thpb]
\centering
\includegraphics[width=.6\columnwidth]{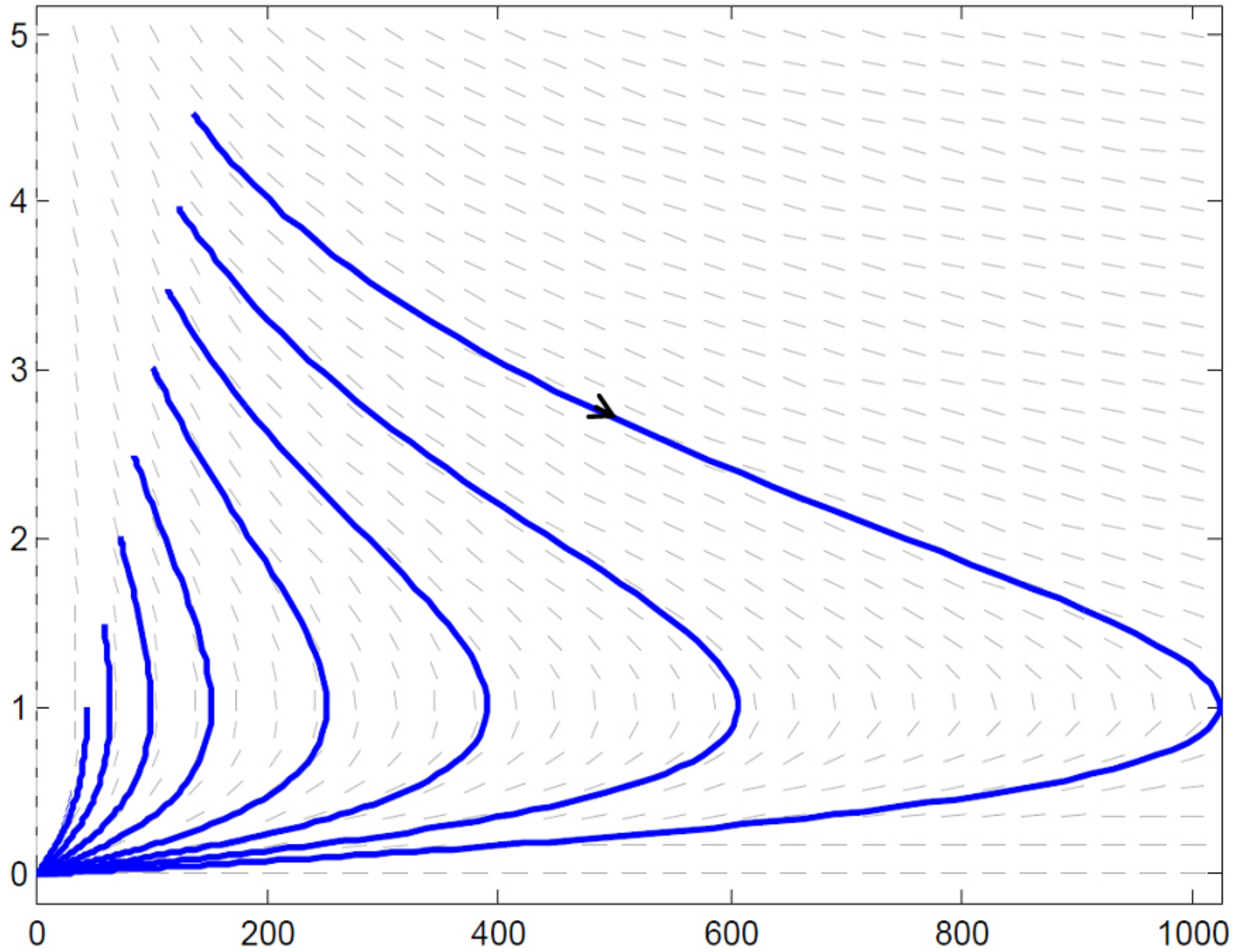}
\caption{Typical trajectories of the vector field in
(\ref{eq:counterexample.sys}) starting from initial conditions in
the nonnegative orthant.} \label{fig:trajectories}
\end{figure}

To do this formally, we start by noting that we can explicitly
solve
%\footnote{It is easy to give a slightly longer proof of
%Theorem~\ref{thm:GAS.poly.no.poly.Lyap} without computing the
%solutions explicitly.}
for the solution $(x(t),y(t))$ of the vector field in
(\ref{eq:counterexample.sys}) starting from any initial condition
$(x(0),y(0))$:

\begin{equation}\label{eq:explicit.soln}
\begin{array}{rll}
x(t)&=&x(0) e^{[y(0)-y(0)e^{-t}-t]}\\
y(t)&=&y(0) e^{-t}.
\end{array}
\end{equation}
Consider initial conditions
\begin{equation}\nonumber
(x(0),y(0))=(k,\alpha k)
\end{equation}
parameterized by $k>1$ and for some fixed constant $\alpha>0$.
From the explicit solution in (\ref{eq:explicit.soln}) we have
that the time $t^*$ it takes for the trajectory to cross the line
$y=\alpha$ is
$$t^*=\ln (k),$$ and that the location of this crossing is given by
$$(x(t^*),y(t^*))=(e^{\alpha(k-1)},\alpha).$$
Consider now any candidate nonnegative polynomial function
$V(x,y)$ that depends on both $x$ and $y$ (as any Lyapunov
function should). Since $k>1$ (and thus, $t^* > 0$), for $V(x,y)$
to be a valid Lyapunov function, it must satisfy $V(x(t^*),y(t^*))
< V(x(0),y(0))$, i.e.,
\begin{equation}\nonumber
V(e^{\alpha(k-1)},\alpha)<V(k,\alpha k).
\end{equation}
However, this inequality cannot hold for $k$ large enough, since
for a generic fixed $\alpha$, the left hand side grows
exponentially in $k$ whereas the right hand side grows only
polynomially in $k$. The only subtlety arises from the fact that
$V(e^{\alpha(k-1)},\alpha)$ could potentially be a constant for
some particular choices of $\alpha$. However, for any polynomial
$V(x,y)$ with nontrivial dependence on $y$, this may happen for at
most finitely many values of $\alpha$. Therefore, any generic
choice of $\alpha$ would make the argument work.
\end{proof}

\paragraph{Example of Bacciotti and Rosier.} After our
counterexample above was submitted for publication, Christian
Ebenbauer brought to our attention an earlier counterexample of
Bacciotti and Rosier~\cite[Prop.
5.2]{Bacciotti.Rosier.Liapunov.Book} that achieves the same goal
(though by using irrational coefficients). We will explain the
differences between the two examples below. At the time of
submission of our result, we were under the impression that no
such examples were known, partly because of a recent reference in
the controls literature that ends its conclusion with the
following
statement~\cite{Peet.Antonis.converse.sos.CDC},~\cite{Peet.Antonis.converse.sos.journal}:

\begin{itemize}
\item[] ``Still unresolved is the fundamental question of whether
\emph{globally} stable vector fields will also admit
sum-of-squares Lyapunov functions.''
\end{itemize}

In~\cite{Peet.Antonis.converse.sos.CDC},~\cite{Peet.Antonis.converse.sos.journal},
what is referred to as a sum of squares Lyapunov function (in
contrast to our terminology here) is a Lyapunov function that is a
sum of squares, with no sos requirements on its derivative.
Therefore, the fundamental question referred to above is on
existence of a polynomial Lyapunov function. If one were to exist,
then we could simply square it to get another polynomial Lyapunov
function that is a sum of squares (see Lemma~\ref{lem:W=V^2}).

The example of Bacciotti and Rosier is a vector field in $2$
variables and degree $5$ that is GAS but has no polynomial (and no
analytic) Lyapunov function even around the origin. Their very
clever construction is complementary to our example in the sense
that what creates trouble for existence of polynomial Lyapunov
functions in our Theorem~\ref{thm:GAS.poly.no.poly.Lyap} is growth
rates \emph{arbitrarily far} away from the origin, whereas the
problem arising in their example is slow decay rates
\emph{arbitrarily close} to the origin. The example crucially
relies on a parameter that appears as part of the coefficients of
the vector field being \emph{irrational}. (Indeed, one easily sees
that if that parameter is rational, their vector field does admit
a polynomial Lyapunov function.) In practical applications where
computational techniques for searching over Lyapunov functions on
finite precision machines are used, such issues with irrationality
of the input cannot occur. By contrast, the example in
(\ref{eq:counterexample.sys}) is much less contrived and
demonstrates that non-existence of polynomial Lyapunov functions
can happen for extremely simple systems that may very well appear
in applications.

%Since existence of an \emph{analytic} Lyapunov function on an open
%set around the origin implies existence of a polynomial one *, and
%since the example Bacciotti and Rosier admits no polynomial
%Lyapunov function even locally, it follows that their vector field
%also does not admit an analytic Lyapunov function *.
In~\cite{Peet.exp.stability}, Peet has shown that locally
exponentially stable polynomial vector fields admit polynomial
Lyapunov functions on compact sets. The example of Bacciotti and
Rosier implies that the assumption of exponential stability indeed
cannot be dropped.

\section{(Non)-existence of sum of squares Lyapunov
functions}\label{sec:(non)-existence.sos.lyap}

In this section, we suppose that the polynomial vector field at
hand admits a polynomial Lyapunov function, and we would like to
investigate whether such a Lyapunov function can be found with sos
programming. In other words, we would like to see whether the
constrains in (\ref{eq:V.SOS}) and (\ref{eq:-Vdot.SOS}) are more
conservative than the true Lyapunov inequalities in
(\ref{eq:V.positive}) and (\ref{eq:Vdot.negative}). We think of
the sos Lyapunov conditions in (\ref{eq:V.SOS}) and
(\ref{eq:-Vdot.SOS}) as sufficient conditions for the strict
inequalities in (\ref{eq:V.positive}) and (\ref{eq:Vdot.negative})
even though sos decomposition in general merely guarantees
non-strict inequalities. The reason for this is that when an sos
feasibility problem is strictly feasible, the polynomials returned
by interior point algorithms are automatically positive definite
(see~\cite[p. 41]{AAA_MS_Thesis} for more discussion).\footnote{We
expect the reader to recall the basic definitions and concepts
from Subsection~\ref{subsec:nonnegativity.sos.basics} of the
previous chapter. Throughout, when we say a Lyapunov function (or
the negative of its derivative) is \emph{positive definite}, we
mean that it is positive everywhere except possibly at the
origin.}

We shall emphasize that existence of nonnegative polynomials that
are not sums of squares does not imply on its own that the sos
conditions in (\ref{eq:V.SOS}) and (\ref{eq:-Vdot.SOS}) are more
conservative than the Lyapunov inequalities in
(\ref{eq:V.positive}) and (\ref{eq:Vdot.negative}). Since Lyapunov
functions are not in general unique, it could happen that within
the set of valid polynomial Lyapunov functions of a given degree,
there is always at least one that satisfies the sos conditions
(\ref{eq:V.SOS}) and (\ref{eq:-Vdot.SOS}). Moreover, many of the
known examples of nonnegative polynomials that are not sos have
multiple zeros and local minima~\cite{Reznick} and therefore
cannot serve as Lyapunov functions. Indeed, if a function has a
local minimum other than the origin, then its value evaluated on a
trajectory starting from the local minimum would not be
decreasing.

%On the other hand, the asymptotic picture is known to be quite
%different. Blekherman~\cite{BlekhermanSOS} has shown that for
%fixed degree $d\geq4$, as the number of variables $n$ goes to
%infinity, most nonnegative polynomials are in fact not sums of
%squares. However, the values of $n$ for which the asymptotics of
%these results kick in are

%A related question is Hilbert's 17th problem which asks if every
%nonnegative polynomial can be represented as a sum of squares of
%rational functions. The question was answered in the affirmative
%by Artin; see e.g.~\cite{Reznick}. Although we do not know how to
%check for such a decomposition via semidefinite programming, there
%are results on uniformity of denominators in this representation,
%such as the result of Reznick in~\cite{Reznick_Unif_denominator},
%that give us a hierarchy of semidefinite programs that approximate
%the set of nonnegative polynomials with increasing accuracy. In
%fact, the recent result of Scheiderer (Theorem~\ref{thm:claus})
%which we will crucially use in
%Section~\ref{sec:converse.sos.results} is a result of this type.

\subsection{A motivating example}\label{subsec:motivating.example}

The following example will help motivate the kind of questions
that we are addressing in this section.
\begin{example}\label{ex:poly8}
Consider the dynamical system
\begin{equation} \label{eq:poly8}
\begin{array}{lll}
\dot{x_{1}}&=&-0.15x_1^7+200x_1^6x_2-10.5x_1^5x_2^2-807x_1^4x_2^3\\
\ &\ &+14x_1^3x_2^4+600x_1^2x_2^5-3.5x_1x_2^6+9x_2^7
\\
\ &\ & \ \\
\dot{x_{2}}&=&-9x_1^7-3.5x_1^6x_2-600x_1^5x_2^2+14x_1^4x_2^3\\
\ &\ &+807x_1^3x_2^4-10.5x_1^2x_2^5-200x_1x_2^6-0.15x_2^7.
\end{array}
\end{equation}
A typical trajectory of the system that starts from the initial
condition $x_0=(2, 2)^T$ is plotted in Figure~\ref{fig:poly8}. Our
goal is to establish global asymptotic stability of the origin by
searching for a polynomial Lyapunov function. Since the vector
field is homogeneous, the search can be restricted to homogeneous
Lyapunov functions~\cite{Hahn_stability_book},~\cite{HomogHomog}.
To employ the sos technique, we can use the software package
SOSTOOLS~\cite{sostools} to search for a Lyapunov function
satisfying the sos conditions (\ref{eq:V.SOS}) and
(\ref{eq:-Vdot.SOS}). However, if we do this, we will not find any
Lyapunov functions of degree $2$, $4$, or $6$. If needed, a
certificate from the dual semidefinite program can be obtained,
which would prove that no polynomial of degree up to $6$ can
satisfy the sos requirements (\ref{eq:V.SOS}) and
(\ref{eq:-Vdot.SOS}).

At this point we are faced with the following question. Does the
system really not admit a Lyapunov function of degree $6$ that
satisfies the true Lyapunov inequalities in (\ref{eq:V.positive}),
(\ref{eq:Vdot.negative})? Or is the failure due to the fact that
the sos conditions in (\ref{eq:V.SOS}), (\ref{eq:-Vdot.SOS}) are
more conservative?

Note that when searching for a degree $6$ Lyapunov function, the
sos constraint in (\ref{eq:V.SOS}) is requiring a homogeneous
polynomial in $2$ variables and of degree $6$ to be a sum of
squares. The sos condition (\ref{eq:-Vdot.SOS}) on the derivative
is also a condition on a homogeneous polynomial in $2$ variables,
but in this case of degree $12$. (This is easy to see from
$\dot{V}=\langle \nabla V,f \rangle$.) Recall from
Theorem~\ref{thm:Hilbert} of the previous chapter that
nonnegativity and sum of squares are equivalent notions for
homogeneous bivariate polynomials, irrespective of the degree.
Hence, we now have a proof that this dynamical system truly does
not have a Lyapunov function of degree $6$ (or lower).

This fact is perhaps geometrically intuitive.
Figure~\ref{fig:poly8} shows that the trajectory of this system is
stretching out in $8$ different directions. So, we would expect
the degree of the Lyapunov function to be at least $8$. Indeed,
when we increase the degree of the candidate function to $8$,
SOSTOOLS and the SDP solver SeDuMi~\cite{sedumi} succeed in
finding the following Lyapunov function:
\begin{eqnarray}\nonumber
V(x)&=&0.02x_1^8+0.015x_1^7x_2+1.743x_1^6x_2^2-0.106x_1^5x_2^3
\nonumber \\
 \ &\ &-3.517x_1^4x_2^4+0.106x_1^3x_2^5+1.743x_1^2x_2^6
 \nonumber\\
 \ &\ &-0.015x_1x_2^7+0.02x_2^8. \nonumber
\end{eqnarray}
\begin{figure}%[thpb]
\centering \scalebox{0.38} {\includegraphics{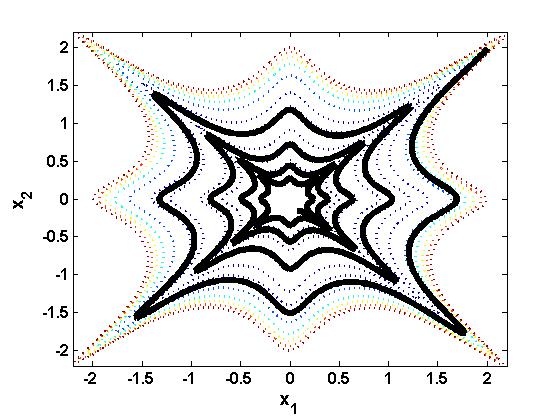}}
\caption{A typical trajectory of the vector filed in
Example~\ref{ex:poly8} (solid), level sets of a degree $8$
polynomial Lyapunov function (dotted).} \label{fig:poly8}
\end{figure}
The level sets of this Lyapunov function are plotted in
Figure~\ref{fig:poly8} and are clearly invariant under the
trajectory. $\triangle$
\end{example}

\subsection{A counterexample}\label{subsec:the.counterexample}
Unlike the scenario in the previous example, we now show that a
failure in finding a Lyapunov function of a particular degree via
sum of squares programming can also be due to the gap between
nonnegativity and sum of squares. What will be conservative in the
following counterexample is the sos condition on the
derivative.\footnote{This counterexample has appeared in our
earlier work~\cite{AAA_MS_Thesis} but not with a complete proof.}

%Recall that every time we search for a Lyapunov function using
%sos programming, we use the sos relaxation twice. First for
%nonnegativity of $V$, and second for nonpositivity of $\dot{V}$.
%Each of these relaxations may in general be conservative. Many of
%the well-known examples of nonnegative polynomials that are not
%sos have more than one local minimum~\cite{Reznick}. On the other
%hand, for continuous time systems, Lyapunov functions cannot have
%local minima, except for a single global minimum at the origin. To
%see this, not that that if a trajectory starts exactly at the
%local minima of $V$, irrespective of the direction in which $f$
%moves the trajectory, the Lyapunov function will locally
%increase\footnote{There are Lyapunov functions that are not
%required to decrease monotonically but still prove stability of
%dynamical systems; see~\cite{AAA_MS_Thesis}. These Lyapunov
%functions are allowed to increase locally and therefore can, in
%theory, have local minimas.}. As a consequence, many of the
%well-known nonnegative polynomials that are not sos cannot serve
%as Lyapunov functions. On the other hand, the following example
%shows that an sos relaxation on $\dot{V}$ can be conservative.

Consider the dynamical system
\begin{equation} \label{eq:sos.conservative.dynamics}
\begin{array}{lll}
\dot{x_{1}}&=&-x_1^3x_2^2+2x_1^3x_2-x_1^3+4x_1^2x_2^2-8x_1^2x_2+4x_1^2 \\
\ &\ &-x_1x_2^4+4x_1x_2^3-4x_1+10x_2^2
\\ \ &\ & \ \\
\dot{x_{2}}&=&-9x_1^2x_2+10x_1^2+2x_1x_2^3-8x_1x_2^2-4x_1-x_2^3
\\
\ &\ &+4x_2^2-4x_2.
\end{array}
\end{equation}
One can verify that the origin is the only equilibrium point for
this system, and therefore it makes sense to investigate global
asymptotic stability. If we search for a quadratic Lyapunov
function for (\ref{eq:sos.conservative.dynamics}) using sos
programming, we will not find one. It will turn out that the
corresponding semidefinite program is infeasible. We will prove
shortly why this is the case, i.e, why no quadratic function $V$
can satisfy
\begin{equation}\label{eq:both.sos.conditions}
\begin{array}{rl}
V & \mbox{sos} \\
-\dot{V} & \mbox{sos.}
\end{array}
\end{equation}
Nevertheless, we claim that
\begin{equation}\label{eq:V.0.5x1^2+.5x2^2}
V(x)=\frac{1}{2}x_1^2+\frac{1}{2}x_2^2
\end{equation}
is a valid Lyapunov function. Indeed, one can check that
\begin{equation}\label{eq:Vdot.-Motzkin}
\dot{V}(x)=x_1\dot{x}_1+x_2\dot{x}_2=-M(x_1-1,x_2-1),
\end{equation}
where $M(x_1,x_2)$ is the Motzkin polynomial~\cite{MotzkinSOS}:
\begin{equation}\nonumber%\label{eq:Motzkin}
M(x_1,x_2)=x_1^4x_2^2+x_1^2x_2^4-3x_1^2x_2^2+1.
\end{equation}
This polynomial is just a dehomogenized version of the Motzkin
form presented before, and it has the property of being
nonnegative but not a sum of squares. The polynomial $\dot{V}$ is
strictly negative everywhere, except for the origin and three
other points $(0,2)^{T}$, $(2,0)^{T}$, and $(2,2)^{T}$, where
$\dot{V}$ is zero. However, at each of these three points we have
$\dot{x}\neq0$. Once the trajectory reaches any of these three
points, it will be kicked out to a region where $\dot{V}$ is
strictly negative. Therefore, by LaSalle's invariance principle
(see e.g. \cite[p. 128]{Khalil:3rd.Ed}), the quadratic Lyapunov
function in (\ref{eq:V.0.5x1^2+.5x2^2}) proves global asymptotic
stability of the origin of (\ref{eq:sos.conservative.dynamics}).

\begin{figure}[h]
\begin{center}
    \mbox{
      \subfigure[Shifted Motzkin polynomial is nonnegative but not sos.]
      {\label{subfig:shifted.Motzkin}\scalebox{0.28}{\includegraphics{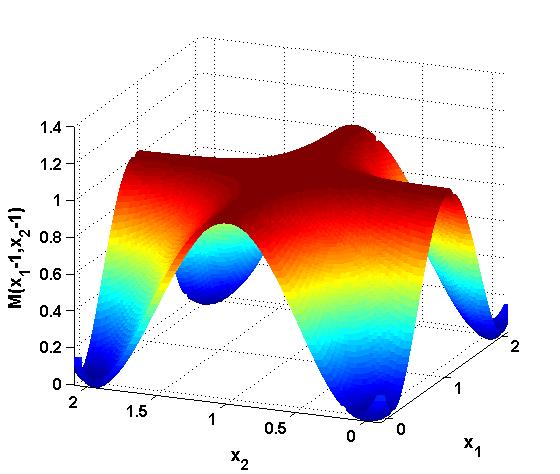}}}}
\mbox{
      \subfigure[Typical trajectories of (\ref{eq:sos.conservative.dynamics}) (solid), level sets of $V$ (dotted).]
      {\label{subfig:sos.hurt.trajec}\scalebox{0.25}{\includegraphics{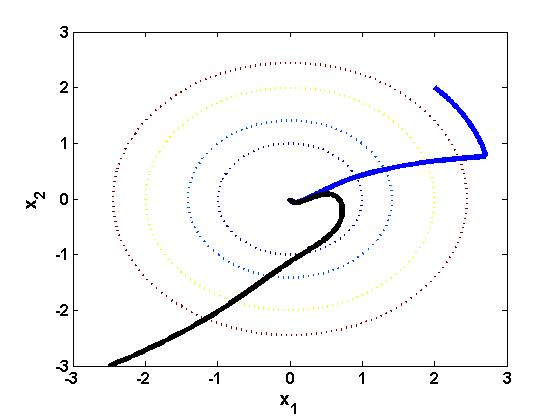}}}

      \subfigure[Level sets of a quartic Lyapunov function found through sos~programming.]
      {\label{subfig:sos.hurt.Lyap4}\scalebox{0.25}{\includegraphics{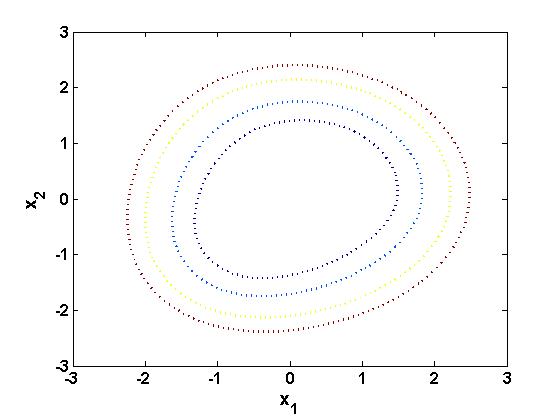}}} }

    \caption{The quadratic polynomial $\frac{1}{2}x_1^2+\frac{1}{2}x_2^2$ is a valid Lyapunov function for the vector field in (\ref{eq:sos.conservative.dynamics}) but it is not detected through sos programming.}
\label{fig:sos.hurt.Motzkin}
\end{center}
\end{figure}

The fact that $\dot{V}$ is zero at three points other than the
origin is not the reason why sos programming is failing. After
all, when we impose the condition that $-\dot{V}$ should be sos,
we allow for the possibility of a non-strict inequality. The
reason why our sos program does not recognize
(\ref{eq:V.0.5x1^2+.5x2^2}) as a Lyapunov function is that the
shifted Motzkin polynomial in (\ref{eq:Vdot.-Motzkin}) is
nonnegative but it is not a sum of squares. This sextic polynomial
is plotted in Figure~\ref{subfig:shifted.Motzkin}. Trajectories of
(\ref{eq:sos.conservative.dynamics}) starting at $(2,2)^{T}$ and
$(-2.5,-3)^{T}$ along with level sets of $V$ are shown in
Figure~\ref{subfig:sos.hurt.trajec}.

So far, we have shown that $V$ in (\ref{eq:V.0.5x1^2+.5x2^2}) is a
valid Lyapunov function but does not satisfy the sos conditions in
(\ref{eq:both.sos.conditions}). We still need to show why no other
quadratic Lyapunov function
\begin{equation}\label{eq:U(x).general.quadratic}
U(x)=c_1x_1^2+c_2x_1x_2+c_3x_2^2
\end{equation}
can satisfy the sos conditions either.\footnote{Since we can
assume that the Lyapunov function $U$ and its gradient vanish at
the origin, linear or constant terms are not needed in
(\ref{eq:U(x).general.quadratic}).} We will in fact prove the
stronger statement that $V$ in (\ref{eq:V.0.5x1^2+.5x2^2}) is the
only valid quadratic Lyapunov function for this system up to
scaling, i.e., any quadratic function $U$ that is not a scalar
multiple of $\frac{1}{2}x_1^2+\frac{1}{2}x_2^2$ cannot satisfy
$U\geq0$ and $-\dot{U}\geq0$. It will even be the case that no
such $U$ can satisfy $-\dot{U}\geq0$ alone. (The latter fact is to
be expected since global asymptotic stability of
(\ref{eq:sos.conservative.dynamics}) together with $-\dot{U}\geq0$
would automatically imply $U\geq0$; see~\cite[Theorem
1.1]{AAA_PP_ACC11_Lyap_High_Deriv}.)

So, let us show that $-\dot{U}\geq0$ implies $U$ is a scalar
multiple of $\frac{1}{2}x_1^2+\frac{1}{2}x_2^2$. Because Lyapunov
functions are closed under positive scalings, without loss of
generality we can take $c_1~=~1$. One can check that
$$-\dot{U}(0,2)=-80c_2,$$ so to have $-\dot{U}\geq0$, we need
$c_2\leq0$. Similarly, $$-\dot{U}(2,2)=-288c_1+288c_3,$$ which
implies that $c_3\geq1$. Let us now look at
\begin{equation}\label{eq:-Udot(x1,1)}
\begin{array}{lll}
-\dot{U}(x_1,1)&=& -c_2x_1^3 + 10c_2x_1^2 + 2c_2x_1 - 10c_2 -
2c_3x_1^2\\
\ &\ & + 20c_3x_1 + 2c_3 + 2x_1^2 - 20x_1.
\end{array}
\end{equation}
If we let $x_1\rightarrow -\infty$, the term $-c_2x_1^3$ dominates
this polynomial. Since $c_2\leq0$ and $-\dot{U}\geq0$, we conclude
that $c_2=0$. Once $c_2$ is set to zero in (\ref{eq:-Udot(x1,1)}),
the dominating term for $x_1$ large will be $(2-2c_3)x_1^2$.
Therefore to have $-\dot{U}(x_1,1)\geq0$ as
$x_1\rightarrow\pm\infty$ we must have $c_3\leq1$. Hence, we
conclude that $c_1=1, c_2=0, c_3=1$, and this finishes the proof.

Even though sos programming failed to prove stability of the
system in (\ref{eq:sos.conservative.dynamics}) with a quadratic
Lyapunov function, if we increase the degree of the candidate
Lyapunov function from $2$ to $4$, then SOSTOOLS succeeds in
finding a quartic Lyapunov function
\begin{eqnarray}\nonumber %\label{eq:sos.hurt.lyap.fn.poly4}
W(x)&=&0.08x_1^4-0.04x_1^3+0.13x_1^2x_2^2+0.03x_1^2x_2
\nonumber \\
 \ &\
 &+0.13x_1^2+0.04x_1x_2^2-0.15x_1x_2\nonumber
 \\
 \ &\ &+0.07x_2^4-0.01x_2^3+0.12x_2^2, \nonumber
\end{eqnarray}
which satisfies the sos conditions in
(\ref{eq:both.sos.conditions}). The level sets of this function
are close to circles and are plotted in
Figure~\ref{subfig:sos.hurt.Lyap4}.

Motivated by this example, it is natural to ask whether it is
always true that upon increasing the degree of the Lyapunov
function one will find Lyapunov functions that satisfy the sum of
squares conditions in (\ref{eq:both.sos.conditions}). In the next
subsection, we will prove that this is indeed the case, at least
for planar systems such as the one in this example, and also for
systems that are homogeneous.

%statement is indeed true at least for homogeneous vector fields.

%We should mention that the example we just described was contrived
%to make our point. For many practical problems, sos programming
%has provably shown to be a powerful technique~\cite{PabloLyap},
%\cite{PapP02}, \cite{PraP03}. There are some recent
%results~\cite{Blekhermansos}, however, that show for a fixed
%degree, as the dimension goes up the gap between nonnegativity and
%sos broadens. The extent to which this can impact existence of
%Lyapunov functions in higher dimensions is yet to be investigated.

\subsection{Converse sos Lyapunov
theorems}\label{subsec:converse.sos.results}

In~\cite{Peet.Antonis.converse.sos.CDC},
\cite{Peet.Antonis.converse.sos.journal}, it is shown that if a
system admits a polynomial Lyapunov function, then it also admits
one that is a sum of squares. However, the results there do not
lead to any conclusions as to whether the negative of the
derivative of the Lyapunov function is sos, i.e, whether condition
(\ref{eq:-Vdot.SOS}) is satisfied. As we remarked before, there is
therefore no guarantee that the semidefinite program can find such
a Lyapunov function. Indeed, our counterexample in the previous
subsection demonstrated this very phenomenon.

The proof technique used
in~\cite{Peet.Antonis.converse.sos.CDC},\cite{Peet.Antonis.converse.sos.journal}
is based on approximating the solution map using the Picard
iteration and is interesting in itself, though the actual
conclusion that a Lyapunov function that is sos exists has a far
simpler proof which we give in the next lemma.

\begin{lemma}\label{lem:W=V^2}
If a polynomial dynamical system has a positive definite
polynomial Lyapunov function $V$ with a negative definite
derivative $\dot{V}$, then it also admits a positive definite
polynomial Lyapunov function $W$ which is a sum of squares.
\end{lemma}
\begin{proof}
Take $W=V^2$. The negative of the derivative $-\dot{W}=-2V\dot{V}$
is clearly positive definite (though it may not be sos).
\end{proof}

We will next prove a converse sos Lyapunov theorem that guarantees
the derivative of the Lyapunov function will also satisfy the sos
condition, though this result is restricted to homogeneous
systems. The proof of this theorem relies on the following
Positivstellensatz result due to Scheiderer.

\begin{theorem}[Scheiderer,~\cite{Claus_Hilbert17}] \label{thm:claus}
Given any two positive definite homogeneous polynomials $p$ and
$q$, there exists an integer $k$ such that $pq^k$ is a sum of
squares.
\end{theorem}

\begin{theorem}\label{thm:poly.lyap.then.sos.lyap}
Given a homogeneous polynomial vector field, suppose there exists
a homogeneous polynomial Lyapunov function $V$ such that $V$ and
$-\dot{V}$ are positive definite. Then, there also exists a
homogeneous polynomial Lyapunov function $W$ such that $W$ is sos
and $-\dot{W}$ is sos.
\end{theorem}

\begin{proof}
Observe that $V^2$ and $-2V\dot{V}$ are both positive definite and
homogeneous polynomials. Applying Theorem~\ref{thm:claus} to these
two polynomials, we conclude the existence of an integer $k$ such
that $(-2V\dot{V})(V^2)^k$ is sos. Let $$W=V^{2k+2}.$$ Then, $W$
is clearly sos since it is a perfect even power. Moreover,
$$-\dot{W}=-(2k+2)V^{2k+1}\dot{V}=-(k+1)2V^{2k}V\dot{V}$$
is also sos by the previous claim.\footnote{Note that $W$
constructed in this proof proves GAS since $-\dot{W}$ is positive
definite and $W$ itself being homogeneous and positive definite is
automatically radially unbounded.}
\end{proof}

Next, we develop a similar theorem that removes the homogeneity
assumption from the vector field, but instead is restricted to
vector fields on the plane. For this, we need another result of
Scheiderer.

\begin{theorem}[{Scheiderer,~\cite[Cor. 3.12]{Claus_3vars_sos}}] \label{thm:claus.3vars}
Let $p\mathrel{\mathop:}=p(x_1,x_2,x_3)$ and
$q\mathrel{\mathop:}=q(x_1,x_2,x_3)$ be two homogeneous
polynomials in three variables, with $p$ positive semidefinite and
$q$ positive definite. Then, there exists an integer $k$ such that
$pq^k$ is a sum of squares.
\end{theorem}

\begin{theorem}\label{thm:poly.lyap.then.sos.lyap.PLANAR}
Given a (not necessarily homogeneous) polynomial vector field in
two variables, suppose there exists a positive definite polynomial
Lyapunov function $V,$ with $-\dot{V}$ positive definite, and such
that the highest order term of $V$ has no zeros\footnote{This
requirement is only slightly stronger than the requirement of
radial unboundedness, which is imposed on $V$ by Lyapunov's
theorem anyway.}. Then, there also exists a polynomial Lyapunov
function $W$ such that $W$ is sos and $-\dot{W}$ is sos.
\end{theorem}

\begin{proof}
Let $\tilde{V}=V+1$. So, $\dot{\tilde{V}}=\dot{V}$. Consider the
(non-homogeneous) polynomials $\tilde{V}^2$ and
$-2\tilde{V}\dot{\tilde{V}}$ in the variables
$x\mathrel{\mathop:}=(x_1,x_2)$. Let us denote the (even) degrees
of these polynomials respectively by $d_1$ and $d_2$. Note that
$\tilde{V}^2$ is nowhere zero and $-2\tilde{V}\dot{\tilde{V}}$ is
only zero at the origin. Our first step is to homogenize these
polynomials by introducing a new variable $y$. Observing that the
homogenization of products of polynomials equals the product of
homogenizations, we obtain the following two trivariate forms:
\begin{equation}\label{eq:V^2.homoegenized}
y^{2d_1}\tilde{V}^2(\textstyle{\frac{x}{y}}),
\end{equation}
\begin{equation}\label{eq:-2V.Vdot.homogenized}
-2y^{d_1}y^{d_2}\tilde{V}(\textstyle{\frac{x}{y}})\dot{\tilde{V}}(\textstyle{\frac{x}{y}}).
\end{equation}
Since by assumption the highest order term of $V$ has no zeros,
the form in (\ref{eq:V^2.homoegenized}) is positive definite . The
form in (\ref{eq:-2V.Vdot.homogenized}), however, is only positive
semidefinite. In particular, since $\dot{\tilde{V}}=\dot{V}$ has
to vanish at the origin, the form in
(\ref{eq:-2V.Vdot.homogenized}) has a zero at the point
$(x_1,x_2,y)=(0,0,1)$. Nevertheless, since
Theorem~\ref{thm:claus.3vars} allows for positive semidefiniteness
of one of the two forms, by applying it to the forms in
(\ref{eq:V^2.homoegenized}) and (\ref{eq:-2V.Vdot.homogenized}),
we conclude that there exists an integer $k$ such that
\begin{equation}\label{eq:-2V.Vdot.homog.*.V^2.homog^k}
-2y^{d_1(2k+1)}y^{d_2}\tilde{V}(\textstyle{\frac{x}{y}})\dot{\tilde{V}}(\textstyle{\frac{x}{y}})\tilde{V}^{2k}(\textstyle{\frac{x}{y}})
\end{equation}
is sos. Let $W=\tilde{V}^{2k+2}.$ Then, $W$ is clearly sos.
Moreover,
$$-\dot{W}=-(2k+2)\tilde{V}^{2k+1}\dot{\tilde{V}}=-(k+1)2\tilde{V}^{2k}\tilde{V}\dot{\tilde{V}}$$
is also sos because this polynomial is obtained from
(\ref{eq:-2V.Vdot.homog.*.V^2.homog^k}) by setting
$y=1$.\footnote{Once again, we note that the function $W$
constructed in this proof is radially unbounded, achieves its
global minimum at the origin, and has $-\dot{W}$ positive
definite. Therefore, $W$ proves global asymptotic stability.}
\end{proof}

\section{Existence of sos Lyapunov functions for switched linear
systems}\label{sec:extension.to.switched.sys} The result of
Theorem~\ref{thm:poly.lyap.then.sos.lyap} extends in a
straightforward manner to Lyapunov analysis of switched systems.
In particular, we are interested in the highly-studied problem of
stability analysis of arbitrary switched linear systems:
\begin{equation}\label{eq:switched.linear.system}
\dot{x}=A_i x, \quad i\in\{1,\ldots,m\},
\end{equation}
$A_i\in\mathbb{R}^{n\times n}$. We assume the minimum dwell time
of the system is bounded away from zero. This guarantees that the
solutions of (\ref{eq:switched.linear.system}) are well-defined.
Existence of a common Lyapunov function is necessary and
sufficient for (global) asymptotic stability under arbitrary
switching (ASUAS) of system (\ref{eq:switched.linear.system}). The
ASUAS of system (\ref{eq:switched.linear.system}) is equivalent to
asymptotic stability of the linear differential inclusion
\begin{equation}\nonumber
\dot{x}\in co\{A_i\}x, \quad i\in\{1,\ldots,m\},
\end{equation}
where $co$ here denotes the convex hull.
%This is obvious in view
%of the fact that a common Lyapunov function decreasing with
%respect to the matrices $A_i$ would also decrease with respect to
%any matrix in the convex hull.
It is also known that ASUAS of (\ref{eq:switched.linear.system})
is equivalent to exponential stability under arbitrary
switching~\cite{Angeli.homog.switched}. A common approach for
analyzing the stability of these systems is to use the sos
technique to search for a common polynomial Lyapunov
function~\cite{PraP03},\cite{Chest.et.al.sos.robust.stability}. We
will prove the following result.

\begin{theorem}\label{thm:converse.sos.switched.sys}
The switched linear system in (\ref{eq:switched.linear.system}) is
asymptotically stable under arbitrary switching if and only if
there exists a common homogeneous polynomial Lyapunov function $W$
such that
\begin{equation}\nonumber
\begin{array}{rl}
W & \mbox{sos} \\
-\dot{W}_i=-\langle \nabla W(x), A_ix\rangle & \mbox{sos},
\end{array}
\end{equation}
for $i=1,\ldots, m$, where the polynomials $W$ and $-\dot{W}_i$
are all positive definite.
\end{theorem}
To prove this result, we will use the following theorem of Mason
et al.

\begin{theorem}[Mason et al.,~\cite{switch.common.poly.Lyap}] \label{thm:switch.poly.exists} If the switched linear system
in (\ref{eq:switched.linear.system}) is asymptotically stable
under arbitrary switching, then there exists a common homogeneous
polynomial Lyapunov function $V$ such that
\begin{equation}\nonumber
\begin{array}{rll}
V &>&0 \ \ \forall x\neq 0  \\
-\dot{V}_i(x)=-\langle \nabla V(x), A_ix\rangle &>&0 \ \  \forall
x\neq 0,
\end{array}
\end{equation}
for $i=1,\ldots,m$.
\end{theorem}

The next proposition is an extension of
Theorem~\ref{thm:poly.lyap.then.sos.lyap} to switched systems (not
necessarily linear).

\begin{proposition}\label{prop:switch.poly.then.sos.poly}
Consider an arbitrary switched dynamical system
\begin{equation}\nonumber %\label{eq:switched.polynomial.system}
\dot{x}=f_i(x), \quad i\in\{1,\ldots,m\},
\end{equation}
where $f_i(x)$ is a homogeneous polynomial vector field of degree
$d_i$ (the degrees of the different vector fields can be
different). Suppose there exists a common positive definite
homogeneous polynomial Lyapunov function $V$ such that
$$-\dot{V}_i(x)=-\langle \nabla V(x), f_i(x)\rangle$$
is positive definite for all $i\in\{1,\ldots, m\}$. Then there
exists a common homogeneous polynomial Lyapunov function $W$ such
that $W$ is sos and the polynomials $$-\dot{W}_i=-\langle \nabla
W(x), f_i(x)\rangle,$$ for all $i\in\{1,\ldots, m\}$, are also
sos.
\end{proposition}

\begin{proof}
Observe that for each $i$, the polynomials $V^2$ and
$-2V\dot{V}_i$ are both positive definite and homogeneous.
Applying Theorem~\ref{thm:claus} $m$ times to these pairs of
polynomials, we conclude the existence of positive integers $k_i$
such that
\begin{equation}\label{eq:-2VVdotV^2V^k_i.sos}
(-2V\dot{V}_i)(V^2)^{k_i} \  \mbox{is sos,}
\end{equation}
for $i=1,\ldots,m$. Let $$k=\max\{k_1,\ldots,k_m\},$$ and let
$$W=V^{2k+2}.$$ Then, $W$ is clearly sos. Moreover, for each $i$,
the polynomial
\begin{equation}\nonumber
\begin{array}{rll}
-\dot{W}_i &=&-(2k+2)V^{2k+1}\dot{V}_i \\
\ &=&-(k+1)2V\dot{V}_iV^{2k_i}V^{2(k-k_i)}
\end{array}
\end{equation}
is sos since $(-2V\dot{V}_i)(V^{2k_i})$ is sos by
(\ref{eq:-2VVdotV^2V^k_i.sos}), $V^{2(k-k_i)}$ is sos as an even
power, and products of sos polynomials are sos.
\end{proof}
The proof of Theorem~\ref{thm:converse.sos.switched.sys} now
simply follows from Theorem~\ref{thm:switch.poly.exists} and
Proposition~\ref{prop:switch.poly.then.sos.poly} in the special
case where $d_i=1$ for all $i$.

Analysis of switched linear systems is also of great interest to
us in discrete time. In fact, the subject of the next chapter will
be on the study of systems of the type
\begin{equation}\label{eq:switched.linear.system.in.DT}
x_{k+1}=A_i x_k, \quad i\in\{1,\ldots,m\},
\end{equation}
where at each time step the update rule can be given by any of the
$m$ matrices $A_i$. The analogue of
Theorem~\ref{thm:converse.sos.switched.sys} for these systems has
already been proven by Parrilo and Jadbabaie
in~\cite{Pablo_Jadbabaie_JSR_journal}. It is shown that if
(\ref{eq:switched.linear.system.in.DT}) is asymptotically stable
under arbitrary switching, then there exists a homogeneous
polynomial Lyapunov function $W$ such that
\begin{equation}\nonumber
\begin{array}{rl}
W(x) & \mbox{sos} \\
W(x)-W(A_ix) & \mbox{sos},
\end{array}
\end{equation}
for $i=1,\ldots,m$. We will end this section by proving two
related propositions of a slightly different flavor. It will be
shown that for switched linear systems, both in discrete time and
in continuous time, the sos condition on the Lyapunov function
itself is never conservative, in the sense that if one of the
``decrease inequalities'' is sos, then the Lyapunov function is
automatically sos. These propositions are really statements about
linear systems, so we will present them that way. However, since
stable linear systems always admit quadratic Lyapunov functions,
the propositions are only interesting in the context where a
common polynomial Lyapunov function for a switched linear system
is seeked.

\begin{proposition}\label{prop:switch.DT.V.automa.sos}
Consider the linear dynamical system $x_{k+1}=Ax_k$ in discrete
time. Suppose there exists a positive definite polynomial Lyapunov
function $V$ such that $V(x)-V(Ax)$ is positive definite and sos.
Then, $V$ is sos.
\end{proposition}

\begin{proof}
Consider the polynomial $V(x)-V(Ax)$ that is sos by assumption. If
we replace $x$ by $Ax$ in this polynomial, we conclude that the
polynomial $V(Ax)-V(A^2 x)$ is also sos. Hence, by adding these
two sos polynomials, we get that $V(x)-V(A^2x)$ is sos. This
procedure can obviously be repeated to infer that for any integer
$k\geq 1$, the polynomial
\begin{equation}\label{eq:V-V(A^k)}
V(x)-V(A^kx)
\end{equation}
is sos. Since by assumption $V$ and $V(x)-V(Ax)$ are positive
definite, the linear system must be GAS, and hence $A^k$ converges
to the zero matrix as $k\rightarrow\infty$. Observe that for all
$k$, the polynomials in (\ref{eq:V-V(A^k)}) have degree equal to
the degree of $V$, and that the coefficients of $V(x)-V(A^kx)$
converge to the coefficients of $V$ as $k\rightarrow\infty$. Since
for a fixed degree and dimension the cone of sos polynomials is
closed~\cite{RobinsonSOS}, it follows that $V$ is sos.
\end{proof}

Similarly, in continuous time, we have the following proposition.
\begin{proposition}\label{prop:switch.CT.V.automa.sos}
Consider the linear dynamical system $\dot{x}=Ax$ in continuous
time. Suppose there exists a positive definite polynomial Lyapunov
function $V$ such that $-\dot{V}=-\langle \nabla V(x), Ax \rangle$
is positive definite and sos. Then, $V$ is sos.
\end{proposition}
\begin{proof}
The value of the polynomial $V$ along the trajectories of the
dynamical system satisfies the relation
$$V(x(t))=V(x(0))+\int_o^t \dot{V}(x(\tau)) d_\tau.$$ Since the
assumptions imply that the system is GAS, $V(x(t))\rightarrow 0$
as $t$ goes to infinity. (Here, we are assuming, without loss of
generality, that $V$ vanishes at the origin.) By evaluating the
above equation at $t=\infty$, rearranging terms, and substituting
$e^{A\tau}x$ for the solution of the linear system at time $\tau$
starting at initial condition $x$, we obtain
$$V(x)=\int_0^\infty -\dot{V}(e^{A\tau}x) d_\tau.$$ By assumption, $-\dot{V}$ is sos and therefore for any value of $\tau$, the integrand
$-\dot{V}(e^{A\tau}x)$ is an sos polynomial. Since converging
integrals of sos polynomials are sos, it follows that $V$ is sos.
\end{proof}

\begin{remark}
The previous proposition does not hold if the system is not
linear. For example, consider any positive form $V$ that is not a
sum of squares and define a dynamical system by $\dot{x}=-\nabla
V(x)$. In this case, both $V$ and $-\dot{V}=||\nabla V(x)||^2$ are
positive definite and $-\dot{V}$ is sos, though $V$ is not sos.
\end{remark}

\section{Some open questions}\label{sec:summary.future.work}
Some open questions related to the problems studied in this
chapter are the following. Regarding complexity, of course the
interesting problem is to formally answer the questions of Arnold
on undecidability of determining stability for polynomial vector
fields. Regarding existence of polynomial Lyapunov functions, Mark
Tobenkin asked whether a globally exponentially stable polynomial
vector field admits a polynomial Lyapunov function. Our
counterexample in Section~\ref{sec:no.poly.Lyap}, though GAS and
locally exponentially stable, is not globally exponentially stable
because of exponential growth rates in the large. The
counterexample of Bacciotti and Rosier
in~\cite{Bacciotti.Rosier.Liapunov.Book} is not even locally
exponentially stable. Another future direction is to prove that
GAS homogeneous polynomial vector fields admit homogeneous
polynomial Lyapunov functions. This, together with
Theorem~\ref{thm:poly.lyap.then.sos.lyap}, would imply that
asymptotic stability of homogeneous polynomial systems can always
be decided via sum of squares programming. Also, it is not clear
to us whether the assumption of homogeneity and planarity can be
removed from Theorems~\ref{thm:poly.lyap.then.sos.lyap}
and~\ref{thm:poly.lyap.then.sos.lyap.PLANAR} on existence of sos
Lyapunov functions. Finally, another research direction would be
to obtain upper bounds on the degree of polynomial or sos
polynomial Lyapunov functions. Some degree bounds are known for
Lyapunov analysis of locally exponentially stable
systems~\cite{Peet.Antonis.converse.sos.journal}, but they depend
on uncomputable properties of the solution such as convergence
rate. Degree bounds on Positivstellensatz result of the type in
Theorems~\ref{thm:claus} and~\ref{thm:claus.3vars} are known, but
typically exponential in size and not very encouraging for
practical purposes.

\chapter{Joint Spectral Radius and Path-Complete Graph Lyapunov
Functions}\label{chap:jsr}
\newcommand{\rmj}[1]{{#1}}

%\newcommand{\comrj}[1]{\begin{tt}[RJ: #1]\end{tt}}
%%Black and white version
%\newcommand{\rmj}[1]{{#1}}

%%%%%%

%%%First journal revision by AAA:
%\long\def\aaa#1{{\color{blue}#1}}
%%Black and white version
\long\def\aaa#1{{#1}}

In this chapter, we introduce the framework of path-complete graph
Lyapunov functions for analysis of switched systems. The
methodology is presented in the context of approximation of the
joint spectral radius. The content of this chapter is based on an
extended version of the work in~\cite{HSCC_JSR_Path_complete}.

\section{Introduction}
Given a finite set of square matrices
$\mathcal{A}\mathrel{\mathop:}=\left\{ A_{1},...,A_{m}\right\} $,
their \emph{joint spectral radius} $\rho(\mathcal{A})$ is defined
as
\begin{equation} \rho\left(\mathcal{A}\right)
=\lim_{k\rightarrow\infty}\max_{\sigma
\in\left\{  1,...,m\right\}  ^{k}}\left\Vert A_{\sigma_{k}}...A_{\sigma_{2}%
}A_{\sigma_{1}}\right\Vert ^{1/k},\label{eq:def.jsr}%
\end{equation}
where the quantity $\rho(\mathcal{A})$ is independent of the norm
used in (\ref{eq:def.jsr}). The joint spectral radius (JSR) is a
natural generalization of the spectral radius of a single square
matrix and it characterizes the maximal growth rate that can be
obtained by taking products, of arbitrary length, of all possible
permutations of $A_{1},...,A_{m}$. This concept was introduced by
Rota and Strang~\cite{RoSt60} in the early 60s and has since been
the subject of extensive research within the engineering and the
mathematics communities alike. Aside from a wealth of fascinating
mathematical questions that arise from the JSR, the notion emerges
in many areas of application such as stability of switched linear
dynamical systems, computation of the capacity of codes,
continuity of wavelet functions, convergence of consensus
algorithms, trackability of graphs, and many others.
See~\cite{Raphael_Book} and references therein for a recent survey
of the theory and applications of the JSR.

Motivated by the abundance of applications, there has been much
work on efficient computation of the joint spectral radius; see
e.g.~\cite{BlNT04},~\cite{BlNes05},~\cite{Pablo_Jadbabaie_JSR_journal},
and references therein. Unfortunately, the negative results in the
literature certainly restrict the horizon of possibilities.
In~\cite{BlTi2}, Blondel and Tsitsiklis prove that even when the
set $\mathcal{A}$ consists of only two matrices, the question of
testing whether $\rho(\mathcal{A})\leq1$ is undecidable. They also
show that unless P=NP, one cannot compute an approximation
$\hat{\rho}$ of $\rho$ that satisfies
$|\hat{\rho}-\rho|\leq\epsilon\rho$, in a number of steps
polynomial in the bit size of $\mathcal{A}$ and the bit size of
$\epsilon$~\cite{BlTi3}. It is not difficult to show that the
spectral radius of any finite product of length $k$ raised to the
power of $1/k$ gives a lower bound on $\rho$~\cite{Raphael_Book}.
However, for reasons that we explain next, our focus will be on
computing upper bounds for $\rho$.

There is an attractive connection between the joint spectral
radius and the stability properties of an arbitrary switched
linear system; i.e., dynamical systems of the form
\begin{equation}\label{eq:switched.linear.sys}
x_{k+1}=A_{\sigma\left(  k\right)  }x_{k},
\end{equation}
where $\sigma:\mathbb{Z\rightarrow}\left\{  1,...,m\right\}$ is a
map from the set of integers to the set of indices. It is
well-known that $\rho<1$ if and only if system
(\ref{eq:switched.linear.sys}) is \emph{absolutely asymptotically
stable} (AAS), that is, (globally) asymptotically stable for all
switching sequences. Moreover, it is
known~\cite{switched_system_survey} that absolute asymptotic
stability of (\ref{eq:switched.linear.sys}) is equivalent to
absolute asymptotic stability of the linear difference inclusion
\begin{equation}\label{eq:linear.difference.inclusion}
x_{k+1}\in \mbox{co}{\mathcal{A}}\ x_{k},
\end{equation}
where $\mbox{co}{\mathcal{A}}$ here denotes the convex hull of the
set $\mathcal{A}$. Therefore, any method for obtaining upper
bounds on the joint spectral radius provides sufficient conditions
for stability of systems of type (\ref{eq:switched.linear.sys}) or
(\ref{eq:linear.difference.inclusion}). Conversely, if we can
prove absolute asymptotic stability of
(\ref{eq:switched.linear.sys}) or
(\ref{eq:linear.difference.inclusion}) for the set
$\mathcal{A}_\gamma\mathrel{\mathop:}=\{ \gamma
A_{1},\ldots,\gamma A_{m}\}$ for some positive scalar $\gamma$,
then we get an upper bound of $\frac{1}{\gamma}$ on
$\rho(\mathcal{A})$. (This follows from the scaling property of
the JSR: $\rho(\mathcal{A}_\gamma)=\gamma\rho(\mathcal{A})$.) One
advantage of working with the notion of the joint spectral radius
is that it gives a way of rigorously quantifying the performance
guarantee of different techniques for stability analysis of
systems (\ref{eq:switched.linear.sys}) or
(\ref{eq:linear.difference.inclusion}).

Perhaps the most well-established technique for proving stability
of switched systems is the use of a \emph{common (or simultaneous)
Lyapunov function}. The idea here is that if there is a
continuous, positive, and homogeneous (Lyapunov) function
$V(x):\mathbb{R}^n\rightarrow\mathbb{R}$ that for some $\gamma>1$
satisfies
\begin{equation}
V(\gamma A_ix)\leq V(x) \quad \forall i=1,\ldots,m,\ \ \forall
x\in\mathbb{R}^n,
\end{equation}
(i.e., $V(x)$ decreases no matter which matrix is applied), then
the system in (\ref{eq:switched.linear.sys}) (or in
(\ref{eq:linear.difference.inclusion})) is AAS. Conversely, it is
known that if the system is AAS, then there exists a \emph{convex}
common Lyapunov function (in fact a norm); see e.g.~\cite[p.
24]{Raphael_Book}. However, this function is not in general
finitely constructable. A popular approach has been to try to
approximate this function by a class of functions that we can
efficiently search for using convex optimization and in particular
semidefinite programming. As we mentioned in our introductory
chapters, semidefinite programs (SDPs) can be solved with
arbitrary accuracy in polynomial time and lead to efficient
computational methods for approximation of the JSR. As an example,
if we take the Lyapunov function to be quadratic (i.e.,
$V(x)=x^TPx$), then the search for such a Lyapunov function can be
formulated as the following SDP:
\begin{equation}\label{eq:Lyap.CQ.SDP}
\begin{array}{rll}
P&\succ&0 \\
\gamma^2 A_i^TPA_i&\preceq&P \quad \forall i=1,\ldots,m.
\end{array}
\end{equation}

The quality of approximation of common quadratic Lyapunov
functions is a well-studied topic. In particular, it is
known~\cite{BlNT04} that the estimate
$\hat{\rho}_{\mathcal{V}^{2}}$ obtained by this
method\footnote{The estimate $\hat{\rho}_{\mathcal{V}^{2}}$ is the
reciprocal of the largest $\gamma$ that satisfies
(\ref{eq:Lyap.CQ.SDP}) and can be found by bisection.} satisfies
\begin{equation}\label{eq:CQ.bound}
\frac{1}{\sqrt{n}}\hat{\rho}_{\mathcal{V}^{2}}(\mathcal{A})\leq\rho(\mathcal{A})\leq\hat{\rho}_{\mathcal{V}^{2}}(\mathcal{A}),
\end{equation}
where $n$ is the dimension of the matrices. This bound is a direct
consequence of John's ellipsoid theorem and is known to be
tight~\cite{Ando98}.

In~\cite{Pablo_Jadbabaie_JSR_journal}, the use of sum of squares
(sos) polynomial Lyapunov functions of degree $2d$ was proposed as
a common Lyapunov function for the switched system in
(\ref{eq:switched.linear.sys}). As we know, the search for such a
Lyapunov function can again be formulated as a semidefinite
program. This method does considerably better than a common
quadratic Lyapunov function in practice and its estimate
$\hat{\rho}_{\mathcal{V}^{SOS,2d}}$ satisfies the bound
\begin{equation}\label{eq:SOS.bound}
\frac{1}{\sqrt[2d]{\eta}}\hat{\rho}_{\mathcal{V}^{SOS,2d}}(\mathcal{A})\leq\rho(\mathcal{A})\leq\hat{\rho}_{\mathcal{V}^{SOS,2d}}(\mathcal{A}),
\end{equation}
where $\eta=\min\{m,{n+d-1\choose d}\}$. Furthermore, as the
degree $2d$ goes to infinity, the estimate
$\hat{\rho}_{\mathcal{V}^{SOS,2d}}$ converges to the true value of
$\rho$~\cite{Pablo_Jadbabaie_JSR_journal}. The semidefinite
programming based methods for approximation of the JSR have been
recently generalized and put in the framework of conic
programming~\cite{protasov-jungers-blondel09}.

\subsection{Contributions and organization of this chapter}
It is natural to ask whether one can develop better approximation
schemes for the joint spectral radius by using multiple Lyapunov
functions as opposed to requiring simultaneous contractibility of
a single Lyapunov function with respect to all the matrices. More
concretely, our goal is to understand how we can write
inequalities among, say, $k$ different Lyapunov functions
$V_1(x),\ldots,V_k(x)$ that imply absolute asymptotic stability of
(\ref{eq:switched.linear.sys}) and can be checked via semidefinite
programming.

The general idea of using several Lyapunov functions for analysis
of switched systems is a very natural one and has already appeared
in the literature (although to our knowledge not in the context of
the approximation of the JSR); see e.g.
\cite{JohRan_PWQ},~\cite{multiple_lyap_Branicky},
\cite{composite_Lyap}, \cite{composite_Lyap2},
\cite{convex_conjugate_Lyap}. Perhaps one of the earliest
references is the work on ``piecewise quadratic Lyapunov
functions'' in~\cite{JohRan_PWQ}. However, this work is in the
different framework of state dependent switching, where the
dynamics switches depending on which region of the space the
trajectory is traversing (as opposed to arbitrary switching). In
this setting, there is a natural way of using several Lyapunov
functions: assign one Lyapunov function per region and ``glue them
together''. Closer to our setting, there is a body of work in the
literature that gives sufficient conditions for existence of
piecewise Lyapunov functions of the type $\max\{x^TP_1x,\ldots,
x^TP_kx\}$, $\min\{x^TP_1x,\ldots,x^TP_kx\}$, and
$\mbox{conv}\{x^TP_1x,\ldots, x^TP_kx\}$, i.e, the pointwise
maximum, the pointwise minimum, and the convex envelope of a set
of quadratic functions \cite{composite_Lyap},
\cite{composite_Lyap2}, \cite{convex_conjugate_Lyap},
\cite{hu-ma-lin}. These works are mostly concerned with analysis
of linear differential inclusions in continuous time, but they
have obvious discrete time counterparts. The main drawback of
these methods is that in their greatest generality, they involve
solving bilinear matrix inequalities, which are non-convex and in
general NP-hard. One therefore has to turn to heuristics, which
have no performance guarantees and their computation time quickly
becomes prohibitive when the dimension of the system increases.
Moreover, all of these methods solely provide sufficient
conditions for stability with no performance guarantees.

There are several unanswered questions that in our view deserve a
more thorough study: (i) With a focus on conditions that are
amenable to convex optimization, what are the different ways to
write a set of inequalities among $k$ Lyapunov functions that
imply absolute asymptotic stability of
(\ref{eq:switched.linear.sys})? Can we give a unifying framework
that includes the previously proposed Lyapunov functions and
perhaps also introduces new ones? (ii) Among the different sets of
inequalities that imply stability, can we identify some that are
less conservative than some other? (iii) The available methods on
piecewise Lyapunov functions solely provide sufficient conditions
for stability with no guarantee on their performance. Can we give
converse theorems that guarantee the existence of a feasible
solution to our search for a given accuracy?

\aaa{The contributions of this chapter to these questions are as
follows.} We propose a unifying framework based on a
representation of Lyapunov inequalities with labeled graphs and by
making some connections with basic concepts in automata theory.
This is done in Section~\ref{sec:graphs.jsr}, where we define the
notion of a path-complete graph
(Definition~\ref{def:path-complete}) and prove that any such graph
provides an approximation scheme for the JSR
(Theorem~\ref{thm:path.complete.implies.stability}). In
Section~\ref{sec:duality.and.some.families.of.path.complete}, we
give examples of families of path-complete graphs and show that
many of the previously proposed techniques come from particular
classes of simple path-complete graphs (e.g.,
Corollary~\ref{cor:min.of.quadratics},
Corollary~\ref{cor:max.of.quadratics}, \aaa{and
Remark~\ref{rmk:Lee-Dellerud_Daafouz}}). In
Section~\ref{sec:who.beats.who}, we characterize all the
path-complete graphs with two nodes for the analysis of the JSR of
two matrices.  We determine how the approximations obtained from
all of these graphs compare
(Proposition~\ref{prop:who.beats.who}). In Section~\ref{sec:hscc},
we study in more depth the approximation properties of a
particular pair of ``dual'' path-complete graphs that seem to
perform very well in practice.
Subsection~\ref{subsec:duality.and.transposition} contains more
general results about duality within path-complete graphs and its
connection to transposition of matrices
(Theorem~\ref{thm:transpose.bound.dual.bound}).
Subsection~\ref{subsec:HSCC.bound} gives an approximation
guarantee for the graphs studied in Section~\ref{sec:hscc}
(Theorem~\ref{thm:HSCC.bound}), and
Subsection~\ref{subsec:numerical.examples} contains some numerical
examples. In Section~\ref{sec:converse.thms}, we prove a converse
theorem for the method of max-of-quadratics Lyapunov functions
(Theorem~\ref{thm:converse.max.of.quadratics}) and an
approximation guarantee for a new class of methods for proving
stability of switched systems (Theorem~\ref{thm-bound-codes}).
Finally, some concluding remarks and future directions are
presented in Section~\ref{sec:conclusions.future.directions}.

\section{Path-complete graphs and the joint spectral
radius}\label{sec:graphs.jsr}

In what follows, we will think of the set of matrices
$\mathcal{A}\mathrel{\mathop:}=\left\{ A_{1},...,A_{m}\right\} $
as a finite alphabet and we will often refer to a finite product
of matrices from this set as a \emph{word}. We denote the set of
all words ${A_i}_t\ldots{A_i}_1$ of length $t$ by $\mathcal{A}^t$.
Contrary to the standard convention in automata theory, our
convention is to read a word from right to left. This is in
accordance with the order of matrix multiplication. The set of all
finite words is denoted by $\mathcal{A}^*$; i.e.,
$\mathcal{A}^*=\bigcup\limits_{t\in\mathbb{Z}^+} \mathcal{A}^t$.

The basic idea behind our framework is to represent through a
graph all the possible occurrences of products that can appear in
a run of the dynamical system in (\ref{eq:switched.linear.sys}),
and assert via some Lyapunov inequalities that no matter what
occurrence appears, the product must remain stable. A convenient
way of representing these Lyapunov inequalities is via a directed
labeled graph $G(N, E)$. Each node of this graph is associated
with a (continuous, positive definite, and homogeneous) Lyapunov
function $V_i(x):\mathbb{R}^n\rightarrow\mathbb{R}$, and each edge
is labeled by a finite product of matrices, i.e., by a word from
the set $\mathcal{A}^*$. As illustrated in
Figure~\ref{fig:node.arc}, given two nodes with Lyapunov functions
$V_i(x)$ and $V_j(x)$ and an edge going from node $i$ to node $j$
labeled with the matrix $A_l$, we write the Lyapunov inequality:
\begin{equation}\label{eq:lyap.inequality.rule}
V_j(A_lx)\leq V_i(x) \quad \forall x\in\mathbb{R}^n.
\end{equation}

\begin{figure}[h]
\centering \scalebox{0.25} {\includegraphics{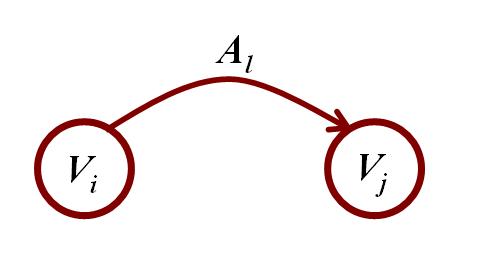}}  %AAA: I think cannot use JPG to creat a ps file...
\caption{Graphical representation of Lyapunov inequalities. The
edge in the graph above corresponds to the Lyapunov inequality
$V_j(A_lx)\leq V_i(x)$. Here, $A_l$ can be a single matrix from
$\mathcal{A}$ or a finite product of matrices from $\mathcal{A}$.}
\label{fig:node.arc}
\end{figure}

The problem that we are interested in is to understand which sets
of Lyapunov inequalities imply stability of the switched system in
(\ref{eq:switched.linear.sys}). We will answer this question based
on the corresponding graph.

For reasons that will become clear shortly, we would like to
reduce graphs whose edges have arbitrary labels from the set
$\mathcal{A}^*$ to graphs whose edges have labels from the set
$\mathcal{A}$, i.e, labels of length one. This is explained next.

\begin{definition}\label{def:expanded.graph}
Given a labeled directed graph $G(N,E)$, we define its
\emph{expanded graph} $G^e(N^e,E^e)$ as the outcome of the
following procedure. For every edge $(i,j)\in E$ with label
${A_i}_k\ldots{A_i}_1\in\mathcal{A}^k$, where $k>1$, we remove the
edge $(i,j)$ and replace it with $k$ new edges $(s_q,s_{q+1})\in
E^e\setminus E:\ q\in\{0,\ldots,k-1\}$, where $s_0=i$ and
$s_k=j$.\footnote{It is understood that the node index $s_q$
depends on the original nodes $i$ and $j$. To keep the notation
simple we write $s_q$ instead of $s_{q}^{ij}$.} (These new edges
go from node $i$ through $k-1$ newly added nodes
$s_1,\ldots,s_{k-1}$ and then to node $j$.) We then label the new
edges $(i,s_1),\ldots,(s_q,s_{q+1}),\ldots,(s_{k-1},j)$ with
${A_i}_1,\ldots,{A_i}_k$ respectively.
\end{definition}
\begin{figure}[h]
\centering \scalebox{0.3}
{\includegraphics{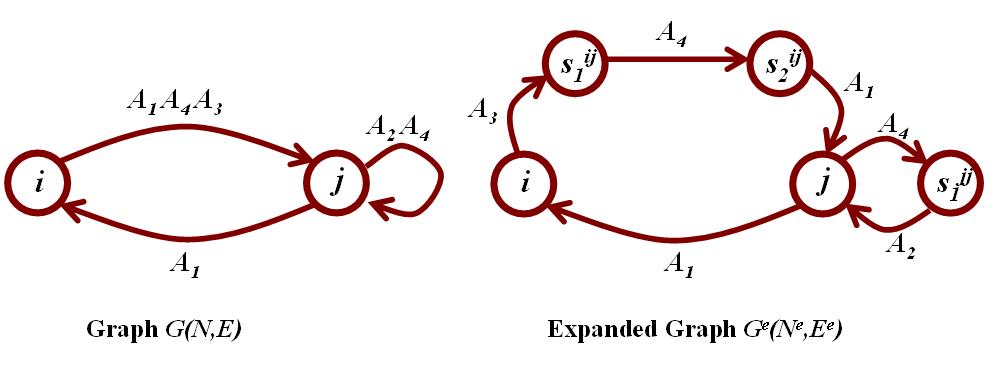}} \caption{Graph
expansion: edges with labels of length more than one are broken
into new edges with labels of length one.}
\label{fig:graph.expansion}
\end{figure}

An example of a graph and its expansion is given in
Figure~\ref{fig:graph.expansion}. Note that if a graph has only
labels of length one, then its expanded graph equals itself. The
next definition is central to our development.
\begin{definition}\label{def:path-complete}
Given a directed graph $G(N,E)$ whose edges are labeled with words
from the set $\mathcal{A}^*$, we say that the graph is
\emph{path-complete}, if for all finite words $A_{\sigma_k}\ldots
A_{\sigma_1}$ of any length $k$ (i.e., for all words in
$\mathcal{A}^*$), there is a directed path in its expanded graph
$G^e(N^e,E^e)$ such that the labels \aaa{on} the edges \aaa{of}
this path are \rmj{the} labels $A_{\sigma_1}$ up to
$A_{\sigma_k}$.
\end{definition}

In Figure~\ref{fig:jsr.graphs}, we present seven path-complete
graphs on the alphabet $\mathcal{A}=\{A_1,A_2\}$. The fact that
these graphs are path-complete is easy to see for graphs $H_1,
H_2, G_3,$ and $G_4$, but perhaps not so obvious for graphs $H_3,
G_1,$ and $G_2$. One way to check if a graph is path-complete is
to think of it as a finite automaton by introducing an auxiliary
start node (state) with free transitions to every node and by
making all the other nodes be accepting states. Then, there are
well-known algorithms (see e.g.~\cite[Chap.
4]{Hopcroft_Motwani_Ullman_automata_Book}) that check whether the
language accepted by an automaton is $\mathcal{A}^*$, which is
equivalent to the graph being path-complete. At least for the
cases where the automata are deterministic (i.e., when all
outgoing edges from any node have different labels), these
algorithms are very efficient and have running time of only
$O(|N|^2)$. Similar algorithms exist in the symbolic dynamics
literature; see e.g.~\cite[Chap. 3]{Lind_Marcus_symbolic_Book}.
Our interest in path-complete graphs stems from the
Theorem~\ref{thm:path.complete.implies.stability} below that
establishes that any such graph gives a method for approximation
of the JSR. We introduce one last definition before we state this
theorem.
%
%We are not so concerned with running times of algorithms for
%recognizing path-complete graphs since our interest in these
%graphs solely stems from their potential for stability analysis of
%arbitrary switched systems. As the following theorem illustrates,
%once we know a graph is path-complete, we can fix it and then
%build an approximation scheme for the JSR based on it.
%

\begin{figure}[h]
\centering \scalebox{0.4}
{\includegraphics{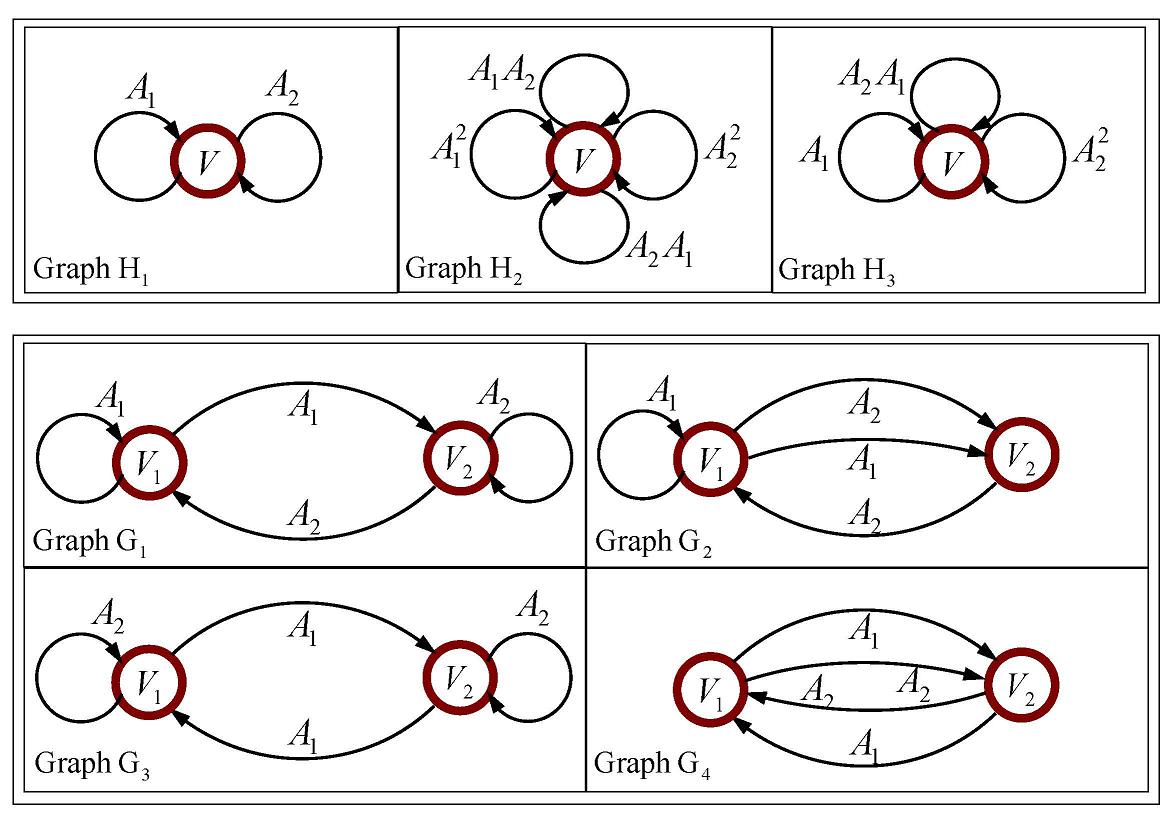}} \caption{Examples
of path-complete graphs for the alphabet $\{A_1,A_2\}$. If
Lyapunov functions satisfying the inequalities associated with any
of these graphs are found, then we get an upper bound of unity on
$\rho(A_1,A_2)$.} \label{fig:jsr.graphs}
\end{figure}

\begin{definition}
Let $\mathcal{A}=\left\{ A_{1},\ldots,A_{m}\right\}  $ be a set of
matrices. Given a path-complete graph $G\left( N,E\right) $ and
$|N|$ functions $V_i(x)$, we say that
$\{V_{i}(x)|i=1,\ldots,\left\vert N\right\vert \}$ is a
\emph{graph Lyapunov function (GLF) associated with $G\left(
N,E\right)  $} if
\[
V_{j}\left(  L\left( (i,j)  \right)  x\right)  \leq V_{i}\left(
x\right)  \text{\qquad}\forall x\in \mathbb{R}^n,\ \ \forall\
(i,j)\in E,
\]
where $L\left(  (i,j)  \right) \in\mathcal{A}^{\ast}$ is the label
associated with edge $(i,j)  \in E$ going from node $i$ to node
$j$.
\end{definition}

\begin{theorem}\label{thm:path.complete.implies.stability}
%Consider a finite set of matrices
%$\mathcal{A}=\{A_1,\\\ldots,A_m\}$. Let $G(N,E)$ be a
%path-complete graph whose edges are labeled with words from
%$\mathcal{A}^*$. If there exist positive, continuous, and
%homogeneous\footnote{The requirement of homogeneity can be
%replaced by the requirement of radially unboundedness which is
%implied by homogeneity and positivity. However, since the
%dynamical system in (\ref{eq:switched.linear.sys}) is homogeneous
%in space, there is no conservatism in asking $V_i(x)$ to be
%homogeneous~\cite{HomogHomog}.} Lyapunov functions $V_i(x)$, one
%per node of the graph, that for some scalar $\gamma>0$ satisfy the
%Lyapunov inequalities associated to the edges of $G$ (cf.
%Figure~\ref{fig:node.arc}) with scaled matrices $\gamma
%A_1,\ldots,\gamma A_m$, then
%$\rho(\mathcal{A})\leq\frac{1}{\gamma}$.
Consider a finite set of matrices
$\mathcal{A}=\{A_1,\ldots,A_m\}$. For a scalar $\gamma>0$, let
$\mathcal{A}_\gamma\mathrel{\mathop:}=\{ \gamma
A_{1},\ldots,\gamma A_{m}\}$. Let $G(N,E)$ be a path-complete
graph whose edges are labeled with words from
$\mathcal{A}_\gamma^*$. If there exist positive, continuous, and
homogeneous\footnote{The requirement of homogeneity can be
replaced by radial unboundedness which is implied by homogeneity
and positivity. However, since the dynamical system in
(\ref{eq:switched.linear.sys}) is homogeneous, there is no
conservatism in asking $V_i(x)$ to be homogeneous.} functions
$V_i(x)$, one per node of the graph, such that
$\{V_{i}(x)~|~i=1,\ldots,\left\vert N\right\vert \}$ is a graph
Lyapunov function associated with $G(N,E)$, then
$\rho(\mathcal{A})\leq\frac{1}{\gamma}$.
\end{theorem}

\begin{proof}
We will first prove the claim for the special case where the edge
labels of $G(N,E)$ belong to $\mathcal{A}_\gamma$ and therefore
$G(N,E)=G^e(N^e,E^e)$. The general case will be reduced to this
case afterwards. Let $d$ be the degree of homogeneity of the
Lyapunov functions $V_i(x)$, i.e., $V_i(\lambda x)=\lambda^d
V_i(x)$ for all $\lambda\in\mathbb{R}$. (The actual value of $d$
is irrelevant.) By positivity, continuity, and homogeneity of
$V_i(x)$, there exist scalars $\alpha_i$ and $\beta_i$ with
$0<\alpha_i\leq\beta_i$ for $i=1,\ldots,|N|$, such that
\begin{equation}\label{eq:homog.bounds}
\alpha_i||x||^d\leq V_i(x)\leq\beta_i||x||^d,
\end{equation}
for all $x\in\mathbb{R}^n$ and for all $i=1,\ldots,|N|$, where
$||x||$ here denotes the Euclidean norm of $x$. Let
\begin{equation}\label{eq:xi}
\xi=\max_{i,j\in\{1,\ldots,|N|\}^2} \frac{\beta_i}{\alpha_j}.
\end{equation}
Now consider an arbitrary product $A_{\sigma_{k}}\ldots
A_{\sigma_{1}}$ of length $k$. Because the graph is path-complete,
there will be a directed path corresponding to this product that
consists of $k$ edges, and goes from some node $i$ to some node
$j$. If we write the chain of $k$ Lyapunov inequalities associated
with these edges (cf. Figure~\ref{fig:node.arc}), then we get
%for the scaled matrices $\gamma A_{\sigma_{1}},\ldots,\gamma
%A_{\sigma_{k}}$, then we get
\begin{equation}\nonumber
V_j(\gamma^k A_{\sigma_{k}}\ldots A_{\sigma_{1}}x)\leq V_i(x),
\end{equation}
which by homogeneity of the Lyapunov functions can be rearranged
to
\begin{equation}\label{eq:Vj/Vi.bound}
\left(\frac{V_j(A_{\sigma_{k}}\ldots
A_{\sigma_{1}}x)}{V_i(x)}\right)^{\frac{1}{d}}\leq\frac{1}{\gamma^k}.
\end{equation}
We can now bound the spectral norm of $A_{\sigma_{k}}\ldots
A_{\sigma_{1}}$ as follows:
%\begin{equation}\nonumber
%\begin{array}{lll}
\begin{eqnarray}\nonumber
||A_{\sigma_{k}}\ldots A_{\sigma_{1}}||&\leq&\max_{x}
\frac{||A_{\sigma_{k}}\ldots A_{\sigma_{1}}x||}{||x||} \\
\nonumber\
&\leq&\left(\frac{\beta_i}{\alpha_j}\right)^{\frac{1}{d}}\max_{x}\frac{V_j^{\frac{1}{d}}(A_{\sigma_{k}}\ldots
A_{\sigma_{1}}x)}{V_i^{\frac{1}{d}}(x)} \\\nonumber \
&\leq&\left(\frac{\beta_i}{\alpha_j}\right)^\frac{1}{d}\frac{1}{\gamma^k}
\\\nonumber
\ &\leq&\xi^{\frac{1}{d}}\frac{1}{\gamma^k},
\end{eqnarray}
%\end{array}
%\end{equation}
where the last three inequalities follow from
(\ref{eq:homog.bounds}), (\ref{eq:Vj/Vi.bound}), and (\ref{eq:xi})
respectively. From the definition of the JSR in
(\ref{eq:def.jsr}), after taking the $k$-th root and the limit
$k\rightarrow\infty$, we get that
$\rho(\mathcal{A})\leq\frac{1}{\gamma}$ and the claim is
established.

Now consider the case where at least one edge of $G(N,E)$ has a
label of length more than one and hence $G^e(N^e,E^e)\neq G(N,E).$
We will start with the Lyapunov functions $V_i(x)$ assigned to the
nodes of $G(N,E)$ and from them we will explicitly construct
$|N^e|$ Lyapunov functions for the nodes of $G^e(N^e,E^e)$ that
satisfy the Lyapunov inequalities associated to the edges in
$E^e$. Once this is done, in view of our preceding argument and
the fact that the edges of $G^e(N^e,E^e)$ have labels of length
one by definition, the proof will be completed.

%\texttt{The remainder of this proof and
%Remark~\ref{rmk:invertibility}\\ is from Mardavij's tex-file.}

For $j\in N^e$, let us denote the new Lyapunov functions by
$V_j^e(x)$. We give the construction for the case where
$\left\vert N^{e}\right\vert =\left\vert N\right\vert +1.$ The
result for the general case follows by iterating this simple
construction. Let $s\in N^{e}\backslash N$ be the added node in
the expanded graph, and $q,r\in N$ be such that $\left( s,q\right)
\in E^{e}$ and $\left( r,s\right)  \in E^{e}$ with $A_{sq}$ and
$A_{rs}$ as the corresponding labels respectively. Define
\begin{equation}
V_{j}^{e}\left(  x\right)  =\left\{
\begin{array}
[c]{lll}%
V_{j}\left(  x\right)  ,\text{ } & \text{if} & j\in N\medskip\\
V_{q}\left(  A_{sq}x\right)  ,\text{ } & \text{if} & j=s.
\end{array}
\right.  \label{Eqtwo}%
\end{equation}
By construction, $r$ and $q,$ and subsequently, $A_{sq}$ and
$A_{rs}$ are uniquely defined and hence, $\left\{  V_{j}^e\left(
x\right)  ~|~j\in N^{e}\right\}  $ is well defined. We only need
to show that
\begin{align}
V_{q}\left(  A_{sq}x\right)   &  \leq V_{s}^{e}\left(  x\right)
\label{thefirst}\\
V_{s}^{e}\left(  A_{rs}x\right)   &  \leq V_{r}\left(  x\right).
\label{thesecond}%
\end{align}
Inequality (\ref{thefirst}) follows trivially from (\ref{Eqtwo}).
Furthermore,
it follows from (\ref{Eqtwo}) that%
\begin{align*}
V_{s}^{e}\left(  A_{rs}x\right)   &  =V_{q}\left(  A_{sq}A_{rs}x\right)  \\
&  \leq V_{r}\left(  x\right),
\end{align*}
where the inequality follows from the fact that for $i\in N$, the
functions $V_i(x)$ satisfy the Lyapunov inequalities of the edges
of $G\left( N,E\right).$
\end{proof}

\begin{remark}\label{rmk:invertibility}
If the matrix $A_{sq}$ is not invertible, the extended function
$V_{j}^{e}(x)$ as defined in (\ref{Eqtwo}) will only be positive
semidefinite. However, since our goal is to approximate the JSR,
we will never be concerned with invertibility of the matrices in
$\mathcal{A}$. Indeed, since the JSR is continuous in the entries
of the matrices~\cite{Raphael_Book}, we can always perturb the
matrices slightly to make them invertible without changing the JSR
by much. In particular, for any $\alpha>0,$ there exist
$0<\varepsilon, \delta <\alpha$ such that
\[
\hat{A}_{sq}=\frac{A_{sq}+\delta I}{1+\varepsilon}%
\]
is invertible and (\ref{Eqtwo})$-$(\ref{thesecond}) are satisfied
with $A_{sq}=\hat{A}_{sq}.$
\end{remark}

To understand the generality of the framework of ``path-complete
graph Lyapunov funcitons'' more clearly, let us revisit the
path-complete graphs in Figure~\ref{fig:jsr.graphs} for the study
of the case where the set $\mathcal{A}=\{A_1,A_2\}$ consists of
only two matrices. For all of these graphs if our choice for the
Lyapunov functions $V(x)$ or $V_1(x)$ and $V_2(x)$ are quadratic
functions or sum of squares polynomial functions, then we can
formulate the well-established semidefinite programs that search
for these candidate Lyapunov functions.

Graph $H_1$, which is clearly the simplest possible one,
corresponds to the well-known common Lyapunov function approach.
Graph $H_2$ is a common Lyapunov function applied to all products
of length two. This graph also obviously implies
stability.\footnote{By slight abuse of terminology, we say that a
graph implies stability meaning that the associated Lyapunov
inequalities imply stability.} But graph $H_3$ tells us that if we
find a Lyapunov function that decreases whenever $A_1$, $A_2^2$,
and $A_2A_1$ are applied (but with no requirement when $A_1A_2$ is
applied), then we still get stability. This is a priori not
obvious and we believe this approach has not appeared in the
literature before. Graph $H_3$ is also an example that explains
why we needed the expansion process. Note that for the unexpanded
graph, there is no path for any word of the form $(A_1A_2)^k$ or
of the form $A_2^{2k-1}$, for any $k\in \mathbb{N}.$ However, one
can check that in the expanded graph of graph $H_3$, there is a
path for every finite word, and this in turn allows us to conclude
stability from the Lyapunov inequalities of graph $H_3$.

The remaining graphs in Figure~\ref{fig:jsr.graphs} which all have
two nodes and four edges with labels of length one have a
connection to the method of min-of-quadratics or max-of-quadratics
Lyapunov functions~\cite{composite_Lyap},~\cite{composite_Lyap2},
\cite{convex_conjugate_Lyap}, \cite{hu-ma-lin}. If Lyapunov
inequalities associated with any of these four graphs are
satisfied, then either $\min\{V_1(x),V_2(x)\}$ or
$\max\{V_1(x),V_2(x)\}$ or both serve as a common Lyapunov
function for the switched system. In the next section, we assert
these facts in a more general setting
(Corollaries~\ref{cor:min.of.quadratics}
and~\ref{cor:max.of.quadratics}) and show that these graphs in
some sense belong to ``simplest'' families of path-complete
graphs.

%Let us comment now on the graphs with two nodes and four edges,
%which each impose four Lyapunov inequalities. We can show that if
%$V_1(x)$ and $V_2(x)$ satisfy the inequalities of any of the
%graphs $G_2, G_3$, or $G_4$, then $\max\{V_1(x),V_2(x)\}$ is a
%common Lyapunov function for the switched system. If $V_1(x)$ and
%$V_2(x)$ satisfy the inequalities of the graphs $G_3$ and $G_4$,
%then $\min\{V_1(x),V_2(x)\}$ is a common Lyapunov function. These
%arguments serve as alternative proofs of stability and in the case
%where $V_1$ and $V_2$ are quadratic functions, they correspond to
%the works in~\cite{composite_Lyap},~\cite{composite_Lyap2},
%\cite{convex_conjugate_Lyap}, \cite{hu-ma-lin}.
%
%In the next section, we prove these statements in a more general
%setting (Corollaries~\ref{cor:min.of.quadratics}
%and~\ref{cor:max.of.quadratics}) and show that these graphs in
%some sense belong to ``simplest'' families of path-complete
%graphs.

\section{Duality and examples of families of path-complete
graphs}\label{sec:duality.and.some.families.of.path.complete}

Now that we have shown that \emph{any} path-complete graph
introduces a method for proving stability of switched systems, our
next focus is naturally on showing how one can produce graphs that
are path-complete. Before we proceed to some basic constructions
of such graphs, let us define a notion of \emph{duality} among
graphs which essentially doubles the number of path-complete
graphs that we can generate.

\begin{definition}\label{def:dual.graph}
Given a directed graph $G(N,E)$ whose edges are labeled from the
words in $\mathcal{A}^*$, we define its \emph{dual graph}
$G'(N,E')$ to be the graph obtained by reversing the direction of
the edges of $G$, and changing the labels $A_{\sigma_k}\ldots
A_{\sigma_1}$ of every edge of $G$ to its reversed version
$A_{\sigma_1}\ldots A_{\sigma_k}$.
\end{definition}

\begin{figure}[h]
\centering \scalebox{0.25} {\includegraphics{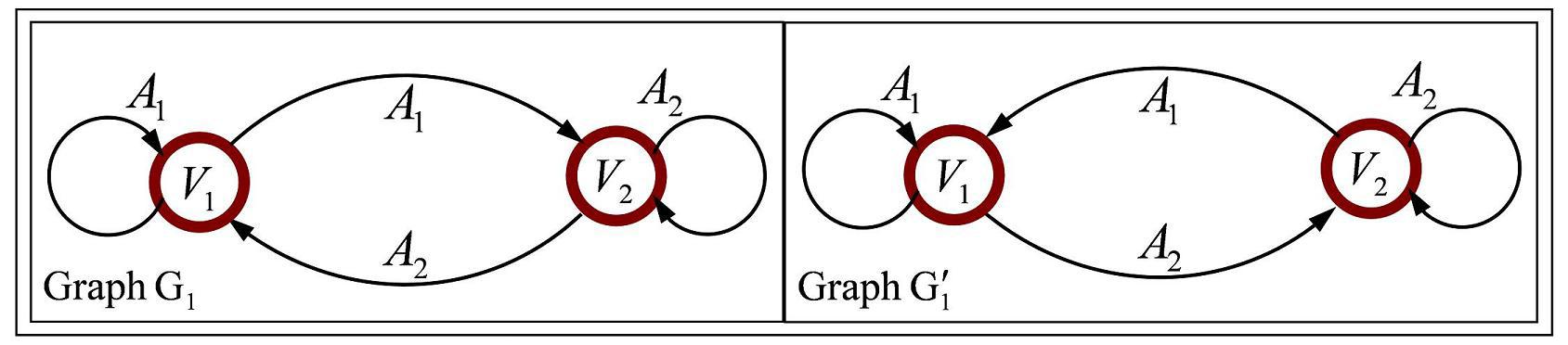}}  %AAA: I think cannot use JPG to creat a ps file...
\caption{An example of a pair of dual graphs.}
\label{fig:dual.graphs}
\end{figure}

An example of a pair of dual graphs with labels of length one is
given in Figure~\ref{fig:dual.graphs}. The following theorem
relates dual graphs and path-completeness.

\begin{theorem}\label{thm:path.complete.dual.path.complete}
If a graph $G(N,E)$ is path-complete, then its dual graph
$G'(N,E')$ is also path-complete.
\end{theorem}

\begin{proof}
Consider an arbitrary finite word $A_{i_k}\ldots A_{i_1}$. By
definition of what it means for a graph to be path-complete, our
task is to show that there exists a path corresponding to this
word in the expanded graph of the dual graph $G'$. It is easy to
see that the expanded graph of the dual graph of $G$ is the same
as the dual graph of the expanded graph of $G$; i.e,
$G'^e(N^e,E'^e)=G^{e^{'}}(N^e,E^{e^{'}})$. Therefore, we show a
path for $A_{i_k}\ldots A_{i_1}$ in $G^{e^{'}}$. Consider the
reversed word $A_{i_i}\ldots A_{i_k}$. Since $G$ is path-complete,
there is a path corresponding to this reversed word in $G^e$. Now
if we just trace this path backwards, we get exactly a path for
the original word $A_{i_k}\ldots A_{i_1}$ in $G^{e^{'}}$. This
completes the proof.
\end{proof}

The next proposition offers a very simple construction for
obtaining a large family of path-complete graphs with labels of
length one.

\begin{proposition}\label{prop:in.out.going.edges.path.complete}
A graph having any of the two properties below is path-complete.

Property (i): every node has outgoing edges with all the labels in
$\mathcal{A}$.

Property (ii): every node has incoming edges with all the labels
in $\mathcal{A}$.
\end{proposition}

\begin{proof}
If a graph has Property (i), then it is obviously path-complete.
If a graph has Property (ii), then its dual has Property (i) and
therefore by Theorem~\ref{thm:path.complete.dual.path.complete} it
is path-complete.
\end{proof}

Examples of path-complete graphs that fall in the category of this
proposition include graphs $G_1,G_2,G_3,$ and $G_4$ in
Figure~\ref{fig:jsr.graphs} and all of their dual graphs. By
combining the previous proposition with
Theorem~\ref{thm:path.complete.implies.stability}, we obtain the
following two simple corollaries which unify several linear matrix
inequalities (LMIs) that have been previously proposed in the
literature. These corollaries also provide a link to
min/max-of-quadratics Lyapunov functions. Different special cases
of these LMIs have appeared
in~\cite{composite_Lyap},~\cite{composite_Lyap2},
\cite{convex_conjugate_Lyap}, \cite{hu-ma-lin}, \cite{LeeD06},
\cite{daafouzbernussou}. Note that the framework of path-complete
graph Lyapunov functions makes the proof of the fact that these
LMIs imply stability immediate.

\begin{corollary}\label{cor:min.of.quadratics}
Consider a set of $m$ matrices and the switched linear system in
(\ref{eq:switched.linear.sys}) or
(\ref{eq:linear.difference.inclusion}). If there exist $k$
positive definite matrices $P_j$ such that
\begin{eqnarray}\label{eq:min.quadratics.LMIs}
\forall (i,k)\in\{1,\ldots,m\}^2,\ \exists j\in\{1,\ldots,m\}\
\nonumber
\\
\mbox{such that}\quad \quad \gamma^2A_i^TP_jA_i\preceq P_k,
\end{eqnarray}
for some $\gamma>1$, then the system is absolutely asymptotically
stable. Moreover, the pointwise minimum
$$\min\{x^TP_1x,\ldots,x^TP_kx\}$$ of the quadratic functions serves
as a common Lyapunov function.
\end{corollary}

\begin{proof}
The inequalities in (\ref{eq:min.quadratics.LMIs}) imply that
every node of the associated graph has outgoing edges labeled with
all the different $m$ matrices. Therefore, by
Proposition~\ref{prop:in.out.going.edges.path.complete} the graph
is path-complete, and by
Theorem~\ref{thm:path.complete.implies.stability} this implies
absolute asymptotic stability. The proof that the pointwise
minimum of the quadratics is a common Lyapunov function is easy
and left to the reader.
\end{proof}

\begin{corollary}\label{cor:max.of.quadratics}
Consider a set of $m$ matrices and the switched linear system in
(\ref{eq:switched.linear.sys}) or
(\ref{eq:linear.difference.inclusion}). If there exist $k$
positive definite matrices $P_j$ such that
\begin{eqnarray}\label{eq:max.quadratics.LMIs}
\forall (i,j)\in\{1,\ldots,m\}^2,\ \exists k\in\{1,\ldots,m\}\
\nonumber
\\
\mbox{such that}\quad \quad \gamma^2A_i^TP_jA_i\preceq P_k,
\end{eqnarray}
for some $\gamma>1$, then the system is absolutely asymptotically
stable. Moreover, the pointwise maximum
$$\max\{x^TP_1x,\ldots,x^TP_kx\}$$ of the quadratic functions serves
as a common Lyapunov function.
\end{corollary}

\begin{proof}
The inequalities in (\ref{eq:max.quadratics.LMIs}) imply that
every node of the associated graph has incoming edges labeled with
all the different $m$ matrices. Therefore, by
Proposition~\ref{prop:in.out.going.edges.path.complete} the graph
is path-complete and the proof of absolute asymptotic stability
then follows. The proof that the pointwise maximum of the
quadratics is a common Lyapunov function is again left to the
reader.
\end{proof}

\begin{remark} The linear matrix inequalities in (\ref{eq:min.quadratics.LMIs}) and (\ref{eq:max.quadratics.LMIs}) are (convex) sufficient
conditions for existence of min-of-quadratics or max-of-quadratics
Lyapunov functions. The converse is not true. The works in
~\cite{composite_Lyap},~\cite{composite_Lyap2},
\cite{convex_conjugate_Lyap}, \cite{hu-ma-lin} have additional
multipliers in (\ref{eq:min.quadratics.LMIs}) and
(\ref{eq:max.quadratics.LMIs}) that make the inequalities
non-convex but when solved with a heuristic method contain a
larger family of min-of-quadratics and max-of-quadratics Lyapunov
functions. Even if the non-convex inequalities with multipliers
could be solved exactly, except for special cases where the
$\mathcal{S}$-procedure is exact (e.g., the case of two quadratic
functions), these methods still do not completely characterize
min-of-quadratics and max-of-quadratics functions.
%
%As we will see shortly, the framework of path-complete graphs
%gives us a wealth of techniques amenable to convex optimization
%formulations and with performance guarantees that in our view it
%is not really needed to resort to non-convex approaches and
%heuristic solvers.
\end{remark}
\aaa{\begin{remark} \label{rmk:Lee-Dellerud_Daafouz} The work in
\cite{LeeD06} on ``path-dependent quadratic Lyapunov functions''
and the work in \cite{daafouzbernussou} on ``parameter dependent
Lyapunov functions''--when specialized to the analysis of
arbitrary switched linear systems--are special cases of
Corollary~\ref{cor:min.of.quadratics}
and~\ref{cor:max.of.quadratics} respectively. This observation
makes a connection between these techniques and
min/max-of-quadratics Lyapunov functions which is not established
in~\cite{LeeD06},~\cite{daafouzbernussou}. It is also interesting
to note that the path-complete graph corresponding to the LMIs
proposed in \cite{LeeD06} (see Theorem 9 there) is the well-known
De Bruijn graph~\cite{GraphTheory_Handbook}.
%Two other well-established references in the literature that (when
%specialized to the analysis of arbitrary switched linear systems)
%turn out to be particular classes of path-complete graphs are the
%work in \cite{LeeD06} on ``path-dependent quadratic Lyapunov
%functions'', and the work in \cite{daafouzbernussou} on
%``parameter dependent Lyapunov functions''. In fact, the LMIs
%suggested in these works are special cases of
%Corollary~\ref{cor:min.of.quadratics}
%and~\ref{cor:max.of.quadratics} respectively, hence revealing a
%connection to the min/max-of-quadratics type Lyapunov functions. A
%curious observation is that the path-complete graph corresponding
%to the LMIs proposed in \cite{LeeD06} (see Theorem 9 there) is the
%De Bruijn graph~\cite{GraphTheory_Handbook}.
\end{remark}}

The set of path-complete graphs is much broader than the set of
simple family of graphs constructed in
Proposition~\ref{prop:in.out.going.edges.path.complete}. Indeed,
there are many graphs that are path-complete without having
outgoing (or incoming) edges with all the labels on every node;
see e.g. graph $H_4^e$ in
Figure~\ref{fig:non.trivial.path.complete}. This in turn means
that there are several more sophisticated Lyapunov inequalities
that we can explore for proving stability of switched systems.
Below, we give one particular example of such ``non-obvious''
inequalities for the case of switching between two matrices.

\begin{figure}[h]
\centering \scalebox{0.5}
{\includegraphics{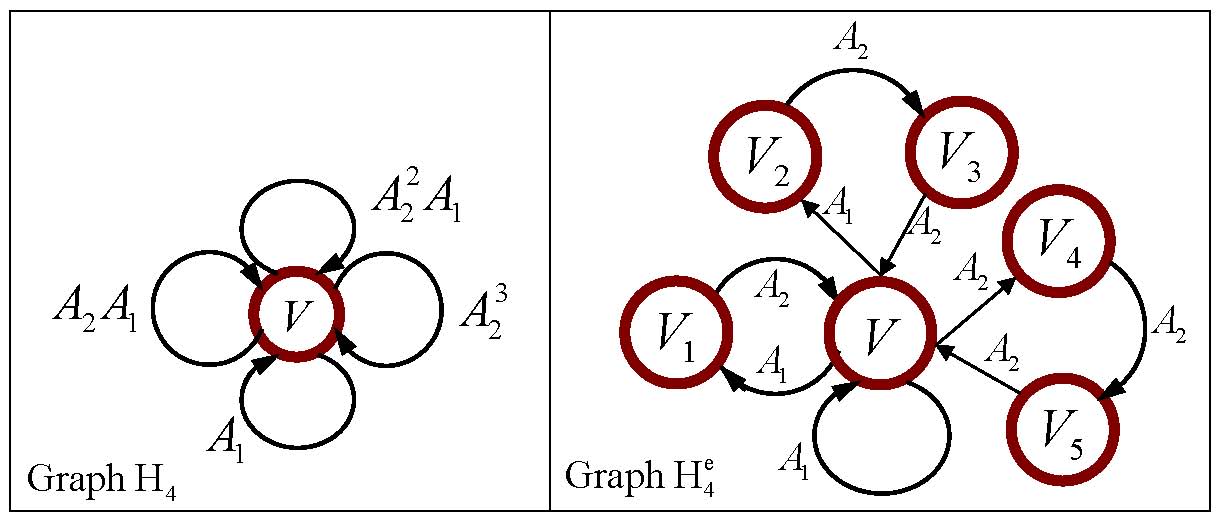}}
\caption{The path-complete graphs corresponding to
Proposition~\ref{prop:non.trivial.path.complete.graph}.}
\label{fig:non.trivial.path.complete}
\end{figure}

\begin{proposition}\label{prop:non.trivial.path.complete.graph}
Consider the set $\mathcal{A}=\{A_1,A_2\}$ and the switched linear
system in (\ref{eq:switched.linear.sys}) or
(\ref{eq:linear.difference.inclusion}). If there exist a positive
definite matrix $P$ such that
\begin{eqnarray}\label{eq:non.trivial.path.complete.graph}
\gamma^2A_1^TPA_1\preceq P, \nonumber \\
\gamma^4(A_2A_1)^TP(A_2A_1)\preceq P, \nonumber \\
\gamma^6(A_2^2A_1)^TP(A_2^2A_1)\preceq P, \nonumber \\
\gamma^6A_2^{3^T}PA_2^3\preceq P, \nonumber
\end{eqnarray}
for some $\gamma>1$, then the system is absolutely asymptotically
stable.
\end{proposition}
\begin{proof}
The graph $H_4$ associated with the LMIs above and its expanded
version $H_4^e$ are drawn in
Figure~\ref{fig:non.trivial.path.complete}. We leave it as an
exercise for the reader to show (e.g. by induction on the length
of the word) that there is path for every finite word in $H_4^e$.
Therefore, $H_4$ is path-complete and in view of
Theorem~\ref{thm:path.complete.implies.stability} the claim is
established.
\end{proof}

\begin{remark}
Proposition~\ref{prop:non.trivial.path.complete.graph} can be
generalized as follows: If a single Lyapunov function decreases
with respect to the matrix products
$$\{A_1,A_2A_1,A_2^2A_1,\ldots,A_2^{k-1}A_1,A_2^k\}$$ for some
integer $k\geq 1$, then the arbitrary switched system consisting
of the two matrices $A_1$ and $A_2$ is absolutely asymptotically
stable. We omit the proof of this generalization due to space
limitations. We will later prove (Theorem~\ref{thm-bound-codes}) a
bound for the quality of approximation of path-complete graphs of
this type, where a common Lyapunov function is required to
decrease with respect to products of different lengths.
\end{remark}

When we have so many different ways of imposing conditions for
stability, it is natural to ask which ones are better. The answer
clearly depends on the combinatorial structure of the graphs and
does not seem to be easy in general. Nevertheless, in the next
section, we compare the performance of all path-complete graphs
with two nodes for analysis of switched systems with two matrices.
The connections between the bounds obtained from these graphs are
not always obvious. For example, we will see that the graphs $H_1,
G_3,$ and $G_4$ always give the same bound on the joint spectral
radius; i.e, one graph will succeed in proving stability if and
only if the other will. So, there is no point in increasing the
number of decision variables and the number of constraints and
impose $G_3$ or $G_4$ in place of $H_1$. The same is true for the
graphs in $H_3$ and $G_2$, which makes graph $H_3$ preferable to
graph $G_2$. (See Proposition~\ref{prop:who.beats.who}.)

\section{Path-complete graphs with two nodes}\label{sec:who.beats.who}

In this section, we characterize the set of all path-complete
graphs consisting of two nodes, an alphabet set
$\mathcal{A}=\{A_{1},A_{2}\},$ and edge labels of unit length. We
will elaborate on the set of all admissible topologies arising in
this setup and compare the performance---in the sense of
conservatism of the ensuing analysis---of different path-complete
graph topologies.

\subsection{The set of path-complete graphs}

%Let us first establish a terminology that we will be using
%frequently.

%Given a labeled graph $G\left(  N,E\right)  ,$ we denote by
%$L\left(  e\right)  $ the label associated with edge $e\in E.$
%Also, $\mathcal{D}\left(  e\right)  $ denotes the destination node
%of $e\in E.$ Given a path-complete graph $G\left( N,E\right)  ,$
%we say that $\{V_{i}~|~i=1,...,\left\vert N\right\vert \}$ is a
%Graph Lyapunov Function (GLF) associated with $G\left(
%N,E\right)  $ for the switched system defined by
%$\mathcal{A}=\left\{  A_{1},\cdots,A_{m}\right\}  $ if
%\[
%V_{j}\left(  L\left(  \left(  i,j\right)  \right)  x\right)  \leq
%V_{i}\left( x\right)  ,\text{\qquad}\forall\left(  i,j\right)  \in
%E,
%\]
%where $L\left(  \left(  i,j\right)  \right) \in\mathcal{A}^{\ast}$
%is the label associated with edge $\left( i,j\right)  \in E.$ %%

%\textbf{Terminology:} Given a path-complete graph $G\left(
%N,E\right) ,$ we say that $\{V_{i}~|~i=1,...,\left\vert
%N\right\vert \}$ is a Graph Lyapunov Function (GLF) associated
%with $G\left( N,E\right)  $ for the switched system defined by
%$\mathcal{A}=\left\{ A_{1},\cdots,A_{m}\right\}  $ if
%\[
%V_{j}\left(  L\left( e  \right)  x\right)  \leq V_{i}\left(
%x\right)  ,\text{\qquad}\forall\ e  \in E,
%\]
%where $L\left(  e  \right) \in\mathcal{A}^{\ast}$ is the label
%associated with edge $e  \in E.$

The next lemma establishes that for thorough analysis of the case
of two matrices and two nodes, we only need to examine graphs with
four or fewer edges.

\begin{lemma}
\label{lessthan5}Let $G\left(  \left\{  1,2\right\}  ,E\right)  $
be a path-complete graph with labels of length one for
$\mathcal{A}=\{A_{1},A_{2}\}$. Let $\left\{ V_{1},V_{2}\right\}$
be a graph Lyapunov function for $G.$ If $\left\vert E\right\vert
>4,$ then, either
\newline\hspace*{0.15in}(i) there exists $\hat{e}\in E$ such that $G\left(  \left\{  1,2\right\}, E\backslash \hat{e}\right)  $ is a path-complete graph,
\newline\hspace*{0.15in} or
\newline\hspace*{0.15in}(ii) either $V_{1}$ or $V_{2}$ or both are common Lyapunov functions for {$\mathcal{A}.$}
\end{lemma}

\begin{proof}
If $\left\vert E\right\vert >4,$ then at least one node has three
or more outgoing edges. Without loss of generality let node $1$ be
a node with exactly three outgoing edges $e_{1},e_{2},e_{3}$, and
let $L\left( e_{1}\right)  =L\left( e_{2}\right)  =A_{1}.$ Let
$\mathcal{D}\left(  e\right)  $ denote the destination node of an
edge $e\in E.$ If $\mathcal{D}\left(
e_{1}\right)  =\mathcal{D}\left(  e_{2}\right),$ then $e_{1}$ (or $e_{2}%
$)\ can be removed without changing the output set of words. If
$\mathcal{D}\left( e_{1}\right)  \neq\mathcal{D}\left(
e_{2}\right)  ,$ assume, without loss of generality, that
$\mathcal{D}\left(  e_{1}\right)  =1$ and $\mathcal{D}\left(
e_{2}\right)  =2.$ Now, if $L\left(  e_{3}\right)  =A_{1},$ then
regardless of its destination node, $e_{3}$ can be removed. If
$L\left(  e_{3}\right)  =A_{2}$ and $\mathcal{D}\left(
e_{3}\right) =1$, then $V_1$ is a common Lyapunov function for
$\mathcal{A}$. The only remaining possibility is that $L\left(
e_{3}\right)  =A_{2}$ and $\mathcal{D}\left(  e_{3}\right) =2. $
Note that there must be an edge $e_4 \in E$ from node $2$ to node
$1$, otherwise either node $2$ would have two self-edges with the
same label or $V_2$ would be a common Lyapunov function for
$\mathcal{A}$. If $L(e_4)=A_2$ then it can be verified that
$G(\{1,2\},\{e_1,e_2,e_3,e_4\})$ is path-complete and thus all
other edge can be removed. If there is no edge from node $2$ to
node $1$ with label $A_2$ then $L(e_4)=A_1$ and node $2$ must have
a self-edge $e_5 \in E$ with label $L(e_5)=A_2$, otherwise the
graph would not be path-complete. In this case, it can be verified
that $e_{2}$ can be removed without affecting the output set of
words.
\end{proof}

It can be verified that a path-complete graph with two nodes and
less than four edges must necessarily place two self-loops with
different labels on one node, which necessitates existence of a
common Lyapunov function for the underlying switched system. Since
we are interested in exploiting the favorable properties of
\aaa{graph Lyapunov functions} in approximation of the JSR, we
will focus on graphs with four edges.

Before we proceed, for convenience we introduce the following
notation: Given a labeled graph $G(N,E)$ associated with two
matrices $A_1$ and $A_2$, we denote by $\overline{G}(N,E)$, the
graph obtained by swapping of $A_1$ and $A_2$ in all the labels on
every edge.
%%%This notion can be generalized to more than two matrices, however, we need not elaborate this as in this paper we will use the notation for the case of two matrices only.
%TCIMACRO{\FRAME{ftbpF}{3.5276in}{3.0268in}{0pt}{}{\Qlb{Fig X}}{slide1_1.jpg}%
%width 3.5276in;  height 3.0268in;  depth 0pt;  original-width 6.9998in;
%original-height 6.0001in;  cropleft "0";  croptop "1";  cropright "1";
%cropbottom "0";  filename 'slide1_1.jpg';file-properties "XNPEU";}}}%
%BeginExpansion
%\begin{figure}
%[ptb]
%\begin{center}
%\includegraphics[
%natheight=6.000100in, natwidth=6.999800in, height=3.0268in,
%width=3.5276in
%]%
%{slide1_1.jpg}%
%\label{Fig X}%
%\end{center}
%\end{figure}
%EndExpansion

\subsection{Comparison of performance}\label{subsec:comparison.of.who.beats.who}

It can be verified that for path-complete graphs with two nodes,
four edges, and two matrices, and without multiple self-loops on a
single node, there are a total of nine distinct graph topologies
to consider. Of the nine graphs, six have the property that every
node has two incoming edges with different labels. These are
graphs $G_1,~G_2,~\overline{G}_2,~G_3,~\overline{G}_3,$ and $G_4$
(Figure \ref{fig:jsr.graphs}). Note that $\overline{G}_1=G_1$ and
$\overline{G}_4=G_4$. The duals of these six graphs, i.e.,
$G_1^\prime,~G_2^\prime,~\overline{G}_2^\prime,~G_3^\prime=G_3,~\overline{G}_3^\prime=\overline{G}_3,$
and $G_4^\prime=G_4$ have the property that every node has two
outgoing edges with different labels. Evidently,
$G_3,~\overline{G}_3,$ and $G_4$ are \emph{self-dual graphs},
i.e., they are isomorphic to their dual graphs. The self-dual
graphs are least interesting to us since, as we will show, they
necessitate existence of a common Lyapunov function for
$\mathcal{A}$ (cf. Proposition \ref{prop8}, equation
(\ref{theselfduals}))$.$

%%The (strictly) primal graphs are Graph $G_{1}\ $(Figure
%%\ref{fig:jsr.graphs} (g)), Graph $G_{2},$ (Figure
%%\ref{fig:jsr.graphs} (d)), and Graph $\overline{G}_2$ which is obtained by
%%swapping the roles of $A_{1}$ and $A_{2}$ in $G_{2}$ (not shown).
%%The self-dual graphs are Graph $G_{4}$ (Figure
%%\ref{fig:jsr.graphs} (f)), Graph $G_{5}$ (Figure
%%\ref{fig:jsr.graphs} (e)), and Graph $G_{6}$ which is obtained by
%%swapping the roles of $A_{1}$ and $A_{2}$ in $G_{5}$ (not shown).
%%The (strictly) dual graphs are obtained by reversing the direction
%%of the arrows in the primals and are denoted by $G_{1}^{\prime},$
%%$G_{2}^{\prime},$ $\overline{G}_2^{\prime}$ respectively. For instance,
%%$G_{1}^{\prime}$ is the graph shown in Figure \ref{fig:jsr.graphs}
%%(h). The rest of the dual graphs are not shown.

Note that all of these graphs perform at least as well as a common
Lyapunov function because we can always take $V_{1}\left( x\right)
=V_{2}\left( x\right)  $. Furthermore, we know from
Corollaries~\ref{cor:max.of.quadratics}
and~\ref{cor:min.of.quadratics} that if Lyapunov inequalities
associated with $G_1,~G_2,~\overline{G}_2,~G_3,~\overline{G}_3,$
and $G_4$ are satisfied, then $\max\left\{ V_{1}\left( x\right)
,V_{2}\left( x\right) \right\} $ is a common Lyapunov function,
whereas, in the case of graphs
$G_1^\prime,~G_2^\prime,~\overline{G}_2^\prime,~G_3^\prime,~\overline{G}_3^\prime$,
and $G_4^\prime$, the function  $\min\left\{ V_{1}\left( x\right)
,V_{2}\left( x\right) \right\} $ would serve as a common Lyapunov
function. Clearly, for the self-dual graphs $G_3,~\overline{G}_3,$
and $G_4$ both $\max\left\{  V_{1}\left( x\right) ,V_{2}\left(
x\right) \right\}$ and $\min\left\{ V_{1}\left(  x\right)
,V_{2}\left( x\right) \right\}$ are common Lyapunov functions.

\textbf{Notation:} Given a set of matrices $\mathcal{A}=\left\{  A_{1}%
,\cdots,A_{m}\right\}  ,$ a path-complete graph $G\left(
N,E\right)  ,$ and a
class of functions $\mathcal{V},$ we denote by $\hat{\rho}_{\mathcal{V}}%
,_{G}\left(  \mathcal{A}\right)  ,$ the upper bound on the JSR of
$\mathcal{A}$ that can be obtained by numerical optimization of
GLFs $V_{i}\in \mathcal{V},~i\in N  ,$ defined over $G.$ With a
slight abuse of notation, we denote by $\hat{\rho
}_{\mathcal{V}}\left( \mathcal{A}\right)  ,$ the upper bound that
is obtained by using a common Lyapunov function $V\in\mathcal{V}.$

\aaa{

\begin{proposition}
\label{prop:who.beats.who} \label{prop8} Consider the set
$\mathcal{A}=\left\{A_{1},A_{2}\right\}  ,$  and let
$G_1,~G_2,~G_3,~G_4$, and $H_3$ be the path-complete graphs shown
in Figure \ref{fig:jsr.graphs}. Then, the upper bounds on the JSR
of $\mathcal{A}$ obtained by analysis via the associated GLFs
satisfy the following relations:
\begin{equation}
\hat{\rho}_{\mathcal{V}},_{G_{1}}\left(  \mathcal{A}\right)
=\hat{\rho }_{\mathcal{V}},_{G_{1}^{\prime}}\left(
\mathcal{A}\right)
\label{hscceqdual}%
\end{equation}
and
\begin{equation}
\hat{\rho }_{\mathcal{V}}\left(
\mathcal{A}\right)=\hat{\rho}_{\mathcal{V}},_{G_{3}}\left(
\mathcal{A}\right)
=\hat{\rho}_{\mathcal{V}},_{\overline{G}_{3}}\left(
\mathcal{A}\right) =\hat{\rho}_{\mathcal{V}},_{G_{4}}\left(
\mathcal{A}\right)
\label{theselfduals}%
\end{equation}
and
\begin{equation}
\hat{\rho}_{\mathcal{V}},_{G_{2}}\left(  \mathcal{A}\right)
=\hat{\rho }_{\mathcal{V}},_{{{H}_3}}\left(\mathcal{A}\right)
,\text{\qquad}\hat{\rho}_{\mathcal{V}},_{\overline{G}_2}\left(\mathcal{A}\right)
=\hat{\rho}_{\mathcal{V}},_{{\overline{{H}}_3}}\left(\mathcal{A}\right)
\label{theprimals}%
\end{equation}
and
\begin{equation}
\hat{\rho}_{\mathcal{V}},_{G_{2}^{\prime}}\left(
\mathcal{A}\right)
=\hat{\rho}_{\mathcal{V}},_{{{H}_3^{\prime}}}\left(
\mathcal{A}\right)  ,\text{\qquad}\hat{\rho}_{\mathcal{V}},_{\overline{G}_2^{\prime}%
}\left(  \mathcal{A}\right)  =\hat{\rho}_{\mathcal{V}},_{{\overline{{H}}_3^{\prime}}}\left(  \mathcal{A}\right).  \label{theduals}%
\end{equation}
\end{proposition}
}

\begin{proof}
A proof of \eqref{hscceqdual} in more generality is provided in
Section \ref{sec:hscc} (cf. Corollary
\ref{cor:HSCC.invariance.under.transpose}). The proof of
(\ref{theselfduals}) is based on symmetry arguments. Let
$\left\{V_{1},V_{2}\right\}$ be a GLF associated with $G_{3}$
($V_1$ is associated with node $1$ and $V_2$ is associated with
node $2$). Then, by symmetry, $\left\{V_{2},V_{1}\right\}$ is also
a GLF for $G_{3}$ (where $V_1$ is associated with node $2$ and
$V_2$ is associated with node $1$). Therefore, letting
$V=V_{1}+V_{2}$, we have that $\left\{V,V\right\}$ is a GLF for
$G_{3}$ and thus, $V=V_1+V_2$ is also a common Lyapunov function
for $\mathcal{A},$ which implies that
$\hat{\rho}_{\mathcal{V}},_{G_{3}}\left( \mathcal{A}\right)
\geq\hat{\rho }_{\mathcal{V}}\left( \mathcal{A}\right)  .$ The
other direction is trivial: If $V\in\mathcal{V}$ is a common
Lyapunov function for $\mathcal{A},$ then $\left\{
V_{1},V_{2}~|~V_{1}=V_{2}=V\right\}  $ is a GLF associated with
$G_{3},$ and hence, $\hat{\rho}_{\mathcal{V}},_{G_{3}}\left(  \mathcal{A}%
\right)  \leq\hat{\rho}_{\mathcal{V}}\left(  \mathcal{A}\right) .$
Identical arguments based on symmetry hold for
${\overline{G}_{3}}$ and ${G_{4}}$. We now prove the left equality
in (\ref{theprimals}), the proofs for the remaining equalities in
(\ref{theprimals}) and (\ref{theduals}) are analogous. The
equivalence between $G_2$ and $H_3$ is a special case of the
relation between a graph and its \emph{reduced} model, obtained by
removing a node without any self-loops, adding a new edge per each
pair of incoming and outgoing edges to that node, and then
labeling the new edges by taking the composition of the labels of
the corresponding incoming and outgoing edges in the original
graph; see~\cite{Roozbehani2008},~\cite[Chap.
5]{MardavijRoozbehani2008}. Note that $H_3$ is an offspring of
$G_2$ in this sense. This intuition helps construct a proof. Let
$\left\{ V_{1},V_{2}\right\}  $ be a GLF associated with $G_{2}.$
It can be verified that $V_{1}$ is a Lyapunov function associated
with ${{H}_3},$ and therefore,
$\hat{\rho}_{\mathcal{V}},_{{{H}_3}}\left(\mathcal{A}\right)
\leq\hat{\rho}_{\mathcal{V}},_{G_{2}}\left(\mathcal{A}\right) .$
Similarly, if $V\in\mathcal{V}$ is a Lyapunov function associated
with ${{H}_3},$ then one can check that $\left\{
V_{1},V_{2}~|~V_{1}\left(  x\right)  = V\left( x\right)
,V_{2}\left(  x\right)  =V\left(  A_{2}x\right) \right\} $ is a
GLF associated with $G_{2},$ and hence,
$\hat{\rho}_{\mathcal{V}},_{{{H}_3 }}\left(  \mathcal{A}\right)
\geq\hat{\rho}_{\mathcal{V}},_{G_{2}}\left( \mathcal{A}\right).$
%We refer the reader to~\cite{Roozbehani2008},
%and~\cite{MardavijRoozbehani2008}, Chapter 5 for a more general
%understanding of how the Lyapunov inequalities associated to
%certain pairs of graphs relate to each other.
%%%The proof of (\ref{hscceqdual}) is based on similar
%%%arguments;\ the GLFs associated with $G_{1}$ and $G_{1}^{\prime}$
%%%can be derived from one another via $V_{1}^{\prime}\left(
%%%A_{1}x\right)  =V_{1}\left(  x\right)  ,$ and
%%%$V_{2}^{\prime}\left(  A_{2}x\right)  =V_{2}\left(  x\right)  .$
\end{proof}

%Amirali: left out the following remark to save space. Put it back in for journal version.
%\begin{remark}
%The graphs $G_{1}$ and $G_{1}^{\prime}$ are the only ones for
%which we have not established the equivalence of GLF analysis to
%analysis via a common Lyapunov function for a finite set of
%matrices. What can be shown is that if $\left\{
%V_{1},V_{2}\right\}  $ is a GLF associated with $G_{1}$ then
%$V_{1}$
%is a Lyapunov function for $\mathcal{A}_{\infty}=\left\{  A_{2}^{k}%
%A_{1}~|~k=0,1,\cdots\right\}  .$ Similarly, if $\left\{  V_{1}^{\prime}%
%,V_{2}^{\prime}\right\}  $ is a GLF associated with
%$G_{1}^{\prime}$ then $V_{1}^{\prime}$ is a Lyapunov function for
%$\mathcal{A}_{\infty}^{\prime }=\left\{
%A_{1}A_{2}^{k}~|~k=0,1,\cdots\right\}  .$ Moreover, it can be
%verified that the finite sequences $\mathcal{A}_{k}=\left\{  A_{1},A_{2}%
%A_{1},\cdots,A_{2}^{k-1}A_{1},A_{2}^{k}\right\}  $ and $\mathcal{A}%
%_{k}^{\prime}=\left\{  A_{1},A_{1}A_{2},\cdots,A_{1}A_{2}^{k-1},A_{2}%
%^{k}\right\}  $ can be associated with a path-complete graph with
%a single node. Computation of the JSR using a common Lyapunov
%function for the finite sequence $\mathcal{A}_{k}$ (or
%$\mathcal{A}_{k}^{\prime}$) often leads to a good approximation of
%the JSR at a modest computational cost. The quality of
%approximation improves as $k$ increases.
%\end{remark}

\begin{figure}[h]
\centering \scalebox{0.28}
{\includegraphics{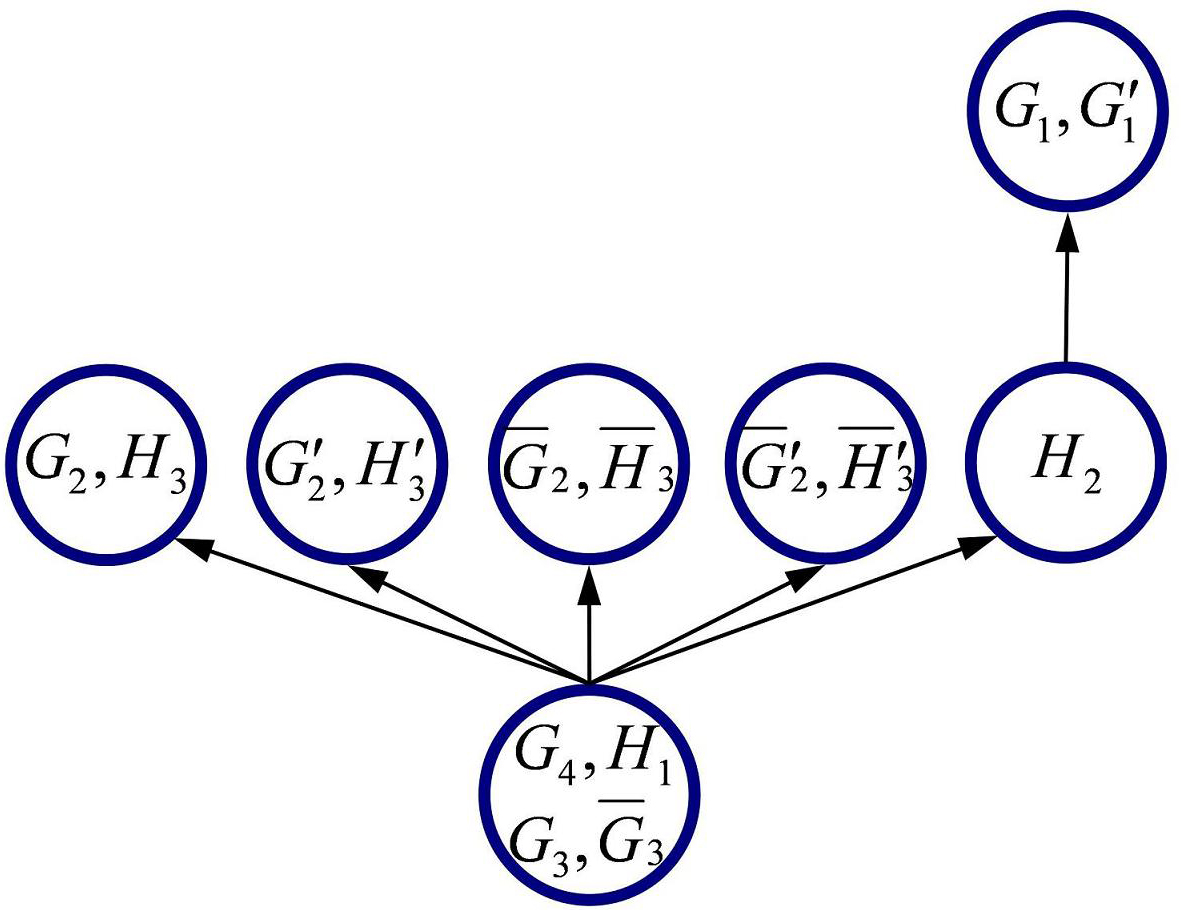}} \caption{\aaa{A
diagram describing the relative performance of the path-complete
graphs of Figure~\ref{fig:jsr.graphs} together with their duals
and label permutations. The graphs placed in the same circle
always give the same approximation of the JSR. A graph at the end
of an arrow results in an approximation of the JSR that is always
at least as good as that of the graph at the start of the arrow.
When there is no directed path between two graphs in this diagram,
either graph can outperform the other depending on the set of
matrices $\mathcal{A}$.}} \label{fig:Hasse.diag}
\end{figure}

\begin{remark}
\label{unverifiedremark}Proposition \ref{prop8} (equation
\ref{hscceqdual}) establishes the equivalence of the bounds
obtained from the pair of dual graphs $G_{1}$ and
$G_{1}^{\prime}$. This, however, is not true for graphs $G_{2}$
and $\overline{G}_2$ as there exist examples for which
\begin{align*}
\hat{\rho}_{\mathcal{V}},_{G_{2}}\left(  \mathcal{A}\right)    &
\neq\hat
{\rho}_{\mathcal{V}},_{G_{2}^{\prime}}\left(  \mathcal{A}\right)  ,\text{ }\\
\hat{\rho}_{\mathcal{V}},_{\overline{G}_2}\left(
\mathcal{A}\right)    & \neq\hat
{\rho}_{\mathcal{V}},_{\overline{G}_2^{\prime}}\left(
\mathcal{A}\right)  .
\end{align*}
\end{remark}

The diagram in Figure~\ref{fig:Hasse.diag} summarizes the results
of this section. We remark that no relations other than the ones
given in Figure~\ref{fig:Hasse.diag} can be made among these
path-complete graphs. Indeed, whenever there are no relations
between two graphs in Figure~\ref{fig:Hasse.diag}, we have
examples of matrices $A_1,A_2$ (not presented here) for which one
graph can outperform the other.

The graphs $G_1$ and $G^\prime_1$ seem to statistically perform
better than all other graphs in Figure~\ref{fig:Hasse.diag}. For
example, we ran experiments on a set of $100$ random $5\times5$
matrices $\{A_1,A_2\}$ with elements uniformly distributed in
$\left[ -1,1\right]$ to compare the performance of graphs $G_1,
G_2$ and $\overline{G}_2$. If in each case we also consider the
relabeled matrices (i.e., $\{A_2,A_1\}$) as our input, then, out
of the total $200$ instances, graph $G_1$ produced strictly better
bounds on the JSR $58$ times, whereas graphs $G_2$ and
$\overline{G}_2$ each produced the best bound of the three graphs
only $23$ times. (The numbers do not add up to $200$ due to ties.)
In addition to this superior performance, the bound
$\hat{\rho}_{\mathcal{V}},_{G_{1}}\left( \left\{
A_{1},A_{2}\right\}  \right)  $ obtained by analysis via the graph
$G_1$ is invariant under (i) permutation of the labels $A_{1}$ and
$A_{2}$ (obvious), and (ii) transposing of $A_{1}$ and $A_{2}$
(Corollary~\ref{cor:HSCC.invariance.under.transpose}). These are
desirable properties which fail to hold for $G_{2}$ and
$\overline{G}_2$ or their duals. Motivated by these observations,
we generalize $G_{1}$ and its dual $G_{1}^{\prime}$ in the next
section to the case of $m$ matrices and $m$ Lyapunov functions and
establish that they have certain appealing properties. We will
prove (cf. Theorem \ref{thm:HSCC.bound}) that these graphs always
perform better than a common Lyapunov function in 2 steps (i.e.,
the graph $H_2$ in Figure~\ref{fig:jsr.graphs}), whereas, this is
not the case for $G_{2}$ and $\overline{G}_2$ or their duals.

%\aaa{It is natural to define a partial order among path-complete
%graphs in terms of their relative performance in approximating the
%joint spectral radius. In Figure~\ref{fig:Hasse.diag}, we show the
%Hasse diagram of this partial order for the eight graphs given in
%Figure~\ref{fig:jsr.graphs}.}
%\aaa{The diagram in Figure~\ref{fig:Hasse.diag} summarizes the
%results of this section by demonstrating the relationship between
%the graphs shown in Figure~\ref{fig:jsr.graphs}, as well as their
%duals, and/or label permutations in terms of their quality of
%approximation of the JSR.}

\section{Further analysis of a particular family of path-complete graphs}\label{sec:hscc}
The framework of path-complete graphs provides a multitude of
semidefinite programming based techniques for the approximation of
the JSR whose performance vary with computational cost. For
instance, as we increase the number of nodes of the graph, or the
degree of the polynomial Lyapunov functions assigned to the nodes,
or the number of edges of the graph that instead of labels of
length one have labels of higher length, we obtain better results
but at a higher computational cost. Many of these approximation
techniques are asymptotically tight, so in theory they can be used
to achieve any desired accuracy of approximation. For example,
$$\hat{\rho}_{\mathcal{V}^{SOS,2d}}(\mathcal{A})\rightarrow\rho(\mathcal{A})
\ \mbox{as} \ 2d\rightarrow\infty,$$ where $\mathcal{V}^{SOS,2d}$
denotes the class of sum of squares homogeneous polynomial
Lyapunov functions of degree $2d$. (Recall our notation for bounds
from Section~\ref{subsec:comparison.of.who.beats.who}.) It is also
true that a common quadratic Lyapunov function for products of
higher length achieves the true JSR
asymptotically~\cite{Raphael_Book}; i.e.\footnote{By
$\mathcal{V}^2$ we denote the class of quadratic homogeneous
polynomials. We drop the superscript ``SOS'' because nonnegative
quadratic polynomials are always sums of squares.},
$$\sqrt[t]{\hat{\rho}_{\mathcal{V}^{2}}(\mathcal{A}^t)}\rightarrow\rho(\mathcal{A})\ \mbox{as} \ t\rightarrow\infty.$$

Nevertheless, it is desirable for practical purposes to identify a
class of path-complete graphs that provide a good tradeoff between
quality of approximation and computational cost. Towards this
objective, we propose the use of $m$ quadratic Lyapunov functions
assigned to the nodes of the De Bruijn graph of order $1$ on $m$
symbols for the approximation of the JSR of a set of $m$ matrices.
This graph and its dual are particular path-complete graphs with
$m$ nodes and $m^2$ edges and will be the subject of study in this
section. If we denote the quadratic Lyapunov functions by
$x^TP_ix$, then we are proposing the use of linear matrix
inequalities
\begin{equation}\label{eq:HSCC.lmis}
\begin{array}{rll}
P_i&\succ&0 \quad  \forall i=1,\ldots, m, \\
\gamma^2A_i^TP_jA_i&\preceq&P_i \quad \forall i,j=\{1,\ldots,
m\}^2
\end{array}
\end{equation}
or the set of LMIs
\begin{equation}\label{eq:dual.HSCC.lmis}
\begin{array}{rll}
P_i&\succ&0 \quad  \forall i=1,\ldots, m, \\
\gamma^2A_i^TP_iA_i&\preceq&P_j \quad \forall i,j=\{1,\ldots,
m\}^2
\end{array}
\end{equation}
for the approximation of the JSR of $m$ matrices. Throughout this
section, we denote the path-complete graphs associated with
(\ref{eq:HSCC.lmis}) and (\ref{eq:dual.HSCC.lmis}) with $G_1$ and
$G_1^\prime$ respectively. (The De Bruijn graph of order $1$, by
standard convention, is actually the graph $G_1^\prime$.) Observe
that $G_1$ and $G_1^\prime$ are indeed dual graphs as they can be
obtained from each other by reversing the direction of the edges.
For the case $m=2$, our notation is consistent with the previous
section and these graphs are illustrated in
Figure~\ref{fig:dual.graphs}. Also observe from
Corollary~\ref{cor:min.of.quadratics} and
Corollary~\ref{cor:max.of.quadratics} that the LMIs in
(\ref{eq:HSCC.lmis}) give rise to max-of-quadratics Lyapunov
functions, whereas the LMIs in (\ref{eq:dual.HSCC.lmis}) lead to
min-of-quadratics Lyapunov functions. We will prove in this
section that the approximation bound obtained by these LMIs (i.e.,
the reciprocal of the largest $\gamma$ for which the LMIs
(\ref{eq:HSCC.lmis}) or (\ref{eq:dual.HSCC.lmis}) hold) is always
the same and lies within a multiplicative factor of
$\frac{1}{\sqrt[4]{n}}$ of the true JSR, where $n$ is the
dimension of the matrices. The relation between the bound obtained
by a pair of dual path-complete graphs has a connection to
transposition of the matrices in the set $\mathcal{A}$. We explain
this next.

\subsection{Duality and invariance under transposition}\label{subsec:duality.and.transposition}
In~\cite{dual_LMI_diff_inclusions},~\cite{convex_conjugate_Lyap},
it is shown that absolute asymptotic stability of the linear
difference inclusion in (\ref{eq:linear.difference.inclusion})
defined by the matrices $\mathcal{A}=\{A_1,\ldots,A_m\}$ is
equivalent to absolute asymptotic stability of
(\ref{eq:linear.difference.inclusion}) for the transposed matrices
$\mathcal{A}^T\mathrel{\mathop:}=\{A_1^T,\ldots,A_m^T\}$. Note
that this fact is immediately seen from the definition of the JSR
in (\ref{eq:def.jsr}), since
$\rho(\mathcal{A})=\rho(\mathcal{A}^T)$. It is also well-known
that
$$\hat{\rho}_{\mathcal{V}^2}(\mathcal{A})=\hat{\rho}_{\mathcal{V}^2}(\mathcal{A}^T).$$
Indeed, if $x^TPx$ is a common quadratic Lyapunov function for the
set $\mathcal{A}$, then it is easy to show that $x^TP^{-1}x$ is a
common quadratic Lyapunov function for the set $\mathcal{A}^T$.
However, this nice property is not true for the bound obtained
from some other techniques. For instance, the next example shows
that
%\footnote{We have examples that show the statement in
%(\ref{eq:bound.sos.not.transp.invar}), which we do not present
%because of space limitations. See~\cite{dual_LMI_diff_inclusions}
%for such an example in the continuous time setting.}
\begin{equation}\label{eq:bound.sos.not.transp.invar}
\hat{\rho}_{\mathcal{V}^{SOS,4}}(\mathcal{A})\neq\hat{\rho}_{\mathcal{V}^{SOS,4}}(\mathcal{A}^T),
\end{equation}
i.e, the upper bound obtained by searching for a common quartic
sos polynomial is not invariant under transposition.

\begin{example}\label{ex:quartic.sos.not.transpose.invariant}
Consider the set of matrices $\mathcal{A}=\{A_1,A_2,A_3,A_4\},$
with \scalefont{.5}
\begin{align*}
A_{1} =\left[
\begin{array}
[c]{rrr}%
10 & -6 & -1 \\
8 & 1 & -16 \\
-8 & 0 & 17
\end{array}
\right] , A_{2}=\left[
\begin{array}
[c]{rrr}%
-5 & 9 & -14\\
1 & 5 & 10 \\
3 & 2 & 16
\end{array}
\right],  A_{3} =\left[
\begin{array}
[c]{rrr}%
-14 & 1 & 0\\
-15 & -8 & -12 \\
-1 & -6 & 7
\end{array}
\right],  A_{4} =\left[
\begin{array}
[c]{rrr}%
1 & -8 & -2\\
1 & 16 & 3 \\
16 & 11 & 14
\end{array}
\right].
\end{align*}
\normalsize We have
$\hat{\rho}_{\mathcal{V}^{SOS,4}}(\mathcal{A})=21.411,$ but
$\hat{\rho}_{\mathcal{V}^{SOS,4}}(\mathcal{A}^T)=21.214$ (up to
three significant digits). $\triangle$
\end{example}

Similarly, the bound obtained by non-convex inequalities proposed
in~\cite{dual_LMI_diff_inclusions} is not invariant under
transposing the matrices. For such methods, one would have to run
the numerical optimization twice---once for the set $\mathcal{A}$
and once for the set $\mathcal{A}^T$---and then pick the better
bound of the two. We will show that by contrast, the bound
obtained from the LMIs in (\ref{eq:HSCC.lmis}) and
(\ref{eq:dual.HSCC.lmis}) are invariant under transposing the
matrices. Before we do that, let us prove a general result which
states that for path-complete graphs with quadratic Lyapunov
functions as nodes, transposing the matrices has the same effect
as dualizing the graph.

\begin{theorem}\label{thm:transpose.bound.dual.bound}
Let $G(N,E)$ be a path-complete graph, and let
$G^\prime(N,E^\prime)$ be its dual graph. Then,
\begin{equation}\label{eq:rho.hat.A^T=rho.hat.G'}
\hat{\rho}_{\mathcal{V}^2,
G}(\mathcal{A}^T)=\hat{\rho}_{\mathcal{V}^2,
G^\prime}(\mathcal{A}).
\end{equation}
\end{theorem}

\begin{proof}
For ease of notation, we prove the claim for the case where the
edge labels of $G(N,E)$ have length one. The proof of the general
case is identical. Pick an arbitrary edge $(i,j)\in E$ going from
node $i$ to node $j$ and labeled with some matrix
$A_l\in\mathcal{A}$. By the application of the Schur complement we
have
\begin{equation}\nonumber
A_lP_jA_l^T\preceq P_i \ \Leftrightarrow\
\begin{bmatrix}P_i & A_l\\ A_l^T & P_j^{-1} \end{bmatrix}\succeq0 \ \Leftrightarrow\
A_l^TP_i^{-1}A_l\preceq P_j^{-1}.
\end{equation}
But this already establishes the claim since we see that $P_i$ and
$P_j$ satisfy the LMI associated with edge $(i,j)\in E$ when the
matrix $A_l$ is transposed if and only if $P_j^{-1}$ and
$P_i^{-1}$ satisfy the LMI associated with edge $(j,i)\in
E^\prime$.
%
%OLD PROOF:
%associated constraint be given by
%\begin{equation}\nonumber
%A_l^TP_jA_l\preceq P_i,
%\end{equation}
%for some $A_l\in\mathcal{A}$. Suppose the constraint associated to
%this edge is satisfied for the transposed matrix $A_l^T$, i.e.,
%there exists positive definite matrices $\tilde{P}_i$ and
%$\tilde{P}_j$ such that
%\begin{equation}\nonumber
%A_l\tilde{P}_jA_l^T\preceq \tilde{P}_i.
%\end{equation}
%By applying the Schur complement twice, we see that the last
%inequality holds if and only if
%\begin{equation}\nonumber
%A_l^T\tilde{P}_i^{-1}A_l\preceq \tilde{P}_j^{-1}.
%\end{equation}
%But this inequality shows that $\tilde{P}_i^{-1}$ and
%$\tilde{P}_j^{-1}$ satisfy the constraint associated with edge
%$(j,i)\in E^\prime$. This establishes the claim.
\end{proof}

\begin{corollary}\label{cor:transpose.eq.iff.dual.eq}
$\hat{\rho}_{\mathcal{V}^2,
G}(\mathcal{A})=\hat{\rho}_{\mathcal{V}^2, G}(\mathcal{A}^T)$ if
and only if $\hat{\rho}_{\mathcal{V}^2,
G}(\mathcal{A})=\hat{\rho}_{\mathcal{V}^2,
G^\prime}(\mathcal{A})$.
\end{corollary}

\begin{proof}
This is an immediate consequence of the equality in
(\ref{eq:rho.hat.A^T=rho.hat.G'}).
\end{proof}

It is an interesting question for future research to characterize
the topologies of path-complete graphs for which one has
$\hat{\rho}_{\mathcal{V}^2,
G}(\mathcal{A})=\hat{\rho}_{\mathcal{V}^2, G}(\mathcal{A}^T).$ For
example, the above corollary shows that this is obviously the case
for any path-complete graph that is self-dual. Let us show next
that this is also the case for graphs $G_1$ and $G_1^\prime$
despite the fact that they are not self-dual.

\begin{corollary}\label{cor:HSCC.invariance.under.transpose}
For the path-complete graphs $G_1$ and $G_1^\prime$ associated
with the inequalities in (\ref{eq:HSCC.lmis}) and
(\ref{eq:dual.HSCC.lmis}), and for any class of continuous,
homogeneous, and positive definite functions $\mathcal{V}$, we
have
\begin{equation}\label{eq:bound.G=bound.G'.general.Lyap}
\hat{\rho}_{\mathcal{V},G_1}(\mathcal{A})=\hat{\rho}_{\mathcal{V},G_1^\prime}(\mathcal{A}).
\end{equation}
Moreover, if quadratic Lyapunov functions are assigned to the
nodes of $G_1$ and $G_1^\prime$, then we have
\begin{equation}\label{eq:bound.G1=G1'=G1A^T=G1'A^T}
\hat{\rho}_{\mathcal{V}^2,
G_1}(\mathcal{A})=\hat{\rho}_{\mathcal{V}^2,
G_1}(\mathcal{A}^T)=\hat{\rho}_{\mathcal{V}^2,
G_1^\prime}(\mathcal{A})=\hat{\rho}_{\mathcal{V}^2,
G_1^\prime}(\mathcal{A}^T).
\end{equation}
\end{corollary}

\begin{proof}
The proof of (\ref{eq:bound.G=bound.G'.general.Lyap}) is
established by observing that the GLFs associated with $G_1$ and
$G_1^\prime$ can be derived from one another via
$V_i^\prime(A_ix)=V_i(x).$ (Note that we are relying here on the
assumption that the matrices $A_i$ are invertible, which as we
noted in Remark~\ref{rmk:invertibility}, is not a limiting
assumption.) Since (\ref{eq:bound.G=bound.G'.general.Lyap}) in
particular implies that $\hat{\rho}_{\mathcal{V}^2,
G_1}(\mathcal{A})=\hat{\rho}_{\mathcal{V}^2,
G_1^\prime}(\mathcal{A})$, we get the rest of the equalities in
(\ref{eq:bound.G1=G1'=G1A^T=G1'A^T}) immediately from
Corollary~\ref{cor:transpose.eq.iff.dual.eq} and this finishes the
proof. For concreteness, let us also prove the leftmost equality
in (\ref{eq:bound.G1=G1'=G1A^T=G1'A^T}) directly. Let $P_i$,
$i=1,\ldots,m,$ satisfy the LMIs in (\ref{eq:HSCC.lmis}) for the
set of matrices $\mathcal{A}$. Then, the reader can check that
$$\tilde{P}_i=A_i P_i^{-1}A_i^T,\quad i=1,\ldots,m,$$
satisfy the LMIs in (\ref{eq:HSCC.lmis}) for the set of matrices
$\mathcal{A}^T$.
\end{proof}

%\begin{remark}\label{rmk:G1.bound.equal.G1'.bound}
%The equivalence of the bounds obtained by the path-complete graphs
%$G_1$ and $G_1^\prime$ is not special to the case where quadratic
%Lyapunov functions are assigned to the nodes. Indeed, for any
%class of continuous, homogeneous, and positive definite functions
%$\mathcal{V}$, we can directly show that $\hat{\rho}_{\mathcal{V},
%G_1}(\mathcal{A})=\hat{\rho}_{\mathcal{V},
%G_1^\prime}(\mathcal{A})$ by observing that the GLFs associated
%with $G_1$ and $G_1^\prime$ can be derived from one another via
%$$V_i^\prime(A_ix)=V_i(x).$$
%\end{remark}

\subsection{An approximation guarantee}\label{subsec:HSCC.bound}
The next theorem gives a bound on the quality of approximation of
the estimate resulting from the LMIs in (\ref{eq:HSCC.lmis}) and
(\ref{eq:dual.HSCC.lmis}). Since we have already shown that
$\hat{\rho}_{\mathcal{V}^2,
G_1}(\mathcal{A})=\hat{\rho}_{\mathcal{V}^2,
G_1^\prime}(\mathcal{A}),$ it is enough to prove this bound for
the LMIs in (\ref{eq:HSCC.lmis}).

\begin{theorem}\label{thm:HSCC.bound}
Let $\mathcal{A}$ be a set of $m$ matrices in $\mathbb{R}^{n\times
n}$ with JSR $\rho(\mathcal{A})$. Let $\hat{\rho}_{\mathcal{V}^2,
G_1}(\mathcal{A})$ be the bound on the JSR obtained from the LMIs
in (\ref{eq:HSCC.lmis}). Then,
\begin{equation}\label{eq:HSCC.bound.4throot.of.n}
\frac{1}{\sqrt[4]{n}}\hat{\rho}_{\mathcal{V}^2,
G_1}(\mathcal{A})\leq\rho(\mathcal{A})\leq\hat{\rho}_{\mathcal{V}^2,
G_1}(\mathcal{A}).
\end{equation}
%and
%\begin{equation}\label{eq:dual.HSCC.bound.4throot.of.n}
%\frac{1}{\sqrt[4]{n}}\hat{\rho}_{\mathcal{V}^2,
%G_1^\prime}(\mathcal{A})\leq\rho(\mathcal{A})\leq\hat{\rho}_{\mathcal{V}^2,
%G_1^\prime}(\mathcal{A}).
%\end{equation}
\end{theorem}

\begin{proof}
The right inequality is just a consequence of $G_1$ being a
path-complete graph
(Theorem~\ref{thm:path.complete.implies.stability}). To prove the
left inequality, consider the set $\mathcal{A}^2$ consisting of
all $m^2$ products of length two. In view of (\ref{eq:CQ.bound}),
a common quadratic Lyapunov function for this set satisfies the
bound
\begin{equation}\nonumber
\frac{1}{\sqrt{n}}\hat{\rho}_{\mathcal{V}^2}(\mathcal{A}^2)\leq\rho(\mathcal{A}^2).
\end{equation}
It is easy to show that
$$\rho(\mathcal{A}^2)=\rho^2(\mathcal{A}).$$
See e.g.~\cite{Raphael_Book}. Therefore,
\begin{equation}\label{eq:bound.of.CQ.2.steps}
\frac{1}{\sqrt[4]{n}}\hat{\rho}_{\mathcal{V}^2}^{\frac{1}{2}}(\mathcal{A}^2)\leq\rho(\mathcal{A}).
\end{equation}
Now suppose for some $\gamma>0$, $x^TQx$ is a common quadratic
Lyapunov function for the matrices in $\mathcal{A}_\gamma^2$;
i.e., it satisfies
\begin{equation}\nonumber
\begin{array}{rll}
Q&\succ&0 \\
\gamma^4(A_iA_j)^TQA_iA_j&\preceq&Q \quad \forall i,j=\{1,\ldots,
m\}^2.
\end{array}
\end{equation}
Then, we leave it to the reader to check that
\begin{equation}\nonumber
P_i=Q+A_i^TQA_i,\quad i=1,\ldots,m
\end{equation}
satisfy (\ref{eq:HSCC.lmis}). Hence,
\begin{equation}\nonumber
\hat{\rho}_{\mathcal{V}^2,
G_1}(\mathcal{A})\leq\hat{\rho}_{\mathcal{V}^2}^{\frac{1}{2}}(\mathcal{A}^2),
\end{equation}
and in view of (\ref{eq:bound.of.CQ.2.steps}) the claim is
established.
\end{proof}

Note that the bound in (\ref{eq:HSCC.bound.4throot.of.n}) is
independent of the number of matrices. Moreover, we remark that
this bound is tighter, in terms of its dependence on $n$, than the
known bounds for $\hat{\rho}_{\mathcal{V}^{SOS,2d}}$ for any
finite degree $2d$ of the sum of squares polynomials. The reader
can check that the bound in (\ref{eq:SOS.bound}) goes
asymptotically as $\frac{1}{\sqrt{n}}$. Numerical evidence
suggests that the performance of both the bound obtained by sum of
squares polynomials and the bound obtained by the LMIs in
(\ref{eq:HSCC.lmis}) and (\ref{eq:dual.HSCC.lmis}) is much better
than the provable bounds in (\ref{eq:SOS.bound}) and in
Theorem~\ref{thm:HSCC.bound}. The problem of improving these
bounds or establishing their tightness is open. It goes without
saying that instead of quadratic functions, we can associate sum
of squares polynomials to the nodes of $G_1$ and obtain a more
powerful technique for which we can also prove better bounds with
the exact same arguments.

\subsection{Numerical examples}\label{subsec:numerical.examples}
In the proof of Theorem~\ref{thm:HSCC.bound}, we essentially
showed that the bound obtained from LMIs in (\ref{eq:HSCC.lmis})
is tighter than the bound obtained from a common quadratic applied
to products of length two. Our first example shows that the LMIs
in (\ref{eq:HSCC.lmis}) can in fact do better than a common
quadratic applied to products of \emph{any} finite length.

\begin{example}
Consider the set of matrices $\mathcal{A}=\{A_1,A_2\},$ with
\[
A_{1}=\left[
\begin{array}
[c]{cc}%
1 & 0\\
1 & 0
\end{array}
\right]  ,\text{ }A_{2}=\left[
\begin{array}
[c]{cr}%
0 & 1\\
0 & -1
\end{array}
\right].
\]
This is a benchmark set of matrices that has been studied
in~\cite{Ando98},~\cite{Pablo_Jadbabaie_JSR_journal},~\cite{AAA_PP_CDC08_non_monotonic}
because it gives the worst case approximation ratio of a common
quadratic Lyapunov function. Indeed, it is easy to show that
$\rho(\mathcal{A})=1$, but
$\hat{\rho}_{\mathcal{V}^2}(\mathcal{A})=\sqrt{2}$. Moreover, the
bound obtained by a common quadratic function applied to the set
$\mathcal{A}^t$ is
$$\hat{\rho}_{\mathcal{V}^2}^{\frac{1}{t}}(\mathcal{A}^t)=2^{\frac{1}{2t}},$$
which for no finite value of $t$ is exact. On the other hand, we
show that the LMIs in (\ref{eq:HSCC.lmis}) give the exact bound;
i.e., $\hat{\rho}_{\mathcal{V}^2, G_1}(\mathcal{A})=1$. Due to the
simple structure of $A_{1}$ and $A_{2}$, we can even give an
analytical expression for our Lyapunov functions. Given any
$\varepsilon>0$, the LMIs in (\ref{eq:HSCC.lmis}) with
$\gamma=1/\left( 1+\varepsilon \right)  $ are feasible with
\[
P_{1}=\left[
\begin{array}
[c]{cc}%
a & 0\\
0 & b
\end{array}
\right]  ,\text{\qquad}P_{2}=\left[
\begin{array}
[c]{cc}%
b & 0\\
0 & a
\end{array}
\right],
\]
for any $b>0$ and $a>b/2\varepsilon.$ $\triangle$
\end{example}

\begin{example}
%Consider the set of randomly generated matrices
%$\mathcal{A}=\{A_1,A_2,A_3\},$ with \scalefont{.7}
%\begin{align*}
%A_{1}  & =\left[
%\begin{array}
%[c]{rrrrr}%
%0 & -2 & 2 & 2 & 4\\
%0 & 0 & -4 & -1 & -6\\
%2 & 6 & 0 & -8 & 0\\
%-2 & -2 & -3 & 1 & -3\\
%-1 & -5 & 2 & 6 & -4
%\end{array}
%\right]  ,\text{ }A_{2}=\left[
%\begin{array}
%[c]{rrrrr}%
%-5 & -2 & -4 & \text{ \ \hspace*{0.01in}}6 & -1\\
%1 & 1 & 4 & 3 & -5\\
%-2 & 3 & -2 & 8 & -1\\
%0 & 8 & -6 & 2 & 5\\
%-1 & -5 & 1 & 7 & -4
%\end{array}
%\right] \\[0.1in]
%A_{3}  & =\left[
%\begin{array}
%[c]{rrrrr}%
%3 & -8 & -3 & 2 & -4\\
%-2 & -2 & -9 & 4 & -1\\
%2 & 2 & -5 & -8 & 6\\
%-4 & -1 & 4 & -3 & 0\\
%0 & 5 & 0 & -3 & 5
%\end{array}
%\right].
%\end{align*}
%\normalsize
%%Let's make the matrices smaller to save room:
\aaa{ Consider the set of randomly generated matrices
$\mathcal{A}=\{A_1,A_2,A_3\},$ with \scalefont{.5}
\begin{align*}
A_{1} =\left[
\begin{array}
[c]{rrrrr}%
0 & -2 & 2 & 2 & 4\\
0 & 0 & -4 & -1 & -6\\
2 & 6 & 0 & -8 & 0\\
-2 & -2 & -3 & 1 & -3\\
-1 & -5 & 2 & 6 & -4
\end{array}
\right] , A_{2}=\left[
\begin{array}
[c]{rrrrr}%
-5 & -2 & -4 & \text{ \ \hspace*{0.01in}}6 & -1\\
1 & 1 & 4 & 3 & -5\\
-2 & 3 & -2 & 8 & -1\\
0 & 8 & -6 & 2 & 5\\
-1 & -5 & 1 & 7 & -4
\end{array}
\right],  A_{3} =\left[
\begin{array}
[c]{rrrrr}%
3 & -8 & -3 & 2 & -4\\
-2 & -2 & -9 & 4 & -1\\
2 & 2 & -5 & -8 & 6\\
-4 & -1 & 4 & -3 & 0\\
0 & 5 & 0 & -3 & 5
\end{array}
\right].
\end{align*}
\normalsize

}

A lower bound on $\rho(\mathcal{A})$ is
$\rho(A_{1}A_{2}A_{2})^{1/3}=11.8015$. The upper approximations
for $\rho(\mathcal{A})$ that we computed for this example are
\aaa{as follows}:
\begin{equation}
\begin{array}{rll}
\hat{\rho}_{\mathcal{V}^2}(\mathcal{A})&=& 12.5683   \\
\hat{\rho}_{\mathcal{V}^2}^{\frac{1}{2}}(\mathcal{A}^2)&=&  11.9575    \\
\hat{\rho}_{\mathcal{V}^2,G_1}(\mathcal{A})&=&   11.8097 \\
\hat{\rho}_{\mathcal{V}^{SOS,4}}(\mathcal{A})&=&   11.8015.
\end{array}
\end{equation}
%let's save some room:
%$\hat{\rho}_{\mathcal{V}^2}(\mathcal{A})= 12.5786,\
%\hat{\rho}_{\mathcal{V}^2}^{\frac{1}{2}}(\mathcal{A}^2)= 12.0337,\
%\hat{\rho}_{\mathcal{V}^2,G_1}(\mathcal{A})=   11.8106,\
%\hat{\rho}_{\mathcal{V}^{SOS,4}}(\mathcal{A})=   11.8022,\
%\hat{\rho}_{\mathcal{V}^{SOS,4},G_1}(\mathcal{A})= 11.8015.$
The bound $\hat{\rho}_{\mathcal{V}^{SOS,4}}$ matches the lower
bound numerically and is most likely exact for this example. This
bound is slightly better than $\hat{\rho}_{\mathcal{V}^2,G_1}$.
However, a simple calculation shows that the semidefinite program
resulting in $\hat{\rho}_{\mathcal{V}^{SOS,4}}$ has 25 more
decision variables than the one for
$\hat{\rho}_{\mathcal{V}^2,G_1}$. Also, the running time of the
algorithm leading to $\hat{\rho}_{\mathcal{V}^{SOS,4}}$ is
noticeably larger than the one leading to
$\hat{\rho}_{\mathcal{V}^2,G_1}$. In general, when the dimension
of the matrices is large, it can often be cost-effective to
increase the number of the nodes of our path-complete graphs but
keep the degree of the polynomial Lyapunov functions assigned to
its nodes relatively low. $\triangle$
%\begin{table}[tbh]
%\caption{Comparison of the performance and complexity of our
%technique versus
%existing methods for numerical approximation of an upper-bound on the JSR.}%
%
%\label{Table II}
%\begin{center}
%$%
%\begin{tabular}
%[c]{|p{1.4in}|p{0.79in}|p{0.53in}|p{0.62in}|}\hline
%Approx. scheme & Upper bound \hspace*{0.06in}on the JSR & $\hspace*{0.03in}%
%$$\sharp$ of decision variables & $\sharp$ of constraints\\\hline
%$\hat{\rho}_{\mathcal{V}^2}(\mathcal{A})$ &
%$\hspace*{0.17in}12.5786$ & $\hspace*{0.2in}15$ &
%$\hspace*{0.27in}3$\\\hline
%$\hat{\rho}_{\mathcal{V}^2}^{\frac{1}{2}}(\mathcal{A}^2) $& $\hspace*{0.17in}%
%12.0337$ & $\hspace*{0.2in}15$ & $\hspace*{0.27in}9$\\\hline
%$\hat{\rho}_{\mathcal{V}^2,G_1}(\mathcal{A}) $&
%$\hspace*{0.17in}11.8106$ & $\hspace *{0.2in}45$ &
%$\hspace*{0.27in}9$\\\hline
%$\hat{\rho}_{\mathcal{V}^{SOS,4}}(\mathcal{A})$&
%$\hspace*{0.17in}11.8022$ & $\hspace*{0.2in}70$ &
%$\hspace*{0.27in}3$\\\hline
%$\hat{\rho}_{\mathcal{V}^{SOS,4},G_1}(\mathcal{A}) $&
%$\hspace*{0.17in}11.8015$ & $\hspace *{0.17in}210$ &
%$\hspace*{0.27in}9$\\\hline
%\end{tabular}
%$
%\end{center}
%\end{table}
\end{example}

\section{Converse Lyapunov theorems and approximation with arbitrary accuracy}\label{sec:converse.thms}
It is well-known that existence of a Lyapunov function which is
the pointwise maximum of quadratics is not only sufficient but
also necessary for absolute asymptotic stability of
(\ref{eq:switched.linear.sys}) or
(\ref{eq:linear.difference.inclusion}); see
e.g.~\cite{pyatnitskiy_max_lyap}.
%AAA: To save some space, I'm removing a few lines:
This is perhaps an intuitive fact if we recall that switched
systems of type (\ref{eq:switched.linear.sys}) and
(\ref{eq:linear.difference.inclusion}) always admit a convex
Lyapunov function. Indeed, if we take ``enough'' quadratics, the
convex and compact unit sublevel set of a convex Lyapunov function
can be approximated arbitrarily well with sublevel sets of
max-of-quadratics Lyapunov functions, which are intersections of
ellipsoids. This of course implies that the bound obtained from
max-of-quadratics Lyapunov functions is asymptotically tight for
the approximation of the JSR. However, this converse Lyapunov
theorem does not answer two natural questions of importance in
practice: (i) How many quadratic functions do we need to achieve a
desired quality of approximation? (ii) Can we search for these
quadratic functions via semidefinite programming or do we need to
resort to non-convex formulations? Our next theorem provides an
answer to these questions.

\begin{theorem}\label{thm:converse.max.of.quadratics}
Let $\mathcal{A}$ be a set of $m$ matrices in $\mathbb{R}^{n\times
n}$. Given any positive integer $l$, there exists an explicit
path-complete graph $G$ consisting of $m^{l-1}$ nodes assigned to
quadratic Lyapunov functions and $m^l$ edges with labels of length
one such that the linear matrix inequalities associated with $G$
imply existence of a max-of-quadratics Lyapunov function and the
resulting bound obtained from the LMIs satisfies
\begin{equation}\label{eq:bound.converse.thm.max}
\frac{1}{\sqrt[2l]{n}}\hat{\rho}_{\mathcal{V}^2,
G}(\mathcal{A})\leq\rho(\mathcal{A})\leq\hat{\rho}_{\mathcal{V}^2,
G}(\mathcal{A}).
\end{equation}
\end{theorem}
\begin{proof}
Let us denote the $m^{l-1}$ quadratic Lyapunov functions by
$x^TP_{i_1\ldots i_{l-1}}x$, where $i_1\ldots
i_{l-1}\in\{1,\ldots,m\}^{l-1}$ is a multi-index used for ease of
reference to our Lyapunov functions. We claim that we can let $G$
be the graph dual to the De~Bruijn graph of order $l-1$ on $m$
symbols. The LMIs associated to this graph are given by
\begin{equation}\label{eq:converse.thm.LMIs}
\begin{array}{rll}
P_{i_1i_2\ldots i_{l-2}i_{l-1}}&\succ& 0\ \ \ \  \forall i_1\ldots i_{l-1}\in\{1,\ldots,m\}^{l-1}\\
A_j^TP_{i_1i_2\ldots i_{l-2}i_{l-1}}A_j&\preceq&P_{i_2i_3\ldots
i_{l-1}j}\\
\ &\ & \forall i_1\ldots i_{l-1}\in\{1,\ldots,m\}^{l-1},\\
\ &\ & \forall j\in\{1,\ldots,m\}.
\end{array}
\end{equation}
The fact that $G$ is path-complete and that the LMIs imply
existence of a max-of-quadratics Lyapunov function follows from
Corollary~\ref{cor:max.of.quadratics}. The proof that these LMIs
satisfy the bound in (\ref{eq:bound.converse.thm.max}) is a
straightforward generalization of the proof of
Theorem~\ref{thm:HSCC.bound}. By the same arguments we have
\begin{equation}\label{eq:bound.of.CQ.l.steps}
\frac{1}{\sqrt[2l]{n}}\hat{\rho}_{\mathcal{V}^2}^{\frac{1}{l}}(\mathcal{A}^l)\leq\rho(\mathcal{A}).
\end{equation}
Suppose $x^TQx$ is a common quadratic Lyapunov function for the
matrices in $\mathcal{A}^l$; i.e., it satisfies
\begin{equation}\nonumber
\begin{array}{rll}
Q&\succ&0 \\
(A_{i_1}\ldots A_{i_l})^TQA_{i_1}\ldots A_{i_l}&\preceq&Q \quad
\forall i_1\ldots i_{l}\in\{1,\ldots,m\}^{l}.
\end{array}
\end{equation}
Then, it is easy to check that\footnote{The construction of the
Lyapunov function here is a special case of a general scheme for
constructing Lyapunov functions that are monotonically decreasing
from those that decrease only every few steps; see~\cite[p.
58]{AAA_MS_Thesis}.}
\begin{equation}\nonumber
\begin{array}{lll}
P_{i_1i_2\ldots i_{l-2}i_{l-1}}=Q+A_{i_{l-1}}^TQA_{i_{l-1}}&\ &\ \\
+(A_{i_{l-2}}A_{i_{l-1}})^TQ(A_{i_{l-2}}A_{i_{l-1}})+\cdots&\ &\ \\
+ (A_{i_1}A_{i_2}\ldots
A_{i_{l-2}}A_{i_{l-1}})^TQ(A_{i_1}A_{i_2}\ldots
A_{i_{l-2}}A_{i_{l-1}}), &\ &\ \\
i_1\ldots i_{l-1}\in\{1,\ldots,m\}^{l-1},&\ &\
\end{array}
\end{equation}
satisfy (\ref{eq:converse.thm.LMIs}). Hence,
\begin{equation}\nonumber
\hat{\rho}_{\mathcal{V}^2,
G}(\mathcal{A})\leq\hat{\rho}_{\mathcal{V}^2}^{\frac{1}{l}}(\mathcal{A}^l),
\end{equation}
and in view of (\ref{eq:bound.of.CQ.l.steps}) the claim is
established.
\end{proof}

%Don't think the next remark is accurate. So, I'm removing it for now.
%
%\begin{remark}Theorem~\ref{thm:converse.max.of.quadratics} establishes that
%for any fixed $l$, there is a polynomial time algorithm based on
%max-of-quadratics Lyapunov functions for the approximation of the
%JSR that achieves the approximation guarantee of
%(\ref{eq:bound.converse.thm.max}). In view of NP-hardness of
%approximation of the JSR~\cite{BlTi3}, the fact that the number of
%quadratic functions and the number of LMIs grow exponentially in
%$l$ is to be expected.
%\end{remark}

\begin{remark}
A converse Lyapunov theorem identical to
Theorem~\ref{thm:converse.max.of.quadratics} can be proven for the
min-of-quadratics Lyapunov functions. The only difference is that
the LMIs in (\ref{eq:converse.thm.LMIs}) would get replaced by the
ones corresponding to the dual graph of $G$.
\end{remark}

Our last theorem establishes approximation bounds for a family of
path-complete graphs with one single node but several edges
labeled with words of different lengths. Examples of such
path-complete graphs include graph $H_3$ in
Figure~\ref{fig:jsr.graphs} and graph $H_4$ in
Figure~\ref{fig:non.trivial.path.complete}.

\begin{theorem}
\label{thm-bound-codes} Let $\mathcal{A}$ be a set of matrices in
$\mathbb{R}^{n\times n}.$ Let $\tilde{G}\left(  \left\{ 1\right\}
,E\right)  $ be a path-complete graph, and $l$ be the length of
the shortest word in $\tilde{\mathcal{A}}=\left\{ L\left( e\right)
:e\in E\right\}.$ Then
$\hat{\rho}_{\mathcal{V}^{2}},_{\tilde{G}}(\mathcal{A})$ provides
an estimate of $\rho\left(  \mathcal{A}\right)  $ that satisfies
\[
\frac{1}{\sqrt[2l]{n}}\hat{\rho}_{\mathcal{V}^{2}},_{\tilde{G}}(\mathcal{A}%
)\leq\rho(\mathcal{A})\leq\hat{\rho}_{\mathcal{V}^{2}},_{\tilde{G}%
}(\mathcal{A}).
\]

\end{theorem}

\begin{proof}
The right inequality is obvious, we prove the left one. Since both
$\hat{\rho }_{\mathcal{V}^{2}},_{\tilde{G}}(\mathcal{A})$ and
$\rho$ are homogeneous in $\mathcal{A},$ we may assume, without
loss of generality, that $\hat{\rho
}_{\mathcal{V}^{2}},_{\tilde{G}}(\mathcal{A})=1$. Suppose for the
sake of contradiction that
\begin{equation}\label{eq:supposition-thm-words}\rho(\mathcal{A})<1/\sqrt[2l]{n}.\end{equation}
We will show that this implies that
$\hat{\rho}_{\mathcal{V}^{2}},_{\tilde{G}}(\mathcal{A})<1$.
Towards this goal, let us first prove that
$\rho(\tilde{\mathcal{A}})\leq\rho^l(\mathcal{A}).$
%new proof:
\aaa{ Indeed, if we had
$\rho(\tilde{\mathcal{A}})>\rho^{l}(\mathcal{A})$, then there
would exist\footnote{Here, we are appealing to the well-known fact
about the JSR of a general set of matrices $\mathcal{B}$:
$\rho(\mathcal{B})~=~\limsup_{k\rightarrow\infty}
\max_{B\in\mathcal{B}^k} \rho^\frac{1}{k}(B).$ See
e.g.~\cite[Chap. 1]{Raphael_Book}.} an integer $i$ and a product
$A_\sigma\in\tilde{\mathcal{A}}^i$ such that
\begin{equation}\label{eq:rho^1/i>rho^l}
\rho^{\frac{1}{i}}(A_\sigma)>\rho^{l}(\mathcal{A}).
\end{equation}
Since we also have $A_\sigma\in\mathcal{A}^j$ (for some $j\geq
il$), it follows that
\begin{equation}\label{eq:rho^1/j<rho}
\rho^{\frac{1}{j}}(A_\sigma)\leq\rho(\mathcal{A}).
\end{equation}
The inequality in (\ref{eq:rho^1/i>rho^l}) together with
$\rho(\mathcal{A})\leq 1$ gives
$$\rho^{\frac{1}{j}}(A_\sigma)>\rho^{\frac{il}{j}}(\mathcal{A})\geq\rho(\mathcal{A}).$$
But this contradicts (\ref{eq:rho^1/j<rho}). Hence we have shown
$$\rho(\tilde{\mathcal{A}})\leq \rho^{l}(\mathcal{A}).$$
}
%%%OLD PROOF
%It is a well known result (see \aaa{\cite[Sec.
%3.6]{Raphael_Book}}) that for any set of matrices $\mathcal{A}$
%with JSR $\rho,$ there exists an increasing polynomial \aaa{$p$},
%positive on $\mathbb{R}^{+},$ such that for any \rmj{$Y\in \mathbb{N},$ and for any} product $A_{\sigma}%
%\in\mathcal{A}^{Y},$
%\[
%\left\Vert A_{\sigma}\right\Vert \leq p(Y)\rho^{Y}.
%\]
%Let $L$ be the length of the largest word in
%$\tilde{\mathcal{A}},$ and consider a product of matrices
%$A_{\sigma}\in\tilde{\mathcal{A}}^{y}.$ Then $A_{\sigma}$
%corresponds to a product of matrices in $\mathcal{A}^{Y},$ for
%some $Y,$%
%\begin{equation}
%yl\leq Y\leq yL.\label{ThmBoundsEq1}%
%\end{equation}
%Thus, we have the following bound:
%\begin{align}
%\left\Vert A_{\sigma}\right\Vert  &  \leq p(Y)(\rho(\mathcal{A}))^{Y}%
%\nonumber\\
%&  \leq p(yL)(\rho(\mathcal{A})^{l})^{y},\label{ThmBoundsEq2}%
%\end{align}
%where (\ref{ThmBoundsEq2}) follows from $\rho(\mathcal{A})\leq1$
%and (\ref{ThmBoundsEq1}). Since $$\lim_{y\rightarrow\infty}\left(
%p(yL)\right) ^{1/y}=1,$$ the bound (\ref{ThmBoundsEq2}) on the
%norm of the products in $\tilde{\mathcal{A}}^{y},$ \rmj{together
%with the definition (\ref{eq:def.jsr}),} implies that
%$$\rho(\tilde{\mathcal{A}})\leq \rho(\mathcal{A})^{l}.$$
%%%END OF OLD PROOF
\rmj{Now, by our hypothesis (\ref{eq:supposition-thm-words})
above, we have that $ \rho(\tilde{\mathcal{A}})<1/\sqrt{n}.$}
Therefore, there
exists $\epsilon>0$ such that $\rho((1+\epsilon)\tilde{\mathcal{A}}%
)<1/\sqrt{n}.$ It then follows from (\ref{eq:CQ.bound}) that there
exists a common quadratic Lyapunov function for
$(1+\epsilon)\tilde{\mathcal{A}}.$ \aaa{Hence}, $\hat
{\rho}_{\mathcal{V}^{2}}((1+\epsilon)\tilde{\mathcal{A}})\leq1,$
which
immediately implies that $\hat{\rho}_{\mathcal{V}^{2}},_{\tilde{G}%
}(\mathcal{A})<1,$ a contradiction.
\end{proof}

\aaa{ A noteworthy immediate corollary of
Theorem~\ref{thm-bound-codes} (obtained by setting
$\tilde{\mathcal{A}}=\bigcup_{t=r}^k \mathcal{A}^t)$ is the
following: If $\rho(\mathcal{A})<\frac{1}{\sqrt[2r]{n}}$, then
there exists a quadratic Lyapunov function that decreases
simultaneously for all products of lengths $r,r+1,\ldots,r+k$, for
any desired value of $k$. Note that this fact is obvious for
$r=1$, but nonobvious for $r\geq 2$.}

\section{Conclusions and future directions}\label{sec:conclusions.future.directions}

We introduced the framework of path-complete graph Lyapunov
functions for the formulation of semidefinite programming based
algorithms for approximating the joint spectral radius (or
equivalently establishing absolute asymptotic stability of an
arbitrary switched linear system). We defined the notion of a
path-complete graph, which was inspired by concepts in automata
theory. We showed that every path-complete graph gives rise to a
technique for the approximation of the JSR. This provided a
unifying framework that includes many of the previously proposed
techniques and also introduces new ones. \rmj{\aaa{(In fact}, all
families of LMIs that we are aware of \aaa{are} particular cases
of our method.\aaa{)}} \aaa{We shall also emphasize that although
we focused on switched \emph{linear} systems because of our
interest in the JSR, the analysis technique of multiple Lyapunov
functions on path-complete graphs is clearly valid for switched
\emph{nonlinear} systems as well.}

We compared the quality of the bound obtained from certain classes
of path-complete graphs, including all path-complete graphs with
two nodes on an alphabet of two matrices, and also a certain
family of dual path-complete graphs. We proposed a specific class
of such graphs that appear to work particularly well in practice
and proved that the bound obtained from these graphs is invariant
under transposition of the matrices and is always within a
multiplicative factor of $1/\sqrt[4]{n}$ from the true JSR.
Finally, we presented two converse Lyapunov theorems, one for the
well-known methods of minimum and maximum-of-quadratics Lyapunov
functions, and the other for a new class of methods that propose
the use of a common quadratic Lyapunov function for a set of words
of possibly different lengths.

%AAA: removed the following lines to save room:
%These theorems yield explicit and systematic constructions of
%semidefinite programs that achieve any desired accuracy of
%approximation.

%Amirali: I deferred the following comment of Raph to the journal version.
%\rmj{In their greatest generality, our methods can be applied by using conic programming with any cone (see \cite{protasov-jungers-blondel09}). They improve the quality of the bounds when the function is obtained thanks to optimization on the cone of SDP matrices, on the cone of Sum-Of-Squares polynomials, and on the positive orthant.  We conjecture that for any cone, our methods provide an improvement over classical methods. }

We believe the methodology proposed in this chapter should
straightforwardly extend to the case of \emph{constrained
switching} by requiring the graphs to have a path not for all the
words, but only the words allowed by the constraints on the
switching. A rigorous treatment of this idea is left for future
work.

Vincent Blondel showed that when the underlying automaton is not
deterministic, checking path-completeness of a labeled directed
graph is an NP-hard problem (personal communication). In general,
the problem of deciding whether a non-deterministic finite
automaton accepts all finite words is known to be
PSPACE-complete~\cite[p. 265]{GareyJohnson_Book}. However, we are
yet to investigate whether the same is true for automata arising
from path-complete graphs which have a little more structure. At
the moment, the NP-hardness proof of Blondel remains as the
strongest negative result we have on this problem. Of course, the
step of checking path-completeness of a graph is done offline and
prior to the run of our algorithms for approximating the JSR.
Therefore, while checking path-completeness is in general
difficult, the approximation algorithms that we presented indeed
run in polynomial time since they work with a fixed (a priori
chosen) path-complete graph. Nevertheless, the question on
complexity of checking path-completeness is interesting in many
other settings, e.g., when deciding whether a given set of
Lyapunov inequalities imply stability of an arbitrary switched
system.

Some other interesting questions that can be explored in the
future are the following. What are some other classes of
path-complete graphs that lead to new techniques for proving
stability of switched systems? How can we compare the performance
of different path-complete graphs in a systematic way? Given a set
of matrices, a class of Lyapunov functions, and a fixed size for
the graph, can we efficiently come up with the least conservative
topology of a path-complete graph? Within the framework that we
proposed, do all the Lyapunov inequalities that prove stability
come from path-complete graphs? \aaa{What are the analogues of the
results of this chapter for continuous time switched systems?} To
what extent do the results carry over to the synthesis (controller
design) problem for switched systems? These questions and several
others show potential for much follow-up work on path-complete
graph Lyapunov functions.

%%%%%%%%%%%%%%
%% Back Matter
%%%%%%%%%%%%%%

%% appendices
%\appendix
%\include{appendix_directSampler}
%\include{appendix_blockedSampler}
%\include{appendix_hypers}
%\include{appendix_messagePassing}
%\include{appendix_HDPHMMKF}
%\include{appendix_dynamicParamPost}

% bibliography
\cleardoublepage
\phantomsection
\addcontentsline{toc}{chapter}{References}
\bibliographystyle{abbrv}
\bibliography{pablo_amirali}

\def\cprime{$'$}
\begin{thebibliography}{100}

\bibitem{AAA_MS_Thesis}
A.~A. Ahmadi.
\newblock Non-monotonic {L}yapunov functions for stability of nonlinear and
  switched systems: theory and computation.
\newblock Master's thesis, Massachusetts Institute of Technology, June 2008.
  Available from \texttt{http://dspace.mit.edu/handle/1721.1/44206}.

\bibitem{AAA_GB_PP_Convex_ternary_quartics}
A.~A. Ahmadi, G.~Blekherman, and P.~A. Parrilo.
\newblock Convex ternary quartics are sos-convex.
\newblock In preparation, 2011.

\bibitem{HSCC_JSR_Path_complete}
A.~A. Ahmadi, R.~Jungers, P.~A. Parrilo, and M.~Roozbehani.
\newblock Analysis of the joint spectral radius via {L}yapunov functions on
  path-complete graphs.
\newblock In {\em Hybrid Systems: Computation and Control 2011}, Lecture Notes
  in Computer Science. Springer, 2011.

\bibitem{AAA_MK_PP_CDC11_no_Poly_Lyap}
A.~A. Ahmadi, M.~Krstic, and P.~A. Parrilo.
\newblock A globally asymptotically stable polynomial vector field with no
  polynomial {L}yapunov function.
\newblock In {\em Proceedings of the 50$^{th}$ IEEE Conference on Decision and
  Control}, 2011.

\bibitem{NPhard_Convexity_arxiv}
A.~A. Ahmadi, A.~Olshevsky, P.~A. Parrilo, and J.~N. Tsitsiklis.
\newblock {NP}-hardness of deciding convexity of quartic polynomials and
  related problems.
\newblock {\em Mathematical Programming}, 2011.
\newblock Accepted for publication. Online version available at
  arXiv:.1012.1908.

\bibitem{AAA_PP_CDC08_non_monotonic}
A.~A. Ahmadi and P.~A. Parrilo.
\newblock Non-monotonic {L}yapunov functions for stability of discrete time
  nonlinear and switched systems.
\newblock In {\em Proceedings of the 47$^{th}$ IEEE Conference on Decision and
  Control}, 2008.

\bibitem{AAA_PP_CDC09_HessianNotFactor}
A.~A. Ahmadi and P.~A. Parrilo.
\newblock A positive definite polynomial {H}essian that does not factor.
\newblock In {\em Proceedings of the 48$^{th}$ IEEE Conference on Decision and
  Control}, 2009.

\bibitem{AAA_PP_CDC10_algeb_convex}
A.~A. Ahmadi and P.~A. Parrilo.
\newblock On the equivalence of algebraic conditions for convexity and
  quasiconvexity of polynomials.
\newblock In {\em Proceedings of the 49$^{th}$ IEEE Conference on Decision and
  Control}, 2010.

\bibitem{AAA_PP_table_sos-convexity}
A.~A. Ahmadi and P.~A. Parrilo.
\newblock A complete characterization of the gap between convexity and
  sos-convexity.
\newblock In preparation, 2011.

\bibitem{AAA_PP_CDC11_converseSOS_Lyap}
A.~A. Ahmadi and P.~A. Parrilo.
\newblock Converse results on existence of sum of squares {L}yapunov functions.
\newblock In {\em Proceedings of the 50$^{th}$ IEEE Conference on Decision and
  Control}, 2011.

\bibitem{AAA_PP_not_sos_convex_journal}
A.~A. Ahmadi and P.~A. Parrilo.
\newblock A convex polynomial that is not sos-convex.
\newblock {\em Mathematical Programming}, 2011.
\newblock DOI: 10.1007/s10107-011-0457-z.

\bibitem{AAA_PP_ACC11_Lyap_High_Deriv}
A.~A. Ahmadi and P.~A. Parrilo.
\newblock On higher order derivatives of {L}yapunov functions.
\newblock In {\em Proceedings of the 2011 American Control Conference}, 2011.

\bibitem{Ando98}
T.~Ando and M.-H. Shih.
\newblock Simultaneous contractibility.
\newblock {\em SIAM Journal on Matrix Analysis and Applications}, 19:487--498,
  1998.

\bibitem{Stabilize_Homog}
A.~Andreini, A.~Bacciotti, and G.~Stefani.
\newblock Global stabilizability of homogeneous vector fields of odd degree.
\newblock {\em Systems and Control Letters}, 10(4):251--256, 1988.

\bibitem{Angeli.homog.switched}
D.~Angeli.
\newblock A note on stability of arbitrarily switched homogeneous systems.
\newblock 1999.
\newblock Preprint.

\bibitem{Arnold_algebraic_unsolve}
V.~I. Arnold.
\newblock Algebraic unsolvability of the problem of {L}yapunov stability and
  the problem of topological classification of singular points of an analytic
  system of differential equations.
\newblock {\em Functional Analysis and its Applications}, 4(3):173--180.
\newblock Translated from Funktsional'nyi Analiz i Ego Prilozheniya (1970).

\bibitem{Arnold_Problems_for_Math}
V.~I. Arnold.
\newblock Problems of present day mathematics{, XVII (Dynamical systems and
  differential equations)}.
\newblock {\em Proc. Symp. Pure Math.}, 28(59), 1976.

\bibitem{Quasiconcave_programming}
K.~J. Arrow and A.~C. Enthoven.
\newblock Quasi-concave programming.
\newblock {\em Econometrica}, 29(4):779--800, 1961.

\bibitem{Artin_Hilbert17}
E.~Artin.
\newblock \"{U}ber die {Z}erlegung {D}efiniter {F}unktionen in {Q}uadrate.
\newblock {\em Hamb. Abh.}, 5:100--115, 1927.

\bibitem{SOS_KYP}
E.~M. Aylward, S.~M. Itani, and P.~A. Parrilo.
\newblock Explicit {SOS} decomposition of univariate polynomial matrices and
  the {K}alman-{Y}akubovich-{P}opov lemma.
\newblock In {\em Proceedings of the 46$^{th}$ IEEE Conference on Decision and
  Control}, 2007.

\bibitem{Erin_Pablo_Contraction}
E.~M. Aylward, P.~A. Parrilo, and J.~J.~E. Slotine.
\newblock Stability and robustness analysis of nonlinear systems via
  contraction metrics and {SOS} programming.
\newblock {\em Automatica}, 44(8):2163--2170, 2008.

\bibitem{Bacciotti.Rosier.Liapunov.Book}
A.~Bacciotti and L.~Rosier.
\newblock {\em {L}iapunov Functions and Stability in Control Theory}.
\newblock Springer, 2005.

\bibitem{Baillieul_Homog_geometry}
J.~Baillieul.
\newblock The geometry of homogeneous polynomial dynamical systems.
\newblock {\em Nonlinear analysis, Theory, Methods and Applications},
  4(5):879--900, 1980.

\bibitem{Algo_real_algeb_geom_Book}
S.~Basu, R.~Pollack, and M.~F. Roy.
\newblock {\em Algorithms in Real Algebraic Geometry}, volume~10 of {\em
  Algorithms and Computation in Mathematics}.
\newblock Springer-Verlag, Berlin, second edition, 2006.

\bibitem{NLP_Book_Bazaraa}
M.~S. Bazaraa, H.~D. Sherali, and C.~M. Shetty.
\newblock {\em Nonlinear Programming}.
\newblock Wiley-Interscience, 2006.
\newblock Third edition.

\bibitem{Blekherman_convex_not_sos}
G.~Blekherman.
\newblock Convex forms that are not sums of squares.
\newblock arXiv:0910.0656, 2009.

\bibitem{Blekherman_nonnegative_and_sos}
G.~Blekherman.
\newblock Nonnegative polynomials and sums of squares.
\newblock arXiv:1010.3465, 2010.

\bibitem{Symmetric_quartics_sos}
G.~Blekherman and C.~B. Riener.
\newblock Symmetric sums of squares in degree four.
\newblock In preparation, 2011.

\bibitem{Deciding_stab_mortal_PWA}
V.~D. Blondel, O.~Bournez, P.~Koiran, C.~H. Papadimitriou, and J.~N.
  Tsitsiklis.
\newblock Deciding stability and mortality of piecewise affine systems.
\newblock {\em Theoretical Computer Science}, 255(1-2):687--696, 2001.

\bibitem{TsiLinSat}
V.~D. Blondel, O.~Bournez, P.~Koiran, and J.~N. Tsitsiklis.
\newblock The stability of saturated linear dynamical systems is undecidable.
\newblock {\em J. Comput. System Sci.}, 62(3):442--462, 2001.

\bibitem{BlNes05}
V.~D. Blondel and Y.~Nesterov.
\newblock Computationally efficient approximations of the joint spectral
  radius.
\newblock {\em SIAM J. Matrix Anal. Appl.}, 27(1):256--272, 2005.

\bibitem{BlNT04}
V.~D. Blondel, Y.~Nesterov, and J.~Theys.
\newblock On the accuracy of the ellipsoidal norm approximation of the joint
  spectral radius.
\newblock {\em Linear Algebra Appl.}, 394:91--107, 2005.

\bibitem{BlTi_complexity_3classes}
V.~D. Blondel and J.~N. Tsitsiklis.
\newblock Overview of complexity and decidability results for three classes of
  elementary nonlinear systems.
\newblock In {\em Learning, Control and Hybrid Systems}, pages 46--58.
  Springer, 1998.

\bibitem{BlTi_stab_contr_hybrid}
V.~D. Blondel and J.~N. Tsitsiklis.
\newblock Complexity of stability and controllability of elementary hybrid
  system.
\newblock {\em Automatica}, 35:479--489, 1999.

\bibitem{BlTi2}
V.~D. Blondel and J.~N. Tsitsiklis.
\newblock The boundedness of all products of a pair of matrices is undecidable.
\newblock {\em Systems and Control Letters}, 41:135--140, 2000.

\bibitem{BlTi1}
V.~D. Blondel and J.~N. Tsitsiklis.
\newblock A survey of computational complexity results in systems and control.
\newblock {\em Automatica}, 36(9):1249--1274, 2000.

\bibitem{Survey_CT_Computation}
O.~Bournez and M.~L. Campagnolo.
\newblock A survey on continuous time computations.
\newblock {\em New Computational Paradigms}, 4:383--423, 2008.

\bibitem{BoydBook}
S.~Boyd and L.~Vandenberghe.
\newblock {\em Convex Optimization}.
\newblock Cambridge University Press, 2004.

\bibitem{multiple_lyap_Branicky}
M.~S. Branicky.
\newblock Multiple {L}yapunov functions and other analysis tools for switched
  and hybrid systems.
\newblock {\em IEEE Transactions on Automatic Control}, 43(4):475--482, 1998.

\bibitem{Canny_PSPACE}
J.~Canny.
\newblock Some algebraic and geometric computations in {PSPACE}.
\newblock In {\em Proceedings of the {T}wentieth {A}nnual {ACM} {S}ymposium on
  {T}heory of {C}omputing}, pages 460--469, New York, 1988. ACM.

\bibitem{AutControlSpecial_PositivePolys}
G.~Chesi and {D. Henrion (editors)}.
\newblock Special issue on positive polynomials in control.
\newblock {\em IEEE Trans. Automat. Control}, 54(5), 2009.

\bibitem{Chest.et.al.sos.robust.stability}
G.~Chesi, A.~Garulli, A.~Tesi, and A.~Vicino.
\newblock Polynomially parameter-dependent {L}yapunov functions for robust
  stability of polytopic systems: an {LMI} approach.
\newblock {\em IEEE Trans. Automat. Control}, 50(3):365--370, 2005.

\bibitem{Chesi_book}
G.~Chesi, A.~Garulli, A.~Tesi, and A.~Vicino.
\newblock {\em Homogeneous polynomial forms for robustness analysis of
  uncertain systems}.
\newblock Number 390 in Lecture Notes in Control and Information Sciences.
  Springer, 2009.

\bibitem{Chesi_Hung_journal}
G.~Chesi and Y.~S. Hung.
\newblock Establishing convexity of polynomial {L}yapunov functions and their
  sublevel sets.
\newblock {\em IEEE Trans. Automat. Control}, 53(10):2431--2436, 2008.

\bibitem{Choi_Biquadratic}
M.~D. Choi.
\newblock Positive semidefinite biquadratic forms.
\newblock {\em Linear Algebra and its Applications}, 12:95--100, 1975.

\bibitem{Choi_Lam_extremalPSDforms}
M.~D. Choi and T.~Y. Lam.
\newblock Extremal positive semidefinite forms.
\newblock {\em Math. Ann.}, 231:1--18, 1977.

\bibitem{CLRrealzeros}
M.~D. Choi, T.-Y. Lam, and B.~Reznick.
\newblock Real zeros of positive semidefinite forms. {I}.
\newblock {\em Math. Z.}, 171(1):1--26, 1980.

\bibitem{pseudoconvex_nonnegative_vars}
R.~W. Cottle and J.~A. Ferland.
\newblock On pseudo-convex functions of nonnegative variables.
\newblock {\em Math. Programming}, 1(1):95--101, 1971.

\bibitem{Criteria_comparison_quasiconvexity_pseudoconvexity}
J.~P. Crouzeix and J.~A. Ferland.
\newblock Criteria for quasiconvexity and pseudoconvexity: relationships and
  comparisons.
\newblock {\em Math. Programming}, 23(2):193--205, 1982.

\bibitem{Crusius_thesis}
C.~A.~R. Crusius.
\newblock {\em Automated analysis of convexity properties of nonlinear
  programs}.
\newblock PhD thesis, Department of Electrical Engineering, Stanford
  University, 2002.

\bibitem{Costa_Doria_undecidabiliy}
N.~C.~A. {da Costa} and F.~A. Doria.
\newblock On {Arnold's Hilbert} symposium problems.
\newblock In {\em Computational Logic and Proof Theory}, volume 713 of {\em
  Lecture Notes in Computer Science}, pages 152--158. Springer, 1993.

\bibitem{Costa_Doria_Hopf}
N.~C.~A. {da Costa} and F.~A. Doria.
\newblock Undecidable {H}opf bifurcation with undecidable fixed point.
\newblock {\em International Journal of Theoretical Physics,}, 33(9), 1994.

\bibitem{daafouzbernussou}
J.~Daafouz and J.~Bernussou.
\newblock Parameter dependent {L}yapunov functions for discrete time systems
  with time varying parametric uncertainties.
\newblock {\em Systems and Control Letters}, 43(5):355--359, 2001.

\bibitem{deKlerk_complexity_survey}
E.~de~Klerk.
\newblock The complexity of optimizing over a simplex, hypercube or sphere: a
  short survey.
\newblock {\em CEJOR Cent. Eur. J. Oper. Res.}, 16(2):111--125, 2008.

\bibitem{Monique_Etienne_Convex}
E.~de~Klerk and M.~Laurent.
\newblock On the {L}asserre hierarchy of semidefinite programming relaxations
  of convex polynomial optimization problems.
\newblock Available at
  \texttt{http://www.optimization-online.org/DB-FILE/2010/11/2800.pdf}, 2010.

\bibitem{deKlerk_StableSet_copositive}
E.~de~Klerk and D.~V. Pasechnik.
\newblock Approximation of the stability number of a graph via copositive
  programming.
\newblock {\em SIAM Journal on Optimization}, 12(4):875--892, 2002.

\bibitem{Even_quartics_4vars_sos}
P.~H. Diananda.
\newblock On non-negative forms in real variables some or all of which are
  non-negative.
\newblock {\em Proceedings of the {C}ambridge philosophical society},
  58:17--25, 1962.

\bibitem{copositivity_NPhard}
P.~Dickinson and L.~Gijben.
\newblock On the computational complexity of membership problems for the
  completely positive cone and its dual.
\newblock Available at
  \texttt{http://www.optimization-online.org/DB-FILE/2011/05/3041.pdf}, 2011.

\bibitem{Pablo_Sep_Entang_States}
A.~C. Doherty, P.~A. Parrilo, and F.~M. Spedalieri.
\newblock Distinguishing separable and entangled states.
\newblock {\em Physical {R}eview {L}etters}, 88(18), 2002.

\bibitem{Matrix_theoretic_quasiconvexity}
J.~A. Ferland.
\newblock Matrix-theoretic criteria for the quasiconvexity of twice
  continuously differentiable functions.
\newblock {\em Linear Algebra Appl.}, 38:51--63, 1981.

\bibitem{GareyJohnson_Book}
M.~R. Garey and D.~S. Johnson.
\newblock {\em Computers and Intractability}.
\newblock W. H. Freeman and Co., San Francisco, Calif., 1979.

\bibitem{Symmetry_groups_Gatermann_Pablo}
K.~Gatermann and P.~A. Parrilo.
\newblock Symmetry groups, semidefinite programs, and sums of squares.
\newblock {\em Journal of Pure and Applied Algebra}, 192:95--128, 2004.

\bibitem{dual_LMI_diff_inclusions}
R.~Goebel, T.~Hu, and A.~R. Teel.
\newblock Dual matrix inequalities in stability and performance analysis of
  linear differential/difference inclusions.
\newblock In {\em Current Trends in Nonlinear Systems and Control}, pages
  103--122. 2006.

\bibitem{convex_conjugate_Lyap}
R.~Goebel, A.~R. Teel, T.~Hu, and Z.~Lin.
\newblock Conjugate convex {L}yapunov functions for dual linear differential
  inclusions.
\newblock {\em IEEE Transactions on Automatic Control}, 51(4):661--666, 2006.

\bibitem{Bounded_Defined_ODE_Undecidable}
D.~S. Graca, J.~Buescu, and M.~L. Campagnolo.
\newblock Boundedness of the domain of definition is undecidable for polynomial
  {ODEs}.
\newblock In {\em Proceedings of the Fourth International Conference of
  Computability and Complexity in Analysis}, 2007.

\bibitem{cvx}
M.~Grant and S.~Boyd.
\newblock {CVX}: Matlab software for disciplined convex programming, version
  1.21.
\newblock \texttt{http://cvxr.com/cvx}, May 2010.

\bibitem{GraphTheory_Handbook}
J.~L. Gross and J.~Yellen.
\newblock {\em Handbook of Graph Theory (Discrete Mathematics and Its
  Applications)}.
\newblock CRC Press, 2003.

\bibitem{homog.feedback}
L.~Gr\"{u}ne.
\newblock Homogeneous state feedback stabilization of homogeneous systems.
\newblock In {\em Proceedings of the 39$^{th}$ IEEE Conference on Decision and
  Control}, 2000.

\bibitem{complexity_simplex_convexity}
B.~Guo.
\newblock On the difficulty of deciding the convexity of polynomials over
  simplexes.
\newblock {\em International Journal of Computational Geometry and
  Applications}, 6(2):227--229, 1996.

\bibitem{Gurvits_quantum_entag_hard}
L.~Gurvits.
\newblock Classical deterministic complexity of {E}dmonds' problem and quantum
  entanglement.
\newblock In {\em Proceedings of the {T}hirty-{F}ifth {A}nnual {ACM}
  {S}ymposium on {T}heory of {C}omputing}, pages 10--19 (electronic), New York,
  2003. ACM.

\bibitem{Stability_number_SOS}
N.~Gvozdenovi\'c and M.~Laurent.
\newblock Semidefinite bounds for the stability number of a graph via sums of
  squares of polynomials.
\newblock {\em Mathematical Programming}, 110(1):145--173, 2007.

\bibitem{Hahn_stability_book}
W.~Hahn.
\newblock {\em Stability of Motion}.
\newblock Springer-Verlag, New York, 1967.

\bibitem{Undecidability_vec_fields_survey}
E.~Hainry.
\newblock Decidability and undecidability in dynamical systems.
\newblock Research report. Available at
  \texttt{\texttt{http://hal.inria.fr/inria-00429965/PDF/dynsys.pdf}}, 2009.

\bibitem{Cubic_Homog_Planar}
M.~A. Hammamia and H.~Jerbia.
\newblock The stabilization of homogeneous cubic vector fields in the plane.
\newblock {\em Applied Mathematics Letters}, 7(4):95--99, 1994.

\bibitem{Helton_Nie_SDP_repres_2}
J.~W. Helton and J.~Nie.
\newblock Semidefinite representation of convex sets.
\newblock {\em Mathematical Programming}, 122(1, Ser. A):21--64, 2010.

\bibitem{PositivePolyInControlBook}
D.~Henrion and A.~Garulli, editors.
\newblock {\em Positive polynomials in control}, volume 312 of {\em Lecture
  Notes in Control and Information Sciences}.
\newblock Springer, 2005.

\bibitem{Hilbert_1888}
D.~Hilbert.
\newblock \"{U}ber die {D}arstellung {D}efiniter {F}ormen als {S}umme von
  {F}ormenquadraten.
\newblock {\em Math. Ann.}, 32, 1888.

\bibitem{Hopcroft_Motwani_Ullman_automata_Book}
J.~E. Hopcroft, R.~Motwani, and J.~D. Ullman.
\newblock {\em Introduction to Automata Theory, Languages, and Computation}.
\newblock Addison Wesley, 2001.

\bibitem{HJ_Matrix_Analysis_Book}
R.~A. Horn and C.~R. Johnson.
\newblock {\em Matrix Analysis}.
\newblock Cambridge University Press, 1995.

\bibitem{composite_Lyap2}
T.~Hu and Z.~Lin.
\newblock Absolute stability analysis of discrete-time systems with composite
  quadratic {L}yapunov functions.
\newblock {\em IEEE Transactions on Automatic Control}, 50(6):781--797, 2005.

\bibitem{composite_Lyap}
T.~Hu, L.~Ma, and Z.~Li.
\newblock On several composite quadratic {L}yapunov functions for switched
  systems.
\newblock In {\em Proceedings of the 45$^{th}$ IEEE Conference on Decision and
  Control}, 2006.

\bibitem{hu-ma-lin}
T.~Hu, L.~Ma, and Z.~Lin.
\newblock Stabilization of switched systems via composite quadratic functions.
\newblock {\em IEEE Transactions on Automatic Control}, 53(11):2571 -- 2585,
  2008.

\bibitem{ControlAppsSOS}
Z.~Jarvis-Wloszek, R.~Feeley, W.~Tan, K.~Sun, and A.~Packard.
\newblock Some controls applications of sum of squares programming.
\newblock In {\em Proceedings of the 42$^{th}$ IEEE Conference on Decision and
  Control}, pages 4676--4681, 2003.

\bibitem{QuadraticIP_undecidable}
R.~G. Jeroslow.
\newblock There cannot be an algorithm for integer programming with quadratic
  constraints.
\newblock {\em Operations Research}, 21(1):221--224, 1973.

\bibitem{Raphael_Book}
R.~Jungers.
\newblock {\em The joint spectral radius: theory and applications}, volume 385
  of {\em Lecture Notes in Control and Information Sciences}.
\newblock Springer, 2009.

\bibitem{Khalil:3rd.Ed}
H.~Khalil.
\newblock {\em Nonlinear systems}.
\newblock Prentice Hall, 2002.
\newblock Third edition.

\bibitem{Kharitonov_interval_poly}
V.~L. Kharitonov.
\newblock Asymptotic stability of an equilibrium position of a family of
  systems of linear differential equations.
\newblock {\em Differential Equations}, 14:1483--1485, 1979.

\bibitem{Kojima_SOS_matrix}
M.~Kojima.
\newblock Sums of squares relaxations of polynomial semidefinite programs.
\newblock {\em Research report B-397, Dept. of Mathematical and Computing
  Sciences. Tokyo Institute of Technology}, 2003.

\bibitem{Lasserre_Jensen_inequality}
J.~B. Lasserre.
\newblock Convexity in semialgebraic geometry and polynomial optimization.
\newblock {\em SIAM Journal on Optimization}, 19(4):1995--2014, 2008.

\bibitem{Lasserre_Convex_Positive}
J.~B. Lasserre.
\newblock {Representation of nonnegative convex polynomials}.
\newblock {\em Archiv der Mathematik}, 91(2):126--130, 2008.

\bibitem{Lasserre_set_convexity}
J.~B. Lasserre.
\newblock Certificates of convexity for basic semi-algebraic sets.
\newblock {\em Applied Mathematics Letters}, 23(8):912--916, 2010.

\bibitem{Testing_convexity_economics}
L.~J. Lau.
\newblock Testing and imposing monotonicity, convexity and quasiconvexity
  constraints.
\newblock In M.~A. Fuss and D.~L. McFadden, editors, {\em Production Economics:
  A Dual Approach to Theory and Applications}, pages 409--453. North-Holland
  Publishing Company, 1978.

\bibitem{LeeD06}
J.~W. Lee and G.~E. Dullerud.
\newblock Uniform stabilization of discrete-time switched and {M}arkovian jump
  linear systems.
\newblock {\em Automatica}, 42(2):205--218, 2006.

\bibitem{Lenstra_IP}
H.~W. Lenstra.
\newblock Integer programming with a fixed number of variables.
\newblock {\em Mathematics of Operations Research}, 8(4):538--548, 1983.

\bibitem{switched_system_survey}
H.~Lin and P.~J. Antsaklis.
\newblock Stability and stabilizability of switched linear systems: a short
  survey of recent results.
\newblock In {\em Proceedings of IEEE International Symposium on Intelligent
  Control}, 2005.

\bibitem{Lind_Marcus_symbolic_Book}
D.~Lind and B.~Marcus.
\newblock {\em An Introduction to Symbolic Dynamics and Coding}.
\newblock Cambridge University Press, 1995.

\bibitem{Ling_et_al_Biquadratic}
C.~Ling, J.~Nie, L.~Qi, and Y.~Ye.
\newblock Biquadratic optimization over unit spheres and semidefinite
  programming relaxations.
\newblock {\em SIAM Journal on Optimization}, 20(3):1286--1310, 2009.

\bibitem{yalmip}
J.~L\"ofberg.
\newblock Yalmip : A toolbox for modeling and optimization in {MATLAB}.
\newblock In {\em Proceedings of the CACSD Conference}, 2004.
\newblock Available from
  \texttt{http://control.ee.ethz.ch/\~{}joloef/yalmip.php}.

\bibitem{PhD:Lyapunov}
A.~M. Lyapunov.
\newblock {\em General problem of the stability of motion}.
\newblock PhD thesis, Kharkov Mathematical Society, 1892.
\newblock In Russian.

\bibitem{convex_fitting}
A.~Magnani, S.~Lall, and S.~Boyd.
\newblock Tractable fitting with convex polynomials via sum of squares.
\newblock In {\em Proceedings of the 44$^{th}$ IEEE Conference on Decision and
  Control}, 2005.

\bibitem{Mangasarian_Pseudoconvex_fns}
O.~L. Mangasarian.
\newblock Pseudo-convex functions.
\newblock {\em J. Soc. Indust. Appl. Math. Ser. A Control}, 3:281--290, 1965.

\bibitem{NLP_Book_Mangasarian}
O.~L. Mangasarian.
\newblock {\em Nonlinear Programming}.
\newblock SIAM, 1994.
\newblock First published in 1969 by the McGraw-Hill Book Company, New York.

\bibitem{switch.common.poly.Lyap}
P.~Mason, U.~Boscain, and Y.~Chitour.
\newblock Common polynomial {L}yapunov functions for linear switched systems.
\newblock {\em SIAM Journal on Optimization and Control}, 45(1), 2006.

\bibitem{Second_order_pseudoconvexity}
P.~Mereau and J.~G. Paquet.
\newblock Second order conditions for pseudo-convex functions.
\newblock {\em SIAM J. Appl. Math.}, 27:131--137, 1974.

\bibitem{pyatnitskiy_max_lyap}
A.~Molchanov and Y.~Pyatnitskiy.
\newblock Criteria of asymptotic stability of differential and difference
  inclusions encountered in control theory.
\newblock {\em Systems and Control Letters}, 13:59--64, 1989.

\bibitem{homog.systems}
L.~Moreau, D.~Aeyels, J.~Peuteman, and R.~Sepulchre.
\newblock Homogeneous systems: stability, boundedness and duality.
\newblock In {\em Proceedings of the 14th Symposium on Mathematical Theory of
  Networks and Systems}, 2000.

\bibitem{MotzkinSOS}
T.~S. Motzkin.
\newblock The arithmetic-geometric inequality.
\newblock In {\em Inequalities ({P}roc. {S}ympos. {W}right-{P}atterson {A}ir
  {F}orce {B}ase, {O}hio, 1965)}, pages 205--224. Academic Press, New York,
  1967.

\bibitem{Motzkin_Straus}
T.~S. Motzkin and E.~G. Straus.
\newblock Maxima for graphs and a new proof of a theorem of {T}ur\'{a}n.
\newblock {\em Canadian Journal of Mathematics}, 17:533--540, 1965.

\bibitem{nonnegativity_NP_hard}
K.~G. Murty and S.~N. Kabadi.
\newblock Some {NP}-complete problems in quadratic and nonlinear programming.
\newblock {\em Mathematical Programming}, 39:117--129, 1987.

\bibitem{Nemirovskii_interval_matrix_NPhard}
A.~Nemirovskii.
\newblock Several {NP}-hard problems arising in robust stability analysis.
\newblock {\em Mathematics of Control, Signals, and Systems}, 6:99--105, 1993.

\bibitem{NN}
Y.~E. Nesterov and A.~Nemirovski.
\newblock {\em Interior point polynomial methods in convex programming},
  volume~13 of {\em Studies in Applied Mathematics}.
\newblock {SIAM}, {Philadelphia, PA}, 1994.

\bibitem{Nie_PMI_SDP_repres.}
J.~Nie.
\newblock Polynomial matrix inequality and semidefinite representation.
\newblock arXiv:0908.0364., 2009.

\bibitem{Gradient_Ideal_SOS}
J.~Nie, J.~Demmel, and B.~Sturmfels.
\newblock Minimizing polynomials via sum of squares over the gradient ideal.
\newblock {\em Mathematical Programming}, 106(3, Ser. A):587--606, 2006.

\bibitem{PapP02}
A.~Papachristodoulou and S.~Prajna.
\newblock On the construction of {L}yapunov functions using the sum of squares
  decomposition.
\newblock In {\em IEEE Conference on Decision and Control}, 2002.

\bibitem{Papadimitriou_IP}
C.~H. Papadimitriou.
\newblock On the complexity of integer programming.
\newblock {\em Journal of the {ACM}}, 28:765--768, 1981.

\bibitem{Pardalos_Book_Complexity_num_opt}
P.~M. Pardalos, editor.
\newblock {\em Complexity in Numerical Optimization}.
\newblock World Scientific Publishing Co. Inc., River Edge, NJ, 1993.

\bibitem{open_complexity}
P.~M. Pardalos and S.~A. Vavasis.
\newblock Open questions in complexity theory for numerical optimization.
\newblock {\em Mathematical Programming}, 57(2):337--339, 1992.

\bibitem{PhD:Parrilo}
P.~A. Parrilo.
\newblock {\em Structured semidefinite programs and semialgebraic geometry
  methods in robustness and optimization}.
\newblock PhD thesis, California Institute of Technology, May 2000.

\bibitem{sdprelax}
P.~A. Parrilo.
\newblock Semidefinite programming relaxations for semialgebraic problems.
\newblock {\em Mathematical Programming}, 96(2, Ser. B):293--320, 2003.

\bibitem{Pablo_poly_games}
P.~A. Parrilo.
\newblock Polynomial games and sum of squares optimization.
\newblock In {\em Proceedings of the 45$^{th}$ IEEE Conference on Decision and
  Control}, 2006.

\bibitem{PabloLyap}
P.~A. Parrilo.
\newblock On a decomposition of multivariable forms via {LMI} methods.
\newblock {\em American Control Conference, 2000.}, 1(6):322--326 vol.1, Sep
  2000.

\bibitem{Pablo_Jadbabaie_JSR_journal}
P.~A. Parrilo and A.~Jadbabaie.
\newblock Approximation of the joint spectral radius using sum of squares.
\newblock {\em Linear Algebra Appl.}, 428(10):2385--2402, 2008.

\bibitem{Pablo_Geometry_Packing_SOS}
P.~A. Parrilo and R.~Peretz.
\newblock An inequality for circle packings proved by semidefinite programming.
\newblock {\em Discrete and Computational Geometry}, 31(3):357--367, 2004.

\bibitem{Minimize_poly_Pablo}
P.~A. Parrilo and B.~Sturmfels.
\newblock Minimizing polynomial functions.
\newblock {\em Algorithmic and Quantitative Real Algebraic Geometry, DIMACS
  Series in Discrete Mathematics and Theoretical Computer Science}, 60:83--99,
  2003.

\bibitem{Peet.exp.stability}
M.~M. Peet.
\newblock Exponentially stable nonlinear systems have polynomial {L}yapunov
  functions on bounded regions.
\newblock {\em IEEE Trans. Automat. Control}, 54(5):979--987, 2009.

\bibitem{Peet.Antonis.converse.sos.CDC}
M.~M. Peet and A.~Papachristodoulou.
\newblock A converse sum of squares {L}yapunov result: an existence proof based
  on the {P}icard iteration.
\newblock In {\em Proceedings of the 49$^{th}$ IEEE Conference on Decision and
  Control}, 2010.

\bibitem{Peet.Antonis.converse.sos.journal}
M.~M. Peet and A.~Papachristodoulou.
\newblock A converse sum of squares {L}yapunov result with a degree bound.
\newblock {\em IEEE Trans. Automat. Control}, 2011.
\newblock To appear.

\bibitem{Scheiderer_ternary_quartic}
A.~Pfister and C.~Scheiderer.
\newblock An elementary proof of {H}ilbert's theorem on ternary quartics.
\newblock arXiv:1009.3144, 2010.

\bibitem{Undecidable_Hilbert10Survey}
B.~Poonen.
\newblock Undecidability in number theory.
\newblock {\em Notices of the Amer. Math. Soc.}, 55(3):344--350, 2008.

\bibitem{NewApproach_Hilbert_Ternary_Quatrics}
V.~Powers, B.~Reznick, C.~Scheiderer, and F.~Sottile.
\newblock A new approach to {H}ilbert's theorem on ternary quartics.
\newblock {\em Comptes Rendus Mathematique}, 339(9):617 -- 620, 2004.

\bibitem{PraP03}
S.~Prajna and A.~Papachristodoulou.
\newblock Analysis of switched and hybrid systems -- beyond piecewise quadratic
  methods.
\newblock In {\em Proceedings of the American Control Conference}, 2003.

\bibitem{sostools}
S.~Prajna, A.~Papachristodoulou, and P.~A. Parrilo.
\newblock {\em {SOSTOOLS}: Sum of squares optimization toolbox for {MATLAB}},
  2002-05.
\newblock Available from \texttt{http://www.cds.caltech.edu/sostools} and
  \texttt{http://www.mit.edu/\~{}parrilo/sostools}.

\bibitem{Pablo_Rantzer_synthesis}
S.~Prajna, P.~A. Parrilo, and A.~Rantzer.
\newblock Nonlinear control synthesis by convex optimization.
\newblock {\em IEEE Trans. Automat. Control}, 49(2):310--314, 2004.

\bibitem{protasov-jungers-blondel09}
V.~Y. Protasov, R.~M. Jungers, and V.~D. Blondel.
\newblock Joint spectral characteristics of matrices: a conic programming
  approach.
\newblock {\em SIAM Journal on Matrix Analysis and Applications},
  31(4):2146--2162, 2010.

\bibitem{Test_Geom_Convexity}
L.~Rademacher and S.~Vempala.
\newblock Testing geometric convexity.
\newblock In {\em F{STTCS} 2004: {F}oundations of Software Technology and
  Theoretical Computer Science}, volume 3328 of {\em Lecture Notes in Comput.
  Sci.}, pages 469--480. Springer, Berlin, 2004.

\bibitem{JohRan_PWQ}
M.~J.~A. Rantzer.
\newblock Computation of piecewise quadratic {L}yapunov functions for hybrid
  systems.
\newblock {\em IEEE Trans. Automat. Control}, 43(4):555--559, 1998.

\bibitem{Reznick_Unif_denominator}
B.~Reznick.
\newblock Uniform denominators in {H}ilbert's 17th problem.
\newblock {\em Math Z.}, 220(1):75--97, 1995.

\bibitem{Reznick}
B.~Reznick.
\newblock Some concrete aspects of {H}ilbert's 17th problem.
\newblock In {\em Contemporary Mathematics}, volume 253, pages 251--272.
  American Mathematical Society, 2000.

\bibitem{Reznick_Hilbert_construciton}
B.~Reznick.
\newblock On {H}ilbert's construction of positive polynomials.
\newblock arXiv:0707.2156., 2007.

\bibitem{Blenders_Reznick}
B.~Reznick.
\newblock Blenders.
\newblock arXiv:1008.4533, 2010.

\bibitem{RobinsonSOS}
R.~M. Robinson.
\newblock Some definite polynomials which are not sums of squares of real
  polynomials.
\newblock In {\em Selected questions of algebra and logic (collection dedicated
  to the memory of {A}. {I}. {M}al\cprime cev) ({R}ussian)}, pages 264--282.
  Izdat. ``Nauka'' Sibirsk. Otdel., Novosibirsk, 1973.

\bibitem{Rockafellar}
R.~T. Rockafellar.
\newblock {\em Convex Analysis}.
\newblock Princeton University Press, Princeton, New Jersey, 1970.

\bibitem{Roc_SIAM_Lagrange}
R.~T. Rockafellar.
\newblock Lagrange multipliers and optimality.
\newblock {\em SIAM Review}, 35:183--238, 1993.

\bibitem{MardavijRoozbehani2008}
M.~Roozbehani.
\newblock {\em Optimization of {L}yapunov invariants in analysis and
  implementation of safety-critical software systems}.
\newblock PhD thesis, Massachusetts Institute of Technology, 2008.

\bibitem{Roozbehani2008}
M.~Roozbehani, A.~Megretski, E.~Frazzoli, and E.~Feron.
\newblock Distributed {L}yapunov functions in analysis of graph models of
  software.
\newblock {\em Springer Lecture Notes in Computer Science}, 4981:443--456,
  2008.

\bibitem{HomogHomog}
L.~Rosier.
\newblock Homogeneous {L}yapunov function for homogeneous continuous vector
  fields.
\newblock {\em Systems and Control Letters}, 19(6):467--473, 1992.

\bibitem{RoSt60}
G.~C. Rota and W.~G. Strang.
\newblock A note on the joint spectral radius.
\newblock {\em Indag. Math.}, 22:379--381, 1960.

\bibitem{Rudin_RealComplexAnalysis}
W.~Rudin.
\newblock {\em Real and complex analysis}.
\newblock Mc{G}raw-{H}ill series in higher mathematics, 1987.
\newblock Third edition.

\bibitem{Stability_homog_poly_ODE}
N.~Samardzija.
\newblock Stability properties of autonomous homogeneous polynomial
  differential systems.
\newblock {\em Journal of Differential Equations}, 48(1):60--70, 1983.

\bibitem{Claus_3vars_sos}
C.~Scheiderer.
\newblock Sums of squares on real algebraic surfaces.
\newblock {\em Manuscr. Math.}, 119(4):395--410, 2006.

\bibitem{Claus_Hilbert17}
C.~Scheiderer.
\newblock A {P}ositivstellensatz for projective real varieties.
\newblock arXiv:1104.1772, 2011.

\bibitem{matrix_sos_Hol}
C.~W. Scherer and C.~W.~J. Hol.
\newblock Matrix sum of squares relaxations for robust semidefinite programs.
\newblock {\em Mathematical Programming}, 107:189--211, 2006.

\bibitem{Schrijver_LP_IP_Book}
A.~Schrijver.
\newblock {\em Theory of Linear and Integer Programming}.
\newblock John Wiley \& sons, 1998.

\bibitem{Seidenberg_quantifier_elim}
A.~Seidenberg.
\newblock A new decision method for elementary algebra.
\newblock {\em Ann. of Math. (2)}, 60:365--374, 1954.

\bibitem{Shor}
N.~Z. Shor.
\newblock Class of global minimum bounds of polynomial functions.
\newblock {\em Cybernetics}, 23(6):731--734, 1987.
\newblock (Russian orig.: Kibernetika, No. 6, (1987), 9--11).

\bibitem{Sontag_complexity_comparison}
E.~D. Sontag.
\newblock From linear to nonlinear: some complexity comparisons.
\newblock In {\em Proceedings of the 34$^{th}$ IEEE Conference on Decision and
  Control}, 1995.

\bibitem{sedumi}
J.~Sturm.
\newblock {\em {SeDuMi} version 1.05}, Oct. 2001.
\newblock Latest version available at \texttt{http://sedumi.ie.lehigh.edu/}.

\bibitem{Tarski_quantifier_elim}
A.~Tarski.
\newblock {\em A decision method for elementary algebra and geometry}.
\newblock University of California Press, Berkeley and Los Angeles, Calif.,
  1951.
\newblock Second edition.

\bibitem{Tedrake_LQR_Trees}
R.~Tedrake, I.~R. Manchester, M.~M. Tobenkin, and J.~W. Roberts.
\newblock {LQR}-{T}rees: {F}eedback motion planning via sums of squares
  verification.
\newblock {\em International Journal of Robotics Research}, 29:1038--1052,
  2010.

\bibitem{Tits_lec.notes}
A.~L. Tits.
\newblock Lecture notes on optimal control.
\newblock Available from \texttt{http://www.isr.umd.edu/\~{}andre/664.pdf},
  2008.

\bibitem{BlTi3}
J.~N. Tsitsiklis and V.~Blondel.
\newblock The {L}yapunov exponent and joint spectral radius of pairs of
  matrices are hard- when not impossible- to compute and to approximate.
\newblock {\em Mathematics of Control, Signals, and Systems}, 10:31--40, 1997.

\bibitem{VaB:96}
L.~Vandenberghe and S.~Boyd.
\newblock Semidefinite programming.
\newblock {\em SIAM Review}, 38(1):49--95, Mar. 1996.

\end{thebibliography}
\cleardoublepage

\phantomsection
%\addcontentsline{toc}{chapter}{Notation}
%\include{defn}

%\bibliography{\bibdir/StickyHDPHMM_Bib,\bibdir/Bibliography,\bibdir/Asilomar_Bib,\bibdir/Fusion_Bib,\bibdir/NIPS_Bib,\bibdir/Proposal_Bib}
%\bibliography{Bibliography}
% index
%\newpage
%\mbox{}
%\newpage
%\addcontentsline{toc}{chapter}{Index}
%\printindex

\end{document}